\documentclass[11pt]{article}
\usepackage{setspace}
\usepackage{cprotect}
\usepackage{amsmath,amssymb}
\usepackage{amsthm}
\usepackage[noend]{algorithmic}
\usepackage[ruled,vlined]{algorithm2e}
\usepackage{url}
\usepackage{fullpage}
\usepackage{makeidx}
\usepackage{enumerate}
\usepackage{fullpage}
\usepackage{graphicx,float,psfrag,epsfig,caption,subcaption}
\usepackage{epstopdf}
\usepackage{color}
\usepackage{hyperref}
\usepackage{xcolor}
\hypersetup{
    colorlinks,
    linkcolor={blue!80!black},
    citecolor={green!50!black},
}
\colorlet{linkequation}{blue}
\usepackage[title]{appendix}
\usepackage{cases}
\usepackage{footmisc}

\usepackage[mathscr]{euscript}

\DeclareSymbolFont{rsfs}{U}{rsfs}{m}{n}
\DeclareSymbolFontAlphabet{\mathscrsfs}{rsfs}

\numberwithin{equation}{section}

\newtheoremstyle{myexample} % name
    {\topsep}                    % Space above
    {\topsep}                    % Space below
    {\rm }                   % Body font
    {}                           % Indent amount
    {\bf }                   % Theorem head font
    {.}                          % Punctuation after theorem head
    {.5em}                       % Space after theorem head
    {}  % Theorem head spec (can be left empty, meaning normal)

\newtheoremstyle{myremark} % name
    {\topsep}                    % Space above
    {\topsep}                    % Space below
    {\rm}                        % Body font
    {}                           % Indent amount
    {\bf}                        % Theorem head font
    {.}                          % Punctuation after theorem head
    {.5em}                       % Space after theorem head
    {}  % Theorem head spec (can be left empty, meaning normal)

\newtheorem{claim}{Claim}[section]
\newtheorem{lemma}[claim]{Lemma}

\newtheorem{conjecture}[claim]{Conjecture}
\newtheorem{theorem}{Theorem}
\newtheorem{proposition}[claim]{Proposition}
\newtheorem{corollary}[claim]{Corollary}
\newtheorem{definition}[claim]{Definition}

\theoremstyle{myremark}
\newtheorem{remark}{Remark}[section]

\theoremstyle{myremark}

\theoremstyle{myexample}

\makeatletter
\newcommand\footnoteref[1]{\protected@xdef\@thefnmark{\ref{#1}}\@footnotemark}
\makeatother

\def\snr{\mathsf{snr}}

\def\<{\langle}
\def\>{\rangle}
\def\P{{\mathbb P}}

\def\toprob{\stackrel{{\mathrm p}}{\rightarrow}}
\def\toW{\stackrel{{\mathrm W}}{\rightarrow}}
\def\simprob{\stackrel{{\mathrm p}}{\simeq}}

\def\integers{{\mathbb Z}}

\def\E{{\mathbb E}} %expectation
\def\Var{{\sf{Var}}}

\def\de{{\rm d}}

\def\normal{{\sf N}}

\def\reals{\mathbb{R}}

\def\normal{{\sf N}}

\def\cB{{\mathcal{B}}}
\def\cC{{\mathcal{C}}}

\def\cG{{\mathcal{G}}}

\def\cP{{\mathcal{P}}}

\def\cT{{\mathcal{T}}}

\def\cV{{\mathcal{V}}}

\def\bzero{{\boldsymbol 0}}
\def\binfty{{\boldsymbol \infty}}
\def\bdelta{{\boldsymbol \delta}}

\def\bg{{\boldsymbol g}}
\def\bB{{\boldsymbol B}}

\def\by{{\boldsymbol y}}
\def\bu{{\boldsymbol u}}
\def\br{{\boldsymbol r}}
\def\bk{{\boldsymbol k}}

\def\bS{{\boldsymbol S}}
\def\bT{{\boldsymbol T}}
\def\bQ{{\boldsymbol Q}}

\def\bI{{\boldsymbol I}}

\def\bx{{\boldsymbol x}}
\def\bz{{\boldsymbol z}}
\def\ba{{\boldsymbol a}}
\def\bp{{\boldsymbol p}}

\def\bh{\boldsymbol{h}}

\def\bT{{\boldsymbol T}}
\def\bX{{\boldsymbol X}}

\def\bv{{\boldsymbol v}}
\def\bu{{\boldsymbol u}}
\def\bw{{\boldsymbol w}}

\def\btheta{{\boldsymbol \theta}}

\def\bkappa{{\boldsymbol \kappa}}
\def\bPhi{{\boldsymbol \Phi}}

\def\bc{{\boldsymbol c}}
\def\bb{{\boldsymbol b}}

\def\bU{{\boldsymbol U}}
\def\bV{{\boldsymbol V}}

\def\bA{{\boldsymbol A}}
\def\bB{{\boldsymbol B}}

\def\bg{{\boldsymbol g}}
\def\bbeta{{\boldsymbol \beta}}

\def\bm{{\boldsymbol m}}

\def\U{{\rm U}}
\def\U1{{\rm U}(1)}

\def\Tr{{\rm Tr}}

\def\eps{{\varepsilon}}

\def\ones{{\mathbf 1}}

\def\cE{{\mathcal E}}

\def\lambdaorc{\lambda_{\mathsf{orc}}}
\def\tauorc{\tau_{\mathsf{orc}}}

\title{Fundamental Barriers to High-Dimensional Regression\\ with Convex Penalties}
\author{Michael Celentano\thanks{Department of Statistics, University of California, Berkeley} \;\; and \;\; Andrea Montanari\thanks{Department of Electrical Engineering and Department of Statistics, Stanford University} }
\date{\today}

\begin{document}

\maketitle

\begin{abstract}
    In high-dimensional regression, we attempt to estimate a parameter vector $\bbeta_0\in\reals^p$ from $n\lesssim p$
    observations $\{(y_i,\bx_i)\}_{i\le n}$ where $\bx_i\in\reals^p$ is a vector of predictors and $y_i$ is a response variable.
    A well-established approach uses convex regularizers to promote specific structures (e.g. sparsity) of the estimate $\widehat{\bbeta}$,
    while allowing for practical algorithms.
    Theoretical analysis implies that convex penalization schemes have nearly optimal estimation properties in certain settings. 
    However, in general the gaps between statistically optimal estimation (with unbounded computational resources) and convex methods
    are poorly understood.

    We show that when the statistican has very simple structural information about the distribution of the entries of $\bbeta_0$, a large gap frequently exists between the best performance achieved by \emph{any convex regularizer} satisfying a mild technical condition and either \emph{(i)}~the optimal statistical error or \emph{(ii)}~the statistical error achieved by optimal approximate message passing algorithms.
    Remarkably, a gap occurs at high enough signal-to-noise ratio if and only if the distribution of the coordinates of
        $\bbeta_0$ is not log-concave.
    These conclusions follow from an analysis of standard Gaussian designs.
    Our lower bounds are expected to be generally tight, and we prove tightness under certain conditions. 
\end{abstract}

\tableofcontents

\section{Introduction}

Consider the classical linear regression  model
\begin{equation}\label{linear-model}
\by = \bX \bbeta_0 + \bw,
\end{equation}
where $\bX \in \reals^{n \times p}$.
The statistician observes $\by$ and $\bX$ but not $\bbeta_0$ or $\bw$, and she seeks to estimate $\bbeta_0$. 
We assume she approximately knows the $\ell_2$-norm of the noise $\bw$ and the empirical distribution of the coordinates of $\bbeta_0$ in 
senses we will make precise below.

We are interested in the high-dimensional regime in which $p$ is comparable to $n$, and both are large. In this regime, computational considerations
are crucial: only estimators which can be implemented by polynomial-time algorithms are relevant to statistical practice. 

This paper develops precise lower bounds that characterize a broad class of estimators which are attractive in large part for their computational tractability.
These are penalized least-squares estimators of the form:
\begin{equation}\label{linear-cvx-estimator}
\widehat \bbeta_{\mathsf{cvx}} \in \arg\min_{\bbeta} \left\{\frac1n\|\by - \bX \bbeta\|^2 + \rho(\bbeta)\right\},
\end{equation}
where 
$\rho: \mathbb{R}^p \rightarrow \mathbb{R}\cup \{\infty\}$ 
is a lower semi-continuous (lsc), proper, convex function. 
The penalty $\rho$ is selected to incorporate prior knowledge on the structure of $\bbeta_0$ into the estimation procedure.
Convexity typically yields an estimator which is efficiently computable. 
Concretely, we address the following question:
\begin{itemize}
\item[] \emph{How well can we hope estimator \eqref{linear-cvx-estimator} to perform in the high-dimensional regime
by optimally designing $\rho$? 
How does this performance compare to other polynomial-time algorithms and to conjectured computational lower bounds?}
\end{itemize}

The design of optimal penalties or loss functions was considered only when the distribution of the noise or --in the case of Bayesian models-- the prior had log-concave density with respect to Lebesgue measure \cite{Bean2013,advani2016statistical}.
Log-concavity excludes important structural assumptions like sparsity, and,
as we will show, is exactly the condition which leads to gaps between convex procedures and important computational or information-theoretic benchmarks. 
Thus, the case of non-log-concave priors is both practically important and algorithmically more subtle.

We will illustrate our conclusions with two small simulation studies.

\subsection{A surprise: Exact recovery of a vector from 3-point prior}\label{sec:noiseless-3pt-recovery}

Consider the case of noiseless linear measurements, namely $\bw=\bzero$ in  Eq.~\eqref{linear-model}. We assume that the
empirical distribution of $\bbeta_0$ is known, and let $S$ be the set of vectors with that empirical distribution (i.e., vectors obtained by permuting the entries of $\bbeta_0$).
If we had unbounded computational resources, we would attempt reconstruction by finding  $\bbeta\in S$ such that $\by = \bX\bbeta$. 
If only one such vector exists, then 
we are sure it coincides $\bbeta_0$.
Otherwise, exact recovery is impossible.

What is the best we can achieve by convex procedures and practical (polynomial-time) algorithms? Most researchers with a knowledge of compressed sensing or high-dimensional
statistics would consider the following convex relaxation 
\begin{align}
\begin{split}\label{convex-feasibility-prob}
    \text{find} \quad & \bbeta \in \mathsf{conv}(S)\, ,\\
    \text{subject to}\quad & \by = \bX \bbeta.
\end{split}
\end{align}
This is the tightest possible relaxation of the combinatorial constraint $\bbeta\in S$.
It can be written in the form \eqref{linear-cvx-estimator}, where, setting $C:=\mathsf{conv}(S)$, the penalty is 
$\rho(\bbeta) = \mathbb{I}_{C}(\bbeta)$, and $\mathbb{I}_C(\bbeta):= 0$ if $\bbeta\in C$, $\mathbb{I}_C(\bbeta) := \infty$  otherwise.

Notice that the approach \eqref{convex-feasibility-prob} is at least as effective as ---for instance--- basis pursuit
\cite{BP95}, which minimizes  $\|\bbeta \|_1$ subject to $\by = \bX \bbeta$.
To see this, notice that (for a generic $\bX$) the approach \eqref{convex-feasibility-prob} fails if and only if there exists $\bbeta_*$
in the interior of $\mathsf{conv}(S)$ such that $\by = \bX\bbeta_*$. Since $S\subseteq\{\bbeta:\, \|\bbeta\|_1\le \|\bbeta_0\|_1\}$,
this implies $\|\bbeta_*\|_1<\|\bbeta\|_1$ and therefore basis pursuit fails as well.

Is replacing the combinatorial constraint $S$ with its tightest convex relaxation  $C\equiv\mathsf{conv}(S)$ the best we can do?
We report the results of a simulation study, with $p = 2000$, $n = 0.4\cdot p = 800$. We generate a parameter vector $\bbeta_0$ in which 
$0.75\cdot p = 1500$ coordinates are equal to $0$, $0.15p = 300$
coordinates are equal to $0.2/\sqrt{p}$, and $0.1\cdot p = 200$ coordinates are equal to $1/\sqrt{p}$.
In particular, the empirical distribution of the coordinates of $\sqrt{p}\bbeta_0$ is $\pi:= .75 \cdot \delta_0 + .15 \cdot \delta_{0.2} + .1\cdot \delta_1$, which is far from being log-concave.
We generate Gaussian features $(X_{ij})_{i\leq n, j \leq p} \stackrel{\mathrm{iid}}\sim \mathsf{N}(0,1)$ and response $\by$ according to linear model \eqref{linear-model} with $\bw = \bzero$.

We attempt to recover $\bbeta_0$ using two different methods: $(i)$ an accelerated proximal gradient method to solve \eqref{convex-feasibility-prob}, and $(ii)$ a Bayes-optimal approximate message passing (Bayes-AMP) algorithm at prior $\pi$
(see Section \ref{sec:AlgoLB}).
The former is a convex optimization method, while
the latter is an efficient but non-convex procedure.
We generate 500 independent realizations of the data, and for each realization, we attempt to recover $\bbeta_0$ by each method.
In Table \ref{tbl:3pt-full-recovery}, we report the percentage of simulations in which full recovery was achieved by each method.
For 498 of the 500 realizations of the data, Bayes-AMP achieved full recovery; that is, $\widehat \bbeta = \bbeta_0$
up to machine precision. 
In contrast, the convex procedure never fully recovered $\bbeta_0$.
We also report the median, minimal, and maximal value of the relative estimation error
$\|\widehat \bbeta - \bbeta_0\|^2 /\|\bbeta_0\|_2^2$.
The relative errors displayed indicate that projection denoising never comes close to achieving exact recovery of the true parameter vector. 
% latex table generated in R 3.5.2 by xtable 1.8-3 package
% Sun Apr 14 19:22:21 2019
\begin{table}[ht]
\centering
\begin{tabular}{rrr}
  \hline
 & Projection Denoising & Bayes-AMP \\ 
  \hline
\% Full Recovery & 0.00 & 99.60 \\ 
  Median Est. Error & 0.14 & 0.00 \\ 
  Min Est. Error & 0.06 & 0.00 \\ 
  Max Est. Error & 0.22 & 0.03 \\ 
  Theory Lower Bounds & 0.06 & 0.00 \\ 
   \hline
\end{tabular}
\caption{Percentage of simulations in which full recovery is achieved by convex projection (estimator \eqref{convex-feasibility-prob}) and by Bayes-AMP, as well as median, minimum, and maximum value of $\|\widehat \bbeta - \bbeta_0\|^2 /\|\bbeta_0\|_2^2$ observed over 500 independent realization of the data. Full recovery for Bayes-AMP means $\widehat \bbeta = \bbeta_0$ up to machine precision. ``Theory lower bounds'' are high-probability asymptotic lower bounds on $\|\widehat \bbeta - \bbeta_0\|^2 /\|\bbeta_0\|_2^2$ for any convex procedure (left) and for Bayes-AMP (right).}
\label{tbl:3pt-full-recovery}
\end{table}

This study supports the perhaps surprising conclusion that estimator \eqref{convex-feasibility-prob} is sub-optimal among polynomial-time estimators 
for the task of noiseless recovery of a parameter vector whose coordinates have known empirical distribution $\pi$.
In fact, this paper rigorously establishes a substantially more powerful conclusions, namely, that \emph{(i)} \emph{any} convex estimator of the form \eqref{linear-cvx-estimator} will 
with high-probability not only fail to recover the true signal, but also have estimation error lower-bounded by a constant (we refer to Section \ref{sec:lower-bounds-and-benchmarks} for precise asymptotic statements).
This lower bound is reported in Table \ref{tbl:3pt-full-recovery}.
Thus, in this case full recovery is possible both information theoretically and in polynomial-time but not via convex procedures. 
As we will see, this gap is driven by the non log-concavity of $\pi$.
In fact, the convex estimator \eqref{convex-feasibility-prob} is suboptimal with respect to $\ell_2$-estimation error even among convex procedures.

In contrast to convex procedures, Bayes-AMP achieves vanishingly small reconstruction error in the current setting with probability approaching 1. 
Let us mention that for noiseless or nearly noiseless observations, an alternative polynomial-time algorithm that achieves exact recovery for
discrete priors was recently developed in \cite{gamarnik2017high}. 
However, the approach of \cite{gamarnik2017high} does not apply when the signal-to-noise ratio is of order one, which is the main
focus of the present paper.

\begin{figure}[t!]
\begin{tabular}{cc}
\centerline{\phantom{A}\hspace{-1cm}\includegraphics[width=0.96\columnwidth]{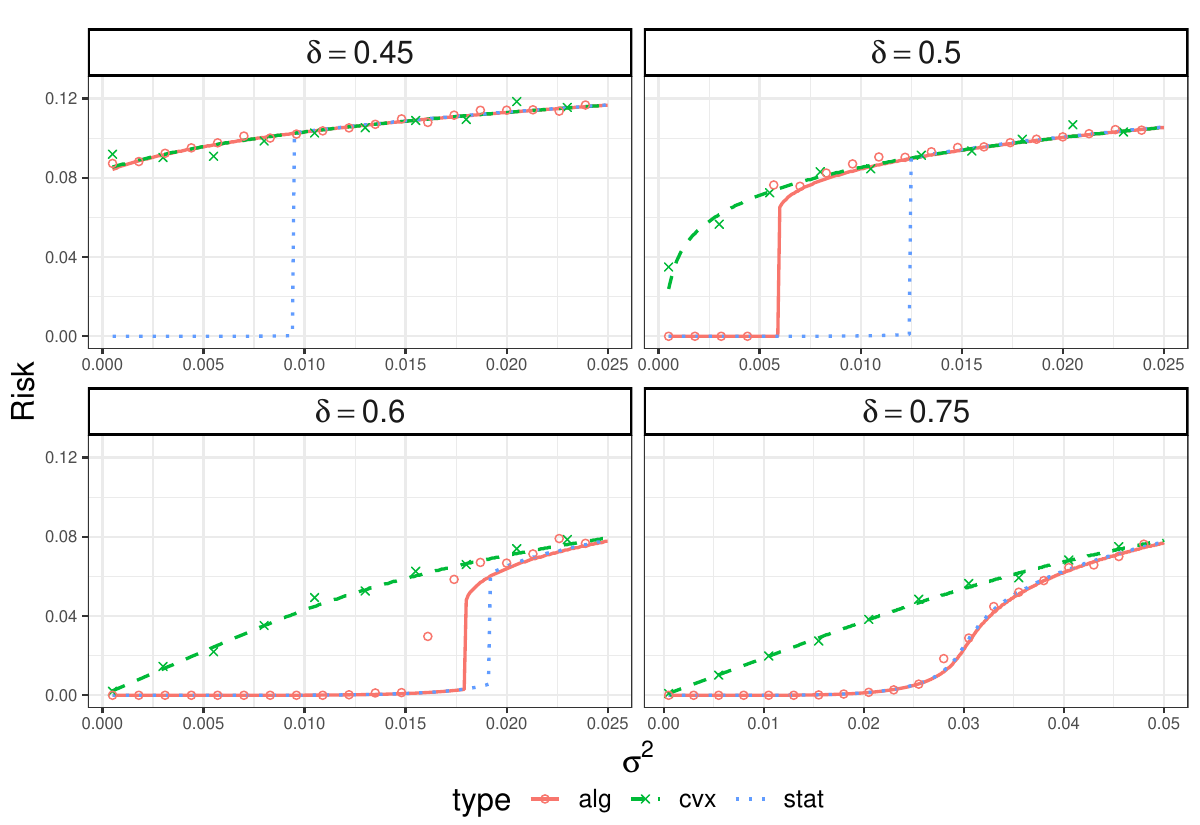}}
\end{tabular}
\caption{Median squared error of estimation in  high-dimensional regression. Symbols refer to simulations for two different polynomial-time algorithms. 
Crosses: M-estimator \eqref{linear-cvx-estimator} for a certain optimized penalty $\rho(\bbeta)$. Circles: Bayes-Approximate Message Passing.
Dashed and solid lines correspond to our theoretical predictions for the asymptotic behavior of these algorithms. Dotted line corresponds to the asymptotics of the Bayes error.
See main text for further details.} 
\label{fig:Simulations}
\end{figure}

\subsection{An example: Noisy estimation of a sparse vector}\label{sec:noisy-sparse-estimation}

Gaps between the performance of convex procedures and optimal polynomial-time algorithm persist in the presence of noise. 
They may also occur in regimes in which all known polynomial-time algorithms are suboptimal information theoretically.
To illustrate these claims, in Figure \ref{fig:Simulations} we report the results of a simulation study for $p = 2000$, $n= 2000\delta$.
We 
generated Gaussian features $(X_{ij})_{i\le n, j\le p}\stackrel{\mathrm{iid}}\sim\normal(0,1)$, noise $\bw \sim \mathsf{Unif}(\sqrt{n}\sigma S^{n-1})$ the uniform distribution on the sphere of radius $\sqrt{n}\sigma$ in $\reals^n$, and $\bbeta_0$ such that $0.1 p$ coefficients are $1/\sqrt{p}$, $0.1 p$ coefficients are $-1/\sqrt{p}$, and $0.8 p$ coefficients are $0$.
Observe that the empirical distribution of the coordinates of $\sqrt{p}\bbeta_0$ is $\pi := (\eps/2)\delta_{-1} + (1-\eps)\delta_0 + (\eps/2)\delta_1$ with $\eps = 0.2$, which is of course non log-concave.
We generated response variables $\by$ according to the linear model \eqref{linear-model} and attempted to estimate the parameter vector
$\bbeta_0$ using two different methods: 
$(i)$ a convex M-estimator of the form \eqref{linear-cvx-estimator}, with a penalty $\rho(\bbeta)$ which was carefully optimized for the prior $\pi$,
$(ii)$ an approximate message passing (AMP) algorithm called Bayes AMP (which is optimal among AMP algorithms
for the prior $\pi$, but not always Bayes optimal).

The choice of Bayes-AMP as a reference algorithm is not arbitrary. It is in fact justified by the following conjecture, which is motivated
by ideas in statistical physics and has appeared informally several times in the literature. 
In the context of statistical estimation problems arising in information theory, this
conjecture appears in  Chapters 15 and 21 of \cite{mezard2009information}.
For tutorials discussing it in the context of statistical estimation, see 
Sections III E and IV B of \cite{Zdeborov2015StatisticalPO}; Sections 4.2 and 4.3 of \cite{bandeira2018notes}.
For recent contributions mentioning this idea or analogous ones in the context of matrix estimation,
see \cite{barbier2017,lelarge2019fundamental,banks2021local}.
\begin{conjecture}\label{conj:Gap}
Consider the problem of estimating $\bbeta_0$ in the linear model  \eqref{linear-model} with standard Gaussian features
$(X_{ij})_{i\le n, j\le p}\stackrel{\mathrm{iid}}\sim\normal(0,1)$, noise $(w_i)_{i\le n}\stackrel{\mathrm{iid}}\sim\normal(0,\sigma^2)$ with $\sigma > 0$, and coefficients such that $(\sqrt{p}\beta_{0,i})_{i\le p}\stackrel{\mathrm{iid}}\sim \pi$, with $\pi$ a distribution
 with finite second moment. Assume $\pi$ is known to the statistician. 
Then Bayes-AMP achieves the minimum mean square estimation error among all polynomial-time algorithms in the limit $n,p\to\infty$ with $n/p \rightarrow \delta$ fixed.
\end{conjecture}

We plot the median error under square loss achieved by these two  estimators, as a function of the noise level, for four values of $\delta =n/p$. We also plot:
$(i')$~the  asymptotic Bayes risk, as predicted by \cite{thrampoulidis2018precise,barbier2017,barbier2018optimal} 
(see Section \ref{subsection-comparing-stat-and-cvx}); $(ii')$~the predicted performance of Bayes-AMP (see Section \ref{sec:AlgoLB}); $(iii')$~our lower 
bound on the risk of convex M-estimators (cf.\ Theorem \ref{thm-cvx-lower-bound}).
Three qualitatively different behaviors can be discerned:
\begin{itemize}
\item For $\delta=0.45$, optimal convex M-estimators matches the performance of Bayes-AMP, and they are both substantially 
suboptimal with respect to Bayes estimation.
\item For $\delta\in \{0.5,0.6\}$, optimal convex M-estimation is suboptimal compared to Bayes AMP, and --in turn-- they are both inferior to Bayes estimation.
\item For $\delta = 0.75$, Bayes-AMP is Bayes optimal for all noise levels $\sigma$, and both Bayes-AMP and Bayes estimation are superior to optimal convex M-estimation.
\end{itemize}
We further note that our lower bound for convex M-estimation is nearly matched by the error achieved by the specific regularizer used in simulations. 
Our results rigorously establish the existence of these three qualitative behaviors,
and, as we will see, are driven by the non log-concavity of $\pi$ convolved with various levels of Gaussian noise.
Moreover, our convex lower bounds appear to be tight and are consistent with the conjectured computational lower bound achieved by Bayes AMP.

\subsection{Summary of contributions}

The present  paper establishes  the scenario illustrated by Figure \ref{fig:Simulations} and Table \ref{tbl:3pt-full-recovery} in a precise way.
Our results hold for the case of standard Gaussian  features. Since convex regularizers are thought to perform well in this setting,
establishing lower bounds in this case is particularly informative. 
Namely:
\begin{enumerate}
\item We prove that, for any given convex penalty, a solution to a certain system of equations provides a
  lower bound on the asymptotic estimation error achieved by this penalty. Further, this lower bound is tight 
--and hence precisely characterizes the asymptotic mean square error-- if the penalty $\rho$ is strongly convex.

\item  We prove the lower bound on the error of any convex M-estimator
  plotted in Figure \ref{fig:Simulations} and reported in Table \ref{tbl:3pt-full-recovery}.
  This lower bound applies to both log-concave and non log-concave priors for $\bbeta_0$.
\item  We  prove that the three behaviors illustrated by Figure \ref{fig:Simulations} are the only possible and that they indeed occur. Namely,
the Bayes error is smaller than the Bayes-AMP error, and sometimes strictly smaller, and the Bayes-AMP error is always smaller than the
convex M-estimation error, and sometimes strictly smaller.
\item The occurrence of these three phases is determined by the log-concavity or not of the prior convolved with Gaussian noise at a certain variance which we specify. 
Importantly, non-trivial phase diagrams occur exactly when the prior is non log-concave.
In particular, we provide a nearly complete characterization of when convex M-estimation achieves Bayes-optimal error, and when it does not.
In order get a quantitative understanding on the statistical-convex gap, we characterize it in  the high and low 
signal-to-noise ratio regimes.
\item Finally, our general lower bound holds under a certain technical condition on the regularizers $\rho$, which we call 
\emph{$\delta$-bounded width.}
We illustrate our results by considering a number of convex penalties introduced in the literature, including separable penalties,
convex constraints, SLOPE,  and OWL norms. We show that, in each of these cases, the bounded width condition holds.
\end{enumerate}
Our work is consistent with
Conjecture \ref{conj:Gap} in showing that no convex M-estimator of the form \eqref{linear-model} can surpass the
postulated lower bound on polynomial-time algorithms.  Further, we believe that the characterization
mentioned at the first point  holds beyond strongly convex penalties: since we are mostly interested in the lower bound,
we do not attempt to prove such general result.

The asymptotic characterization of Bayes-AMP is completely explicit and can be easily evaluated, 
hence it can provide concrete guidance in specific problems.
We expect that universality arguments \cite{korada2011applications,bayati2015universality,oymak2018universality}
can be used to show that the same asymptotics hold for iid non-Gaussian features. 

Finally, let us emphasize that we do not advocate the dismissal of
convex penalization method in favor of other approaches, such as message passing algorithms. Convex algorithms
present strong robustness properties that are practically important and not captured by our setting. At the same time,
our work points at  directions for improving their statistical properties. For instance,
Section \ref{sec-beyond-square-error} shows that
a suitable post-processing of a convex M-estimator can nearly bridge the gap to information-theoretically optimal
performance in a large sample size regime (namely for $n/p$ large but of order one). 

\subsection{Related literature}

By far the best-studied estimator of the form \eqref{linear-cvx-estimator} is the Lasso \cite{Tibs96,BP95}, which corresponds to the penalty 
$\rho(\bbeta) = \lambda\|\bbeta\|_1$. An impressive body of theoretical work supports the conclusion that the Lasso
achieves nearly optimal performances when we know that the true vector $\bbeta_0$ is sparse \cite{CandesTao,Dantzig,BickelEtAl,BuhlmannLASSO}. 
Our main conclusion is that, if we attempt to exploit richer
information about the empirical distribution of the coefficients $(\beta_{0,j})_{j\le p}$,
then not only the Lasso, but also any convex estimator \eqref{linear-cvx-estimator}
is substantially suboptimal as compared to the Bayes error or other polynomial-time algorithms. On the other hand,
convex estimators are optimal if the coefficients distribution is log-concave.

Our work builds on a series of recent theoretical advances. 
First, we make use of the sharp analysis of AMP algorithms using state evolution which was developed in \cite{bolthausen2014iterative,BM-MPCS-2011,javanmard2013state}.
In particular, the recent paper \cite{Berthier2017StateFunctions} proves that state evolution holds for certain classes of non-separable
nonlinearities. This is particularly relevant for the present setting, since we are interested in non-separable penalties $\rho(\bbeta)$.

The connection between M-estimation and AMP algorithms was first developed in \cite{DMM09} and subsequently used in 
\cite{BayatiMontanariLASSO} to characterize the asymptotic mean square error of the Lasso for standard Gaussian designs.
The same approach was subsequently used in the context of robust regression in \cite{Donoho2016HighPassing}.
AMP algorithms were developed and analyzed for a number of statistical estimation problems, 
including generalized linear models~\cite{rangan2011generalized}, phase retrieval~\cite{schniter2015compressive,ma2018optimization}, 
and logistic regression~\cite{sur2018modern}.

A different approach to sharp asymptotics in high-dimensional estimation problems makes use of Gaussian comparison inequalities.
This line of work was pioneered by Stojnic~\cite{stojnic2013framework} and
then developed by a number of authors in the context of regularized regression~\cite{thrampoulidis2015regularized},
M-estimation ~\cite{thrampoulidis2018precise}, generalized compressed sensing~\cite{Chandrasekaran2012TheProblems}, binary
compressed sensing~\cite{stojnic2010recovery}, the Lasso \cite{miolane2018distribution}, and so on. 

An independent approach to high-dimensional estimation based on leave-one-out techniques was developed by El Karoui
in the context of ridge-regularized robust regression~\cite{karoui2013asymptotic,el2018impact}.
Closely related to the present work is the paper \cite{Bean2013}, which considers convex M-estimation,
and constructs separable convex losses that match the Bayes optimal error in settings in which the noise distribution is log-concave and hence the gap between the two vanishes.
Our work extends this analysis to cases in which log-concavity assumptions are violated so that the Bayes error cannot be achieved.
In this paper, we focus on the role of regularization rather than the loss function,
though we suspect similar analyses should be possible for general convex losses. 
Optimal convex M-estimators were also studied
---using tools from statistical physics--- in \cite{advani2016statistical}.

As mentioned above, we compare the performance of convex M-estimators to the optimal Bayes error and conjectured computational lower bounds. 
The asymptotic value of the Bayes error for random designs was recently determined in
\cite{barbier2016CS,reeves2016replica}.  Generalizations of this
result were also obtained in~\cite{barbier2018optimal} for other regression problems. 

Finally, the gap between polynomial-time algorithms and statistically optimal estimators has been studied from other points of view as well.
It was noted early on that constrained least square methods (which exhaustively search over supports of given size)
perform accurate regression under weaker conditions than required by the Lasso \cite{wainwright2009information}.
Strong lower bounds for compressed sensing reconstruction were proved in \cite{ba2010lower}
using communication complexity ideas.
Gamarnik and Zadik \cite{gamarnik2017high} study the case of binary coefficients, namely $\bbeta_0\in \{0,1\}^p$,
and standard Gaussian designs $\bX$. They prove existence of a gap between the maximum likelihood estimator (which requires exhaustive search over binary vectors)
and the Lasso. They argue that the failure of polynomial-time algorithms originates in a certain `overlap gap property' which they also characterize.
Further implications of this point of view are investigated in \cite{gamarnik2017sparse}.
After a preprint of this paper appeared online, further work studied the design of optimal penalties and loss functions in classification models and analyzed the achievability of Bayes optimal performance \cite{mignacco2020,taheri2020,taheri2021}.

\subsection{Notations}\label{section-notations}

The Euclidean norm of a vector $\bx \in \reals^p$ is denoted by $\|\bx\|:=\|\bx\|_2$. 
The operator and nuclear norms of a matrix $\bX \in \reals^{n \times p}$ are denoted by $\|\bX\|_{\mathsf{op}}$ and $\|\bX\|_{\mathsf{nuc}}$, respectively.
We denote by $S_+^k$ the set of $k\times k$ positive semi-definite matrices.

Subscripts under the expectation or probability sign, e.g.\ $\E_{\bbeta_0,\bz}$ and $\P_{\bbeta_0,\bz}$ indicate the variables which
are random.
We denote by $\cP_k(\reals)$ the collection of Borel probability measures on $\reals$ with finite $k$-th moment.
For a distribution $\pi \in \cP_k(\reals)$, we will denote by $s_\ell(\pi)$ the $\ell$-th moment of $\pi$. 
We will often extend a distribution $\pi \in \cP_k(\reals)$ to a distribution on $\reals^p$ by taking $\bbeta_0 =(\beta_{0j})_{j\le p}\in \reals^p$ with coordinates 
such that  $(\sqrt{p}\beta_{0j})_{j\le p} \stackrel{\mathrm{iid}}\sim \pi$. We will write this succinctly as $\beta_{0j} \stackrel{\mathrm{iid}}\sim \pi/\sqrt{p}$.
Under this normalization, $\E_{\bbeta_0}[\|\bbeta_0\|^2] = s_2(\pi)$ does not depend on $p$. 
We reserve $z$ and $\bz$ to denote Gaussian random variables and vectors, respectively. We will always take $z \sim \mathsf{N}(0,1)$ and $\bz \sim \mathsf{N}(0,\bI_p/ p)$.
Convolution of probability measures will be denoted by $*$.

We define the Wasserstein distance between two probability measures $\pi,\pi' \in \cP_2(\reals)$ by
\begin{equation}\label{eqdef-wasserstein}
    d_{\mathrm{W}}(\pi,\pi') = \inf_{X,X'} \left(\E_{X,X'}\left[(X-X')^2\right]\right)^{1/2},
\end{equation}
where the infimum is taken over joint distributions of random variables $(X,X')$ with marginal distributions $X \sim \pi$ and $X' \sim \pi'$.
It is well known that this defines a metric on $\cP_2(\reals)$ \cite{Santambrogio2015OptimalMathematicians}.
Convergence in Wasserstein metric will be denoted $\toW$, and we use $\stackrel{\mathrm{p}}\rightarrow$,  $\stackrel{\mathrm{as}}\rightarrow$, $\stackrel{\mathrm{d}}\rightarrow$ for other standard notions of convergence.
For any sequence of real-valued random variables $\{X_p\}$, not necessarily defined on the same probability space, we denote
$$
\liminf^{\mathrm p}_{p \rightarrow \infty} X_p = \sup\Big\{t \in \reals \Big| \lim_{p\rightarrow\infty}\P\left(X_p < t \right) = 0 \Big\},
$$
and $\limsup\limits^{\mathrm p}_{p \rightarrow \infty}X_p = - \liminf\limits^{\mathrm p}_{p \rightarrow \infty} (-X_p)$.
For sequences $\{X_p\}$ and $\{Y_p\}$ of real-valued random variables such that, for each $p$, $X_p$ and $Y_p$ are defined on the same probability space,
we use the notation $X_p \simprob Y_p$ to denote $|X_p - Y_p| \toprob 0$.

We adopt the convention that when the minimizing set in \eqref{linear-cvx-estimator} is empty, $\widehat \bbeta_{\mathsf{cvx}} = \binfty$ and $\|\binfty -\bx\| = \infty$ for any $\bx \in \reals^p$.
Thus, the estimation error is infinite when no minimizer exists.

Finally, a collection of functions $\{\varphi:(\reals^p)^\ell \rightarrow \reals^m\}$, where $p$ and $m$ but not $\ell$ may vary, is said to be \emph{uniformly pseudo-Lipschitz of order $k$} if for all $\varphi$  and $\bx_i,\by_i \in \reals^p,\,i = 1,\ldots,\ell$, we have
\begin{equation}\label{def-uniformly-pseudo-lipschitz}
\|\varphi(\bx_1,\ldots,\bx_\ell) - \varphi(\by_1,\ldots,\by_\ell)\| \leq C \left(1 + \sum_{i=1}^\ell  \|\bx_i\|^{k-1} + \sum_{i=1}^\ell \|\by_i\|^{k-1} \right)\sum_{i=1}^\ell \|\bx_i - \by_i\|,
\end{equation}
for some $C$ which does not depend on $p,m$.

\section{The convex lower bound, the risk of Bayes-AMP, and the Bayes risk}
\label{sec:lower-bounds-and-benchmarks}

In this section, we present a rigorous lower bound on the $\ell_2$ estimation error of convex M-estimators of the form \eqref{linear-cvx-estimator} under proportional asymptotics, Gaussian noise, and structural assumptions on the unknown parameter $\bbeta_0$.
A primary focus will be comparing the convex lower bound to two important benchmarks which have been studied elsewhere \cite{reeves2016replica,barbier2016CS,barbier2018optimal}: 
\begin{itemize}

    \item \textbf{Risk of Bayes-AMP:} The $\ell_2$-estimation error of a certain message passing algorithm conjectured to be optimal among all polynomial-time algorithms (see Conjecture \ref{conj-algorithmic-threshold}).

    \item \textbf{Bayes risk:} The optimal risk over all (possibly computationally unbounded) estimators under a certain Bayesian model for the signal.

\end{itemize}
Before defining these quantities precisely, we may summarize the comparison we will establish by
\begin{gather*}
    \parbox{7em}{\centering Convex\\Lower Bound}
    \geq 
    \parbox{7em}{\centering Risk of\\Bayes AMP}
    \geq 
    \parbox{7em}{\centering Bayes Risk.}
\end{gather*}
While the second inequality holds by the statistical optimality of the Bayes risk, the first is non-trivial.
Previous work established exactly when the second inequality is strict \cite{barbier2018optimal}.
We will likewise specify exactly when the first inequality is strict.
Previous work has only considered optimal convex estimation in regimes in which strict inequality does not occur \cite{Bean2013,advani2016statistical}.

Precisely, we study these three quantities under a certain high-dimensional proportional asymptotics for model \eqref{linear-model}.
\begin{description}
    \item[High Dimensional Asymptotics (HDA)] \hfill\\
    The design matrix satisfies the following assumptions.
    \begin{itemize}
        \item The sample size and number of parameters $n,p \rightarrow \infty$ satisfy $n/p \rightarrow \delta \in (0,\infty)$, a fixed asymptotic aspect ratio.
        \item The matrix $\bX$ has entries $X_{ij}\stackrel{\mathrm{iid}}\sim\mathsf{N}(0,1)$.
    \end{itemize}
\end{description}
\noindent Further, we introduce two sets of assumptions on the unknown parameter $\bbeta_0$ and the 
the noise $\bw$.
\begin{description}
    \item[Deterministic Signal and Noise (DSN)]\hfill\\
    For each $p$ and $n$, we have deterministic parameter vector $\bbeta_0 \in \reals^p$ and noise vector $\bw \in \reals^n$. For some $\pi \in \cP_2(\reals)$ and $\sigma^2 \geq 0$, these satisfy
    \begin{equation}\label{DSN-assumption}
        \widehat \pi_{\bbeta_0} := \frac1p \sum_{j=1}^p \delta_{\sqrt p \beta_{0j}} \stackrel{\mathrm{W}}\rightarrow \pi \qquad \text{and} \qquad \frac1n\|\bw\|^2\rightarrow \sigma^2.
    \end{equation}
    \item[Random Signal and Noise (RSN) Assumption]\hfill\\
    For each $p$ and $n$, we have random parameter vector $\bbeta_0 \in \reals^p$ and noise vector $\bw \in \reals^n$ satisfying
\begin{equation}
    \beta_{0j} \stackrel{\mathrm{iid}}\sim \pi/\sqrt p,\quad \bw \sim \mathsf{N}(0,\sigma^2\bI_n),
\end{equation}
where $\pi \in \cP_2(\reals)$ and $\sigma^2 \geq 0$ do not depend on $p$.
\end{description}
When necessary to indicate where $\bbeta_0$ $\bw$ fall in the sequence of realizations with growing
dimensions, we include indices as $\bbeta_0(p)$ and $\bw(p)$.

Under the DSN assumption,
we will establish a convex lower bound for \emph{symmetric} convex penalties; that is, penalties which are invariant to permutation of the coordinates of their argument.
The DSN assumption specifies the limiting empirical distribution of the coordinates of $\bbeta_0$,
which captures structural information, like sparsity, which is permutation invariant.
Nevertheless, the lower bound applies also to models in which additional information about the order in which the coordinates appear is available: for example, the statistician may know that the coordinates are monotone, have sparse first differences, or satisfy other smoothness conditions.
The lower bound ---which applies only to symmetric convex penalties---
describes a limitation of convex procedures which fail to exploit such information.
%, should it be available.

In contrast, under the RSN assumption, 
we will establish a convex lower bound for \emph{arbitrary} convex penalties.
Here, the statistician can exploit all available information.
But because she has no prior knowledge about the ordering of the coordinates of $\bbeta_0$, she cannot benefit from asymmetric procedures.

The two sets of assumptions are complementary, differing in how they impose symmetry on the problem: either through the method or through the model.
It turns out that the lower bound on the estimation error under the two sets assumptions is the same.

We only make comparisons to information theoretic lower bounds ---that is, the Bayes risk--- under the RSN assumption.
Indeed, the RSN assumption is needed for the Bayes risk to be meaningful.

\subsection{The convex lower bound}\label{sec:convex-lower-bound}
The convex lower bound is defined via a comparison of the linear model \eqref{linear-model} to a simpler  Gaussian sequence model.
In the sequence model, we observe
\begin{equation}\label{sequence-model}
\by_{\mathsf{seq}} = \bbeta_0 + \tau \bz,
\end{equation}
where $\beta_{0j} \stackrel{\mathsf{iid}} \sim \pi/\sqrt{p},\;\bz \sim \mathsf{N}(\bzero,\bI_p/p)$ independent, and $\tau^2 \geq 0$.
Analogously to \eqref{linear-cvx-estimator}, we consider convex M-estimators in the sequence model, also known as \emph{proximal operators}:
\begin{equation}\label{sequence-cvx-estimator}
\widehat \bbeta_{\mathsf{seq}} := \arg\min_{\bbeta} \frac12\|\by_{\mathsf{seq}} - \bbeta\|^2 + \lambda\rho (\bbeta) =: \mathsf{prox}[\lambda\rho](\by_{\mathsf{seq}}).
\end{equation}
By strong convexity, when $\rho$ is lower semi-continuous and proper, the minimizer exists and is unique \cite{Parikh2013ProximalAlgorithms}.

A large body of work exactly characterizes the estimation error of the estimators \eqref{linear-cvx-estimator} in the linear model in terms of the behavior of the estimators \eqref{sequence-cvx-estimator} in the sequence model \cite{BayatiMontanariLASSO,Donoho2016HighPassing,ElKaroui2013OnPredictors.,karoui2013asymptotic,thrampoulidis2015regularized,thrampoulidis2018precise}.
A typical characterization takes the following form.
For a sequence of penalties $\{\rho_p\}$, let $(\tau,\lambda)$ solve
\begin{subequations}\label{eq:body-fixed-pt}
\begin{gather}
    \delta \tau^{2} - \sigma^2 = \lim_{p\rightarrow \infty}\E_{\bz}[\|\mathsf{prox}[\lambda\rho_p](\bbeta_0 + \tau \bz) - \bbeta_0\|^2],\label{eq:body-fixed-pt-1}\\
    2\lambda\left(1 - \frac1{\delta\tau} \lim_{p \rightarrow \infty}\E_{\bz}[\langle \bz , \mathsf{prox}[\lambda \rho_p](\bbeta_0 + \tau \bz)\rangle] \right) = 1.\label{eq:body-fixed-pt-2}
\end{gather}
\end{subequations}
Then under the HDA and DSN assumption,
\begin{equation}
    \|\widehat \bbeta_{\mathsf{cvx}} - \bbeta_0\|^2 
        \stackrel{\mathrm{p}}\rightarrow 
        \delta \tau^2 - \sigma^2 
        =
        \E_{\bz}[\|\mathsf{prox}[\lambda\rho_p](\bbeta_0 + \tau \bz) - \bbeta_0\|^2]. 
\end{equation}
In words, the $\ell_2$ estimation error in the linear model asymptotically agrees with the $\ell_2$  risk in the sequence model at noise variance $\tau^2$ and regularization $\lambda$.
Substantial effort is required to make this rigorous, and many technical assumptions are required.
For example, some work requires strong-convexity assumptions on the cost function \eqref{linear-cvx-estimator} \cite{Donoho2016HighPassing,karoui2013asymptotic}; 
other work involves analysis tailored to a specific penalty like the LASSO or SLOPE \cite{BayatiMontanariLASSO,bu2019}.
We instead provide a lower bound on the estimation error of estimators \eqref{linear-cvx-estimator} which holds simultaneously for a large class of penalties.
We rely on weak assumptions---weaker than what is needed for exact characterizations using existing techniques. 
At a high level, the lower bound follows from controlling the possible solutions to Eq.\ \eqref{eq:body-fixed-pt} and applying exact characterization results.

Denote by 
$\cC_p \subseteq \{ \rho:\reals^p \rightarrow \reals \cup \{\infty\} \}$
any collection of lsc, proper, and convex functions which is closed under scaling;
that is, $\rho_p \in \cC_p$ implies $\lambda \rho_p \in \cC_p$ for all $\lambda > 0$.
Denote by $\cC$ the collection of sequences $\{\rho_p\}_p$ such that $\rho_p \in \cC_p$ for all $p$.
We will mostly be interested in two cases: either $\cC$ consists of all the sequences of convex functions, or
it consists of all convex symmetric functions.

The optimal risk of convex M-estimation using collection $\cC_p$ in the sequence model is
\begin{gather}\label{R-seq-cvx-opt-finite}
\mathsf{R^{opt}_{seq,cvx}}(\tau;\pi,p) := \inf_{\rho \in \cC_p}\E_{\bbeta_0,\bz}\left[\left\|\mathsf{prox}[\rho](\bbeta_0 + \tau \bz) - \bbeta_0\right\|^2\right],
\end{gather}
where $\bbeta_0,\bz$ are as in \eqref{sequence-model},
and the optimal asymptotic risk using the sequences in $\cC$ is 
\begin{equation}\label{R-seq-cvx-opt-asymptotic}
\mathsf{R^{opt}_{seq,cvx}} (\tau;\pi) = \liminf_{p\rightarrow \infty} \mathsf{R_{\mathsf{seq,cvx}}^{opt}}(\tau;\pi ,p) = \inf_{\{\rho_p\} \in \cC} \liminf_{p \rightarrow \infty} \E_{\bbeta_0,\bz}\left[\left\|\mathsf{prox}[\rho_p](\bbeta_0 + \tau \bz) - \bbeta_0\right\|^2\right].
\end{equation}
We will study a quantity similar to \eqref{R-seq-cvx-opt-asymptotic} in the linear model \eqref{linear-model} except that the infimum is taken over a slightly more restrictive collection, which we now define.
\begin{definition}
    For $\pi \in \cP_2(\reals)$ and $\delta \in (0,\infty)$, 
    we say a sequence of lsc, proper, convex functions $\{\rho_p\}$ has \emph{$\delta$-bounded width} at prior $\pi$,  if the following holds:
    \begin{equation}\label{delta-bounded-width}
        \begin{gathered}
            \text{for all compact $T \subset (0,\infty)$, there exists $\bar \lambda =\bar\lambda(T)< \infty$ such that}\\
            \limsup_{p\rightarrow \infty}\sup_{\lambda >  \bar \lambda,\tau \in T} \frac1\tau\E_{\bbeta_0,\bz}\left[\left\langle \bz, \mathsf{prox}\left[\lambda \rho_p\right]\left(\bbeta_0 + \tau\bz\right)\right\rangle\right] < \delta.
            \end{gathered}
    \end{equation}
    For a collection of penalty sequences $\cC$, we denote by $\cC_{\delta,\pi}$ the subset of sequences that satisfy this condition.
\end{definition}
The terminology here is motivated by the resemblance of condition \eqref{delta-bounded-width} with the Gaussian width of convex cones \cite{Chandrasekaran2012TheProblems,amelunxen2014living}, see Section \ref{subsec-convex-constraints}.
It is straightforward to show  that for $\delta>1$ and any $\pi \in \cP_2(\reals)$, 
all sequences of penalties have $\delta$-bounded width at $\pi$ (see Section O, Eq.~(O.11) of the Supplementary Material \cite{supplement}).
Thus,
\begin{equation}\label{C-delta-pi-is-C}
\cC_{\delta,\pi} = \cC \quad \text{if $\delta > 1$}.
\end{equation}
The convex lower bound we establish in the next theorem applies to sequences of penalties in $\cC_{\delta,\pi}$.
\begin{theorem}\label{thm-cvx-lower-bound}
Fix $\pi \in \cP_2(\reals)$, $\delta \in (0,\infty)$, and $\sigma \geq 0$.
Define 
\begin{equation}\label{eqdef-tau-reg-cvx}
\tau_{\mathsf{reg,cvx}}^{2} = \sup\left\{\tau^2 \,\Big|\, \delta\tau^2 - \sigma^2 <  \mathsf{R^{opt}_{seq,cvx}}(\tau;\pi)\right\}.
\end{equation}
Under the HDA and RSN assumptions,\footnote{When the minimizing set has multiple elements, we make no assumption on the mechanism used to break ties.}
\begin{equation}\label{cvx-lower-bound}
\inf_{\{\rho_p\} \in \cC_{\delta,\pi}}\liminf^{\mathrm{p}}_{p\rightarrow \infty}\|\widehat \bbeta_{\mathsf{cvx}} - \bbeta_0\|^2 \geq \delta\tau^{2}_{\mathsf{reg,cvx}} - \sigma^2.
\end{equation}
If $\cC$ contains only symmetric penalties, then the preceding display holds also under DSN assumption.
(Note that we may have $\tau_{\mathsf{reg,cvx}}^2 = \infty$.)

In both cases, for $\delta > 1$, the infimum can be taken over the full collection $\cC$ (instead of $\cC_{\delta,\pi}$),
and the lower bound is tight. 
\end{theorem}
\noindent The proof of Theorem \ref{thm-cvx-lower-bound} is provided in Section E of the Supplemenatary Material \cite{supplement}.
In Section \ref{sec-examples}, we argue through examples that $\cC_{\delta,\pi}$ includes most, if not all, reasonable penalty sequences.
Section I of the Supplementary Material \cite{supplement} discusses the role of the restriction to $\cC_{\delta,\pi}$. 
Because $\mathsf{R^{opt}_{seq,cvx}}(\tau;\pi)$ is continuous in $\tau$ whenever $\cC$ is such that $\tau^{2}_{\mathsf{reg,cvx}}$ is finite (see Lemma C.2 of the Supplementary Material \cite{supplement}),
we will always have $\delta\tau^{2}_{\mathsf{reg,cvx}} - \sigma^2 = \mathsf{R^{opt}_{seq,cvx}}(\tau_{\mathsf{reg,cvx}};\pi)$ in this case.
Thus, Theorem \ref{thm-cvx-lower-bound} should be interpreted as stating:
\begin{itemize}
\item[] \emph{Optimal convex M-estimation in the linear model is no better than optimal convex M-estimation in the sequence model at noise variance $\tau^{2}_{\mathsf{reg,cvx}}$.}
\end{itemize}
Importantly, the convex lower bound applies even when $\pi$ is not log-concave.

Although Theorem \ref{thm-cvx-lower-bound} applies to any potentially restricted collection $\cC$ of convex penalty sequences,
our main interest is to apply it to the largest possible collections.
This is because we are interested in studying \emph{fundamental} barriers to regression with any convex estimators of the form \eqref{linear-cvx-estimator}.
Thus, for the remainder of the paper we will consider only two cases:
under the RSN assumption, we will consider $\cC$ to contain all sequences of convex penalties. 
In this case, $\{ \rho_p \} \in \cC_{\delta,\pi}$ contains any sequence of penalties satisfying \eqref{delta-bounded-width}.
Under the DSN assumption, we will consider $\cC$ to contain all sequences of symmetric convex penalties. 
In this case, $\{ \rho_p \} \in \cC_{\delta,\pi}$ contains any sequence of symmetric penalties satisfying \eqref{delta-bounded-width}.
The convex lower bound in these two cases is the same.
\begin{proposition}\label{claim:equiv-of-lb}
    The parameter $\tau^{2}_{\mathsf{reg,cvx}}$ defined with $\cC$ all sequences of convex penalties or with $\cC$ all sequences of symmetric convex penalties agree.
\end{proposition}
\noindent Although we consider two cases throughout the remainder of the paper, there is only one fundamental convex lower bound, and it applies to both cases.
In the first case---that described by the RSN assumption---the statistician has no information about the order in which the coordinates of the unknown parameter occur,
and the convex lower bound applies to any convex procedure.
In the second case---that described by the DSN assumption---the statistician may have information about the order in which the coordinates of the unknown parameter occur,
and the convex lower bound applies only to symmetric convex procedures.
Thus, the convex lower bound applies either to settings in which information about the order of the coordinates is not available or to settings where such information is not exploited.

\subsection{The risk of Bayes AMP}\label{sec:AlgoLB}
Bayes AMP, which we define below, is a fast iterative scheme for performing estimation in model \eqref{linear-model}.
Analogously to the convex lower bound, its estimation error is defined via a comparison of the linear model \eqref{linear-model} to the sequence model \eqref{sequence-model}.
In particular, define
\begin{equation}\label{R-dir-bayes-finite}
\mathsf{mmse}_\pi(\tau^2) = \E_{\beta_0,z}[(\E_{\beta_0,z}[\beta_0|\beta_0 + \tau z] - \beta_0)^2],
\end{equation}
for random scalars $\beta_0 \sim \pi$, $z \sim \mathsf{N}(0,1)$ independent.
Because
\begin{equation}\label{mmse-large-p}
\mathsf{mmse}_\pi(\tau^2) =  \E_{\bbeta_0,\bz}\left[\left\|\E_{\bbeta_0,\bz}\left[\bbeta_0|\sqrt p \bbeta_0 + \tau \bz\right] - \bbeta_0\right\|^2\right],
\end{equation}
we see that $\mathsf{mmse}_\pi(\tau^2)$ is analogous to \eqref{R-seq-cvx-opt-finite} except that the infimum is taken over all estimators, not just those in a restricted class.
Finally, analogous to \eqref{eqdef-tau-reg-cvx}, define
\begin{equation}\label{alg-bound-sup-to-inf-form-def}
    \tau_{\mathsf{reg,amp}*}^2 :=  \sup\left\{\tau^2 \,\Big|\, \delta \tau^2- \sigma^2 \leq \mathsf{mmse}_\pi(\tau^2)\right\}.
\end{equation}
Note that because $\mathsf{mmse}_\pi(\tau^2)$ is continuous in $\tau$ \cite{DongningGuo2011EstimationError},
\begin{equation}\label{tau-alg-is-stationary}
    \delta \tau_{\mathsf{reg,amp}*}^2 - \sigma^2 = \mathsf{mmse}_\pi(\tau_{\mathsf{reg,amp}*}^2).
\end{equation}
As we will see, Bayes AMP asymptotically achieves estimation error arbitrary close to $\delta \tau_{\mathsf{reg,amp}*}^2 - \sigma^2 = \mathsf{mmse}_\pi(\tau_{\mathsf{reg,amp}*}^2)$ in time $O(np)$.
That is,
\begin{itemize}
\item[] \emph{Bayes AMP in the linear model is exactly as good as Bayesian estimation in the sequence model at noise variance $\tau^{2}_{\mathsf{reg,amp}*}$.}
\end{itemize}
Thus, a comparison of the convex lower bound and the risk of Bayes AMP reduces to a comparison of the parameters $\tau_{\mathsf{reg,cvx}}^2$ and $\tau_{\mathsf{reg,amp}*}^2$.
The following corollary of Theorem \ref{thm-cvx-lower-bound} establishes under generic conditions,
the convex lower bound is no smaller than the estimation error of Bayes AMP, consistent with conjectured optimality of Bayes AMP among polynomial time algorithms.

\begin{corollary}\label{cor-alg-lower-bound}

For any $\pi \in \cP_2(\reals)$, 
\begin{equation}\label{tau-cvx-v-tau-amp}
    \tau_{\mathsf{reg,cvx}}^2 \geq \tau_{\mathsf{reg,amp}*}^2  
\end{equation}
holds for almost every value of $\delta,\sigma$ (w.r.t. Lebesgue measure).
In fact, for any fixed $\sigma$, it holds for almost all values of $\delta$, and for any fixed $\delta$, for almost all values of $\sigma$.

For such values $\delta,\sigma$, 
under the HDA and RSN assumptions,
then
\begin{equation}\label{alg-lower-bound}
\inf_{\{\rho_p\} \in \cC_{\delta,\pi}}\liminf^{\mathrm{p}}_{p\rightarrow \infty}\|\widehat \bbeta_{\mathsf{cvx}} - \bbeta_0\|^2 \geq \delta\tau^{2}_{\mathsf{reg,amp}*} - \sigma^2.
\end{equation}
If $\cC$ contains only symmetric penalties, then the preceding display holds instead under DSN assumption.
\end{corollary}

\begin{proof}[Proof of Corollary \ref{cor-alg-lower-bound}]
Define
\begin{equation}\label{tau-alg-def}
\tau_{\mathsf{reg,amp}}^{2} = \sup\left\{\tau^2 \,\Big|\, \delta\tau^2 - \sigma^2 < \mathsf{mmse}_\pi(\tau^2)\right\}.
\end{equation}
\sloppy
In Section L of the Supplementary Material \cite{supplement}, we show that for any $\pi \in \cP_2(\reals)$, the equality $\tau^{2}_{\mathsf{reg,amp}} = \tau^{2}_{\mathsf{reg,amp}*}$
holds for almost every value of $\delta,\sigma$ (w.r.t. Lebesgue measure).
In fact, for any fixed $\sigma$, it holds for almost all values of $\delta$, and for any fixed $\delta$, for almost all values of $\sigma$. Thus, we only need to establish the result for $\tau^{2}_{\mathsf{reg,amp}}$ in place of $\tau^{2}_{\mathsf{reg,amp}*}$.

By \eqref{R-seq-cvx-opt-finite} and \eqref{mmse-large-p}, we have $\mathsf{mmse}_\pi(\tau^2) \leq \mathsf{R^{opt}_{seq,cvx}}(\tau;\pi,p)$. 
By \eqref{R-seq-cvx-opt-asymptotic}), we obtain $\mathsf{mmse}_\pi(\tau^2) \leq \mathsf{R^{opt}_{seq,cvx}}\tau;\pi)$. 
Thus, the set $\left\{\tau^2 \mid  \delta\tau^2 - \sigma^2 < \mathsf{mmse}_\pi(\tau^2)\right\} \subseteq \left\{\tau^2 \mid \delta\tau^2 - \sigma^2 <  \mathsf{R^{opt}_{seq,cvx}}(\tau^2;\pi)\right\}$, 
and \eqref{tau-cvx-v-tau-amp} follows from \eqref{eqdef-tau-reg-cvx} and \eqref{tau-alg-def}. 
Theorem \ref{thm-cvx-lower-bound} then gives \eqref{alg-lower-bound}.
\end{proof}

In the remainder of this section, we describe the Bayes AMP algorithm and formally characterize its risk.
Bayes AMP and its characterization via state evolution has been derived elsewhere \cite{donoho2010message,barbier2018optimal}.
Define the scalar iteration
\begin{subequations}\label{B-AMP-SE}
\begin{gather}
\tau_0^2 = \frac1\delta \left(\sigma^2 + s_2(\pi)\right),\label{B-AMP-SE-1}\\
\tau_{t+1}^2 = \frac1\delta\left(\sigma^2 +  \mathsf{mmse}_\pi\left(\tau_t^2\right)\right).\label{B-AMP-SE-2}
\end{gather}
\end{subequations}
Moreover, let
\begin{subequations}\label{B-AMP-denoiser}
\begin{gather}
    \eta_{t}(y) = \E_{\beta_0,z}[\beta_0 | \beta_0 + \tau_{t} z = y]
\end{gather}
\end{subequations}
where $\beta_0 \sim \pi,\,z \sim \mathsf{N}(0,1)$ are independent.
Define
\begin{gather}\label{eqdef-bt}
\mathsf{b}_{t} = \frac1\delta \E_{\beta_0,z}\left[{\eta_{t-1}'}\left(\beta_0 + \tau_{t-1}z\right)\right],
\end{gather}
where $\eta'_{t}$ a weak derivative of $\eta_{t}$.
For each $p$, define $\eta_{t}:\reals^p \rightarrow \reals^p$ by 
\begin{equation}\label{B-AMP-full-p-denoiser}
\eta_{t}(\bx)_j = \frac1{\sqrt p} \eta_{t}(\sqrt p x_j),
\end{equation}
where for convenience, we use the same notation $\eta_{t}$ for the multivariate and scalar functions. 
They are distinguished by the nature of their argument. 
The Bayes-AMP iteration is
\begin{equation}\label{B-AMP}
\begin{gathered}
    \br^t = \frac{\by - \bX \widehat \bbeta^t}n + \mathsf{b}_{t} \br^{t-1},\\
    \widehat \bbeta^{t+1} = \eta_{t}\left(\widehat \bbeta^t + \bX^\mathsf{T}\br^t\right),
\end{gathered}
\end{equation}
with initialization $\widehat \bbeta^0 = \bzero$, $\br^{-1} = \bzero$.
For any fixed $t$, we may compute $\widehat \bbeta^t$ in $O(np)$ time.
The following proposition characterizes the asymptotic loss of $\widehat \bbeta^t$ as an estimator of $\bbeta_0$.
\begin{proposition}\label{prop-bAMP-achieves-alg-bound}
   Fix $\pi \in \cP_2(\reals)$, $\delta \in (0,\infty)$, and $\sigma \geq 0$. Assume $s_2(\pi) > 0$. 
   Consider $\tau_t$ as defined by \eqref{B-AMP-SE} and $\widehat \bbeta^t$ as defined by \eqref{B-AMP}.
    Under the HDA and either the DSN or RSN assumptions, for any fixed $t$ we have
    \begin{equation}\label{M-to-infty-iterates-loss}
    \lim_{p \rightarrow \infty}^{\mathrm p}\|\widehat \bbeta^t - \bbeta_0\|^2 = \mathsf{mmse}_\pi\left(\tau_t^2\right)
    \end{equation}
    Further, 
    \begin{equation}\label{large-t-bAMP-SE}
    \lim_{t \rightarrow \infty} \tau_t^2 = \tau_{\mathsf{reg,amp}*}^2.
    \end{equation}
    In particular, for all $\eps > 0$, there exists $t$ fixed such that
    \begin{equation}\label{bAMP-achieves-bound}
    \lim_{p \rightarrow \infty}^{\mathrm{p}} \|\widehat \bbeta^t - \bbeta_0 \|^2 \leq \delta\tau_{\mathsf{reg,amp}*}^2-\sigma^2+ \eps.
    \end{equation}
\end{proposition}
\noindent Proposition \ref{prop-bAMP-achieves-alg-bound} states that the state evolution \eqref{B-AMP-SE} characterizes the large $n,p$ behavior of Bayes AMP.
It follows from standard results in the AMP literature \cite{BM-MPCS-2011}.
A minor technical difficulty is that the main theorem of \cite{BM-MPCS-2011} requires Lipschitz non-linearities in the AMP iteration.
The Bayes estimator $\eta_t$ need not be Lipschitz.
Thus, to apply the results of \cite{BM-MPCS-2011}, we must use a truncation trick.
Though this is a routine proof, we are unaware of a result that immediately implies Proposition
\ref{prop-bAMP-achieves-alg-bound}.
For completeness, we provide this argument in Section L of the Supplementary Material \cite{supplement}. 

Proposition \ref{prop-bAMP-achieves-alg-bound} shows that a polynomial-time (in fact, linear time) algorithm exists which achieves asymptotic loss arbitrarily close to $\delta \tau_{\mathsf{reg,amp}*}^2 - \sigma^2$.  As discussed in the introduction, we do not know of any 
polynomial-time algorithm that achieves asymptotic risk below $\delta\tau_{\mathsf{reg,amp}*}^2 - \sigma^2$. Below is a more precise restatement of 
Conjecture \ref{conj:Gap}.
\begin{conjecture}\label{conj-algorithmic-threshold}
    Fix $\pi \in \cP_2(\reals)$, $\delta \in (0,\infty)$, and $\sigma > 0$.
    Under the HDA and RSN assumptions at $\pi,\delta,\sigma$, 
    no polynomial-time algorithm achieves asymptotic risk smaller than $\delta\tau_{\mathsf{reg,amp}*}^2 - \sigma^2$. 
\end{conjecture}

\subsection{The Bayes risk}

The information theoretic lower bound under the RSN assumption is the Bayes risk 
$$
\E_{\bbeta_0,\bw,\bX}\left[\|\E_{\bbeta_0,\bw,\bX}[\bbeta_0|\by,\bX] - \bbeta_0 \|^2\right],
$$
which cannot be outperformed even in finite samples.
In this section, we recall recent results on the asymptotic value of the Bayes risk on the HDA and RSN assumptions.

Define the potential
\begin{equation}\label{rs-potential}
\phi(\tau^2;\pi,\delta,\sigma) = \frac{\sigma^2}{2\tau^2} - \frac\delta2\log\left(\frac{\sigma^2}{\tau^2}\right)  + i(\tau^2),
\end{equation}
where $i(\tau^2)$ is the base-$e$ mutual information between $\beta_0$ and $y$ in the univariate model $y = \beta_0 +  \tau z$ when $\beta_0 \sim \pi,\,z \sim \mathsf{N}(0,1)$ independent. 
That is,
\begin{equation}\label{eqdef-i}
i(\tau^2) = \E_{\beta_0,z} \left[\log\frac{p(y|\beta_0)}{p(y)}\right] = - \frac12 - \E_{\beta_0,z} \log \left\{\int  e^{-\frac1{2}(y - \beta/\tau)^2}\pi(\mathrm{d}\beta)\right\}\, .
\end{equation}
Also define
\begin{equation}\label{tau-stat-def}
\tau_{\mathsf{reg,stat}}(\pi;\delta,\sigma) = \arg\min_{\tau\geq0} \phi(\tau^2;\pi,\delta,\sigma),
\end{equation}
whenever $\pi,\delta$, and $\sigma$ are such that the minimizer is unique.
The derivative of $\phi$ will be useful in what follows. It is
\begin{equation}\label{potential-derivative}
    \frac{\mathrm{d}}{\mathrm{d}\tau^{-2}} \phi(\tau^2;\pi,\delta,\sigma) = \frac12\left(\sigma^2 - \delta \tau^2 + \mathsf{mmse}_\pi(\tau^2)\right),
\end{equation}
where we have used that $\frac{\mathrm{d}}{\mathrm{d} \tau^{-2}} i(\tau^2) = \frac12 \mathsf{mmse}_\pi(\tau^2)$ by \cite[Corollary 1]{DongningGuo2011EstimationError}.
We see that if $\tau_{\mathsf{reg,stat}} > 0$, then 
\begin{equation}\label{tau-stat-is-stationary}
    \delta \tau_{\mathsf{reg,stat}}^2 - \sigma^2 = \mathsf{mmse}_\pi(\tau_{\mathsf{reg,stat}}^2).
\end{equation}
Equation \eqref{tau-stat-is-stationary} is closely related to  \eqref{tau-alg-def}.
The next result relates the effective noise parameter $\tau_{\mathsf{reg,stat}}$ to the asymptotic Bayes risk in model \eqref{linear-model} under the RSN assumption.
\begin{proposition}[Theorem 2 of \cite{barbier2018optimal}]\label{prop-bayes-risk-asymptotic}
Fix $\pi \in \cP_\infty(\reals)$, $\delta \in (0,\infty)$, and $\sigma > 0$. 
Under the HDA and RSN assumptions,
\begin{equation}\label{bayes-risk-asymptotic}
\lim_{p \rightarrow \infty}\E_{\bbeta_0,\bw,\bX}\left[\|\E_{\bbeta_0,\bw,\bX}[\bbeta_0|\by,\bX] - \bbeta_0 \|^2\right] = \mathsf{mmse}_\pi(\tau_{\mathsf{reg,stat}}^2) = \delta\tau_{\mathsf{reg,stat}}^2 - \sigma^2,
\end{equation}
whenever the minimizer in \eqref{tau-stat-def} is unique. 
This occurs for almost every $(\delta,\sigma)$ (w.r.t. Lebesgue measure).
\end{proposition}
\noindent This is a specific case of Theorem 2 of \cite{barbier2018optimal}. We carry out the conversion from their notation to ours in Section L of the Supplementary Material \cite{supplement}.
This result was previously established under slightly less general conditions in \cite{thrampoulidis2018precise,barbier2017}.
In particular, Proposition \ref{prop-bayes-risk-asymptotic} states that:
\begin{itemize}
\item[] \emph{Bayesian estimation in the linear model is exactly as good as Bayesian estimation in the sequence model at noise variance $\tau^{2}_{\mathsf{reg,stat}}$.}
\end{itemize}
Thus, a comparison of the convex lower bound, the risk of Bayes AMP, and the Bayes risk reduces to a comparison of the noise variances $\tau^2_{\mathsf{reg,cvx}}$, $\tau^2_{\mathsf{reg,amp}*}$, and $\tau^2_{\mathsf{reg,stat}}$.
Because it is simply a lower bound, the convex lower bound could plausibly sometimes be smaller than the Bayes risk.
Fortunately, this does not occur: 
\begin{corollary}
    For all $\pi,\delta,\sigma$, we have
    \begin{equation}\label{compare-cvx-stat-effective-noise}
        \tau^2_{\mathsf{reg,cvx}} \geq \tau^2_{\mathsf{reg,stat}}.
    \end{equation}
\end{corollary}

\begin{proof}
    The inequality $\tau_{\mathsf{reg,cvx}}^2 \geq \tau_{\mathsf{reg,amp}}^2$ holds because the supremum in \eqref{tau-alg-def} is taken over a subset of the supremum in \eqref{eqdef-tau-reg-cvx}.
    Thus, it suffices to show $\tau_{\mathsf{reg,amp}}^2\geq \tau_{\mathsf{reg,stat}}^2$.
    For $\tau' < \tau_{\mathsf{reg,stat}}$,
    \begin{align}
        \phi(\tau_{\mathsf{reg,stat}};\pi,\delta,\sigma) &< \phi(\tau';\pi,\delta,\sigma) \\
        &= \phi(\tau_{\mathsf{reg,stat}};\pi,\delta,\sigma) +  \frac12\int_{\tau_{\mathsf{reg,stat}}^{-2}}^{{\tau'}^{-2}}\left(\sigma^2 - \delta \tau^2 + \mathsf{mmse}_\pi(\tau^2)\right)d\tau^{-2}.
    \end{align}
    Thus, the integral in the previous display must be positive for all $\tau' < \tau_{\mathsf{reg,stat}}$, which implies there exists $\tau' < \tau_{\mathsf{reg,stat}}$ arbitrarily close to $\tau_{\mathsf{reg,stat}}$ for which $\delta {\tau'}^2 - \sigma^2 < \mathsf{mmse}_\pi({\tau'}^2)$. 
    By \eqref{tau-alg-def}, we have $\tau_{\mathsf{reg,amp}} \geq \tau_{\mathsf{reg,stat}}$, as desired. 
\end{proof}

\section{Log-concavity and convex-algorithmic-statistical gaps}

The results in the preceding section establish that \emph{(i)} if $\tau^2_{\mathsf{reg,cvx}} > \tau^2_{\mathsf{reg,amp}*}$, there is a gap between the asymptotic estimation error achieved by convex M-estimators \eqref{linear-cvx-estimator} and that achieved by Bayes AMP, 
and \emph{(ii)} for generic $(\delta,\sigma)$ (i.e., those for which the minimizer in \eqref{tau-stat-def} is unique), if $\tau^2_{\mathsf{reg,cvx}} > \tau^2_{\mathsf{reg,stat}}$, 
there is a gap between the asymptotic estimation error achieved by convex M-estimators \eqref{linear-cvx-estimator} and that achieved by information theoretically optimal estimation. 
Two important questions remain.
\begin{enumerate}

    \item % 1
    Is the converse true? Namely, if $\tau^2_{\mathsf{reg,cvx}} = \tau^2_{\mathsf{reg,amp}*}$ or $\tau^2_{\mathsf{reg,cvx}} = \tau^2_{\mathsf{reg,stat}}$, is convex M-estimation as good as Bayes AMP or Bayesian estimation?

    \item % 2
    Can we provide more interpretable conditions which determine whether the strict inequalities $\tau^2_{\mathsf{reg,cvx}} > \tau^2_{\mathsf{reg,amp}*}$ and $\tau^2_{\mathsf{reg,cvx}} > \tau^2_{\mathsf{reg,stat}}$ occur?

\end{enumerate}
It turns out that the condition we provide to answer the second question  will provide an affirmative answer to the first question.
In particular,
we will show that $\tau^2_{\mathsf{reg,cvx}} = \tau^2_{\mathsf{reg,amp}*}$ (resp.\ $\tau^2_{\mathsf{reg,cvx}} = \tau^2_{\mathsf{reg,stat}}$) if and only if $\pi * \normal(0,\tau^2_{\mathsf{reg,amp}*})$ (resp.\ $\pi * \normal(0,\tau^2_{\mathsf{reg,stat}})$) is log-concave.
Moreover, while when $\delta \leq 1$ we do not guarantee the tightness of the convex lower bound generally, we will guarantee its tightness in the case that $\pi * \normal(0,\tau^2_{\mathsf{reg,cvx}})$ is log-concave.
Because $\tau^2_{\mathsf{reg,cvx}} = \tau^2_{\mathsf{reg,amp}*}$ implies $\pi * \normal(0,\tau^2_{\mathsf{reg,amp}*})$, and hence $\pi * \normal(0,\tau^2_{\mathsf{reg,cvx}})$, is log-concave,
it also implies that convex M-estimation is as good as Bayes AMP. 
A similar line of reasoning follows when $\tau^2_{\mathsf{reg,cvx}} = \tau^2_{\mathsf{reg,stat}}$.
Thus, the converse described in the first question indeed holds.

Before describing this argument in detail, we remark that when $\pi$ itself is log-concave, $\pi * \normal(0,\tau^2)$ is log-concave for all $\tau^2$.
In this case, the convex lower bound, the risk of Bayes AMP, and the Bayes risk agree for all values of $\sigma,\delta$.
Moreover, in this case the convex lower bound is always tight, so that convex M-estimators \eqref{linear-cvx-estimator} always achieve information theoretically optimal performance.
In contrast, we will show that when $\pi$ is not log-concave, 
there exist values of $\sigma,\delta$ for which the convex lower bound is strictly larger than the the risk of Bayes AMP and the Bayes risk.
Thus, non-trivial performance of convex M-estimation relative to computational and information-theoretic benchmarks occurs exactly when $\pi$ is not log-concave.
  
\begin{proposition}\label{prop-achieving-the-bound}
    Consider $\pi \in \cP_\infty(\reals)$, $\delta \in (0,\infty)$, and $\sigma \geq 0$. 
    If $\cC$ consists of all sequences of convex penalties, the following statements hold under the HDA and RSN assumptions; if $\cC$ consists of all sequences of symmetric convex penalties, we may replace the RSN by the DSN assumption.
    \begin{enumerate}[(i)]
    
    \item % i
    If $\tau \geq 0$ is such that $\pi * \mathsf{N}(0,\tau^2)$ has log-concave density (w.r.t.\ Lebesgue measure) and $\delta \tau^2 - \sigma^2 > \mathsf{mmse}_\pi(\tau^2)$, then
    \begin{equation}\label{minimal-loss-upper-bound}
        \inf_{\{\rho_p\} \in \cC_{\delta,\pi}} \lim^{\mathrm{p}}_{p \rightarrow \infty} \|\widehat \bbeta_{\mathsf{cvx}} - \bbeta_0\|^2 \leq \delta\tau^2 - \sigma^2.
    \end{equation}
    Under the RSN assumption, we may replace the limit in probability with $\lim_{p \rightarrow \infty} \E_{\bbeta_0,\bw,\bX}\big[\|\widehat \bbeta_{\mathsf{cvx}} - \bbeta_0\|^2\big]$.
    (We set these limits to $\infty$ when they do not exist.)
    
    \item %ii
    If $\tau \geq 0$ is such that $\pi * \mathsf{N}(0,\tau^2)$ does not have log-concave density (w.r.t.\ Lebesgue measure) and $\delta\tau^2 - \sigma^2 \leq \mathsf{mmse}_\pi(\tau^2)$, then $\tau^2_{\mathsf{reg,cvx}} > \tau^2$ whence
    \begin{equation}\label{minimal-loss-lower-bound}
    \inf_{\{\rho_p\} \in \cC_{\delta,\pi}} \liminf^{\mathrm{p}}_{p \rightarrow \infty}  \|\widehat \bbeta_{\mathsf{cvx}} - \bbeta_0\|^2 > \delta\tau^2 - \sigma^2.
    \end{equation}

    \item %iii
    We have $\tau^2_{\mathsf{reg,cvx}} = \tau^2_{\mathsf{reg,stat}}$ if and only if $\pi * \normal(0,\tau^2_{\mathsf{reg,stat}})$ is log-concave.
    In the (generic) case that $\tau^2_{\mathsf{reg,amp}} = \tau^2_{\mathsf{reg,amp}*}$, 
    we have $\tau^2_{\mathsf{reg,cvx}} = \tau^2_{\mathsf{reg,amp}*}$ if and only if $\pi * \normal(0,\tau^2_{\mathsf{reg,amp}*})$.
    
    \end{enumerate}
\end{proposition}
\noindent The proof of Proposition \ref{prop-achieving-the-bound} is provided in Section J of the Supplementary Material \cite{supplement}. 

While the relevance of the log-concavity of the convolutional density $\pi * \normal(0,\tau^2)$ may seem surprising, it is related to the following fact: 
in the Gaussian sequence model \eqref{sequence-model}, the Bayes estimator is the proximal operator of some convex function if and only if $\pi * \normal(0,\tau^2)$ is log-concave. 
This is a remarkable consequence of Tweedie's formula.
Our construction of penalties achieving \eqref{minimal-loss-upper-bound} involves identifying the penalty whose proximal operator is the Bayes estimator at noise variance $\tau^2$ in the sequence model.
This is related to the construction of \cite{Bean2013}.
See Section J of the Supplementary Material \cite{supplement} for details of this fact and its use in proving Proposition \ref{prop-achieving-the-bound}.

\subsection{Gaps between convex M-estimators and Bayes AMP}\label{subsection-comparing-alg-and-cvx}

Under generic conditions, convex M-estimators achieve the risk of Bayes AMP if and only if $\pi * \normal(0,\tau^2_{\mathsf{reg,amp}*})$ has log-concave density.

\begin{theorem}\label{thm-cvx-cannot-beat-alg}
    Consider $\pi \in \cP_2(\reals)$, $\delta \in (0,\infty)$, $\sigma \geq 0$.
    Assume $\tau_{\mathsf{reg,amp}} = \tau_{\mathsf{reg,amp}*}$ (which holds generically, see the proof of Corollary \ref{cor-alg-lower-bound}, as well as Section L of the Supplementary Material \cite{supplement}).
    If $\cC$ contains all sequences of convex penalties, then under the HDA and RSN assumptions,
    inequality \eqref{alg-lower-bound} holds with equality \emph{if and only if} $\pi * \mathsf{N}(0,\tau_{\mathsf{reg,amp}*}^2)$ has log-concave density (w.r.t.\ Lebesgue measure), which occurs if and only if $\tau^2_{\mathsf{reg,cvx}} = \tau^2_{\mathsf{reg,amp}*}$.
    The same holds if we replace the limits in probability with the limits of expectations in \eqref{alg-lower-bound}.

    If $\cC$ contains all sequences of symmetric convex penalties, the preceding statements hold also under the DSN assumption.
\end{theorem}

When equality occurs in Theorem \ref{thm-cvx-cannot-beat-stat},
the penalty achieving the convex lower bound is (up to a small strong convexity term added for technical reasons) given by the convex function whose proximal operator is the Bayes estimator in the sequence model \eqref{sequence-model} at noise variance $\tau^2_{\mathsf{reg,amp}*}$. 
The existence of such a penalty is a consequence of the log-concavity of $\pi * \mathsf{N}(0,\tau_{\mathsf{reg,amp}}^2)$.
See the remark following Proposition \ref{prop-achieving-the-bound} and the proof of that proposition in Section J of the Supplementary Material \cite{supplement} for further details.

\begin{proof}[Proof of Theorem \ref{thm-cvx-cannot-beat-alg}]
    The equivalence of $\pi * \mathsf{N}(0,\tau_{\mathsf{reg,amp}*}^2)$ having log-concave density and $\tau^2_{\mathsf{reg,cvx}} = \tau^2_{\mathsf{reg,amp}*}$ holds by Proposition \ref{prop-achieving-the-bound}.$(iii)$. We now focus on the remaining parts of the Theorem.

    We first prove the ``if'' direction. 
    By \eqref{alg-bound-sup-to-inf-form-def}, we have for $\tau > \tau_{\mathsf{reg,amp}*}$ that $\delta\tau^2 - \sigma^2 > \mathsf{mmse}_\pi(\tau^2)$. 
    Further, because $\pi * \mathsf{N}(0,\tau_{\mathsf{reg,amp}*}^2)$ has log-concave density, so too does $\pi * \mathsf{N}(0,\tau^2)$ \cite[Proposition 3.5]{Saumard2014Log-concavityReview}. 
    By Proposition \ref{prop-achieving-the-bound}.$(i)$, we have that \eqref{minimal-loss-upper-bound} holds with this choice of $\tau$. 
    Taking $\tau \downarrow \tau_{\mathsf{reg,amp}*} = \tau_{\mathsf{reg,amp}}$, we conclude that \eqref{alg-lower-bound} holds with the inequality reversed, so in fact holds with equality.
    
    We now prove the ``only if'' direction.
    By \eqref{alg-bound-sup-to-inf-form-def} and the continuity of $\mathsf{mmse}_\pi(\tau^2)$ in $\tau^2$ \cite[Proposition 7]{DongningGuo2011EstimationError}, we have 
    \begin{equation}\label{tau-alg-star-is-stationary}
    \delta \tau_{\mathsf{reg,amp}*}^2 - \sigma^2 = \mathsf{mmse}_\pi(\tau_{\mathsf{reg,amp}*}^2).
    \end{equation}
    If $\pi * \mathsf{N}(0,\tau_{\mathsf{reg,amp}*}^2)$ does not have log-concave density, by Proposition \ref{prop-achieving-the-bound}.$(ii)$ Eq.~\eqref{alg-lower-bound} holds with strict inequality. By Lemma K.1 of the Supplementary Material \cite{supplement}, the same holds when replace limits in probability with limits of expectations.
\end{proof}

A corollary of Theorem \ref{thm-cvx-lower-bound} is that when $\pi$ has log-concave density, gaps between convex M-estimation and the risk of Bayes AMP do not occur, 
whereas when $\pi$ does not have log-concave density, they do occur at large enough signal-to-noise ratios.

\begin{corollary}\label{cor-cvx-alg-phase-diagram}
    Consider $\pi \in \cP_2(\reals)$ and $\sigma \geq 0$. Let $\cB\subseteq\reals$ be the set of $\delta>0$ for which $\tau_{\mathsf{reg,amp}} < \tau_{\mathsf{reg,amp}*}$
holds (recall that, by the proof of Corollary \ref{cor-alg-lower-bound}, $\cB$ has zero Lebesgue measure).
We have the following.
        \begin{enumerate}[(a)]
        \item If $\pi$ has log-concave density, then for all $\delta \in\reals_{>0}\setminus\cB$,
 inequality \eqref{alg-lower-bound} holds with equality.
        \item If $\sigma > 0$ and $\pi $ does not have log-concave density, then there exist $0 \leq \delta_{\mathsf{alg}} < \infty$ such that inequality \eqref{alg-lower-bound} holds with equality for $\delta \in (0, \delta_{\mathsf{alg}})\setminus\cB$  and with strict inequality for all $\delta \in(\delta_{\mathsf{alg}},\infty)\setminus \cB$.
    \end{enumerate}
\end{corollary}

\noindent Part $(b)$ states that, if $\pi$ is not log-concave, then either $(i)$ there is always a gap between convex M-estimation and the best algorithm we know of or $(ii)$ for small $\delta$, the algorithmic lower bound is achieved by a convex procedure,
while for large $\delta$ there is a gap between convex M-estimation and the best algorithm that we know of. This might seem counterintuitive, because large $\delta$
corresponds to larger sample size and therefore easier estimation. An intuitive explanation of this result is that, for large $\delta$, we can exploit more of the structure of
the prior $\pi$, and this requires non-convex methods. 

\begin{proof}[Proof of Corollary \ref{cor-cvx-alg-phase-diagram}]\hfill 

    \emph{Part (a):} By \cite[Proposition 3.5]{Saumard2014Log-concavityReview}, $\pi * \mathsf{N}(0,\tau_{\mathsf{reg,amp}}^2)$ has log-concave density. 
    The result follows by Theorem \ref{thm-cvx-cannot-beat-alg}.
    
    \emph{Part (b):} Define $\delta_{\mathsf{alg}} = \inf \{\delta \mid \pi * \mathsf{N}(0,\tau_{\mathsf{reg,amp}}^2)\text{ does not have log-concave density}\}$. 
     By \cite[Proposition 3.5]{Saumard2014Log-concavityReview}, if $\tau < \tau'$ and $\pi * \mathsf{N}(0,\tau^2)$ has log-concave density, then so too does $\pi * \mathsf{N}(0,{\tau'}^2)$. 
By \eqref{tau-alg-def}, $\tau_{\mathsf{reg,amp}}$ is non-increasing in $\delta$.
    Combining these two facts, for $\delta > \delta_{\mathsf{alg}}$ we have $\mathsf{N}(0,\tau_{\mathsf{reg,amp}}^2)$ does not have log-concave density, and for $\delta < \delta_{\mathsf{alg}}$ we have $\mathsf{N}(0,\tau_{\mathsf{reg,amp}}^2)$ does have log-concave density.
    Then, by Theorem \ref{thm-cvx-cannot-beat-alg}, inequality \eqref{alg-lower-bound} holds with equality for $\cB \ni \delta < \delta_{\mathsf{alg}}$ and with strict inequality when $\cB \ni \delta > \delta_{\mathsf{alg}}$.
    We need only check that $\delta_{\mathsf{alg}} < \infty$. 
    By \eqref{tau-alg-is-stationary}, $\tau_{\mathsf{reg,amp}}^2 = \frac1\delta\big( \sigma^2 +\mathsf{mmse}_\pi(\tau_{\mathsf{reg,amp}}^2)\big) \leq \frac1\delta\left(\sigma^2 +  s_2(\pi)\right)$.
    Thus, $\lim_{\delta \rightarrow \infty} \tau_{\mathsf{reg,amp}}^2 = 0$. 
    Because log-concavity is preserved under convergence in distribution \cite[Proposition 3.6]{Saumard2014Log-concavityReview} and $\pi * \mathsf{N}(0,\tau^2) \xrightarrow[\tau \rightarrow 0]{\mathrm{d}} \pi$, we conclude that for $\delta$ sufficiently large, $\pi * \mathsf{N}(0,\tau_{\mathsf{reg,amp}}^2)$ does not have log-concave density, as desired.
\end{proof}

\subsection{Gaps between convex M-estimators and the Bayes risk}\label{subsection-comparing-stat-and-cvx}

Under generic conditions, convex M-estimators achieve the Bayes risk exactly when the convex lower bound is equal to the Bayes risk, which in turn occurs exactly when $\pi * \normal(0,\tau^2_{\mathsf{reg,stat}})$ has log-concave density.

\begin{theorem}\label{thm-cvx-cannot-beat-stat}
    Consider $\pi \in \cP_\infty(\reals)$, $\delta \in (0,\infty)$, and $\sigma > 0$. 
    Assume  
    the potential $\phi$ defined in Eq.~\eqref{rs-potential} has a  unique minimizer. 
    If $\cC$ cosists of  all sequences of convex penalties, then under the HDA and RSN assumptions,
    $\tau^2_{\mathsf{reg,cvx}} = \tau^2_{\mathsf{reg,stat}}$ if and only if
    \begin{equation}\label{convex-risk-achieves-stat-lower-bound}
        \inf_{\{\rho_p\}_p \in \cC_{\delta,\pi}} \liminf_{p \rightarrow \infty}  \E_{\bbeta_0,\bw,\bX} \left[\|\widehat \bbeta_{\mathsf{cvx}} - \bbeta_0\|^2\right] = \lim_{p \rightarrow \infty} \E_{\bbeta_0,\bw,\bX}\left[\|\E_{\bbeta_0,\bw,\bX}[\bbeta_0|\by] - \bbeta_0 \|^2\right],
    \end{equation}
    which in turn occurs if and only if $\pi * \mathsf{N}(0,\tau_{\mathsf{reg,stat}}^2)$ has log-concave density with respect to Lebesgue measure on $\reals$.
\end{theorem}

\noindent 
Analogously to Theorem \ref{thm-cvx-cannot-beat-alg},
when equality occurs in Theorem \ref{thm-cvx-cannot-beat-stat},
the penalty achieving the convex lower bound is (up to a small strong convexity term added for technical reasons) given by the convex function whose proximal operator is the Bayes estimator in the sequence model \eqref{sequence-model} at noise variance $\tau^2_{\mathsf{reg,stat}}$. 
See the remark following Proposition \ref{prop-achieving-the-bound} and the proof of that proposition in Section J of the Supplementary Material \cite{supplement} for further details.
The condition that the minimizer of $\phi$ is unique holds --by analyticity considerations-- for all $(\delta,\sigma)$
except a set of Lebesgue measure zero.

\begin{proof}[Proof of Theorem \ref{thm-cvx-cannot-beat-stat}]
    The equivalence of $\pi * \mathsf{N}(0,\tau_{\mathsf{reg,stat}}^2)$ having log-concave density and $\tau^2_{\mathsf{reg,cvx}} = \tau^2_{\mathsf{reg,stat}}$ holds by Proposition \ref{prop-achieving-the-bound}(iii). We now focus on the remaining parts of the Theorem.

    The right-hand side of \eqref{convex-risk-achieves-stat-lower-bound} is $\delta \tau_{\mathsf{reg,stat}}^2 - \sigma^2$ by Proposition \ref{prop-bayes-risk-asymptotic} (this is where we use $\sigma > 0$). 
    By \eqref{compare-cvx-stat-effective-noise}, if $\tau^2_{\mathsf{reg,cvx}} \neq \tau^2_{\mathsf{reg,stat}}$, then $\tau^2_{\mathsf{reg,cvx}} > \tau^2_{\mathsf{reg,stat}}$. Then by Theorem \ref{thm-cvx-lower-bound}, as well as Lemma K.1 of the Supplementary Material \cite{supplement}, 
    we have under the RSN assumption that \eqref{convex-risk-achieves-stat-lower-bound} holds with equality replace by strict inequality.

    Now consider that $\tau^2_{\mathsf{reg,cvx}} = \tau^2_{\mathsf{reg,stat}}$, or equivalently, that $\pi * \mathsf{N}(0,\tau_{\mathsf{reg,stat}}^2)$ has log-concave density.
    Assume $\mathsf{N}(0,\tau_{\mathsf{reg,stat}}^2)$ has log-concave density, $\sigma > 0$, and $\phi$ has unique minimizer.
    For $\tau' > \tau_{\mathsf{reg,stat}}$ we have
    \begin{align}
        \phi(\tau_{\mathsf{reg,stat}};\pi,\delta,\sigma) &= \phi(\tau';\pi,\delta,\sigma) +  \frac12\int_{\tau'^{-2}}^{\tau_{\mathsf{reg,stat}}^{-2}}\left(\sigma^2 - \delta \tau^2 + \mathsf{mmse}_\pi(\tau^2)\right)d\tau^{-2} \nonumber\\
        &> \phi(\tau_{\mathsf{reg,stat}};\pi,\delta,\sigma) + \frac12\int_{\tau'^{-2}}^{\tau_{\mathsf{reg,stat}}^{-2}}\left(\sigma^2 - \delta \tau^2 + \mathsf{mmse}_\pi(\tau^2)\right)d\tau^{-2},\label{alg-vs-stat-free-energy-compare}
    \end{align}
    where in the inequality we use that the minimizer of $\phi$ is unique.
    Thus, the integral is negative for all $\tau' > \tau_{\mathsf{reg,stat}}$, so there exists $\tau' > \tau_{\mathsf{reg,stat}}$ arbitrarily close to $\tau_{\mathsf{reg,stat}}$ for which $\delta \tau'^2 - \sigma^2 >  \mathsf{mmse}_\pi(\tau'^2)$. 
    By \cite[Proposition 3.5]{Saumard2014Log-concavityReview}, we have for all such $\tau'$ that $\pi * \mathsf{N}(0,\tau'^2)$ has log-concave density.
    Taking $\tau' \downarrow \tau_{\mathsf{reg,stat}}$ along $\tau'$ for which $\delta \tau'^2 - \sigma^2 >  \mathsf{mmse}_\pi(\tau'^2)$ and applying Proposition \ref{prop-achieving-the-bound}.$(i)$, we have under the RSN assumption that
    \begin{equation}\label{tau-stat-DSN-upper-bound}
        \inf_{\{\rho_p\}_p \in \cC_{\delta,\pi}} \lim_{p \rightarrow \infty}  \E_{\bbeta_0,\bw,\bX} \left[\|\widehat \bbeta_{\mathsf{cvx}} - \bbeta_0\|^2\right] \leq \delta\tau_{\mathsf{reg,stat}}^2 - \sigma^2.
    \end{equation}
    By \eqref{bayes-risk-asymptotic}, we have $\delta \tau_{\mathsf{reg,stat}}^2 - \sigma^2$ equals the right-hand side of \eqref{convex-risk-achieves-stat-lower-bound}. 
        The reverse inequality holds by the optimality of the Bayes risk, whence we conclude \eqref{convex-risk-achieves-stat-lower-bound}.
\end{proof}

A corollary of Theorem \ref{thm-cvx-lower-bound} is that when $\pi$ has log-concave density, gaps between convex M-estimation and the Bayes risk do not occur, 
whereas when $\pi$ does have log-concave density, they do occur at large enough signal-to-noise ratios.

\begin{corollary}\label{cor-cvx-stat-phase-diagram}
    Consider $\pi \in \cP_\infty(\reals)$ and $\sigma > 0$.  We have the following.
    \begin{enumerate}[(a)]
        \item If $\pi$ has log-concave density with respect to Lebesgue measure, then for all $\delta > 0$ for which $\phi$ has unique minimizer, equality \eqref{convex-risk-achieves-stat-lower-bound} holds.
        \item If $\pi$ does not have log-concave density with respect to Lebesgue measure, then there exist $0 \leq \delta_{\mathsf{stat}} < \infty$ such that equality \eqref{convex-risk-achieves-stat-lower-bound} holds for all $\delta < \delta_{\mathsf{stat}}$ for which $\phi$ has unique minimizer, 
        and \eqref{convex-risk-achieves-stat-lower-bound} holds with strict inequality replacing equality for all $\delta > \delta_{\mathsf{stat}}$ for which $\phi$ has unique minimizer. Moreover, $\delta_{\mathsf{stat}} \leq \delta_{\mathsf{alg}}$.
    \end{enumerate}
\end{corollary}

\begin{proof}[Proof of Corollary \ref{cor-cvx-stat-phase-diagram}]\hfill

    \emph{Part (a):} 
    By \cite[Proposition 3.5]{Saumard2014Log-concavityReview}, we have $\pi * \mathsf{N}(0,\tau_{\mathsf{reg,stat}}^2)$ has log-concave density with respect to Lebesgue measure.
    The result follws by Theorem \ref{thm-cvx-cannot-beat-stat}.
    
    \emph{Part (b):} Define $\delta_{\mathsf{stat}} = \inf\{\delta \mid \pi * \mathsf{N}(0,\tau_{\mathsf{reg,stat}}^2)\text{ does not have log-concave density}\}$. Because the derivative \eqref{potential-derivative} of $\phi$ with respect to $\tau^{-2}$ is strictly decreasing in $\delta$, we have by \eqref{rs-potential} that $\tau_{\mathsf{reg,stat}}$ is strictly decreasing in $\delta$. 
    As in the proof of Corollary \ref{cor-cvx-alg-phase-diagram}, this implies that for for $\delta > \delta_{\mathsf{stat}}$ we have $\mathsf{N}(0,\tau_{\mathsf{reg,stat}}^2)$ does not have log-concave density and for $\delta < \delta_{\mathsf{stat}}$ we have $\mathsf{N}(0,\tau_{\mathsf{reg,stat}}^2)$ does have log-concave density.
    Then, by Theorem \ref{thm-cvx-cannot-beat-stat}, if $\phi$ has unique minimizer and $\delta > \delta_{\mathsf{stat}}$, 
    then the left-hand side of \eqref{convex-risk-achieves-stat-lower-bound} is strictly larger than the right-hand side, and if $\phi$ has unique minimizer and $\delta < \delta_{\mathsf{stat}}$, equality holds.
    We need only check that $\delta_{\mathsf{stat}} < \infty$. 
    By \eqref{tau-stat-def} and \eqref{potential-derivative}, we have
    $\tau_{\mathsf{reg,stat}}^2 = \frac1\delta\left(\sigma^2 + \mathsf{mmse}_\pi(\tau_{\mathsf{reg,stat}}^2)\right) \leq \frac1\delta \left(\sigma^2 + s_2(\pi)\right)$,
    where $s_2(\pi)$ is the second moment of $\pi$.
    Thus, $\lim_{\delta \rightarrow \infty} \tau_{\mathsf{reg,stat}}^2 = 0$. 
    Because log-concavity is preserved under convergence in distribution \cite[Proposition 3.6]{Saumard2014Log-concavityReview} and $\pi * \mathsf{N}(0,\tau^2)\xrightarrow[\tau \rightarrow 0]{\mathrm{d}} \pi$, we conclude that for sufficiently large $\delta$, $\pi * \mathsf{N}(0,\tau_{\mathsf{reg,stat}}^2)$ is not log-concave, as desired. 
\end{proof}

\section{Quantifying the gap: high and low signal-to-noise ratio (SNR) regimes}

We now provide quantitative estimates of the gap between convex M-estimation and the Bayes risk when such gaps occur. 
Consider $\pi \in \cP_\infty(\reals)$, $\delta \in (0,\infty)$, $\sigma > 0$, and let $\cC$ contain all sequences of convex penalties.
Define the asymptotic gap between convex M-estimation and Bayes error
\begin{align*}
&\Delta(\pi,\delta,\sigma) \equiv\\
&\quad\left(\inf_{\{\rho_p\}_p \in \cC_{\delta,\pi}} \liminf_{p \rightarrow \infty} \E_{\bbeta_0,\bw,\bX} \left[\|\widehat \bbeta_{\mathsf{cvx}} - \bbeta_0\|^2\right]\right) - \left(\lim_{p \rightarrow \infty}  \E_{\bbeta_0,\bw,\bX}\left[\|\E_{\bbeta_0,\bw,\bX}[\bbeta_0|\by,\bX] - \bbeta_0 \|^2\right]\right),
\end{align*}
where the limits are taken under the HDA and RSN assumptions.
The results of  Section \ref{subsection-comparing-stat-and-cvx} characterize whether  $\Delta(\pi,\delta,\sigma) = 0$ or
$\Delta(\pi,\delta,\sigma)> 0$.
Here we provide a more quantitative estimate of its size for large $\delta$ (high SNR) and for large $\sigma$ (low SNR).

\begin{theorem}\label{thm-high-low-snr-gaps}
    Fix $\pi \in \cP_\infty(\reals)$ and let $\cC$ contain all sequences of convex penalties.
    \begin{enumerate}[(i)]
    \item % i
    Restricting ourselves to $\delta,\sigma > 0$ for which the minimizer of \eqref{tau-stat-def} is unique, we have
    \begin{equation}\label{gap-high-snr}
        \Delta(\pi,\delta,\sigma) \geq \mathsf{R^{opt}_{seq,cvx}} \left(\sigma/\sqrt \delta;\pi\right) - \mathsf{mmse}_\pi\left(\sigma^2/\delta\right) + O\left(1/\sqrt \delta\right),
    \end{equation}
    where $O$ hides constants depending only on the moments of $\pi$.
    \item % ii
    Let $\snr = \frac{s_2(\pi)}{\sigma^2}$ denote the signal-to-noise ratio for the sequence model.  
    For any fixed $\delta$, we have $\Delta(\pi,\delta,\sigma) = O(\snr^2)$ as $\snr\to 0$. More precisely
% the minimizer of \eqref{tau-stat-def} is unique for sufficiently large $\sigma$ and
   \begin{equation}\label{gap-low-snr-third-moment}
    \limsup_{\snr \rightarrow 0} \frac{\Delta(\pi,\delta,\sigma)}{\snr^2} \leq s_2(\pi)\delta^2 \frac{s_3^2(\pi)}{2s_2^3(\pi)},
   \end{equation} 
   where the $\limsup$ is taken over $\sigma$ at which \eqref{rs-potential} has unique minimizer.
   \end{enumerate}
\end{theorem}

The proof of this theorem is given in Section M of the Supplementary Material \cite{supplement}. We believe its results provide some useful insight:
\begin{itemize}
\item The large $\delta$ regime of point $(i)$ is most commonly analyzed in the statistics literature, because it ensures high-dimensional consistency.
In this regime, Theorem \ref{thm-high-low-snr-gaps} establishes that 
the gap between convex M-estimation and Bayes error is essentially determined by the analogous gap in the sequence model for noise level $\sigma/\sqrt{\delta}$.
As will be discussed in the next section, in this regime, it makes sense to refine the M-estimate by post-processing.
\item In the low SNR regime (large $\sigma$), the structure of the signal $\bbeta_0$ (and in particular the distribution of the coefficients $\beta_{0j}$) 
is blurred by the Gaussian noise, and the gap vanishes. 
This should be compared with the results of Corollary \ref{cor-cvx-stat-phase-diagram}, which state that gaps, when they occur, occur for small values of $\delta$, which also corresponds to a low SNR regime. 
Both of these results can be traced to the fact that the measure $\pi * \normal(0,\tau^2_{\mathsf{reg,stat}})$ will in some sense be ``more log-concave'' when $\tau^2_{\mathsf{reg,stat}}$ is larger.
Because  $\tau^2_{\mathsf{reg,stat}}$ quantifies, in a certain sense, the intrinsic noisiness of the problem, we see that convex M-estimation comes closer to achieving (or exactly achieves) information theoretic limits at low SNR.
\end{itemize}

\section{Beyond mean square error}\label{sec-beyond-square-error}

A natural concern with the optimality theory we have presented is that it only addresses $\ell_2$ loss.
With a certain type of efficient post-processing,
the optimality theory for general continuous losses is essentially unchanged.
In particular, if we consider two-step procedures in which we first compute a penalized least squares estimator $\widehat \bbeta_{\mathsf{cvx}}$ and second implement simple post-processing detailed below, 
the optimal choice of penalty in the first step should not depend on the loss $\ell$.
The main reason for this is captured by the following result.
(This proposition relies on the notion of strong stationarity introduced in Section B which formalizes the notion of solving the fixed point equations \eqref{eq:body-fixed-pt} and includes a few more technical conditions.
It also uses the collection of penalty sequences $\cC_*$ which are \emph{uniformly strongly convex}, defined below in Definition \ref{def-uniform-strong-convexity}. This is a subset of the collection of convex penalty sequences.)
\begin{proposition}\label{prop-beyond-squared-error}
    Consider $\pi \in \cP_\infty(\reals)$, $\delta \in (0,\infty)$, and $\sigma \geq 0$. 
    Let $\{\rho_p\},\{\tilde \rho_p\}$ be sequences of lsc, proper, convex penalties. 
    Let  $\cT = (\pi,\{\rho_p\})$ and $\tilde \cT = (\pi,\{\tilde \rho_p\})$, and assume $\tau,\lambda,\tilde \tau, \tilde \lambda$ are such that $\tau, \lambda, \delta, \cT$ and $\tilde \tau, \tilde \lambda, \delta, \tilde \cT$ are strongly stationary.
    Without loss of generality, consider $\tilde \tau \leq \tau$.
    Assume either $\delta > 1$ or $\{\rho_p\},\{\tilde\rho_p\} \in \cC_*$ (see Definition \ref{strongly-convex-sequence-C*} below). 
    Let $\widehat \bbeta_{\mathsf{cvx}}$ and $\widehat {\tilde \bbeta}_{\mathsf{cvx}}$ be defined by \eqref{linear-cvx-estimator} with penalties $\rho_p$ and $\tilde \rho_p$ respectively.
    For such sufficiently large $p$, let
    \begin{equation}\label{post-processing}
        \widehat \bbeta_{\mathsf{cvx}+} = \mathsf{prox}\left[\lambda  \rho_p\right]\left(\widehat {\tilde \bbeta}_{\mathsf{cvx}} +  \frac{2\lambda}{n}
\bX^\mathsf{T}(\by - \bX \widehat{\tilde \bbeta}_{\mathsf{cvx}})+ \sqrt{\tau^2 - \tilde \tau^2}\bz\right),
    \end{equation}
    where for each $p$, $\bz \sim \mathsf{N}(\bzero,\bI_p/p)$ is independent of $\bX$.
    
    Under the HDA and RSN assumptions, for any sequence of symmetric, uniformly pseudo-Lipschitz sequence of losses $\ell_p:(\reals^p)^2\rightarrow \reals$ of order $k$ for some $k$, we have 
    \begin{align}\label{post-processing-loss-equivalent}
         \ell_p\left(\bbeta_0,\widehat \bbeta_{\mathsf{cvx}+} \right) \stackrel{\mathrm{p}}\simeq \ell_p\left(\bbeta_0,\widehat \bbeta_{\mathsf{cvx}}\right).
    \end{align}
    If the penalties $\rho_p,\tilde\rho_p$ are symmetric, then the preceding display holds also under the DSN assumption.
\end{proposition}

\noindent We prove Proposition \ref{prop-beyond-squared-error} in Section H of the Supplementary Material \cite{supplement}.
Proposition \ref{prop-beyond-squared-error} establishes that when $\tilde \tau \leq \tau$, we can always post-process $\widehat{\tilde \bbeta}_{\mathsf{cvx}}$ to construct an estimator $\widehat \bbeta_{\mathsf{cvx}+}$ whose performance matches that of $\widehat{ \bbeta}_{\mathsf{cvx}}$ with respect to loss $\ell$.
Proposition \ref{prop-beyond-squared-error} suggests that for any loss, the optimal choice of penalty in the M-estimation step in this two-step procedure is that which minimizes the effective noise parameter $\tau$.
It turns out this is equivalent to choosing a penalty which minimizes $\ell_2$ loss.

A formalization of this discussion is provided in the next theorem. 
\begin{theorem}\label{conj-beyond-squared-error}
    Assume $\eta: \reals \rightarrow \reals$ is the Bayes estimator of $\beta_0$ in the scalar model $y = \beta_0 + \tau_{\mathsf{reg,cvx}} z$ with respect to loss $\ell$. 
    If $\cC$ contains all sequences of convex penalties, then under the HDA and RSN assumption
    \begin{equation}\label{beyond-squared-error-lower-bound}
        \inf_{\{\rho_p\} \in \cC_{*}} \liminf^{\mathrm{p}}_{p \rightarrow \infty} \frac1p \sum_{j=1}^p \ell\left(\sqrt{p} \beta_{0j},\sqrt{p}\widehat \beta_{\mathsf{cvx},j}\right) \geq \E_{\beta_0,z}[\ell(\beta_0,\eta(\beta_0 + \tau_{\mathsf{reg,cvx}}z)].
    \end{equation}
    When $\eta$ is not the proximal operator of a convex function, inequality \eqref{beyond-squared-error-lower-bound} is strict.
    
    Further, when $\delta > 1$,
    \begin{align}
        \inf_{\substack{\{\rho_p\} \in \cC_{*}\\\eta'\text{ Lipschitz} }} \lim^{\mathrm{p}}_{p \rightarrow \infty} \frac1p & \sum_{j=1}^p \ell\left(\sqrt{p}\beta_{0j},\eta'\left(\sqrt{p} \widehat \bbeta_{\mathsf{cvx},j} +  2\lambda\frac{[\bX^\mathsf{T}(\by - \bX \widehat \bbeta_{\mathsf{cvx}})]_j}{n}\right)\right)\nonumber \\
        &\qquad\qquad\qquad= \E_{\beta_0,z}[\ell(\beta_0,\eta(\beta_0 + \tau_{\mathsf{reg,cvx}}z)].\label{beyond-squared-error-achieved}
    \end{align}
    The sequences $\{\rho_p\}$ which minimize the $\ell_2$ loss of $\widehat \bbeta_{\mathsf{cvx}}$ also achieve the infimum in \eqref{beyond-squared-error-achieved}.
    (Note that the infimum over $\eta'$ is taken \emph{after} the
    limit $p\to\infty$, and in particular $\eta'$ does not depend on $p$.)   

    If $\cC$ contains all sequences of symmetric convex penalties, the preceding statements hold also under the DSN assumption.
\end{theorem}

\noindent 
We prove Theorem \ref{conj-beyond-squared-error} in Section H of the Supplementary Material \cite{supplement}.
We expect inequality \eqref{beyond-squared-error-lower-bound} to hold also when the infimum is taken over $\cC_{\delta,\pi}$, 
but we are not aware how to control the estimation error with respect to arbitrary pseudo-Lipschitz losses for $\{\rho_p\} \in \cC_{\delta,\pi}$.
We expect equality \eqref{beyond-squared-error-achieved} to hold also when $\delta \leq 1$, 
but this requires establishing the tightness of the convex lower bound when $\delta \leq 1$, which are are unable to do (see discussion following Theorem \ref{thm-cvx-lower-bound}).
We believe these extensions may be possible using currently available tools but leave it for future work.

For large $\delta$, post-processing nearly closes the gap between convex M-estimation and Bayes AMP.
Indeed, as is shown in Section M of the Supplementary Material \cite{supplement}, when $\delta$ is large (high SNR) --so that \eqref{gap-high-snr} provides a good approximation of the gap $\Delta(\pi,\delta,\sigma)$-- we have $\tau_{\mathsf{reg,cvx}} \approx \tau_{\mathsf{reg,amp}*} \approx \sigma/\sqrt{\delta}$. 
Thus, the gap between the convex lower bound and the Bayes risk in this case is driven not by the difference between $\tau_{\mathsf{reg,cvx}}$ and $\tau_{\mathsf{reg,amp}*}$ but rather by the difference between estimation at that noise level using the optimal proximal operator (as done in \eqref{R-seq-cvx-opt-finite}) and the Bayes estimator (as done in \eqref{mmse-large-p}).
Theorem \ref{conj-beyond-squared-error} states that by post-processing we may effectively replace the proximal operator in Eq.~(H.1) of the Supplementary Material \cite{supplement} by a non-proximal denoiser, which we may take to be the Bayes estimator (or a Lipschitz approximation of it) with respect to $\ell_2$ loss.
This is an important insight because we suspect that the behavior of M-estimation with one step of post-processing is more robust to model misspecification than is the behavior of Bayes AMP, whose finite sample convergence has been observed to be highly sensitive to distributional assumptions on the design matrix $\bX$ (see e.g.\ \cite{rangan2014convergence,rangan2017vector}).

\section{Examples}\label{sec-examples}

Recall that, for $\delta>1$, the assumption that $\rho$ has $\delta$-bounded width does not pose any restriction. 
For $\delta\le 1$, our proof requires $\rho \in\cC_{\delta,\pi}$ for technical reasons, which are discussed Section I of the Supplementary Material \cite{supplement}.  
We believe the conclusion of Theorem \ref{thm-cvx-lower-bound} should hold more generally.  
Nevertheless, as illustrated in the present section,
the assumption $\rho \in\cC_{\delta,\pi}$ is quite weak and is satisfied by broad classes of penalties.

Most proofs are omitted from this section and can be found in Section N of the Supplementary Material \cite{supplement}.
Through this section, we take $\cC$ to contain all sequences of convex penalties, so that $\cC_{\delta,\pi}$ contains all sequences with $\delta$-bounded width.

\subsection{Strongly convex penalties}

We introduce
the notion of uniform strong convexity.
\begin{definition}[Uniform strong convexity]\label{def-uniform-strong-convexity}
    A sequence $\rho_p:\reals^p \rightarrow \reals \cup \{\infty\}$ of lsc, proper, convex functions \emph{has uniform strong-convexity parameter $\gamma\geq 0$} if $\bx \mapsto \rho_p(\bx) - \frac\gamma 2 \|\bx\|^2$ is convex for all $p$. We say that $\{\rho_p\}$ is \emph{uniformly strongly convex} if this holds for some $\gamma > 0$. 
\end{definition}
We define
\begin{equation}\label{strongly-convex-sequence-C*}
    \cC_* = \Big\{\{\rho_p\} \in \cC \mid\text{$\{\rho_p\}$ is uniformly strongly convex} \Big\}.
  \end{equation}
 When the penalties are uniformly strongly convex, the situation is particularly nice. 
\begin{proposition}\label{claim-a2-for-strong-convex-penalty}
    For all $\pi \in \cP_2(\reals)$ and $\delta \in (0,\infty)$, we have $\cC_* \subset \cC_{\delta,\pi}$.
\end{proposition}

\subsection{Convex constraints}\label{subsec-convex-constraints}

Consider
\begin{equation}\label{def-convex-constraint}
\rho_p(\bx) = \mathbb{I}_{C_p}(\bx)
:= 
\begin{cases}
    0 & \bx \in C_p\\
    \infty & \text{otherwise,}
\end{cases}
\end{equation}
where $C_p$ is a closed convex set. 
Convex M-estimation using this penalty is equivalent to defining $\widehat \bbeta_{\mathsf{cvx}}$ via the constrained optimization problem 
\begin{align}\label{constrained-m-estimator}
  \widehat \bbeta_{\mathsf{cvx}} = \arg\min_{\bbeta} \Big\{\frac1n \|\by - \bX \bbeta\|^2\;: \;\;\bbeta \in C_p\Big\}\, .
\end{align}
In this context, the condition \eqref{delta-bounded-width} is closely related to bounding the Gaussian width of 
convex cones \cite{Chandrasekaran2012TheProblems,amelunxen2014living}.
We briefly recall the relevant notions.

Given a closed convex set $K$, we denote by $\Pi_K$ the orthogonal projector onto $K$. Namely
$\Pi_K(\by) := \arg\min_{\bx\in K}\|\by-\bx\|_2$.
Recall that $K$ is a convex cone if $K$ is convex and for every $\alpha > 0$,
$K = \{ \alpha \bx \mid \bx \in K\}$.
For any set $A \subseteq \reals^p$, we define the closed, conic hull of $A$ centered at $\bb \in \reals^p$ by
$$
T_{A}(\bb) := \mathsf{cone}(\{\bx - \bb \mid \bx \in A\}) := \overline{\mathsf{conv}\left(\{\alpha (\bx - \bb) \mid \bx \in A, \alpha \geq 0 \}\right)}\, ,
$$
where the overline denotes closure and $\mathsf{conv}$ denotes the convex hull.
There are several equivalent definitions of the Gaussian width of a closed, convex cone $K$. 
The following translates most readily into our setup (recall that $\bz\sim \normal(0,\bI_p/p)$:
\begin{equation}
w(K) := \E_{\bz}\left[\|\Pi_K(\bz)\|^2\right]\, .
\end{equation}
The Gaussian width is closely related to the geometry of high-dimensional linear inverse problems.
In particular, under the HDA and DSN assumptions,
exact recovery $\widehat \bbeta_{\mathsf{cvx}} = \bbeta_0$ in the noiseless setting (i.e., $\bw = \bzero$) is achieved with high probability by \eqref{constrained-m-estimator} if and only if $\limsup_{p \rightarrow \infty} w(T_{C_p}(\bbeta_0)) < \delta$ \cite{amelunxen2014living,Chandrasekaran2012TheProblems}.
The same condition which guarantees \emph{stable recovery} under noisy measurements, namely, that the error $\|\widehat \bbeta_{\mathsf{cvx}} - \bbeta_0\|$ is bounded, up to a constant, by the norm of the noise $\|\bw\|$.
Thus, when $w(T_{C_p}(\bbeta_0)) > \delta$, we expect the estimation error of $\widehat \bbeta_{\mathsf{cvx}}$ to be uncontrolled. It is therefore reasonable to focus on the case $w(T_{C_p}(\bbeta_0)) < \delta$.

In the case of convex constraints, the $\delta$-bounded width assumption reduces to a slightly weaker condition than
$w(T_{C_p}(\bbeta_0)) < \delta$.
This is perhaps not surprising in light of the fact that for $\rho_p = \mathbb{I}_{C_p}(\bx)$,
the proximal operator $\mathsf{prox}[\lambda \rho_p](\bbeta_0 + \tau \bz) = \Pi_{C_p}(\bbeta_0 + \tau \bz)$ and $\lim_{\tau \rightarrow 0} \frac1\tau \E_{\bz}[\langle \bz , \mathsf{prox}[\lambda \rho_p](\bbeta_0 + \tau \bz) = \Pi_{C_p}(\bbeta_0 + \tau \bz) \rangle] = \E_{\bz}[\| \Pi_{T_{C_p}(\bbeta_0)}(\bz)\|^2]$.
The following proposition makes the relationship between Gaussian widths and the $\delta$-bounded width assumption precise.
\begin{proposition}\label{claim-a2-for-convex-constraints}
    Consider $C_p$ closed, symmetric, convex sets, $\pi \in \cP_2(\reals)$, and $\delta \in (0,\infty)$. Assume that
    \begin{equation}\label{beta0-eventually-in-Cp}
    \lim_{p \rightarrow \infty} \E_{\bbeta_0}\left[d\left(\bbeta_0,C_p\right)\right] = 0\, .
    \end{equation}
    Further assume that 
    \begin{equation}\label{asymptotic-width-bounded-by-delta}
        \lim_{\eps\to 0}\limsup_{p \rightarrow \infty} \E_{\bbeta_0}\left[w(T_{C_p \cap B^c(\bbeta_0,\eps)}(\bbeta_0))\right] < \delta\,,
    \end{equation}
    where $B^c(\bbeta_0,\eps)$ denotes the complement of the ball of radius $\eps$ centered at $\bbeta_0$.
    Then $\{\mathbb{I}_{C_p}\} \in \cC_{\delta,\pi}$.
\end{proposition}
\noindent
The quantity $\lim_{\eps\to 0}w(T_{C_p \cap B^c(\bbeta_0,\eps)}(\bbeta_0))$ agrees with $w(T_{C_p}(\bbeta))$ when $\bbeta_0 \in \partial C_p$.
Thus, when $\bbeta_0 \in \partial C_p$ almost surely,
assumption \eqref{asymptotic-width-bounded-by-delta} of Proposition \ref{claim-a2-for-convex-constraints} is exactly that $\limsup_{p \rightarrow \infty} w(T_{C_p}(\bbeta_0)) < \delta$.
This condition guarantees exact and stable recovery for the convex program \eqref{constrained-m-estimator}.
Thus, Proposition \ref{claim-a2-for-convex-constraints} implies that if constraint sets $\{C_p\}$ guarantee exact and stable recovery,
then $\{ \mathbb{I}_{C_p} \} \in \cC_{\delta,\pi}$.

In the definition of the $\delta$-bounded width assumption (or under the RSN assumption),
$\bbeta_0$ is random.
Thus, it will in general be close to but not exactly on the boundary of $C_p$.
For $\bbeta_0$ in an $\eps$-neighborhood of the boundary but not on the boundary, the quantity $w(T_{C_p \cap B^c(\bbeta_0,\eps)}(\bbeta_0))$ describes the behavior of the convex program \eqref{constrained-m-estimator} and the quantity $w(T_{C_p}(\bbeta))$ does not.
Indeed, $w(T_{C_p}(\bbeta))$ is highly sensitive to small perturbations of $\bbeta_0$: it jumps to 1 when $\bbeta_0$ is in the interior of $C_p$.
In contrast, the behavior of the convex program \eqref{constrained-m-estimator} is not sensitive to such small perturbations.
When $\bbeta_0$ is asymptotically arbitrarily close to but not necessarily exactly on the boundary of $C_p$, the condition of Proposition \ref{claim-a2-for-convex-constraints} is the correct extension of the condition $\limsup_{p\rightarrow \infty} w(T_{C_p}(\bbeta_0)) < \delta$.
It guarantees recovery with asymptotically vanishing error $\|\widehat \bbeta_{\mathsf{cvx}} - \bbeta_0\|^2 \rightarrow 0$ when $d(\bbeta_0,\partial C_p) \rightarrow 0$.
For such $\bbeta_0$, this is the natural replacement of the more stringent notion of exact recovery, 
which will not occur if $\bbeta_0 \not \in \partial C_p$.

\subsection{Separable penalties}

A common class of penalties considered in high-dimensional regression are the separable penalties 
\begin{equation}\label{def-separable-penalty}
    \rho_p(\bx) = \frac1p\sum_{j = 1}^p \rho(\sqrt px_j),
\end{equation}
for an lsc, proper, convex function $\rho: \reals \rightarrow \reals \cup\{\infty\}$ which does not depend on $p$.
Much previous work has analyzed the asymptotic properties of M-estimators which use separable penalties \cite{Bean2013,ElKaroui2013OnPredictors.,Donoho2016HighPassing}, and a few works have broken the separability assumption \cite{thrampoulidis2018precise}.
While Theorem \ref{thm-cvx-lower-bound} is more general, it applies to separable penalties under a mild condition.
\begin{proposition}\label{claim-a2-for-separable-penalties}
    Consider $\rho_p$ as in \eqref{def-separable-penalty} for some lsc, proper, convex $\rho:\reals\rightarrow \reals \cup \{\infty\}$. 
    Let $C \subseteq \reals$ be the set of minimizers of $\rho$ (which is necessarily a closed interval). If $C$ is non-empty, we have
    $$
    \sup_{\tau > \eps} \P_{\beta_0,z}(\beta_0 + \tau z \in C) < \delta \;\; \text{for all} \;\; \eps  > 0,
    $$
    if and only if $\{\rho_p\} \in C_{\delta,\pi}$.
\end{proposition}

\begin{remark}
    Proposition \ref{claim-a2-for-separable-penalties} applies  whenever $C$ is a singleton set because in this case $\P(\beta_0 + \tau z \in C) = 0$ for all $\tau > 0$. Thus, Proposition \ref{claim-a2-for-separable-penalties} covers most, if not all, separable penalties commonly considered in practice (and many more).
\end{remark}

\subsection{SLOPE and OWL norms}

Here we consider the Ordered Weighted $\ell_1$ (OWL) norms defined by
\begin{equation}\label{def-owl-penalty}
    \rho_p(\bx) = \frac1{\sqrt p}\sum_{j=1}^p \kappa_j^{(p)} |x|_{(j)},
\end{equation}
where $\kappa_1^{(p)} \geq \kappa_2^{(p)} \geq \cdots \geq \kappa_p^{(p)} \geq 0$ are the coordinates of $\bkappa^{(p)} \in \reals^p$ and $|x|_{(j)}$ are the decreasing order statistics of the absolute values of the coordinates of $\bx$. 
When $\kappa_j^{(p)} = \Phi^{-1}(1 - jq/(2p))$ for some $q \in (0,1)$ and $\Phi^{-1}$ the standard normal cdf, the estimator \eqref{linear-cvx-estimator} is referred to as SLOPE.
Penalties of the form \eqref{def-owl-penalty} have been used for a few purposes.
SLOPE has recently been proposed for sparse regression because it automatically adapts to sparsity level \cite{Bogdan2015SLOPE---AdaptiveOptimization,Su2016SLOPEMinimax,Bellec2016SlopeOptimality}.
More generally, the use of OWL norms has been argued to produce estimators which are more stable than LASSO under correlated designs \cite{Bondell2008SimultaneousOSCAR,Figueiredo2014SparseRegularization}. 
\begin{proposition}\label{claim-a2-for-owl-penalty}
    Consider $\rho_p$ as in \eqref{def-owl-penalty}. If for all $\eps > 0$ there exists $\xi > 0$ such that $j \leq (1 - \eps) p$ implies $\kappa_j^{(p)} > \xi $, then $\{\rho_p\} \in \cC_{\delta,\pi}$.
\end{proposition}

\section*{Acknowledgements}

MC was supported by the National Science Foundation Graduate Research Fellowship under Grant No. DGE
– 1656518. 
AM was supported by NSF grants CCF-2006489 and the ONR grant N00014-18-1-2729.

% \bibliographystyle{alpha}
% \bibliography{bibliography} 

\newcommand{\etalchar}[1]{$^{#1}$}

\newpage 

\begin{appendix}

\section{Equivalence of lower bounds: proof of Proposition \ref{claim:equiv-of-lb}}

In fact, for any finite $p$, 
$$
    \mathsf{R_{reg,cvx}^{opt}}(\tau;\pi,p) 
        =
        \inf_{\rho \in \cC_1} \E_{\beta_0,z} [(\mathsf{prox}[\rho](\beta_0+\tau z) - \beta_0)^2],
$$
both when $\cC$ contains all lsc, proper, convex functions and when $\cC$ contains all lsc, proper, convex functions.
Here $\cC_1$ is the set of all lsc, proper, convex functions on $\reals$.

First, note that $\mathsf{R_{reg,cvx}^{opt}}(\tau;\pi,p) \leq \inf_{\rho \in \cC_1} \E_{\beta_0,z} [(\mathsf[\rho](\beta_0+\tau z) - \beta_0)^2]$ because $\E_{\beta_0,z} [(\mathsf{prox}[\rho](\beta_0+\tau z) - \beta_0)^2]$ is in fact the risk in the sequence model of dimension $p$ of the procedure which uses separable penalty $\rho_p(\bx) = \frac1p\sum_{j=1}^p \rho(\sqrt{p}x)$. Indeed, $\mathsf{prox}[\rho_p](\by)_j = \mathsf{prox}[\rho](\sqrt{p}y_j)/\sqrt{p}$. (Note that $\rho_p$ is separable and symmetric).

Now note that for any lsc, proper, convex $\rho_p:\reals^p \rightarrow \reals \cup \{\infty\}$,
fixing $\by_{-j}$ the function $y_j \mapsto \mathsf{prox}[\rho](\by)$ is 1-Lipschitz, 
whence in fact $y_j \mapsto \mathsf{prox}[\rho](\by)_j$ is 1-Lipschitz.
Further, by the firm non-expansiveness of the proximal operator (Eq.~\eqref{prox-firm-non-expansive}), we also have that $y_j \mapsto \mathsf{prox}[\rho](\by)_j$ is non-decreasing.
By Fact 2.1 of \cite{celentano2019approximate}, the set $\{\mathsf{prox}[\rho]\}$ as $\rho$ varies over $\cC!$ is exactly the set of 1-Lipschitz and non-decreasing functions on $\reals$. 
Thus, we get $\E[(\mathsf{prox}[\rho_p](\bbeta_0 + \tau \bz)_j - \beta_{0j})^2|\by_{-j}] \geq \E_{\beta_0,z} [(\mathsf{prox}[\rho](\beta_0+\tau z) - \beta_0)^2]/p$ almost surely.
We conclude that 
\begin{equation*}
    \E[\|\mathsf{prox}[\rho_p](\bbeta_0 + \tau \bz) - \beta\|^2] = \sum_{j=1}^p \E[\E[(\mathsf{prox}[\rho_p](\bbeta_0 + \tau \bz)_j - \beta_{0j})^2|\by_{-j}]] \geq \E_{\beta_0,z} [(\mathsf{prox}[\rho](\beta_0+\tau z) - \beta_0)^2].
\end{equation*}
Having established both directions of the inequality completes the proof of Proposition \ref{claim:equiv-of-lb}.

\section{Exact asymptotics for the oracle estimator}\label{sec-exact-loss-characterization}

As discussed in Section \ref{sec:convex-lower-bound},
our proof of the convex lower bound (Theorem \ref{thm-cvx-lower-bound}) leverages exact asymptotics of the estimation error of penalized least squares estimators.
Because we cannot provide exact asymptotics under only the $\delta$-bounded width assumption, 
we will define an \emph{oracle estimator} which performs at least as well as the original estimator \eqref{linear-cvx-estimator} and to which we can apply exact asymptotic results.
For any $\gamma \geq 0$, the oracle estimator is
\begin{equation}\label{orc-cvx-estimator}
    \widehat \bbeta_{\mathsf{orc}}^{(\gamma)} \in \arg\min_{\bbeta}\left\{ \frac1n \|\by-\bX\bbeta\|^2 + \rho(\bbeta) + \frac\gamma2\|\bbeta - \bbeta_0\|^2\right\}.
\end{equation}
That is, we use the perturbed penalty
\begin{equation}\label{oracle-penalty}
\rho^{(\gamma)}(\bbeta) := \rho(\bbeta) + \frac \gamma 2 \|\bbeta - \bbeta_0\|^2,
\end{equation}
which includes a term which shrinks the estimate towards the true value $\bbeta_0$. 
We remark that for $\gamma > 0$, \emph{(i)} using this penalty in practice would require knowledge of the true parameter, so it cannot be implemented by the statistician, and \emph{(ii)} because of its dependence on $\bbeta_0$, the penalty defining the oracle estimator is itself random under the RSN assumption.

Previous work (e.g., \cite{karoui2013asymptotic}) has considered the addition of a small strongly-convex penalty in high-dimensional regression to permit rigorous exact asymptotics.
The oracle term we add also serves this purpose, but is tailored to our goal of establishing estimation error lower bounds.
Indeed, the oracle estimator performs at least as well as the original estimator for every realization of the data.
\begin{lemma}\label{lem-oracle-is-better}
    For $\rho:\reals^p \rightarrow \reals\cup\{\infty\}$ an lsc, proper, convex function, $\bbeta_0 \in \reals^p$, $\bw \in \reals^n$, $\gamma>0$, and all realizations of the design matrix $\bX \in \reals^{p \times n}$ and parameter $\bbeta_0$, we have 
    \begin{equation}\label{oracle-is-better}
        \|\widehat \bbeta_{\mathsf{orc}}^{(\gamma)} - \bbeta_0\|^2 \leq  \|\widehat \bbeta_{\mathsf{cvx}} - \bbeta_0\|^2,
    \end{equation}
    for any  $\bbeta_{\mathsf{orc}}^{(\gamma)}$ satisfying \eqref{orc-cvx-estimator}.
    That is, the $\ell_2$-loss of $\widehat \bbeta_{\mathsf{orc}}^{(\gamma)}$ is no larger than the $\ell_2$-loss of $\widehat \bbeta_{\mathsf{cvx}}$.
\end{lemma}
\begin{proof}[Proof of Lemma \ref{lem-oracle-is-better}]
If the minimizing set of \eqref{linear-cvx-estimator} is empty, then the right-hand side of \eqref{oracle-is-better} is $\infty$ by convention, and there is nothing to show.
    Thus, assume $\widehat \bbeta_{\mathsf{cvx}}$ satisfies \eqref{linear-cvx-estimator}.
    For any $\bbeta \in \reals^p$ with $\|\bbeta - \bbeta_0\| > \|\widehat \bbeta_{\mathsf{cvx}} - \bbeta_0\|$, we have
    \begin{align*}
        \frac1n\|\by - \bX\bbeta\|^2 + \rho(\bbeta) + \frac{\gamma}2\|\bbeta -\bbeta_0\|^2 &> \frac1n\|\by - \bX\bbeta\|^2 + \rho(\bbeta) + \frac{\gamma}2\|\widehat \bbeta_{\mathsf{cvx}}-\bbeta_0\|^2 \\ 
        &\geq \frac1n\|\by - \bX\widehat \bbeta_{\mathsf{cvx}}\|^2 + \rho(\widehat \bbeta_{\mathsf{cvx}}) + \frac{\gamma}2\|\widehat \bbeta_{\mathsf{cvx}} -\bbeta_0\|^2,
    \end{align*}
    where the second inequality follows from the definition of $\widehat \bbeta_{\mathsf{cvx}}$ in \eqref{linear-cvx-estimator}. Thus, $\bbeta$ cannot be a minimizer in \eqref{orc-cvx-estimator}). Moreover, because $\rho^{(\gamma)}$ has strong-convexity parameter $\gamma > 0$, the minimizing set of \eqref{orc-cvx-estimator}) is non-empty. Thus, we have \eqref{oracle-is-better}.
\end{proof}

The exact asymptotic characterization of the oracle estimator requires several definitions.
Denote by $\cT_p$ a pair $(\pi,\rho_p)$ where $\pi \in \cP_2(\reals)$ and $\rho_p:\reals^p \rightarrow \reals \cup \{\infty\}$ is an lsc, proper, convex function.
For any $\tau,\lambda \geq 0$ and $\bT \in S_+^2$, define
\begin{subequations}\label{summary-functions-fixed-finite}
\begin{gather}
\mathsf{R_{reg,cvx}}(\tau,\lambda,\cT_p) := \E_{\bbeta_0,\bz}\left[\|\mathsf{prox}[\lambda \rho_p](\bbeta_0 + \tau \bz) - \bbeta_0\|^2\right],\label{r-finite-p-def}\\
\mathsf{W_{reg,cvx}}(\tau,\lambda,\cT_p) := \frac1{\tau}\E_{\bbeta_0,\bz}\left[\langle \bz, \mathsf{prox}[\lambda\rho_p](\bbeta_0+\tau\bz)\rangle\right],\label{w-finite-p-def}\\
\mathsf{K_{reg,cvx}}(\bT,\lambda,\cT_p) :=  \E_{\bbeta_0,\bz_1,\bz_2}\left[\left\langle\mathsf{prox}\left[\lambda \rho_p\right]\left(\bbeta_0 + \bz_1\right)-\bbeta_0,\mathsf{prox}\left[\lambda\rho_p\right]\left(\bbeta_0 + \bz_2\right)-\bbeta_0\right\rangle\right],\label{k-finite-p-def}
\end{gather}
\end{subequations}
where $\beta_{0j} \stackrel{\mathrm{iid}}\sim \pi/\sqrt{p}$, $\bz \sim \mathsf{N}(\bzero,\bI_p/p)$, and $(\bz_1,\bz_2) \sim \mathsf{N}\left(\bzero ,\bT\otimes \bI_p/p\right)$.
Consider a sequence of penalties $\{\rho_p:\reals^p \rightarrow \reals \cup\{\infty\}\}$. 
Let $\cT = (\pi, \{\rho_p\})$.
Define
\begin{equation}\label{summary-functions-fixed-asymptotic}
\begin{gathered}
\mathsf{R^\infty_{reg,cvx}}(\tau,\lambda,\cT) := \lim_{p\rightarrow \infty} \mathsf{R_{reg,cvx}}(\tau,\lambda,\cT_p), \\
\mathsf{W^\infty_{reg,cvx}}(\tau,\lambda,\cT) := \lim_{p \rightarrow \infty} \mathsf{W_{reg,cvx}}(\tau,\lambda,\cT_p), \\
\mathsf{K^\infty_{reg,cvx}}(\bT,\lambda,\cT) := \lim_{p \rightarrow \infty} \mathsf{K_{reg,cvx}}(\bT,\lambda,\cT_p),
\end{gathered}
\end{equation}
whenever these limits exist. Here $\cT_p$ is related to $\cT$ in the obvious way.
Finally, denote
\begin{equation}\label{oracle-parameters}
\begin{gathered}
    \tau_{\mathsf{orc}} = \tau_{\mathsf{orc}}(\tau,\lambda,\gamma) = \frac{\tau}{\lambda \gamma + 1},
    \quad
    \lambda_{\mathsf{orc}} = \lambda_{\mathsf{orc}}(\lambda,\gamma) = \frac{\lambda}{\lambda\gamma + 1},
    \\
    \bT_{\mathsf{orc}} = \bT_{\mathsf{orc}}(\bT,\lambda,\gamma) = \frac{\bT}{(\lambda \gamma + 1)^2} .
\end{gathered}
\end{equation}

The exact asymptotic characterization is given by a solution $(\tau,\lambda)$ to the following system of equations.
\begin{subequations}\label{fixed-pt-prior-asymptotic}
\begin{gather}
    \delta \tau^{2} - \sigma^2 = \mathsf{R^\infty_{reg,cvx}}(\tau_{\mathsf{orc}},\lambda_{\mathsf{orc}},\cT),\label{se-var-fixed-pt-prior-asymptotic}\\
    2\lambda\left(1 - \frac1{\delta(\lambda\gamma+1)} \mathsf{W^\infty_{reg,cvx}}(\tauorc,\lambdaorc,\cT) \right) = 1.\label{se-lam-fixed-pt-prior-asymptotic}
\end{gather}
\end{subequations}
The following notion will be needed.
\begin{definition}[Strong stationarity]\label{def-stationarity-of-t}
    For any $\tau \geq0$, $\lambda > 0$, $\bT \in S_+^2$, and $\gamma \geq 0$, we denote $\tau_{\mathsf{orc}}$, $\lambda_{\mathsf{orc}}$, and $\bT_{\mathsf{orc}}$ as in \eqref{oracle-parameters}. 
        We say the quintuplet $\tau,\lambda,\gamma,\delta,\cT$ is \emph{strongly stationary} if at $\lambda_{\mathsf{orc}}$ and at all $\tau' \geq 0$, $\bT' \succeq \bzero$, the limits \eqref{summary-functions-fixed-asymptotic} exist, and at $\tau,\lambda, \gamma$, the equations \eqref{fixed-pt-prior-asymptotic} are satisfied.
\end{definition}

We are ready to provide our exact characterization of oracle estimators.
\begin{proposition}\label{prop-strongly-convex-loss}
    Consider $\pi \in \cP_\infty(\reals)$, $\delta \in (0,\infty)$, and $\sigma \geq 0$. 
    Consider a sequence of lsc, proper, convex functions $\rho_p:\reals^p \rightarrow \reals \cup \{\infty\}$.
    Let $\cT = (\pi,\{\rho_p\})$. Assume $\tau,\lambda,\gamma\geq0$ are such that $\tau,\lambda,\gamma,\delta,\cT$ is strongly stationary. 
    For each $p$, let $\widehat \bbeta_{\mathsf{cvx}}^{(\gamma)}$ be a solution to \eqref{linear-cvx-estimator}.
    If either $\delta > 1$, $\gamma > 0$, or the $\rho_p$ have uniform strong convexity parameter $\kappa > 0$, then
    \begin{enumerate}[(i)]
        \item % i
    The solution to \eqref{orc-cvx-estimator})  exists and is unique for all $n$ large enough:
    \begin{equation}\label{unique-m-estimator}
                \P_{\bX}(\text{solution to \eqref{orc-cvx-estimator}) exists and is unique}) = 1 \text{ eventually.}
    \end{equation}
    \item % ii
    Under RSN assumption the loss obeys
     \begin{equation}\label{converging-sequence-l2-loss}
       \|\widehat \bbeta^{(\gamma)}_{\mathsf{cvx}} - \bbeta_0\|^2 \stackrel{\mathrm{p}}\rightarrow \mathsf{R^\infty_{reg,cvx}}(\tau_{\mathsf{orc}},\lambda_{\mathsf{orc}},\cT) = \delta\tau^2-\sigma^2.
    \end{equation}
    If the penalties are symmetric,
    then \eqref{converging-sequence-l2-loss} holds also under the DSN assumption.
    \item % iii
    Consider the case that $\gamma = 0$ and either $\delta > 1$ or $\rho_p$ are uniformly strongly convex.
    Consider any sequence of functions $\varphi_p: (\reals^p)^2 \rightarrow \reals$ which are uniformly pseudo-Lipschitz of order $k$ for some $k$.
    Under the RSN assumption
    \begin{equation}\label{converging-sequence-pl-loss}
            \varphi_p\left(\bbeta_0,\widehat \bbeta_{\mathsf{cvx}} + 2\lambda \frac{\bX^\mathsf{T}(\by - \bX \widehat \bbeta_{\mathsf{cvx}})}{n}\right) \stackrel{\mathrm{p}}\simeq \E_{\bz}\left[\varphi_p(\bbeta_0,\bbeta_0 + \tau \bz)\right].
    \end{equation}
    If the penalties are symmetric,
    then \eqref{converging-sequence-l2-loss} holds under the DSN assumption.
    \end{enumerate}
\end{proposition}
\noindent The proof of Proposition \ref{prop-strongly-convex-loss} is provided in Appendix \ref{app-proof-of-prop-strongly-convex-loss}.

Note that although $\mathsf{K_{reg,cvx}}(\bT,\lambda,\cT_p)$ and $\mathsf{K^\infty_{reg,cvx}}(\bT,\lambda,\cT)$ do not appear in the equations \eqref{fixed-pt-prior-asymptotic}, the existence of the limit \eqref{summary-functions-fixed-asymptotic} will play an essential role in the proof of Proposition \ref{prop-strongly-convex-loss}.
In particular, it will allow us to control the convergence of the iterates of a certain AMP algorithm to the convex M-estimator, which will be important for establishing its characterization (see Section \ref{app-proof-of-prop-strongly-convex-loss} for details).

In addition to its use in establishing the convex lower bound, Proposition \ref{prop-strongly-convex-loss} will play a role in establishing the tightness of the convex lower bound under log-concavity assumptions or when $\delta > 1$. Proposition \ref{prop-strongly-convex-loss}(iii) plays a role in our consideration of non-quadratic losses and post-processing in Section \ref{sec-beyond-square-error}.

Our proof follows closely the proof of the similar result in Theorem 1.2 of \cite{Donoho2016HighPassing}.
The authors of \cite{Donoho2016HighPassing} establish an asymptotic characterization of the loss of M-estimators of the form $\widehat \bbeta = \arg\min_{\bbeta} \sum_{i = 1}^n \rho\left(y_i - [\bX \bbeta]_i\right)$ where $\rho$ is strongly convex and $\delta > 1$. 
Our Proposition \ref{prop-strongly-convex-loss} differs from their theorem in that $(i)$ we impose a penalty on the parameters rather than an arbitrary penalty on the residuals, $(ii)$ we permit non-separable penalties, and $(iii)$ we consider $\delta \leq 1$. 
Nevertheless, our argument follows almost exactly theirs (see Appendix \ref{app-proof-of-prop-strongly-convex-loss}). 
In handling non-separable penalties, we rely on recent results on approximate message passing algorithms with non-separable denoisers \cite{Berthier2017StateFunctions}, which the authors of \cite{Donoho2016HighPassing} did not have access to.

A result similar to Proposition \ref{prop-strongly-convex-loss} was also proved in \cite{thrampoulidis2018precise} using Gaussian comparison inequalities. 
The conditions in \cite{thrampoulidis2018precise} are not directly comparable to the ones of Proposition \ref{prop-strongly-convex-loss}.
We prefer proving an independent statement, since checking the conditions of the general theorem in \cite{thrampoulidis2018precise} is non-trivial.
As an advantage, Proposition \ref{prop-strongly-convex-loss} gives access --via Eq.~\eqref{converging-sequence-pl-loss}-- to the empirical distribution 
of the entries of $\widehat \bbeta_{\mathsf{cvx}}$ and $\widehat \bbeta_{\mathsf{cvx}} + 2\lambda \frac{\bX^\mathsf{T}(\by - \bX \widehat \bbeta_{\mathsf{cvx}})}{n}$, which is not provided by \cite{thrampoulidis2018precise}.
As stated above, this plays a role in our consideration of non-quadratic losses and post-processing.

Finally, Proposition explicitly describes the impact of the oracle term.

\section{Regularity lemmas}\label{app-regularity-lemmas}

This appendix provides several lemmas controlling the regularity of various objects appearing in the exact characterization of Proposition \ref{prop-strongly-convex-loss}.
These will be required in both the proof of Proposition \ref{prop-strongly-convex-loss} and in its applications.

First, for $\tau > 0$, 
\begin{align}
    \mathsf{W_{reg,cvx}}(\tau,\lambda,\cT_p) &= \frac{1}{\tau}\E_{\bbeta_0,\bz}[\langle \bz, \mathsf{prox}[\lambda \rho_p](\bbeta_0 + \tau\bz)\rangle ] = \frac1p \E_{\bbeta_0,\bz}[\mathrm{div}\, \mathsf{prox}[\lambda \rho_p](\bbeta_0 + \tau\bz)] 
    \leq 1,\label{finite-p-width-bounded-by-1}
\end{align}
where in the first equality we have used the definition \eqref{w-finite-p-def}, in the second equality we have used \eqref{prox-gaussian-IBP}, and in the inequality we have used \eqref{prox-divergence-bound}. 
Taking limits, we have (using \eqref{prox-divergence-non-negative})
\begin{align}\label{width-bounds}
    \mathsf{W^\infty_{reg,cvx}}(\tau,\lambda,\cT) \leq 1,
    \quad \text{and} \quad
    \mathsf{W_{reg,cvx}}(\tau,\lambda,\cT_p) \geq 0,
    \quad\text{and}\quad 
    \mathsf{W^\infty_{reg,cvx}}(\tau,\lambda,\cT) \geq 0,
\end{align}
whenever these are defined.

Next, in this and other appendices we will sometimes need the following basic algebraic inequalities which hold for any vectors $\ba,\ba',\bb,\bb' \in \reals^p$ and are straightforward to verify.
\begin{gather}
        |\langle \ba, \bb \rangle - \langle \ba',\bb'\rangle| \leq 2 \underbrace{\max\{\|\ba\|,\|\ba'\|,\|\bb\|,\|\bb'\|\}}_{(*)} \underbrace{(\|\ba - \ba'\| \vee \|\bb - \bb'\|)}_{(**)} ,\label{inner-product-difference-inequality}\\
        \left|\|\ba\|^2 - \|\bb\|^2\right| \leq 2\underbrace{(\|\ba\| \vee \|\bb\|)}_{(*)}\underbrace{\|\ba - \bb\|}_{(**)}.\label{norm-squared-difference-inequality}
\end{gather}
We label the terms on the right-hand sides with $(*)$ and $(**)$ to facilitate future reference. The inequalities are straightforward to verify. In fact, \eqref{norm-squared-difference-inequality} is a special case of \eqref{inner-product-difference-inequality}.

We say a sequence of functions $\{\rho_p\} \in \cC$ \emph{does not shrink towards infinity} if 
\begin{equation}\label{eqdef-bounded-shrinkage-towards-infinity}
    \sup_p \|\mathsf{prox}[\rho_p](\bzero)\| < \infty.
\end{equation}
We define the collection of penalty sequences which do not shrink towards infinity
\begin{equation}\label{eqdef-cB}
    \cB =  \Big\{\{\rho_p\} \in \cC \Bigm| \eqref{eqdef-bounded-shrinkage-towards-infinity} \text{ holds}\Big\}.
\end{equation}

Finally, we provide a series of lemmas which we will need in later sections.
\begin{lemma}\label{lem-integrands-uniformly-pseudo-lipschitz}
    Consider $\{\rho_p\} \in \cB$ (see \eqref{eqdef-cB}). Then for any fixed $\tau,\lambda \geq 0$, the functions
    \begin{gather*}
        \varphi_{\mathsf{R}}^{(p)}(\bbeta_0,\bz) = \|\mathsf{prox}[\lambda \rho_p](\bbeta_0 + \tau \bz) - \bbeta_0\|^2,\\
        \varphi_{\mathsf{W}}^{(p)}(\bbeta_0,\bz) = \frac1{\tau}\left \langle  \bz, \mathsf{prox}[\lambda \rho_p](\bbeta_0 + \tau \bz)\right\rangle,\\
        \varphi_{\mathsf{K}}^{(p)}(\bbeta_0,\bz_1,\bz_2) = \left\langle \mathsf{prox} \left[\lambda \rho_p\right](\bbeta_0 + \bz_1 ) - \bbeta_0,\mathsf{prox} \left[\lambda \rho_p\right](\bbeta_0 + \bz_2 ) - \bbeta_0\right\rangle,
    \end{gather*}
    are uniformly pseudo-Lipschitz of order 2.
\end{lemma}

\begin{proof}[Proof of \ref{lem-integrands-uniformly-pseudo-lipschitz}]
    Let $M := \sup_p \|\mathsf{prox}[\rho_p](\bzero)\|$. We have $M < \infty$ because $\{\rho_p\} \in \cB$. 
    By \eqref{prox-continuous-in-lambda}, we have 
    \begin{align*}
    \|\mathsf{prox}[\lambda \rho_p](\bzero)\| &\leq \|\mathsf{prox}[\rho_p](\bzero)\| + \|\mathsf{prox}[\lambda \rho_p](\bzero) - \mathsf{prox}[\rho_p](\bzero)\| \\
    &\leq \|\mathsf{prox}[\rho_p](\bzero)\| + \|\mathsf{prox}[\lambda \rho_p](\bzero)\|\left|\lambda - 1\right| 
    \leq (2M + 1)\lambda.
    \end{align*}
    Thus, $\|\mathsf{prox}[\lambda \rho_p](\bzero)\|$ is bounded over $p$.
  Further, by \eqref{prox-is-lipschitz}, the functions $(\bbeta_0,\bz) \mapsto \mathsf{prox}[\lambda\rho_p](\bbeta_0 + \tau \bz) - \bbeta_0$ and $(\bbeta_0,\bz) \mapsto \mathsf{prox}[\lambda\rho_p](\bbeta_0 + \tau \bz)$ are uniformly pseudo-Lipschitz of order 1.
  Further, the function $(\bbeta_0,\bz) \mapsto \frac1\tau \bz$ is trivially uniformly pseudo-Lipschitz of order 1.
    Applying Lemma \ref{lem-pseudo-lipschitz-closed-under-inner-product}, the Lemma follows.
\end{proof}

\begin{lemma}\label{lem-r-continuity}
There exists universal constant $C$ such that the functions $\tau\mapsto\mathsf{R^{opt}_{seq,cvx}}(\tau;\pi,p)$ and $\tau\mapsto\mathsf{R^{opt}_{seq,cvx}}(\tau;\pi)$ defined in \eqref{R-seq-cvx-opt-finite} and \eqref{R-seq-cvx-opt-asymptotic} 
satisfy for any $\tau,\tau' \geq 0$ (using $f$ to denote each function)
\begin{equation}\label{eq:f-smoothness}
    |f(\tau')-f(\tau)|
    \leq 
    C(1 + |\tau' - \tau| + f(\tau))|\tau'-\tau|.
\end{equation}
This makes sense even when $\cC$ is such that $\mathsf{R^{opt}_{seq,cvx}}(\tau;\pi)$ is infinite (note, $f$ is always non-negative).

In particular, if $\mathsf{R^{opt}_{seq,cvx}}(\tau;\pi)$ is finite anywhere, it is finite and continuous everywhere.
% are $(1+s_2(\pi))$-Lipschitz continuous in $\tau$ where $s_2(\pi)$ is the second moment of $\pi$.
\end{lemma}

\begin{proof}[Proof of Lemma \ref{lem-r-continuity}]
First, we develop a bound on
$$
\left|\E_{\bbeta_0,\bz}\left[\|\mathsf{prox}[\rho](\bbeta_0 + \tau \bz) - \bbeta_0\|^2\right] - \E_{\bbeta_0,\bz}[\|\mathsf{prox}[\rho](\bbeta_0 + \tau' \bz) - \bbeta_0\|^2]\right|,
$$
for $\rho:\reals^p \rightarrow \reals^p$ an lsc, proper, convex function.
We apply Jensen's inequality to get
\begin{align*}
&\left|\E_{\bbeta_0,\bz}\left[\|\mathsf{prox}[\rho](\bbeta_0 + \tau \bz) - \bbeta_0\|^2\right] - \E_{\bbeta_0,\bz}[\|\mathsf{prox}[\rho](\bbeta_0 + \tau' \bz) - \bbeta_0\|^2]\right|\\
&\qquad\qquad \leq \E_{\bbeta_0,\bz}\left[\left|\|\mathsf{prox}[\rho](\bbeta_0 + \tau \bz) - \bbeta_0\|^2-\|\mathsf{prox}[\rho](\bbeta_0 + \tau' \bz) - \bbeta_0\|^2\right|\right].
\end{align*}
We bound the integrand by applying \eqref{norm-squared-difference-inequality}:
\begin{align*}
&\left|\|\mathsf{prox}[\rho](\bbeta_0 + \tau \bz) - \bbeta_0\|^2-\|\mathsf{prox}[\rho](\bbeta_0 + \tau' \bz) - \bbeta_0\|^2\right|\\
&\qquad\qquad\leq 2 \underbrace{\left(\|\mathsf{prox}[\rho](\bbeta_0 + \tau \bz) - \bbeta_0\| + \|\tau \bz - \tau' \bz\|\right)}_{\text{bound on $(*)$}}\underbrace{ \|\tau \bz - \tau' \bz\| }_{\text{bound on $(**)$}}\\
&\qquad\qquad \leq 2 \|\mathsf{prox}[\rho](\bbeta_0 + \tau \bz) - \bbeta_0\|\|\bz\| |\tau - \tau'| + 2\|\bz\|^2 (\tau - \tau')^2\\
&\qquad\qquad \leq \left(\|\mathsf{prox}[\rho](\bbeta_0 + \tau \bz) - \bbeta_0\|^2 + \|\bz\|^2\right)|\tau - \tau'| + 2\|\bz\|^2 (\tau - \tau')^2.
\end{align*}
Combining the previous two displays,
\begin{align*}
&\left|\E_{\bbeta_0,\bz}\left[\|\mathsf{prox}[\rho](\bbeta_0 + \tau \bz) - \bbeta_0\|^2\right] - \E_{\bbeta_0,\bz}[\|\mathsf{prox}[\rho](\bbeta_0 + \tau' \bz) - \bbeta_0\|^2]\right|\nonumber\\
&\qquad\qquad \leq \left(\E_{\bbeta_0,\bz}\left[\|\mathsf{prox}[\rho](\bbeta_0 + \tau \bz) - \bbeta_0\|^2\right]+ 1\right)|\tau - \tau'| + 2(\tau - \tau')^2.
\end{align*}
Thus, $\tau \mapsto \E_{\bbeta_0,\bz}\left[\|\mathsf{prox}[\rho](\bbeta_0 + \tau \bz) - \bbeta_0\|^2\right]$ satisfies \eqref{eq:f-smoothness}.
To prove \eqref{eq:f-smoothness} for $\mathsf{R^{opt}_{seq,cvx}}(\tau;\pi,p)$ and $\mathsf{R^{opt}_{seq,cvx}}(\tau;\pi)$, we use that the property \eqref{eq:f-smoothness} is preserved by taking point-wise infima of collections of functions, as well as limit infima (provided infinite limit infima are permitted, with \eqref{eq:f-smoothness} interpreted in the natural way in this case).
The finiteness and continuity of $\mathsf{R^{opt}_{seq,cvx}}(\tau;\pi,p)$ and $\mathsf{R^{opt}_{seq,cvx}}(\tau;\pi)$ in the case that these are finite anywhere is then automatic.
\end{proof}

\begin{lemma}\label{lem-emp-to-exp-non-symm}
    Let $\pi \in \cP_2(\reals)$. 
    Let $\varphi_p : \reals^p \rightarrow \reals$ be a sequence of functions which is uniformly pseudo-Lipschitz of order 2.
    Then
    \begin{equation*}
        \varphi_p(\bbeta_0)
        \stackrel{\mathrm{as}}\simeq 
        \E_{\bbeta_0}[\varphi_p(\bbeta_0)],
    \end{equation*}
    where for each $p$, $\beta_{0j} \stackrel{\mathrm{iid}}\sim \pi/\sqrt{p}$, and the $\bbeta_0$ are independent across $p$.

    In particular, if $\{\rho_p\} \in\cB$ and the limits \eqref{summary-functions-fixed-asymptotic} exist (with $\bz \sim \normal(0,\bI_p/p)$),
    then with respect to the randomness in $\bbeta_0$,
    \begin{gather*}
        \E_{\bz}\left[\|\mathsf{prox}[\lambda \rho_p](\bbeta_0 + \tau \bz) - \bbeta_0\|^2\right] \stackrel{\mathrm{as}}\simeq \mathsf{R_{reg,cvx}}(\tau,\lambda,\cT),\\
        \frac1{\tau}\E_{\bz}\left[\langle \bz, \mathsf{prox}[\lambda\rho_p](\bbeta_0+\tau\bz)\rangle\right] \stackrel{\mathrm{as}}\simeq \mathsf{W_{reg,cvx}}(\tau,\lambda,\cT),\\
        \E_{\bz_1,\bz_2}\left[\left\langle\mathsf{prox}\left[\lambda \rho_p\right]\left(\bbeta_0 + \bz_1\right)-\bbeta_0,\mathsf{prox}\left[\lambda\rho_p\right]\left(\bbeta_0 + \bz_2\right)-\bbeta_0\right\rangle\right] \stackrel{\mathrm{as}}\simeq \mathsf{K_{reg,cvx}}(\bT,\lambda,\cT).
    \end{gather*}
\end{lemma}

\begin{proof}
    Throughout the proof, $\beta_0$ will denote a random variable drawn from $\pi$.
    Let $s_2(\pi)^2 = \E[\beta_0^2]$. 
    By assumption, the restriction of $\varphi_p$ to $\{\|\bbeta_0\|_2^2 \leq s_2(\pi)^2 + 1\}$ is $L$-Lipschitz for some $L$ which does not depend on $p$.
    Let $\bar \varphi_p$ be an $L$-Lipschitz extension of $\varphi_p|_{\{\|\bbeta_0\|_2^2 \leq s_2(\pi)^2 \leq s_2(\pi)^2 + 1\}}$ to all of $\reals^p$; 
    that is, it is $L$-Lipschitz and agrees with $\varphi_p$ on $\{\|\bbeta_0\|_2^2 \leq s_2(\pi)^2 + 1\}$.
    For example, one can check that $\bar \varphi_p(\bx) = \sup_{\|\bx'\| \leq R} \{\varphi_p(\bx') - \|\bx-\bx'\|\}$ is a valid Lipschitz extension.

    Fix $1 > \epsilon > 0$.
    Pick $M > 0$ such that $\E[|\beta_0|^2 \mathbf{1}_{|\beta_0|> M}] < \epsilon^2$.
    For each $p$, let $\bh$ have coordinates drawn iid from the Laplace distribution with scale parameter $1/\sqrt{p}$ (i.e., density $\frac{\sqrt{p}}2e^{-\sqrt{p}|x|}$), independent of $\bbeta_0$ and across $p$.
    Define $\bbeta_{0}^\epsilon = (|\beta_{0j}|\mathbf{1}_{\sqrt{p}|\beta_{0j}| \leq M} + h_j)_{j\in[p]}$.
    Then $\sqrt{p}\beta_{0j}^\epsilon$ satisfies a Poincar\'e  inequality (this follows by Corollary 1.6 of \cite{barky2008});
    that is, 
    for any weakly differentiable $f$, $\Var(f(\sqrt{p}\beta_{0j}^\epsilon)) \leq C \E[ f'(\sqrt{p}\beta_{0j}^\epsilon)^2]$ for some constant $C$ which does not depend on $p$ (but may depend on $\epsilon,\pi,M$).
    Then the product measure $\pi^{\otimes p}$ satisfies a Poincar\'e inequality with the same constant $C$: $\Var(\bar \varphi_p(\bbeta)) \leq C \E[\|\nabla \varphi_p(\bbeta_0)\|^2]/p$. 
    Then, by Corollary 4.6 of Ledoux \cite{ledoux1999}, $\bar \varphi_p$ has exponential concentration.
    In particular, there exists a constant $c$, which does not depend on $p$, (but may depend on $\epsilon,\pi,M,L$)
    such that
    \begin{equation*}
        \P\left(\Big|\bar \varphi_p(\bbeta_0) - \E[\bar \varphi_p(\bbeta_0^\epsilon)] \Big| > t \right) 
            \leq 
            2 e^{-c \min(\sqrt{p}t,pt^2)}.
    \end{equation*}
    Taking $t \rightarrow 0$ after $p \rightarrow \infty$ and using the Borel-Cantelli lemma,
    we get that $\bar \varphi_p(\bbeta_0^\epsilon) \stackrel{\mathsf{as}}\simeq \E[\bar \varphi_p(\bbeta_0^\epsilon)]$.

    By the strong law of large numbers and the definition of $\bbeta_0^\epsilon$, 
    we have $\|\bbeta_0 - \bbeta_0^\epsilon\|^2 \stackrel{\mathsf{as}}\rightarrow \E[(\beta_0 - \beta_{0j} \mathbf{1}_{|\beta_0|\leq M} - \epsilon h)^2] < 2\epsilon^2$, where $h$ has Laplace distribution with scale parameter 1 and is independent of $\beta_0$.
    Thus, almost surely we have that for large enough $p$, $|\bar \varphi_p(\bbeta_0) - \bar \varphi_p(\bbeta_0^\epsilon)| < \sqrt{2}L\epsilon$.
    Also by the strong law of large numbers, we have $\|\bbeta_0^2\| \rightarrow s_2(\pi)^2$ and $\|\bbeta_0^\epsilon\|^2 \stackrel{\mathrm{as}}\rightarrow \E[(\sqrt{p}\beta_{0j}^\epsilon)^2] < s_2(\pi)^2 + 1$, where the inequality holds because $\epsilon < 1$. 
    Thus, almost surely we have that for large enough $p$, $\varphi_p(\bbeta_0) = \bar \varphi_p(\bbeta_0)$ and $\varphi_p(\bbeta_0^\epsilon) = \bar \varphi_p(\bbeta_0^\epsilon)$.
    Combining the preceding observations, almost surely we have that for large enough $p$, $|\varphi_p(\bbeta_0) - \E[\bar \varphi_p(\bbeta_0^\epsilon)]| < \sqrt{2}L\epsilon$.
    Finally, $\E[\|\bbeta_0 - \bbeta_0^\epsilon\|] \leq \E[\|\bbeta_0 - \bbeta_0^\epsilon\|^2]^{1/2} < \sqrt{2}\epsilon$.
    Thus, almost surely we have that for large enough $p$, $|\varphi_p(\bbeta_0) - \E[\bar \varphi_p(\bbeta_0)]| < 2\sqrt{2}L\epsilon$.
    
    Because $\varphi_p$ is pseudo-Lipschitz of order 2, $\bar \varphi_p$ is in fact $C(1 + s_2(\pi)^2 + 1)$-Lipschitz, so that $|\varphi(\bx) - \bar \varphi(\bx)| \leq |\varphi(\bx) - \varphi(\bzero)| + |\bar \varphi(\bx) - \bar\varphi(\bzero)| \leq C(2+s_2(\pi)^2)\|\bx\| + C(1 + \|\bx\|)\|\bx\|$. 
    Thus, there exists $C$ which does not depend on $p,\epsilon,M$ such that $|\varphi_p(\bx) - \bar\varphi_p(\bx)| \leq C\|\bx\|^2\mathbf{1}_{\|\bx\|^2 \geq s_2(\pi)^2 + 1}$.
    Thus, $|\E[\bar \varphi_p(\bbeta_0)] - \E[\varphi_p(\bbeta_0)]| < C\E[\|\bbeta_0\|^2\mathbf{1}_{\|\bbeta_0\|^2 \geq s_2(\pi)^2 + 1}] \rightarrow 0$ because $\|\bbeta_0\|^2$ is uniformly integrable and $\P(\|\bbeta_0\|^2 \geq s_2(\pi)^2 + 1) \rightarrow 0$.
    We conclude that almost surely we have that for large enough $p$, $|\varphi_p(\bbeta_0) - \E[ \varphi_p(\bbeta_0)]| < 2\sqrt{2}L\epsilon$.
    Because the left-hand side does not depend on $\epsilon$, in fact $\varphi_p(\bbeta_0) \stackrel{\mathrm{as}}\simeq \E[ \varphi_p(\bbeta_0)]$ as desired.

    The identities involving $\mathsf{R_{reg,cvx}}(\tau,\lambda,\cT)$, $\mathsf{W_{reg,cvx}}(\tau,\lambda,\cT)$, and $\mathsf{K_{reg,cvx}}(\bT,\lambda,\cT)$ now hold because $\bbeta_0 \mapsto \E_{\bz}\left[\|\mathsf{prox}[\lambda \rho_p](\bbeta_0 + \tau \bz) - \bbeta_0\|^2\right]$ is uniformly pseudo-Lipschitz of order 2, and likewise for the remaining relevant functions. The proof is complete.
\end{proof}

\begin{lemma}\label{lem-pseudo-lipschitz-empirical-to-expectation}
Consider a sequence $\{\varphi_p: (\reals^p)^{k+1} \rightarrow \reals\}$ of uniformly pseudo-Lipschitz functions of order 2. 
Moreover, assume $\varphi$ are symmetric in the sense that for any $\sigma \in S_p$, the symmetric group on $[p]$, we have
$$
\varphi_p(\bx_0^\sigma,\ldots,\bx_k^\sigma) = \varphi_p(\bx_0,\ldots,\bx_k),
$$
where $(\bx^\sigma)_i := \bx_{\sigma(i)}$. 
Fix deterministic sequence $\{\bx_0(p)\}$ such that $p^{-1}\sum_{i=1}^p\delta_{\sqrt{p}x_{0,i}}\toW \pi$ for some $\pi \in \cP_2(\reals)$. 
Then
\begin{equation}
    \lim_{p \rightarrow \infty}  \E_{\bz_1,\ldots,\bz_k}[\varphi_p(\bx_0,\bz_1,\ldots,\bz_k)] = \lim_{p \rightarrow \infty}  \E_{\tilde\bx_0,\bz_1,\ldots,\bz_k}[\varphi_p(\tilde \bx_0,\bz_1,\ldots,\bz_k)]
\end{equation}
whenever either of the limits exists, where on the right-hand side we take $\tilde \bx_0$ with coordinates distributed iid from $\pi/\sqrt{p}$. In particular, both limits exist as soon as one of them exists.
\end{lemma}

\begin{proof}[Proof of Lemma \ref{lem-pseudo-lipschitz-empirical-to-expectation}]
 We now drop index $p$ from our notation to avoid clutter.
Consider a probability space on which we have random vectors $\tilde \bx_0,\bz_1,\ldots,\bz_k \in \reals^p$ for each $p$ such that the coordinates of $\tilde \bx_0$ are distributed iid from $\pi /\sqrt p$ and the $\tilde \bx_0$ are independent for different values of $p$.
By \cite[Lemma 8.4]{Bickel1981SomeBootstrap},
 \begin{equation}\label{wasserstein-convergence-under-iid}
     d_{\mathrm{W}}\left(\widehat \pi_{\bx_0}, \widehat \pi_{\tilde \bx_0}\right) \leq d_{\mathrm{W}}\left(\widehat \pi_{\bx_0}, \pi\right) + d_{\mathrm{W}}\left(\pi, \widehat \pi_{\tilde \bx_0}\right) \stackrel{\mathrm as}\rightarrow 0,
 \end{equation}
 where $\widehat \pi_{\bv}\equiv p^{-1}\sum_{i=1}^p\delta_{\sqrt{p}v_{i}}$ denotes the empirical distributions of the entries of $\bv\in\reals^p$.
 For each $p$ and realization $\tilde \bx_0$, there is a permutation $\sigma_p$ (depending on $\tilde \bx_0$) such that 
 $
\left\|\bx_0 - \tilde \bx_0^{\sigma_p}\right\| = d_\mathrm{W}\left(\widehat \pi_{\bx_0}, \widehat \pi_{\tilde \bx_0}\right).
 $
 By the symmetry of $\varphi_p$, we have
 $
 \varphi_p(\tilde \bx_0,\bz_1,\ldots,\bz_k) = \varphi_p\left(\tilde \bx_0^{\sigma_p},\bz_1^{\sigma_p},\ldots,\bz_k^{\sigma_p}\right) \stackrel{\mathrm{d}}= \varphi_p\left(\tilde \bx_0^{\sigma_p},\bz_1,\ldots,\bz_k\right),
 $
 where the equality of distribution follows because $\sigma_p$ is independent of $\bz_1,\ldots,\bz_k$, and the distribution of $\bz_1,\ldots,\bz_k$ is invariant under permutation of the coordinates.
 We have
 \begin{align}
\left|\varphi_p\left(\bx_0,\bz_1,\ldots,\bz_k\right) - \varphi_p\left(\tilde \bx_0^{\sigma_p},\bz_1,\ldots,\bz_k\right)\right|
&\leq  L\left(1 + \|\bx_0\| + \|\tilde \bx_0^{\sigma_p}\|  + 2\sum_{i=1}^k \|\bz_i\|\right)\|\bx_0 - \tilde \bx_0^\sigma\|\nonumber\\
 &= L\left(1 + \|\bx_0\|+ \|\tilde \bx_0\|  + 2\sum_{i=1}^k \|\bz_i\|\right)d_\mathsf{W}\left(\widehat \pi_{\bx_0}, \widehat \pi_{\tilde \bx_0}\right)
 \stackrel{\mathrm p}\rightarrow 0,\label{probabilistic-convergence-under-iid-sampling}
 \end{align}
 where we have used \eqref{wasserstein-convergence-under-iid} and that $\left(1 + \|\bx_0\| + \|\tilde \bx_0\|  + 2\sum_{i=1}^k \|\bz_i\|\right) = O_p(1)$. 
 Further, we check uniform integrability. First,
 \begin{align*}
  \left|\varphi_p\left(\bx_0,\bz_1,\ldots,\bz_k\right) - \varphi_p\left(\tilde \bx_0^{\sigma_p},\bz_1,\ldots,\bz_k\right)\right| 
  \leq L\left(1 + \|\bx_0\| + \|\tilde \bx_0\|  + \sum_{i=1}^k \|\bz_i\|\right)\left(\|\bx_0\| + \|\tilde \bx_0\|\right).
 \end{align*}
 Because $\|\bx_0\|$ is bounded, we only need to check that $\|\tilde \bx_0\|^2$ and $\|\bz_i\|\|\tilde \bx_0\|$ are uniformly integrable over $p$. Observe that $\|\tilde \bx_0\|^2 = \frac1p \sum_{j=1}^p (\sqrt p \tilde x_{0j})^2 $. The random variables $(\sqrt p \tilde x_{0j})^2$ are iid from an $L_1$ probability distribution (which does not depend on $p$), so that $\|\tilde \bx_0\|^2$ are uniformly integrable.
Also, $\|\bz_i\|\|\tilde \bx_0\| \leq \frac12 \left(\|\tilde \bx_0\|^2 + \|\bz_i\|^2\right)$, so these are uniformly integrable for the same reason.
Thus, the probabilistic convergence \eqref{probabilistic-convergence-under-iid-sampling} and Vitali's Convergence Theorem (see e.g.~\cite[Theorem 5.5.2]{Durrett2010Probability:Examples}) implies that
 \begin{align}
& \left| \E_{\bz_1,\ldots,\bz_k}[\varphi_p(\bx_0,\bz_1,\ldots,\bz_k)] - \E_{\tilde \bx_0,\bz_1,\ldots,\bz_k}[ \varphi_p(\tilde \bx_0,\bz_1,\ldots,\bz_k)] \right|\nonumber\\
 &\qquad\qquad\qquad=  \left| \E_{\bz_1,\ldots,\bz_k}\left[\varphi_p\left(\bx_0,\bz_1,\ldots,\bz_k\right)\right] - \E_{\tilde \bx_0,\bz_1,\ldots,\bz_k}\left[ \varphi_p\left(\tilde \bx_0^{\sigma_p},\bz_1,\ldots,\bz_k\right)\right] \right|\nonumber \\
 &\qquad\qquad\qquad\leq \E_{\tilde \bx_0,\bz_1,\ldots,\bz_k}\left[\left|\varphi_p\left(\bx_0,\bz_1,\ldots,\bz_k\right) - \varphi_p\left(\tilde \bx_0^{\sigma_p},\bz_1,\ldots,\bz_k\right)\right|\right] 
 \rightarrow 0.
 \end{align}
 Thus, if $\E_{\bz_1,\ldots,\bz_k}[\varphi_p(\bx_0,\bz_1,\ldots,\bz_k)]$ converges, then $\E_{\tilde \bx_0,\bz_1,\ldots,\bz_k}[ \varphi_p(\tilde \bx_0,\bz_1,\ldots,\bz_k)]$ also converges and has the same limit, and conversely.
\end{proof}

\begin{lemma}\label{lem-uniform-modulus-of-continuity}
Consider $\pi \in \cP_2(\reals)$ and $\{\rho_p\} \in \cB$ (see \eqref{eqdef-cB}).
For each $p$, let $\cT_p = (\pi,\rho_p)$ and consider the functions $\mathsf{R_{reg,cvx}}(\tau,\lambda,\cT_p)$, $\mathsf{W_{reg,cvx}}(\tau,\lambda,\cT_p)$, and $\mathsf{K_{reg,cvx}}(\bT,\lambda,\cT_p)$ defined by \eqref{summary-functions-fixed-finite}. 
Consider $0 < \tau_{\mathsf{min}} \leq \tau_{\mathsf{max}}$ and $0 < \lambda_{\mathsf{min}} \leq \lambda_{\mathsf{max}}$. 
We have the following:
\begin{enumerate}[(i)]
    \item % i
    $\mathsf{R_{reg,cvx}}$ is uniformly (over $p$) Lipschitz continuous in $\tau$ and $\lambda$ for $(\tau,\lambda) \in [0,\tau_{\mathsf{max}}] \times [\lambda_{\mathsf{min}},\lambda_{\mathsf{max}}]$.
    \item % ii
    $\mathsf{W_{reg,cvx}}$ is uniformly (over $p$) Lipschitz continuous in $\tau$ and $\lambda$ for $(\tau,\lambda) \in [\tau_{\mathsf{min}},\tau_{\mathsf{max}}] \times [\lambda_{\mathsf{min}},\lambda_{\mathsf{max}}]$.
    \item %iii
    $\mathsf{K_{reg,cvx}}$ is uniformly (over $p$) equicontinuous in $\bT$ and uniformly Lipschitz continuous in $\lambda$ for $\bzero \preceq \bT \preceq \tau_{\mathsf{max}}^2 \bI_2$ and $\lambda \in [\lambda_{\mathsf{min}},\lambda_{\mathsf{max}}]$.
\end{enumerate}
\end{lemma}

\begin{proof}[Proof of Lemma \ref{lem-uniform-modulus-of-continuity}]
Let $M > \sup_p \|\mathsf{prox}[\rho_p](\bzero)\|$ with $M < \infty$, which is permitted because $\{\rho_p\} \in \cB$.
Throughout the proof, we will denote by $C$ a constant which may depend on $M$, $\pi$, $ \tau_{\mathsf{max}}$, $\lambda_{\mathsf{min}}$, or $ \lambda_{\mathsf{max}}$, but not on $p$ or $\tau_{\mathsf{min}}$, and will denote by $C_+$ a constant which may depend also on $\tau_{\mathsf{min}}$ but not on $p$. Both $C$ and $C_+$ may differ at different appearances, even within the same chain of inequalities, as it absorbs terms. 

Observe that for any $\lambda$ we have 
\begin{align}
    \|\mathsf{prox}[\lambda \rho_p](\bbeta_0 + \tau \bz)\| &\leq \left\|\mathsf{prox}[ \rho_p](\bzero)\right\| + \left\|\mathsf{prox}[\lambda \rho_p](\bzero) - \mathsf{prox}[\rho_p](\bzero)\right\| + \left\|\mathsf{prox}[\lambda \rho_p](\bbeta_0 + \tau \bz) - \mathsf{prox}[\lambda \rho_p](\bzero)\right\| \nonumber \\
    &\leq M + \|\mathsf{prox}[\rho_p](\bzero)\| |\lambda - 1| + \|\bbeta_0\| + \tau \|\bz\|
    \leq M(2 + \lambda_{\mathsf{max}}) + \|\bbeta_0\| + \tau_{\mathsf{max}} \|\bz\| \nonumber\\
    &\leq C\left(1 + \|\bbeta_0\| + \|\bz\|\right),\label{prox-bounded-above}
\end{align}
where in the second inequality we have used \eqref{prox-is-lipschitz} and \eqref{prox-continuous-in-lambda}. 
With one more application of the triangle inequality, we get
\begin{align}\label{prox-error-bounded-above}
    \left\|\mathsf{prox}[\lambda \rho_p](\bbeta_0 + \tau \bz)-\bbeta_0\right\| \leq C\left(1 + \|\bbeta_0\| + \|\bz\|\right).
\end{align}
Further, observe that for $\lambda,\lambda' \in [\lambda_{\mathsf{min}},\lambda_{\mathsf{max}}]$ and $\tau \in [\tau_{\mathsf{min}},\tau_{\mathsf{max}}]$, we have by applying \eqref{prox-bounded-above} and the triangle inequality
\begin{align}
    \left\| \mathsf{prox}[\lambda \rho_p](\bbeta_0 + \tau \bz) - \mathsf{prox}[\lambda' \rho_p](\bbeta_0 + \tau \bz) \right\| & \leq \left\|\bbeta_0 + \tau \bz - \mathsf{prox}[\lambda \rho_p](\bbeta_0 + \tau \bz) \right\| \left|\frac{\lambda'}\lambda - 1\right|\nonumber \\
    &\leq C \left(1 + \|\bbeta_0\| + \|\bz\| \right) |\lambda - \lambda'|,\label{prox-pseudo-lipschitz-in-lambda}
\end{align}
where in the first inequality we have used \eqref{prox-continuous-in-lambda} and in the second inequality we have used\ref{prox-error-bounded-above} and that $\left|\frac{\lambda'}{\lambda} - 1\right| = \frac{|\lambda - \lambda'|}{\lambda } \leq \frac{|\lambda - \lambda'|}{\lambda \wedge \lambda' } \leq C|\lambda - \lambda'|$.

We are ready to demonstrate the claimed continuity properties of $\mathsf{R_{reg,cvx}}$, $\mathsf{W_{reg,cvx}}$, and $\mathsf{K_{reg,cvx}}$.
Fix $\tau,\tau' \in [0,\tau_{\mathsf{max}}]$, $\lambda,\lambda' \in [\lambda_{\mathsf{min}},\lambda_{\mathsf{max}}]$ and $\bzero \preceq \bT,\bT' \preceq \tau_{\mathsf{max}}^2 \bI_2$.
These will remain fixed throughout the remainder of the proof unless otherwise stated.

\noindent {\bf Uniform Lipschitz continuity of $\mathsf{R_{reg,cvx}}$ in $\tau$.} 
We apply \eqref{norm-squared-difference-inequality} identifying $\ba = \mathsf{prox}[\lambda \rho_p](\bbeta_0 + \tau \bz) - \bbeta_0$ and $\bb = \mathsf{prox}[\lambda \rho_p](\bbeta_0 + \tau' \bz) - \bbeta_0$.
Using \ref{prox-error-bounded-above} and \eqref{prox-is-lipschitz} to bound $(*)$ and $(**)$ respectively, we get
\begin{align}
\left|\| \ba\|^2 - \|\bb\|^2\right|
&\leq \underbrace{C\left(1 + \|\bbeta_0\| + \|\bz\| \right)}_{\text{bound on $(*)$}}\cdot \underbrace{ |\tau - \tau'|\|\bz\|}_{\text{bound on $(**)$}} 
\leq C\left(1 + \|\bbeta_0\|^2 + \|\bz\|^2\right)|\tau-\tau'|.\label{R-in-tau-continuity}
\end{align}
We have by Jensen's inequality
\begin{align}
    \Big|\mathsf{R_{reg,cvx}}(\tau,\lambda,\cT_p) &- \mathsf{R_{reg,cvx}}(\tau',\lambda,\cT_p)\Big| = \left|\E_{\bbeta_0,\bz}\left[\|\ba\|^2\right] - \E_{\bbeta_0,\bz}\left[\|\bb\|^2\right]\right|
    \leq \E_{\bbeta_0,\bz}\left[\left|\|\ba\|^2 -\|\bb\|^2\right|\right]\nonumber\\
    & \leq C\E_{\bbeta_0,\bz}\left[\left(1 + \|\bbeta_0\|^2 + \|\bz\|^2\right)\right]|\tau - \tau'| = C|\tau - \tau'|.
\end{align}

\noindent {\bf Uniform Lipschitz continuity of $\mathsf{R_{reg,cvx}}$ in $\lambda$.}
We apply \eqref{norm-squared-difference-inequality} identifying $\ba = \mathsf{prox}[\lambda \rho_p](\bbeta_0 + \tau \bz) - \bbeta_0$ and $\bb = \mathsf{prox}[\lambda' \rho_p](\bbeta_0 + \tau \bz) - \bbeta_0$.
Using \ref{prox-error-bounded-above} and \eqref{prox-pseudo-lipschitz-in-lambda} to bound $(*)$ and $(**)$ respectively, we get
\begin{align}
\left|\| \ba\|^2 - \|\bb\|^2\right|
&\leq \underbrace{C\left(1 + \|\bbeta_0\| + \|\bz\|\right)}_{\text{bound on $(*)$}}\cdot \underbrace{C\left(1 + \|\bbeta_0\| + \|\bz\|\right)|\lambda - \lambda'|}_{\text{bound on $(**)$}}\nonumber\\
&\leq C\left(1 + \|\bbeta_0\|^2 + \|\bz\|^2\right)|\lambda-\lambda'|.
\end{align}
We have by Jensen's inequality
\begin{align}
    \Big|\mathsf{R_{reg,cvx}}(\tau,\lambda,\cT_p) &- \mathsf{R_{reg,cvx}}(\tau,\lambda',\cT_p)\Big| = \left|\E_{\bbeta_0,\bz}\left[\|\ba\|^2\right] - \E_{\bbeta_0,\bz}\left[\|\bb\|^2\right]\right|
    \leq \E_{\bbeta_0,\bz}\left[\left|\|\ba\|^2 -\|\bb\|^2\right|\right]\nonumber\\
    &\leq C\E_{\bbeta_0,\bz}\left[\left(1 + \|\bbeta_0\|^2 + \|\bz\|^2\right)\right]|\lambda - \lambda'|
     = C|\lambda - \lambda'|.
\end{align}

\noindent {\bf Uniform Lipschitz continuity of $\mathsf{W_{reg,cvx}}$ in $\tau$.}
In this section only, we require also that $\tau,\tau' \geq \tau_{\mathsf{min}}$.
We apply \eqref{inner-product-difference-inequality} identifying $\ba = \frac{\bz}{\tau}$, $\bb = \mathsf{prox}[\lambda\rho_p](\bbeta_0 + \tau\bz)$, $\ba' = \frac{\bz}{ \tau'}$, and $\bb' = \mathsf{prox}[\lambda \rho_p](\bbeta_0 + \tau' \bz)$.
Observe that 
$
    \|\ba - \ba'\| = |1/\tau - 1/\tau'|\|\bz\| \leq C_+\|\bz\||\tau - \tau'|,
$
where the last inequality holds because $\tau,\tau' \geq \tau_{\mathsf{min}} > 0$.  
Using \eqref{prox-bounded-above} and \eqref{prox-is-lipschitz} to bound  $\max\{\|\ba\|,\|\ba'\|,\|\bb\|,\|\bb'\|\}$ and $\|\bb - \bb'\|$ respectively, we get
\begin{align}
\left|\langle \ba, \bb\rangle - \langle \ba',\bb'\rangle\right|
&\leq \underbrace{C\left(1 + \|\bbeta_0\| + \|\bz\|\right)}_{\text{bound on $(*)$}}\cdot \underbrace{C_+\|\bz\||\tau - \tau'|}_{\text{bound on $(**)$}}
\leq C_+\left(1 + \|\bbeta_0\|^2 + \|\bz\|^2\right)|\tau-\tau'|.
\end{align}
We have by Jensen's inequality
\begin{align}
    \left|\mathsf{W_{reg,cvx}}(\tau,\lambda,\cT_p) - \mathsf{W_{reg,cvx}}(\tau',\lambda,\cT_p)\right| &= \left|\E_{\bbeta_0,\bz}\left[\langle \ba , \bb \rangle \right] - \E_{\bbeta_0,\bz}[\langle \ba' , \bb' \rangle ]\right|
    \leq \E_{\bbeta_0,\bz}\left[\left|\langle  \ba , \bb \rangle -\langle \ba' , \bb' \rangle \right|\right]\nonumber\\
    &\leq C_+\E_{\bbeta_0,\bz}\left[\left(1 + \|\bbeta_0\|^2 + \|\bz\|^2\right)\right]|\tau - \tau'|
     = C_+|\tau - \tau'|.
\end{align}

\noindent {\bf Uniform Lipschitz continuity of $\mathsf{W_{reg,cvx}}$ in $\lambda$.}
We apply \eqref{inner-product-difference-inequality} identifying $\ba = \frac{\bz}{\tau}$, $\bb = \mathsf{prox}[\lambda \rho_p](\bbeta_0 + \tau\bz)$, $\ba' = \frac{\bz}{ \tau }$, and $\bb' = \mathsf{prox}[\lambda' \rho_p](\bbeta_0 + \tau \bz)$. 
Using \eqref{prox-bounded-above} and \eqref{prox-pseudo-lipschitz-in-lambda} to bound $(*)$ and $(**)$ respectively, we get 
\begin{align}
\left|\langle \ba, \bb\rangle - \langle \ba',\bb'\rangle\right|
&\leq \underbrace{C\left(1 + \|\bbeta_0\| + \|\bz\|\right)}_{\text{bound on $(*)$}}\cdot \underbrace{C\left(1 + \|\bbeta_0\| + \|\bz\|\right)|\lambda - \lambda'|}_{\text{bound on $(**)$}}
\leq C\left(1 + \|\bbeta_0\|^2 + \|\bz\|^2\right)|\lambda-\lambda'|.
\end{align}
We have by Jensen's inequality
\begin{align}
    &\left|\mathsf{W_{reg,cvx}}(\tau,\lambda,\cT_p) - \mathsf{W_{reg,cvx}}(\tau,\lambda',\cT_p)\right| = \left|\E_{\bbeta_0,\bz}\left[\langle \ba , \bb \rangle \right] - \E_{\bbeta_0,\bz}[\langle \ba' , \bb' \rangle ]\right|\leq \E_{\bbeta_0,\bz}\left[\left|\langle  \ba , \bb \rangle -\langle \ba' , \bb' \rangle \right|\right]\nonumber\\
    &\qquad\qquad\qquad \leq C\E_{\bbeta_0,\bz}\left[\left(1 + \|\bbeta_0\|^2 + \|\bz\|^2\right)\right]|\lambda - \lambda'|
     = C|\lambda - \lambda'|.
\end{align}

\noindent {\bf Uniform equicontinuity of $\mathsf{K_{reg,cvx}}$ in $\bT$.}
Let $\bT,\bT' \in S_+^2$. 
By \cite[Proposition 7]{Givens1984ADistributions}, we have 
\begin{equation}
d_{\mathrm{W}}\left(\mathsf{N}(\bzero,\bT),\mathsf{N}(\bzero,\bT')\right) = \sqrt{\Tr\left(\bT + \bT' - 2 (\bT^{1/2}\bT' \bT^{1/2})^{1/2}\right)}.
\end{equation}
By \cite[Proposition 1]{Givens1984ADistributions}, there exists a coupling which achieves the infimum in \eqref{eqdef-wasserstein}.
Let $\nu$ a probability distribution on $\reals^4$ which implements the minimal coupling between $\mathsf{N}(\bzero,\bT)$ and $\mathsf{N}(\bzero,\bT')$. 
Consider a probability space with random vectors $\bbeta_0$ and $\bz_1,\bz_2,\bz_1',\bz_2'$ for all $p$ such that $(z_{1j},z_{2j},z_{1j}',z_{2j}') \stackrel{\mathrm{iid}}\sim \nu/\sqrt p$.
Then
\begin{equation}\label{gaussian-wasserstein-coupling}
\E_{\bz_1,\bz_2,\bz_1',\bz_2'} \left[\|\bz_1 - \bz_1'\|^2 + \|\bz_2 - \bz_2'\|^2\right] = \Tr\left(\bT + \bT' - 2 (\bT^{1/2}\bT' \bT^{1/2})^{1/2}\right).
\end{equation}
We apply \eqref{inner-product-difference-inequality} identifying $\ba = \mathsf{prox}[\lambda \rho_p](\bbeta_0 + \bz_1) - \bbeta_0$, $\bb = \mathsf{prox}[\lambda \rho_p](\bbeta_0 + \bz_2) - \bbeta_0$, $\ba' = \mathsf{prox}[\lambda \rho_p](\bbeta_0 + \bz_1') - \bbeta_0$, and $\bb' = \mathsf{prox}[\lambda \rho_p](\bbeta_0 + \bz_2') - \bbeta_0$.
Using \ref{prox-error-bounded-above} and \eqref{prox-is-lipschitz} to bound $(*)$ and $(**)$ respectively, we get
\begin{align}
\left|\langle \ba, \bb\rangle - \langle \ba',\bb'\rangle\right|
&\leq \underbrace{C\left(1 + \|\bbeta_0\| + \max( \|\bz_1\|, \|\bz_2\|,\|\bz_1'\|,\|\bz_2'\|)\right)}_{\text{bound on $(*)$}} \cdot \underbrace{\max(\|\bz_1 - \bz_1'\|,\|\bz_2-\bz_2'\|)}_{\text{bound on $(**)$}}.
\end{align}
We have by Jensen's inequality and Cauchy-Schwartz
\begin{align}
&\left|\mathsf{K_{reg,cvx}}(\bT,\lambda,\cT_p) - \mathsf{K_{reg,cvx}}(\bT',\lambda,\cT_p)\right|\nonumber\\
&\qquad \qquad \leq C\E_{\bbeta_0,\bz_1,\bz_2,\bz_1',\bz_2'}\left[\left(1 + \|\bbeta_0\| + \max( \|\bz_1\|, \|\bz_2\|,\|\bz_1'\|,\|\bz_2'\|)\right)^2\right]^{1/2} \nonumber\\
&\qquad\qquad\qquad \times \E_{\bz_1,\bz_2,\bz_1',\bz_2'}\left[\max(\|\bz_1 - \bz_1'\|,\|\bz_2-\bz_2'\|)^2\right]^{1/2}.
\end{align}
We have 
\begin{align}
&\E_{\bbeta_0,\bz_1,\bz_2,\bz_1',\bz_2'}\left[\left(1 + \|\bbeta_0\| + \max( \|\bz_1\|, \|\bz_2\|,\|\bz_1'\|,\|\bz_2'\|)\right)^2\right]^{1/2}\nonumber \\
&\qquad\qquad\qquad\leq C\E_{\bbeta_0,\bz_1,\bz_2,\bz_1',\bz_2'}\left[1 + \|\bbeta_0\|^2 + \|\bz_1\|^2 + \|\bz_2|^2 + \|\bz_1'\|^2 + \|\bz_2'\|^2\right]^{1/2} \leq C.
\end{align}
Further, by \eqref{gaussian-wasserstein-coupling}, we have 
\begin{equation}
     \E_{\bz_1,\bz_2,\bz_1',\bz_2'}\left[\max(\|\bz_1 - \bz_1'\|,\|\bz_2-\bz_2'\|)^2\right]^{1/2} \leq  \sqrt{\Tr\left(\bT + \bT' - 2 (\bT^{1/2}\bT' \bT^{1/2})^{1/2}\right)}.
\end{equation}
Thus,
\begin{align}
 \left|\mathsf{K_{reg,cvx}}(\bT,\lambda,\cT_p) - \mathsf{K_{reg,cvx}}(\bT',\lambda,\cT_p)\right| &\leq C \sqrt{\Tr\left(\bT + \bT' - 2 (\bT^{1/2}\bT' \bT^{1/2})^{1/2}\right)}.
 \end{align}
 Now observe that $(\bT,\bT') \mapsto \sqrt{\Tr\left(\bT + \bT' - 2 (\bT^{1/2}\bT' \bT^{1/2})^{1/2}\right)}$ is continuous and is 0 when $\bT = \bT'$.
 Thus, it is uniformly continuous on the compact domain $\{(\bT,\bT') \in (S_+^2)^2 \mid \bzero \preceq \bT,\bT' \preceq \tau_{\mathsf{max}}^2 \bI_2\}$ (where because this is a finite dimensional Euclidean space, continuity holds with respect to any norm by equivalence of norms). 
 Thus, for any $\eps > 0$, there exists $\delta > 0$ such that if $\bzero \preceq \bT,\bT' \preceq \tau_{\mathsf{max}}^2 \bI_2$ and $\|\bT - \bT\|_{\mathsf{op}} < \delta$, then $\sqrt{\Tr\left(\bT + \bT' - 2 (\bT^{1/2}\bT' \bT^{1/2})^{1/2}\right)} < \eps$. 
 Because this modulus of continuity does not depend upon $p$, 
 we have $\mathsf{K_{reg,cvx}}(\bT,\lambda,\cT_p) $ is uniformly (over $p$) equicontinuous in $\bT$.

\noindent{\bf Uniform Lipschitz continuity of $\mathsf{K_{reg,cvx}}$ in $\lambda$.}
Let $(\bz_1,\bz_2) \sim \mathsf{N}\left(0, \bT \otimes \bI_p/p\right)$. 
We apply \eqref{inner-product-difference-inequality} identifying $\ba = \mathsf{prox}[\lambda \rho_p](\bbeta_0 + \tau\bz_1) - \bbeta_0$, $\bb = \mathsf{prox}[\lambda \rho_p](\bbeta_0 + \tau\bz_2) - \bbeta_0$, $\ba' = \mathsf{prox}[\lambda' \rho_p](\bbeta_0 + \tau\bz_1) - \bbeta_0$, and $\bb' = \mathsf{prox}[\lambda' \rho_p](\bbeta_0 + \tau\bz_2) - \bbeta_0$.
Using \ref{prox-error-bounded-above} and \eqref{prox-pseudo-lipschitz-in-lambda} to bound $(*)$ and $(**)$ respectively, we get
\begin{align}
\left|\langle \ba, \bb\rangle - \langle \ba',\bb'\rangle\right| &\leq \underbrace{C\left(1 + \|\bbeta_0\| + \|\bz_1\|\vee \|\bz_2\|\right)}_{\text{bound on $(*)$}} \cdot \underbrace{C\left(1 + \|\bbeta_0\| + \|\bz_1\|\vee \|\bz_2\| \right) |\lambda - \lambda'|}_{\text{bound on $(**)$}}\nonumber\\
&\leq C\left(1 + \|\bbeta_0\|^2 + \|\bz_1\|^2 + \|\bz_2\|^2\right)|\lambda - \lambda'|.
\end{align}
We have by Jensen's inequality
\begin{align}
    \left|\mathsf{K_{reg,cvx}}(\bT,\lambda,\cT_p) - \mathsf{K_{reg,cvx}}(\bT,\lambda',\cT_p)\right| &= \left|\E_{\bbeta_0,\bz}\left[\langle \ba , \bb \rangle \right] - \E_{\bbeta_0,\bz}[\langle \ba' , \bb' \rangle ]\right|
    \leq \E_{\bbeta_0,\bz}\left[\left|\langle  \ba , \bb \rangle -\langle \ba' , \bb' \rangle \right|\right]\nonumber\\
    &\leq C\E_{\bbeta_0,\bz}\left[\left(1 + \|\bbeta_0\|^2 + \|\bz\|^2\right)\right]|\lambda - \lambda'|
    = C|\lambda - \lambda'|.
\end{align}
This completes the proof.
\end{proof}

% ----------------------------------------------------------------- %

\section{Proof of Proposition \ref{prop-strongly-convex-loss}}\label{app-proof-of-prop-strongly-convex-loss}

This argument follows closely that of \cite{Donoho2016HighPassing}. 
In contrast to \cite{Donoho2016HighPassing}, we consider penalized procedures and use non-separable penalties.
Our analysis also establishes the impact of the oracle penalty on the fixed point equations \eqref{fixed-pt-prior-asymptotic}.
We find that using the recent results \cite{Berthier2017StateFunctions} for AMP with non-separable denoisers, their argument goes through.

Throughout the argument, we will frequently (but not always) drop the index $p$ from our notation.
For sequences $\{X_p\}$ and $\{Y_p\}$ of real-valued random variables all defined on the same probability space,
we use the notation $X_p \stackrel{\mathsf{as}}\simeq Y_p$ to denote $|X_p - Y_p| \stackrel{\mathsf{as}} \rightarrow 0$. 

\subsection{Proof of part (i)}

For each $p$, define
\begin{equation}\label{def-L}
L(\bbeta) := \frac1n \|\by - \bX^\mathsf{T}\bbeta\|^2 + \rho_p(\bbeta_0) + \frac\gamma2 \|\bbeta - \bbeta_0\|^2,
\end{equation}
the objective in \eqref{orc-cvx-estimator}).
If $n > p$, $\gamma > 0$, or $\rho_p$ is strongly convex, then $L$ is strongly convex almost surely.
Thus, $\P_{\bX}(\text{solution to \eqref{linear-cvx-estimator} exists and is unique}) = 1$ eventually.
This justifies assuming existence and uniqueness of solutions to \eqref{linear-cvx-estimator} for the remainder of the proof.

We now prove (ii) and (iii), which require substantially more work.

\subsection{Pick a typical sequence of normal vectors}

Without loss of generality, we may assume $p$ is increasing.    
The remainder of the argument will occur conditional on the realization of the sequence of parameters $\{\bbeta_0\}$.
We will be able to carry out all steps under the DSN assumption if the penalties are symmetric,
or almost surely under the RSN. (We will justify this as we go).

We construct a deterministic sequence of vectors $\{\bz^0(p) \in \reals^p\}$ such that for all $\tau' \geq 0$,  
\begin{gather}
\lim_{p \rightarrow \infty} \|\bz^0\|^2 = 1,\nonumber\\
\lim_{p \rightarrow \infty} \|\mathsf{prox}[\lambda \rho_p^{(\gamma)}](\bbeta_0 + \tau \bz^0) - \bbeta_0\|^2 = \delta \tau^2 - \sigma^2,\label{z0-looks-normal}\\
\lim_{p \rightarrow \infty} \E_{\bz} \left[\left\langle \mathsf{prox}[\lambda \rho_p^{(\gamma)}](\bbeta_0 + \tau \bz^0) - \bbeta_0, \mathsf{prox}[\lambda \rho_p^{(\gamma)}](\bbeta_0 + \tau' \bz) - \bbeta_0 \right \rangle \right] = \mathsf{K^\infty_{reg,cvx}}\left(\bT_{\tau'},\lambda,\cT\right),\nonumber
\end{gather}
where $\bz \sim \mathsf{N}(\bzero,\bI_p/p)$ and $\bT_{\tau'} := \begin{pmatrix} \tau^2 & 0 \\ 0 & \tau'^2 \end{pmatrix}$. 
Such a sequence exists because if we draw $\bz^0(p) \sim \mathsf{N}(\bzero,\bI_p/p)$ independently across $p$,
then $\{\bz^0\}$ satisfies the required properties (simultaneously over $\tau'$) almost surely, as we now show.

For such random $\bz^0$,
by Gaussian Lipschitz concentration (Lemma \ref{lem-gaussian-lipschitz-concentration}), we have
\begin{equation}
    \P_{\bz_0}\left(\left| \left\| \bz^0 \right\| - \E_{\bz^0}\left[ \left\| \bz^0 \right\| \right] \right| > t/p^{1/4}\right) \leq 2 e^{-\frac{p^{1/2}}2 t^2}.
\end{equation}
Because the right-hand side is summable over $p$ (recall we assume $p$ is increasing), we have by Borel-Cantelli that $\left \|  \bz^0 \right \| \stackrel{\mathsf{as}}\simeq \E_{\bz^0} \left[\left\|\bz^0\right\|\right] \rightarrow 1$.
Thus, the first identity of \eqref{z0-looks-normal} holds almost surely.

For the remaining two identities, first note that by 
by \eqref{oracle-prox-noise-reduction-form}, 
the second and third identities of \eqref{z0-looks-normal} are equivalent to
\begin{gather*}
    \lim_{p \rightarrow \infty} \|\mathsf{prox}[\lambdaorc \rho_p](\bbeta_0 + \tauorc \bz^0) - \bbeta_0\|^2 = \delta \tau^2 - \sigma^2,\\
    \lim_{p \rightarrow \infty} \E_{\bz} \left[\left\langle \mathsf{prox}[\lambdaorc \rho_p](\bbeta_0 + \tauorc \bz^0) - \bbeta_0, \mathsf{prox}[\lambdaorc \rho_p](\bbeta_0 + \tauorc' \bz) - \bbeta_0 \right \rangle \right] = \mathsf{K^\infty_{reg,cvx}}\left(\bT_{\tau'},\lambda,\cT\right).
\end{gather*}
If the $\rho_n$ are symmetric, $\pi$ has finite second moments, and the $\rho_n$ are symmetric,
then by Lemmas \ref{lem-integrands-uniformly-pseudo-lipschitz} and \ref{lem-pseudo-lipschitz-empirical-to-expectation} and the symmetry of $\rho_p$, 
under the DSN assumption, for all $\tau' \geq 0,$ $\lambda' \geq 0$, $\bT \succeq 0$ fixed,
\begin{gather}
\mathsf{R^\infty_{reg,cvx}}(\tau',\lambda',\cT_p) = \lim_{p \rightarrow \infty}\E_{\bz}\left[\|\mathsf{prox}[\lambda' \rho_p](\bbeta_0 + \tau' \bz) - \bbeta_0\|^2\right],\nonumber\\
\mathsf{W^\infty_{reg,cvx}}(\tau',\lambda',\cT_p) = \lim_{p\rightarrow \infty}\frac1{\tau'}\E_{\bz}\left[\langle \bz, \mathsf{prox}[\lambda'\rho_p](\bbeta_0+\tau'\bz)\rangle\right],\\
\mathsf{K^\infty_{reg,cvx}}(\bT,\lambda',\cT_p) = \lim_{p\rightarrow \infty}\E_{\bz_1,\bz_2}\left[\left\langle\mathsf{prox}\left[\lambda' \rho_p\right]\left(\bbeta_0 + \bz_1\right)-\bbeta_0,\mathsf{prox}\left[\lambda'\rho_p\right]\left(\bbeta_0 + \bz_2\right)-\bbeta_0\right\rangle\right],\nonumber
\end{gather} 
(Note the expectations are only over the Gaussian random vectors and $\bbeta_0$ is fixed).
If the $\rho_n$ are not necessarily symmetric, then under the RSN assumption, the previous display holds almost surely with respect to the realization of $\{\bbeta_0\}$ by Lemma \ref{lem-emp-to-exp-non-symm}.

Because $f_p(\bx;\tauorc) := \mathsf{prox}[\lambda \rho_p](\bbeta_0 + \tauorc \bx) - \bbeta_0$ is $\tauorc$-Lipschitz by \eqref{prox-is-lipschitz}, we have by Gaussian Lipschitz concentration (Lemma \ref{lem-gaussian-lipschitz-concentration}) and Borel-Cantelli that
\begin{equation}\label{B4-as-convergence}
\left\|f_p(\bz^0;\tauorc)\right\| \stackrel{\mathsf{as}}\simeq \E_{\bz^0}\left[\left\|f_p(\bz^0;\tauorc)\right\|\right].
\end{equation}     
Now observe by Jensen's inequality that $
\E_{\bz^0}\left[\left\|f_p(\bz^0;\tauorc)\right\|^2\right] \geq \E_{\bz^0}\left[\left\|f_p(\bz^0;\tauorc)\right\|\right]^2$.
By assumption, the left-hand side and hence the right-hand side is bounded.
By exponential concentration of $\left\|f_p(\bz^0;\tauorc)\right\|$, 
we conclude that $\left\|f_p(\bz^0;\tauorc)\right\|^2$ is uniformly integrable.
Because it concentrates on $\E_{\bz^0}\left[\left\|f_p(\bz^0;\tauorc)\right\|\right]^2$,
we have
\begin{align}\label{first-moment-converges}
\lim_{p \rightarrow \infty} \E_{\bz^0}\left[\left\|f_p(\bz^0;\tauorc)\right\|\right]^2 = \lim_{p \rightarrow \infty}\E_{\bz^0}\left[\left\|f_p(\bz^0;\tauorc)\right\|^2\right] = \delta \tau^2 - \sigma^2.
\end{align}
Then by \eqref{B4-as-convergence}, $\left\|f_p(\bz^0;\tauorc)\right\|^2 \stackrel{\mathrm{as}}\rightarrow \delta \tau^2 - \sigma^2$.
Thus, the second identity of \eqref{z0-looks-normal} hold almost surely.

We now show that almost surely, the third identity \eqref{z0-looks-normal} for all $\tau' \geq 0$. 
Recall by strong stationarity that the limit \eqref{summary-functions-fixed-asymptotic} holds for $\bT_{\tau'} = \begin{pmatrix} \tauorc^2 & 0 \\ 0 & \tau'^2 \end{pmatrix}$ for all $\tau' \geq 0$.         
Now fix a particular $\tau' \geq 0$.
Let $h_p(\bx;\tau') = \E_{\bz}[\langle f_p(\bx;\tauorc), f_p(\bz;\tau')\rangle]$, where $\bz \sim \mathsf{N}(\bzero,\bI_p/p)$.
By Cauchy-Schwartz and because $f_p$ is Lipschitz,
\begin{align}
    \left|h_p(\bx_1;\tau') - h_p(\bx_2;\tau')\right|
    &\leq \tauorc \E_{\bz}\left[\left\|f_p(\bz;\tau')\right\|\right]\|\bx_1 - \bx_2\|.\label{hp-lipschitz}
\end{align} 
By \eqref{first-moment-converges} and \eqref{prox-is-lipschitz}, we have $\E_{\bz}[\left\|f_p(\bz;\tau')\right\|] \leq \E_{\bz}[\left\|f_p(\bz;\tauorc)\right\|] + |\tauorc - \tau'| \E_{\bz}[\left\|\bz\right\|] \leq  \E_{\bz}[\left\|f_p(\bz;\tauorc)\right\|] + |\tauorc - \tau'| $ is bounded in $p$.
Thus, for a fixed $\tau'$,  inequality \eqref{hp-lipschitz} gives that $h_p(\cdot;\tau')$ is uniformly (over $p$) Lipschitz.
Then, applying Lemma \ref{lem-gaussian-lipschitz-concentration} and Borel-Cantelli in the same way we did to establish \eqref{B4-as-convergence}, we have (using also \eqref{k-finite-p-def} and condition \eqref{summary-functions-fixed-asymptotic} of strong stationarity)
\begin{equation}
        h_p(\bz^0;\tau') \stackrel{\mathsf{as}}\simeq \E_{\bz^0}[h_p(\bz^0)] = \mathsf{K_{reg,cvx}}(\bT',\lambda,\cT_p) \rightarrow \mathsf{K^\infty_{reg,cvx}}\left(\bT_{\tau'},\lambda,\cT\right).
\end{equation}
This establishes the results for fixed $\tau' > 0$.
We may extend to all of $\reals_+$ by considering a countable dense subset of $\reals_+$ and using continuity of $h_p$ in $\tau'$.
    
In summary, we have proved that if $\bz^0 \sim \mathsf{N}(0,\bI_p/p)$ for all $p$, then almost surely the limits \eqref{z0-looks-normal} hold simultaneously for all $\tau' \geq 0$.
Therefore, we may choose a deterministic sequence $\{\bz^0\}$ such that these limits all hold.
     
\subsection{The Approximate Message Passing (AMP) iteration}

Let $\{\bz^0\}$ be a deterministic sequence of vectors satisfying limits \eqref{z0-looks-normal} for all $\tau' \geq 0$, as permitted by the previous section.
For each $p$, define the sequence $\{\widehat \bbeta^t\}_{t \geq 0}$ via the following iteration. 
Define
\begin{equation}\label{eqdef-b}
\mathsf{b} = 1 - \frac1{2\lambda},
\end{equation}
and for $t \geq 0$
\begin{subequations}\label{M-AMP}
\begin{gather}
\br^t = \frac{\by - \bX \widehat \bbeta^{t}}{n} + \mathsf{b} \br^{t-1},\label{M-AMP-1}\\
\widehat \bbeta^{t+1} = \mathsf{prox}[\lambda \rho_p^{(\gamma)}]\left(\widehat \bbeta^t + \bX^\mathsf{T}\br^{t}\right),\label{M-AMP-2}\\
\widehat \bbeta^0 = \mathsf{prox}[\lambda \rho_p^{(\gamma)}]\left(\bbeta_0 + \tau \bz^0\right) \quad \text{and} \quad \br^{-1} = \bzero.\label{M-AMP-initialization}
\end{gather}
\end{subequations}
This iteration is an approximate message passing (AMP) algorithm. Several papers (see \cite{Berthier2017StateFunctions} and references therein) precisely characterize the iterates of such algorithms in the $p\rightarrow \infty$ limit, as we will see in Appendix \ref{app-relating-AMP-and-SE} below. 
Further, they satisfy certain identities relating them to $\rho_p$ and $L$.
First, by \eqref{prox-to-rho-subgradient} and \eqref{M-AMP-2}, we have for $t \geq 0$ that
\begin{equation}\label{subgradient-1}
    \widehat \bbeta^t + \bX^\mathsf{T}\br^t - \widehat \bbeta^{t+1} \in \lambda \partial  \rho_p^{(\gamma)}
    \left(\widehat \bbeta^{t+1}\right).
\end{equation}
Second, by \eqref{M-AMP-1},
\begin{equation}\label{subgradient-2}
    \nabla_{\bbeta}\left( \frac1n \|\by - \bX\bbeta\|^2\right)\Big|_{\bbeta = \widehat \bbeta^{t+1}} = \frac2n\bX^\mathsf{T}\left(\bX \widehat \bbeta^{t+1} - \by\right) = 2\bX^\mathsf{T}\left(\mathsf{b}\br^t - \br^{t+1}\right).
\end{equation}
Combining \eqref{subgradient-1} and \eqref{subgradient-2} with \eqref{def-L},
\begin{align}\label{amp-to-subgradient}
    \partial L\left(\widehat \bbeta^{t+1}\right) \ni  2\bX^\mathsf{T}&\left(\mathsf{b} \br^t - \br^{t+1}\right) + \frac{\widehat \bbeta^t + \bX^\mathsf{T}\br^t - \widehat \bbeta^{t+1} }\lambda \nonumber\\
     &=2\mathsf{b}\bX^\mathsf{T}\left( \br^t - \br^{t+1}\right) + \frac{(\widehat \bbeta^t + \bX^\mathsf{T}\br^t) - (\widehat \bbeta^{t+1} + \bX^\mathsf{T}\br^{t+1}) }\lambda,
\end{align}
where in the equality we have used that $\frac1\lambda = 2(1 - \mathsf{b})$.
If $L$ is $\kappa$-strong convex for some $\kappa > 0$ and $\bg \in \partial L \left(\widehat \bbeta^{t+1}\right)$, then for any $\bbeta \in \reals^p$
\begin{equation*}
    L(\bbeta) \geq L\left(\widehat \bbeta^{t+1}\right) + \left \langle \bg , \bbeta - \widehat \bbeta^{t+1} \right\rangle + \frac\kappa 2 \left\|\bbeta - \widehat \bbeta^{t+1}\right\|^2 \geq L\left(\widehat \bbeta^{t+1}\right) - \left \| \bg \right\| \left\|  \bbeta - \widehat \bbeta^{t+1} \right\| + \frac\kappa 2 \left\|\bbeta - \widehat \bbeta^{t+1}\right\|^2.
\end{equation*}
Because $L(\widehat \bbeta_{\mathsf{cvx}}) \leq L(\widehat \bbeta^{t+1})$ by \eqref{linear-cvx-estimator}, we have
\begin{equation}\label{strong-convexity-distance-to-minimizer}
\left\|\widehat \bbeta_{\mathsf{cvx}} - \widehat \bbeta^{t+1}\right\|\leq\frac{2\|\bg\|}{\kappa}\, .
\end{equation}
Combining \eqref{amp-to-subgradient} and \eqref{strong-convexity-distance-to-minimizer}, we have that if $L$ is $\kappa$-strongly convex, then
\begin{equation}\label{iterates-to-minimizer-strong-convexity-bound}
    \left\|\widehat \bbeta_{\mathsf{cvx}} - \widehat \bbeta^{t+1}\right\| \leq \frac2\kappa\left(2\mathsf{b}\left\|\bX^\mathsf{T}\right\|_{\mathsf{op}}\left\| \br^t - \br^{t+1}\right\| + \frac{\left\|(\widehat \bbeta^t + \bX^\mathsf{T}\br^t) - (\widehat \bbeta^{t+1} + \bX^\mathsf{T}\br^{t+1})\right\| }{\lambda}\right).
\end{equation}
Inequality \eqref{iterates-to-minimizer-strong-convexity-bound} allows us to control the distance of the iterates $\widehat \bbeta^t$ from the minimizer $\widehat \bbeta_{\mathsf{cvx}}$ of $L$ in terms of the rate at which the iterates are changing and the strong convexity parameter of $L$. 
We will control this distance in the $p \rightarrow \infty$, fixed $t$ asymptotic regime by controlling the terms on the right-hand side of \eqref{iterates-to-minimizer-strong-convexity-bound}. 

\subsection{The state evolution}\label{app-the-state-evolution}

We now study a certain scalar iteration which, in the following sections, will allow us to characterize the $p\rightarrow\infty$, fixed $t$ behavior of the AMP iteration \eqref{M-AMP}.
For $q \in [0,1]$, define $\bQ_q = \begin{pmatrix} \tauorc^2 & q \tauorc^2 \\
q\tauorc^2 & \tauorc^2 \end{pmatrix}$, and observe that $\bQ_q \succeq \bzero$.
Define $\Psi:[0,1] \rightarrow \reals$ by
\begin{equation}\label{correlation-recursion}
\Psi(q) = \frac1{\delta\tau^2}\left(\sigma^2 + \mathsf{K^\infty_{reg,cvx}}\left(\bQ_q,\lambdaorc,\cT\right)\right).
\end{equation}
Because $\tau,\lambda,\gamma,\delta,\cT$ is strongly stationary, $\Psi(q)$ is well-defined for all $q \in [0,1]$ (recall  strong stationarity requires the limit \eqref{summary-functions-fixed-asymptotic} exist for all $\bT \in S_+^2$). 
Define the doubly-infinite symmetric matrix $Q = (q_{ij})_{i,j = 1}^\infty$ via the following scalar iteration, referred to as the \emph{state evolution}:
\begin{subequations}\label{full-se}
\begin{gather}
q_{1,1} = 1,\, q_{1,i} = q_{i,1} = 0 \text{ for } i > 1,\label{full-se-1}\\
q_{s+1,t+1} = \Psi(q_{s,t}).\label{full-se-2}
\end{gather}
\end{subequations}
In order for \eqref{full-se-2} to make sense, we must verify that $q_{s,t} \in [0,1]$ for all $s,t \geq 1$. 
By induction, it will suffice to show that $\Psi(q) \in [0,1]$ for all $q \in [0,1]$. 
In preparation for what is to come later in the proof, we will in fact show more. 

\begin{lemma}\label{lem-state-evolution-properties}
For any sequence of symmetric convex function $\rho_p$,
\begin{subequations}\label{q-recursion-properties}
\begin{gather}
\Psi(1) = 1,\label{q-recursion-property-1}\\
 \text{$\Psi(q)$ \emph{ is non-decreasing and convex for } $q \in [0,1]$},\label{q-recursion-property-2}\\
\text{$\Psi(q) \geq 1 - \frac1{(\lambda \gamma + 1) \vee \delta} (1 - q)$ \emph{for} $q \in [0,1]$.}\label{q-recursion-property-3}
\end{gather}
\end{subequations}
\end{lemma}

\begin{proof}[Proof of properties \eqref{q-recursion-property-1}, \eqref{q-recursion-property-2} of Lemma \ref{lem-state-evolution-properties}]
    By \eqref{r-finite-p-def}, \eqref{k-finite-p-def}, and \eqref{summary-functions-fixed-asymptotic},
\begin{equation}
    \mathsf{K^\infty_{reg,cvx}}\left(\tauorc^2\bI_2,\lambdaorc,\cT\right) = \mathsf{R^\infty_{reg,cvx}}(\tauorc,\lambdaorc,\cT).\label{k-to-r-identity}
\end{equation}
Because $\tau,\lambda,\gamma,\delta,\cT$ is strongly stationary, \eqref{fixed-pt-prior-asymptotic}, \eqref{correlation-recursion}, and \eqref{k-to-r-identity} imply \eqref{q-recursion-property-1}.

For each $p$, define $\Psi_p :[0,1] \rightarrow \reals$ by
\begin{equation}\label{correlation-recursion-finite-p}
\Psi_p(q) = \frac1{\delta\tau^2}\left(\sigma^2 + \mathsf{K_{reg,cvx}}\left(\bQ_q,\lambdaorc,\cT_p\right)\right),
\end{equation}
where $\cT_p$ is related to $\cT$ in the obvious way.
By strong stationarity condition \eqref{summary-functions-fixed-asymptotic}, $\Psi(q) = \lim_{p \rightarrow \infty} \Psi_p(q)$ for every $q \in [0,1]$. 
It is straightforward to see that properties \eqref{q-recursion-property-2}, \eqref{q-recursion-property-3} will hold if we can establish 
\begin{subequations}\label{q-recursion-finite-p-properties}
    \begin{gather}
        \text{for all $p$, $\Psi_p(q) $ is increasing in $q$ for $q \in [0,1]$,}\label{q-recursion-finite-p-property-1}\\
        \text{for all $p$, $\Psi_p(q) $ is convex in $q$ for $q \in [0,1]$,}\label{q-recursion-finite-p-property-2}\\
        \text{$\limsup_{p \rightarrow \infty} \Psi_p'(1) \leq \frac1{(\lambda \gamma + 1)\vee \delta}$.}\label{q-recursion-finite-p-property-3}
    \end{gather}
\end{subequations}

To show \eqref{q-recursion-finite-p-property-1}, \eqref{q-recursion-finite-p-property-2} we extend the argument of \cite[Lemma 6.9]{Donoho2016HighPassing} and \cite[Lemma C.1]{BayatiMontanariLASSO} to multivariate maps.
Fix $p$. Define $f:\reals^p \rightarrow \reals^p,\,\bx \mapsto \frac1{\sqrt \delta \tau}\left(\mathsf{prox}\left[\lambdaorc \rho_p\right](\bbeta_0 + \tauorc \bx) - \bbeta_0 \right)$.
Let $\{\bz_t\}_{t \geq 0}$ be the $p$-dimensional Ornstein-Uhlenbeck process with mean $\bzero$ and covariance $\E\left[\bz_s\bz_t^\mathsf{T}\right] = e^{-|t-s|}\bI_p/p$. 
By \eqref{correlation-recursion-finite-p} and \eqref{k-finite-p-def}, we may write $\Psi_p(q) = \frac{\sigma^2}{\delta \tau^2} + \E_{\bz_0,\bz_t}[\langle f(\bz_0), f(\bz_t)\rangle]$ for $t = \log(1/q)$.
Denoting the $i^\text{th}$ component of $f$ by $f_i$, we have the spectral representation
\begin{align*}
    f_i(\bx) = \sum_{\bk \in \integers_{\geq 0}^p} c_{i \bk} \prod_{j=1}^p \phi_{k_j}(x_j),
\end{align*}
where for each $k \geq 0$ we have $\phi_{k}$ is the eigenvector of the generator of the univariate Ornstein-Uhlenbeck process corresponding to eigenvalue $k$ \cite[p.~134]{Gardiner1985HandbookScience}. The equality is in $L_2$ with respect to base measure $\left(\frac{p}{2\pi}\right)^{p/2}e^{-\frac{p}{2}\|\bx\|^2}d\bx$.
Then,
\begin{align*}
    \Psi_p(q) &= \sum_{i=1}^p \E_{\bz_0,\bz_t}[f_i(\bz_0)f_i(\bz_t)] = \sum_{i=1}^p\E_{\bz_0}[f_i(\bz_0) \E_{\bz_0,\bz_t}[f_i(\bz_t)|\bz_0]] \\
    &= \sum_{\bk \in \integers^p_{\geq 0}} c_{i\bk} \E_{\bz_0}\left[f_i(\bz_0) \prod_{j=1}^p \phi_{k_j}(z_{0j}) e^{-k_jt}\right] = \sum_{i=1}^p \sum_{\bk \in \integers_{\geq 0}^p} c_{i\bk}^2e^{-\left(\sum_{j=1}^p k_j\right)t}\\
    & = \sum_{i=1}^p \sum_{\bk \in \integers_{\geq 0}^p} c_{i\bk}^2q^{\sum_{j=1}^p k_j},
\end{align*}
whence \eqref{q-recursion-finite-p-property-1}, \eqref{q-recursion-finite-p-property-2} follow.
\end{proof}

\noindent The proof of property \eqref{q-recursion-property-3} of Lemma \ref{lem-state-evolution-properties} requires the following technical lemma.

\begin{lemma}\label{lem-derivative-at-q-1}
    If $h :\reals^p \rightarrow \reals$ is Lipschitz and $\bz_1,\bz_2 \sim \mathsf{N}(\bzero,\bI_p/p)$ are independent, then
    $$
    \frac{\mathrm{d}}{\mathrm{d}q} \E_{\bz_1,\bz_2}\left[h\left(\bz_1\right)h\left(q\, \bz_1 + \sqrt{1-q^2}\bz_2\right)\right] \Big|_{q = 1}= -\frac1p\sum_{j=1}^p\E\left[ (\partial_j h(\bz))^2\right],
    $$
    (where the derivative on the left-hand side is a left-derivative, and the derivatives on the right-hand side exist almost everywhere by \cite[Theorem 3.2]{Evans2015MeasureFunctions}).
\end{lemma}

\begin{proof}
    This is an elementary fact, so we only sketch the proof idea.
    Denote by $F(q)$ the expectation on the left hand side, and $\bg_1:= (\bz_1+(q\, \bz_1 + \sqrt{1-q^2}\bz_2))/2$,
    $\bg_2= (\bz_1-(q\, \bz_1 + \sqrt{1-q^2}\bz_2))/2$ assuming $\nabla^2h$  bounded, Taylor's expansion implies
    \begin{align*}
      2(F(0)-F(q)) &= 4\, \E\{\<\nabla h(\bg_1),\bg_2\>^2\} +O(\E\{\|\bg_2\|^4_2\})
      =  \frac{2}{p}\E\{\|\nabla h(\bg_1)\|^2\} (1-q) +O((1-q)^2)\, .
    \end{align*}
    and the claim follows by dominated convergence.
This is extended to  general Lipschitz $h$ by a routine approximation argument.
  \end{proof}

\noindent We are now ready to prove property \eqref{q-recursion-property-3} of Lemma \ref{lem-state-evolution-properties}.

\begin{proof}[Proof of property \eqref{q-recursion-property-3} of Lemma \ref{lem-state-evolution-properties}]
As in the proof of properties \eqref{q-recursion-property-1}, \eqref{q-recursion-property-2}, define $f:\reals^p \rightarrow \reals^p$ the function which maps $\bx \mapsto \frac1{\sqrt \delta \tau}\left( \mathsf{prox}\left[\lambda \rho_p\right](\bbeta_0 + \tau \bx) - \bbeta_0 \right)$ and let $f_i$ be its $i$th coordinate.
Applying Lemma \ref{lem-derivative-at-q-1} to $h = f_i$ and summing over $i$, we have
\begin{align}
\Psi_p'(1) 
    &= 
    \frac1{p(\lambda\gamma+1)^2}\E_{\bz}[\|\mathrm{D}f(\bz)\|_{\mathsf{F}}^2] 
    = \frac1{\delta p (\lambda\gamma+1)^2 } \E_{\bz}\left[\left\|\mathrm{D}\,\mathsf{prox}\left[\lambdaorc\rho_p](\bbeta_0 + \tauorc \bz)\right]\right\|_{\mathsf{F}}^2\right]\nonumber\\
    &\leq 
    \frac1{\delta p (\lambda\gamma+1)^2 }\E_{\bz}\left[\left\|\mathrm{D}\, \mathsf{prox}\left[\lambdaorc \rho_p\right](\bbeta_0 + \tauorc \bz)\right\|_{\mathsf{op}}\left\|\mathrm{D}\, \mathsf{prox}\left[\lambdaorc \rho_p\right](\bbeta_0 + \tauorc\bz)\right\|_{\mathsf{nuc}}\right]\nonumber\\
    &\leq 
    \frac1{(\lambda \gamma + 1)^2}\frac1{\delta p} \E_{\bz}\left[ \left\|\mathrm{D}\,\mathsf{prox}\left[\lambdaorc \rho_p\right](\bbeta_0 + \tauorc \bz)\right\|_{\mathsf{nuc}}\right]\nonumber\\
    &=
    \frac1{(\lambda \gamma + 1)^2}\frac1{\delta p} \E_{\bz}\left[\mathsf{div}\, \mathsf{prox}\left[\lambdaorc \rho_p\right](\bbeta_0 + \tauorc \bz)\right]\nonumber\\
    &= 
    \frac1{(\lambda \gamma + 1)^2}\cdot \frac{1}{\delta }\mathsf{W_{reg,cvx}}(\tauorc,\lambdaorc,\cT_p).\label{correlation-recursion-derivative-at-1}\
\end{align}
In the first equality, we have used that $\tauorc^2/\tau^2 = 1/(\lambda\gamma+1)^2$.
In the first inequality, we have used that the operator and nuclear norms are dual with respect to the matrix inner product $\langle \bA, \bB \rangle = \Tr(\bA^\mathsf{T}\bB)$, which induces the Frobenius norm.
In the second inequality, we have applied \eqref{prox-jacobian-bound}.
In the second-to-last line we have used that $\|\mathrm{D}\, \mathsf{prox}[\lambda \rho_p](\bbeta_0 + \tau \bz)\|_{\mathsf{nuc}} = \mathsf{div}\,\mathsf{prox}[\lambda \rho_p](\bbeta_0 + \tau \bz)$ because all eigenvalues of $\mathrm{D}\,\mathsf{prox}[\lambda \rho_p](\bbeta_0 + \tau \bz)$ are non-negative by \eqref{prox-jacobian-non-negative-definite}. 
In the last equality, we have used \eqref{w-finite-p-def} and \eqref{prox-gaussian-IBP}.
Because $\tau,\lambda,\gamma,\delta,\cT$ is a strongly stationary quadruplet, by \eqref{fixed-pt-prior-asymptotic} we have $\lim_{p \rightarrow \infty} \frac{1}{\delta (\lambda\gamma+1)}\mathsf{W_{reg,cvx}}(\tauorc,\lambdaorc,\cT_p) < 1$, whence $\limsup_{p \rightarrow \infty} \Psi_p'(1) \leq \frac1{\lambda \gamma + 1}$. Further, by \eqref{w-finite-p-def} and \eqref{prox-width-bound}, we have $\mathsf{W_{reg,cvx}}(\tauorc,\lambdaorc,\cT) \leq 1$, whence we also have $\limsup_{p \rightarrow \infty} \Psi_p'(1) \leq \frac1{\delta}$. Inequality \eqref{q-recursion-finite-p-property-3} follows.
\end{proof}

\noindent 
We are ready to verify that the recursion \eqref{correlation-recursion}, \eqref{full-se} makes sense and establish some of its properties.
By \eqref{q-recursion-property-1} and \eqref{q-recursion-property-2}, we have $\Psi(q) \leq 1$ for $q \in [0,1]$, and by \eqref{q-recursion-property-3}, we have $\Psi(q) \geq 0$ for $q \in [0,1]$. 
Then, inductively we have $q_{s,t} \in [0,1]$ for all $s,t$, so that \eqref{full-se} makes sense.
Further, by \eqref{q-recursion-property-3}, we have for all $t \geq 1$ that $1 - q_{t+1,t+2} = 1 - \Psi(q_{t,t+1}) \leq \frac{1}{(\lambda \gamma + 1) \vee \delta}(1 - q_{t,t+1})$, so that inductively, with base case $1 - q_{1,2} = 1$, we have 
\begin{equation}\label{one-minus-q-goes-to-0}
    1 - q_{t,t+1} \leq \left( \frac{1}{(\lambda \gamma + 1) \vee \delta}\right)^{t-1}.
\end{equation} 
If either $\lambda \gamma > 0$ or $\delta > 1$, we have
\begin{equation}\label{q-goes-to-1}
    q_{t,t+1} \xrightarrow[t \rightarrow \infty]{} 1.
\end{equation} 
Further, by \eqref{full-se} and \eqref{q-recursion-property-1}, we get for all $t \geq 1$,
\begin{equation}\label{q-tt-is-1}
    q_{t,t} = 1.
\end{equation}

\subsection{Relating AMP and state evolution}\label{app-relating-AMP-and-SE}
We will show that for $t \geq 2$,
\begin{subequations}\label{iterates-convergence}
\begin{gather}
    \lim^{\mathrm{p}}_{p \rightarrow \infty} \sqrt{n}\left\| \br^t - \br^{t+1}\right\| = \sqrt{2(1 - q_{t+1,t+2})}\tau,\label{residual-iterates-convergence}\\
    \lim^{\mathrm{p}}_{p \rightarrow \infty} \left\|(\widehat \bbeta^t + \bX^\mathsf{T}\br^t) - (\widehat \bbeta^{t+1} + \bX^\mathsf{T}\br^{t+1})\right\|  = \sqrt{2(1 - q_{t+1,t+2})}\tau.\label{signal-iterates-convergence}
\end{gather}
\end{subequations}
These identities are a consequence of the characterization of the AMP iteration proved in \cite{Berthier2017StateFunctions}, as we now describe. 
The authors of \cite{Berthier2017StateFunctions} study a more general AMP iteration given by
\begin{align}\label{AMP}
    \bv^t &= \frac1{\sqrt n} \bX e_t(\bu^t) - \mathsf{\hat b}_tg_{t-1}(\bv^{t-1}),
    & \bu^{t+1} &= \frac1{\sqrt n}\bX^\mathsf{T} g_t(\bv^t) - \mathsf{\hat d}_t e_t(\bu^t),
\end{align}
with initialization given by deterministic vector 
\begin{equation}\label{AMP-initialization}
    \bu^0 \in \reals^p\quad\text{and}\quad g_{-1}(\cdot) = \bzero,
\end{equation}
and for each $t\geq0$ the functions $e_t:\reals^p \rightarrow \reals$ and $g_t:\reals^n\rightarrow \reals^n$ are uniformly (in $p$) pseudo-Lipschitz of order 1 (a.k.a.\ uniformly Lipschitz).
In \cite{Berthier2017StateFunctions}, iteration \eqref{AMP} is written with respect to a random matrix $\bA$ with entries $A_{ij}\stackrel{\mathrm{iid}}\sim \mathsf{N}(0,1/n)$.
We have replaced this with $\bX/\sqrt n$, which is distributed in this way.
Theorem 1 and Corollary 2 of \cite{Berthier2017StateFunctions} give that certain functionals of the iterates in \eqref{AMP} converge in probability to deterministic constants given by a scalar iteration called state evolution, of which the iteration in Appendix \ref{app-the-state-evolution} is, as we will see, a special case.
In particular, we will show that iteration \eqref{M-AMP} is a special case of the iteration \eqref{AMP}, the scalar recursion \eqref{correlation-recursion} and \eqref{full-se} is the corresponding state evolution, and the limits \eqref{iterates-convergence} are the result of Theorem 1 and Corollary 2 of \cite{Berthier2017StateFunctions}  applied to particular functions.\footnote{In fact, most of this task has already been carried out by Theorem 14 of \cite{Berthier2017StateFunctions}.
Unfortunately, Theorem 14 of \cite{Berthier2017StateFunctions} uses a different initialization than \eqref{M-AMP-initialization} and does not address limits of the form  \eqref{iterates-convergence}. 
Thus, Theorem 14 gives us almost what we need, but not quite. 
To conclude  \eqref{iterates-convergence}, we perform the required change of variables and apply their more general theorem on the iteration \eqref{AMP} ourselves.} To avoid confusion with corollaries which appear in this paper, we will refer to Corollary 2 of \cite{Berthier2017StateFunctions} as Corollary SE.

Iteration \eqref{M-AMP} is equivalent to iteration \eqref{AMP} under the following change of variables. 
\begin{equation}\label{COV}
\begin{aligned}
    \bv^t &= \sqrt{p/n}\bw - \sqrt{np}\br^t, & 
    \bu^{t+1} &= \sqrt{p}\left(\bbeta_0 - \left(\bX^\mathsf{T}\br^t + \widehat \bbeta^t\right)\right),\\
    e_t(\bu) &= \sqrt{p}\big(\mathsf{prox}[\lambda \rho_p^{(\gamma)}]\left(\bbeta_0 - \bu/\sqrt{p}\right) - \bbeta_0\big), \quad t \geq 0, & 
    g_t(\bv) &= \bv - \sqrt{p/n}\bw, \; t \geq 0,\\
    \bu^0 &= \sqrt{p} \tau \bz^0, & 
    \mathsf{\hat b}_t &= -\mathsf{b} \quad\text{and}\quad \mathsf{\hat d}_t = 1.
\end{aligned}
\end{equation}

Due to their different choice of normalization, the authors of \cite{Berthier2017StateFunctions} use a slightly different notion of a collection of functions' being uniformly pseudo-Lipschitz of order $k$ than used in this paper. 
In particular, for them a collection of functions $\{\varphi: (\reals^p)^{\ell} \rightarrow \reals^m\}$, where $p$ and $m$ but not $\ell$ may vary, is uniformly pseudo-Lipschitz of order $k$ if for all $\varphi$ and $\bx_i,\by_i \in \reals^p,\,i=1,\ldots,\ell$, we have
\begin{equation}\label{berthier-pseudo-lipschitz}
\frac{\|\varphi(\bx_1,\ldots,\bx_\ell) - \varphi(\by_1,\ldots,\by_\ell)\|}{\sqrt m} \leq C \left(1 + \sum_{i=1}^\ell  \left(\frac{\|\bx_i\|}{\sqrt p}\right)^{k-1} + \sum_{i=1}^\ell \left(\frac{\|\by_i\|}{\sqrt p}\right)^{k-1}\right)\sum_{i=1}^\ell \frac{\|\bx_i - \by_i\|}{\sqrt p},
 \end{equation}
 for some $C$ which does not depend on $p,m$.
 We will refer to their notion as \cite{Berthier2017StateFunctions}-uniformly pseudo-Lipschitz of order $k$.
 It is exactly equivalent to our notion under a change of normalization.
 In particular, the following claim is easy to check.
 \begin{claim}\label{claim-pseudo-lipschitz-change-of-normalization}
 A collection of functions $\{\varphi\}$ is uniformly pseudo-Lipschitz of order $k$ if and only if the collection of functions $\{\tilde \varphi\}$ defined by $\tilde\varphi(\bx_1,\ldots,\bx_\ell) = \sqrt m \varphi\left(\bx_1/\sqrt p,\ldots,\bx_\ell/\sqrt p\right)$ is \cite{Berthier2017StateFunctions}-uniformly pseudo-Lipschitz of order $k$. 
 \end{claim}
\noindent  This will allow us to translate their conditions and results into our normalization. 
Corollary SE requires six assumptions on the iteration \eqref{AMP}, which the authors label (B1) - (B6). 
In our setting, assumption (B1) holds by assumption; assumption (B2) holds by \eqref{COV}, \eqref{prox-is-lipschitz}, and inspection;
assumption (B3), (B4), and (B5) hold by \eqref{COV}, \eqref{z0-looks-normal}, and the HDA assumption.
Assumption (B6) holds by \eqref{COV}, strong stationarity (i.e.\ definition \eqref{k-finite-p-def} and the existence of the limit \eqref{summary-functions-fixed-asymptotic}) and the proximal operator identity \eqref{oracle-prox-noise-reduction-form}, the HDA assumption $n/p \rightarrow \delta$, and the DSN assumption $\|\bw\|^2/n \rightarrow \sigma^2$. 

Finally, the authors of \cite{Berthier2017StateFunctions} require that 
\begin{equation}\label{correction-factor}
    \mathsf{\hat d}_t \stackrel{\mathrm p}\simeq \frac1n \E_{\bz}\left[\mathrm{div}\, g_t(\Sigma_{t,t}\sqrt n \bz)\right],\quad \mathsf{\hat b}_t \stackrel{\mathrm{p}}\simeq \frac1n \E_{\bz}\left[\mathrm{div}\, e_t(T_{t,t} \sqrt p \bz)\right],
\end{equation}
where $\Sigma_{t,t}$ and $T_{t,t}$ are deterministic scalars which we now define.
The authors of \cite{Berthier2017StateFunctions} define the  double infinite arrays $(\Sigma_{s,t})_{s,t\geq 0}$ and  $(T_{s,t})_{s,t\geq1}$ through the recursion
\begin{subequations}\label{AMP-se}
\begin{gather}
    T_{s+1,t+1} = \lim_{p\rightarrow \infty}\frac1n\E_{\bz_1,\bz_2}\left[\left\langle g_s\left(\sqrt n\bz_1\right),g_t\left(\sqrt n\bz_2\right)\right\rangle\right],\label{AMP-se-1}\\
    \Sigma_{s,t} = \lim_{p \rightarrow \infty} \frac1n \E_{\bz_1,\bz_2}\left[\left\langle e_s\left(\sqrt p \bz_1\right),e_t\left(\sqrt p\bz_2\right)\right\rangle \right],\label{AMP-se-2}\\
    \Sigma_{0,0} = \lim_{p \rightarrow \infty} \frac1n \|e_0(\bu^0)\|^2, \quad \Sigma_{0,i} = \Sigma_{i,0} = 0 \text{ for } i \geq 1,\label{AMP-se-initialization}
\end{gather}
\end{subequations}
where in \eqref{AMP-se-1} we take $(\bz_1,\bz_2) \sim \mathsf{N}\left(\bzero,\begin{pmatrix} \Sigma_{s,s} & \Sigma_{s,t} \\ \Sigma_{t,s} & \Sigma_{t,t} \end{pmatrix} \otimes \bI_n/n \right)$ and in \eqref{AMP-se-2} we take $(\bz_1,\bz_2) \sim \mathsf{N}\left(\bzero,\begin{pmatrix} T_{s,s} & T_{s,t} \\ T_{t,s} & T_{t,t} \end{pmatrix} \otimes \bI_p/p \right)$.
We claim that for all $s,t \geq 1$,
\begin{equation}\label{T-to-q}
    T_{s,t} = q_{s,t}\tau^2.
\end{equation}
We establish this inductively.
By \eqref{COV}, \eqref{z0-looks-normal}, and HDA assumption $n/p \rightarrow \delta$, we have
\begin{align}
    \Sigma_{0,0} &=  \lim_{p \rightarrow \infty} \frac1n\|e_0(\bu^0)\|^2= \lim_{p \rightarrow \infty} \frac pn  \|\mathsf{prox}[\lambda \rho_p^{(\gamma)}](\bbeta_0 + \tau \bz^0) - \bbeta_0\|^2 = \tau^2 - \frac1\delta \sigma^2.\label{Sigma-00}
\end{align}
Moreover, for any $s,t \geq 0$, we have by \eqref{COV}, \eqref{AMP-se-1}, the HDA assumption $n/p \rightarrow \delta$, and the DSN assumption $\|\bw\|^2/n \rightarrow \sigma^2$, that 
\begin{align}\label{T-recursion}
    T_{s+1,t+1} &= \lim_{p\rightarrow \infty}\frac1n\E_{\bz_1,\bz_2}\left[\left\langle \sqrt n \bz_1 - \sqrt{p/n} \bw, \sqrt n \bz_2 - \sqrt{p/n} \bw\right\rangle\right] = \frac1\delta \sigma^2 +  \Sigma_{s,t}.
\end{align}
By \eqref{Sigma-00} and \eqref{T-recursion}, we have $T_{1,1} = \tau^2$, the base case.
Now assume \eqref{T-to-q} holds for all $1\leq s,t\leq l$.
Fix $1\leq s,t\leq l$.
By  \eqref{COV}, \eqref{AMP-se-2}, strong stationarity definition \eqref{k-finite-p-def} and condition \eqref{summary-functions-fixed-asymptotic}, and HDA assumption $n/p \rightarrow \delta$, we have
\begin{align}
    T_{s+1,t+1} &= \frac1\delta \sigma^2 + \lim_{p \rightarrow \infty} \frac pn \E_{\bz_1,\bz_2}\left[\langle \mathsf{prox}[\lambda \rho_p^{(\gamma)}](\bbeta_0 - \bz_1)-\bbeta_0,\mathsf{prox}[\lambda \rho_p^{(\gamma)}](\bbeta_0 - \bz_2) - \bbeta_0\rangle \right]\nonumber\\
    &= \frac1\delta\sigma^2 + \frac1\delta \mathsf{K^\infty_{reg,cvx}}\left(\bQ_{q_{s,t}},\lambdaorc,\cT\right)
    = \Psi(q_{s,t})\tau^2 
    = q_{s+1,t+1}\tau^2,\label{T-recursion-full-calculation}
\end{align}
where we have used in the second equality the HDA assumption $n/p \rightarrow \delta$, the inductive hypothesis that $(\bz_1,\bz_2) \sim \mathsf{N}\left(\bzero , \begin{pmatrix} T_{s,s} & T_{s,t} \\ T_{t,s} & T_{t,t}  \end{pmatrix}\otimes \bI_p / p \right) = \mathsf{N}(\bzero,\bQ_{s,t} \otimes \bI_p/p)$, and the oracle proximal identity \eqref{oracle-prox-noise-reduction-form}; in the third equality, we have used definition \eqref{correlation-recursion}; and in the last equality, we have used \eqref{full-se-2}. 
This confirms the inductive step.
Thus, state evolution \eqref{correlation-recursion}, \eqref{full-se-2} exactly corresponds to the state evolution \eqref{AMP-se}, \eqref{AMP-se-initialization} of \cite{Berthier2017StateFunctions}.

Now we are able to verify assumptions \eqref{correction-factor}, which are the final assumptions the authors of \cite{Berthier2017StateFunctions} require for Corollary SE.
By \eqref{COV}, we see that $\mathrm{div}\,g_t = n$, so that again by \eqref{COV} we see the first identity in \eqref{correction-factor} holds with equality even in finite samples.
By \eqref{T-to-q} and \eqref{q-tt-is-1}, we have $T_{t,t} = \tau^2$ for all $t \geq 1$.
The second identity in \eqref{correction-factor} holds because 
\begin{align}
    \frac1n \E_{\bz}\left[(\mathrm{div}\, e_t)(\tau \sqrt p \bz)\right] &= -\frac{1}{n} \E_{\bz}\left[(\mathrm{div}\, \mathsf{prox}[\lambda \rho_p^{(\gamma)}])(\bbeta_0 - \tau  \bz)\right] 
    = 
    \frac{p}{\tauorc (\lambda\gamma+1) n}\E_{\bz}[\langle \bz , \mathsf{prox}[\lambdaorc \rho_p](\bbeta_0 - \tauorc  \bz) \rangle ]\nonumber\\
    &\stackrel{\mathsf{p}}\simeq 
    -\frac1{\delta(\lambda\gamma+1)}\mathsf{W^\infty_{reg,cvx}}(\tauorc,\lambdaorc,\cT)
    = 
    \frac1{2\lambda} - 1 = - \mathsf{b} =  \mathsf{\widehat b}^t,
\end{align}
where in the first equality we have used \eqref{COV}; in the second equality we have used \eqref{prox-gaussian-IBP}; and in the fourth equality we have used strong stationarity condition \eqref{fixed-pt-prior-asymptotic} and  \eqref{eqdef-b}.
The third equality has two distinct justifications, depending on whether we are working under the DSN assumption, or under the RSN assumption conditional on the realization of $\{\bbeta_0\}$.
Under the DSN assumption and if the penalties are symmetric, we have used strong stationarity definition \eqref{w-finite-p-def}, condition \eqref{summary-functions-fixed-asymptotic}, and Lemma \ref{lem-pseudo-lipschitz-empirical-to-expectation}.
Under the RSN assumption and if the penalties are not necessarily symmetric, we have instead used Lemma \ref{lem-emp-to-exp-non-symm}.
Having verified (B1) - (B6) and \eqref{correction-factor}, we have verified all assumptions required to apply Corollary SE.

Finally, we show that Corollary SE implies \eqref{residual-iterates-convergence}, \eqref{signal-iterates-convergence}.
The collection of maps $(\reals^p)^2 \ni (\bx,\bx') \mapsto \frac{\|\bx - \bx'\|}{\sqrt p}$ is \cite{Berthier2017StateFunctions}-uniformly pseudo-Lipschitz of order 1. Thus, Corollary SE gives (because we have verified its assumptions) that $
    \lim_{p \rightarrow \infty} \frac{\|\bv^{t-1} - \bv^t\|}{\sqrt p} = \sqrt{\Sigma_{t-1,t-1} + \Sigma_{t,t} - 2\Sigma_{t-1,t}}$ and $
    \lim_{p \rightarrow \infty} \frac{\left\|\bu^{t+1} - \bu^{t+2}\right\|}{\sqrt p} = \sqrt{T_{t+1,t+1} + T_{t+2,t+2} - 2T_{t+1,t+2}}$
in probability. 
Under the change of variables \eqref{COV} and using \eqref{q-tt-is-1}, \eqref{T-to-q}, and \eqref{T-recursion}, we get \eqref{residual-iterates-convergence}, \eqref{signal-iterates-convergence}.

\subsection{Relating AMP and convex optimization}

We now complete the proof of parts (ii) and (iii) of Proposition \ref{prop-strongly-convex-loss}.
Observe that by \eqref{COV} and \eqref{prox-is-lipschitz},
\begin{align}
    \frac{\|e_t(\bzero) \|^2}{p} &= \|\mathsf{prox}[\lambda \rho_p^{(\gamma)}](\bbeta_0) - \bbeta_0\|^2 \leq \left(\|\mathsf{prox}[\lambda \rho_p^{(\gamma)}](\bbeta_0 - \tau \bz) - \bbeta_0\| + \tau \|\bz\|\right)^2\nonumber\\
    &\leq 2\|\mathsf{prox}[\lambda \rho_p^{(\gamma)}](\bbeta_0 - \tau \bz) - \bbeta_0\|^2 + 2\tau^2 \|\bz\|^2.
\end{align}
Considering $\bz \sim \mathsf{N}(\bzero,\bI_p/p)$, taking expectations on both sides, and using \eqref{r-finite-p-def} and \eqref{summary-functions-fixed-asymptotic}, we get that $\frac{\|e_t(\bzero)\|}{\sqrt p}$ is bounded.
Moreover, by \eqref{prox-is-lipschitz} and \eqref{COV}, we have $\bx \mapsto \frac{e_t(\sqrt p \bx)}{\sqrt p}$ is uniformly (over $p$) pseudo-Lipschitz of order 1. 
By these two facts, Lemma \ref{lem-pseudo-lipschitz-closed-under-composition} gives that $\bx \mapsto \frac{\|e_t(\sqrt p \bx)\|^2}{p}$ is uniformly pseudo-Lipschitz of order 2. 
By Claim \ref{claim-pseudo-lipschitz-change-of-normalization}, we have $\bu \mapsto \frac{\|e_t(\bu)\|^2}{p}$ is \cite{Berthier2017StateFunctions}-uniformly pseudo-Lipschitz of order 2.
Thus, by \eqref{M-AMP-2}, \eqref{COV}, Corollary SE, oracle proximal identity \eqref{oracle-prox-noise-reduction-form}, and strong stationarity condition \eqref{fixed-pt-prior-asymptotic}, we have
\begin{align}\label{beta-hat-t-loss}
    \left\|\widehat \bbeta^{t+1} - \bbeta_0\right\|^2 &= \frac{\left\|e_t(\bu^{t+1})\right\|^2}{p} \stackrel{\mathrm{p}}\simeq \E_{\bz} \left[\frac{\left\|e_t(\tau\sqrt p \bz)\right\|^2}{p}\right]\stackrel{\mathrm{p}}\simeq \mathsf{R^\infty_{reg,cvx}}(\tauorc,\lambdaorc,\cT).
\end{align}
By the triangle inequality,
\begin{align}
    &\left|\left\|\widehat \bbeta_{\mathsf{cvx}} - \bbeta_0\right\|- \sqrt{\mathsf{R^\infty_{reg,cvx}}(\tauorc,\lambdaorc,\cT)}\right| \nonumber\\
        &\qquad\qquad\qquad  \leq \left\|\widehat \bbeta^{t+1} - \widehat \bbeta_{\mathsf{cvx}}\right\| + \left|\left\|\widehat \bbeta^{t+1} - \bbeta_0\right\| - \sqrt{\mathsf{R^\infty_{reg,cvx}}(\tauorc,\lambdaorc,\cT)}\right|.\nonumber
\end{align}
By \eqref{beta-hat-t-loss} and  \eqref{limsup-p-sums}, we have for fixed $t$
\begin{align}\label{pre-t-post-p-limit}
&\limsup_{p \rightarrow \infty}^{\mathrm{p}} \left|\left\|\widehat \bbeta_{\mathsf{cvx}} - \bbeta_0\right\| - \sqrt{\mathsf{R^\infty_{reg,cvx}}(\tauorc,\lambdaorc,\cT)}\right| \leq \limsup_{p \rightarrow \infty}^{\mathrm{p}}\left\|\widehat \bbeta^{t+1} - \widehat \bbeta_{\mathsf{cvx}}\right\|\nonumber\\
&\qquad\leq  \limsup_{p \rightarrow \infty}^{\mathrm{p}} \frac2\kappa\left(2\mathsf{b}\left\|\bX\right\|_{\mathsf{op}}\left\| \br^t - \br^{t+1}\right\| + \frac{\left\|(\widehat \bbeta^t + \bX^\mathsf{T}\br^t) - (\widehat \bbeta^{t+1} + \bX^\mathsf{T}\br^{t+1})\right\| }{\lambda}\right),
\end{align}
where $\kappa > 0$ is such that with probability going to 1 as $p \rightarrow \infty$, we have $L$ is $\kappa$ strongly convex and in the second inequality, we have used \eqref{iterates-to-minimizer-strong-convexity-bound}. .
Such a $\kappa$ exists whenever $\delta > 0$, $\gamma > 0$, or $\{\rho_p\}$ has positive uniform strong convexity parameter.
By \eqref{q-goes-to-1}, \eqref{residual-iterates-convergence}, and \eqref{limsup-p-products}, we have
\begin{gather}\label{p-then-t-residuals-to-zero}
    \lim_{t \rightarrow \infty} \limsup_{p \rightarrow \infty}^{\mathrm{p}} \left\|\bX\right\|_{\mathsf{op}}\left\| \br^t - \br^{t+1}\right\| \leq \lim_{t \rightarrow \infty} \limsup_{p \rightarrow \infty}^{\mathrm{p}} \frac{\|\bX\|_{\mathsf{op}}}{\sqrt n} \sqrt{2(1-q_{t+1,t+2})} \tau = 0,
\end{gather}
where we have used that $\limsup\limits _{p \rightarrow \infty}^{\mathrm{p}}\|\bX\|_{\mathsf{op}}/\sqrt n < \infty$ (see \cite[Theorem 5.31]{Vershynin2012IntroductionMatrices}).
Similarly, by \eqref{q-goes-to-1} and \eqref{signal-iterates-convergence}, we have
\begin{gather}\label{p-then-t-signals-to-zero}
    \lim_{t \rightarrow \infty} \lim_{p \rightarrow \infty}^{\mathrm{p}} \frac{\left\|(\widehat \bbeta^t + \bX^\mathsf{T}\br^t) - (\widehat \bbeta^{t+1} + \bX^\mathsf{T}\br^{t+1})\right\| }{\lambda} = \lim_{t \rightarrow \infty} \frac{\sqrt{2(1-q_{t+1,t+2})}\tau }{\lambda} = 0.
\end{gather}
We conclude
\begin{equation}\label{p-then-t-iterates-to-estimates}
    \lim_{t \rightarrow \infty} \limsup_{p \rightarrow \infty}^{\mathrm{p}} \left\|\widehat \bbeta^{t+1} - \widehat \bbeta_{\mathsf{cvx}}\right\|  = 0,
\end{equation}
whence again by \eqref{pre-t-post-p-limit}
\begin{equation}\label{p-then-t-estimates-to-beta-0}
    \lim_{t \rightarrow \infty} \limsup_{p \rightarrow \infty}^{\mathrm{p}} \left|\left\|\widehat \bbeta_{\mathsf{cvx}} - \bbeta_0\right\| - \sqrt{\mathsf{R^\infty_{reg,cvx}}(\tauorc,\lambdaorc,\cT)}\right| = 0.
\end{equation}
Thus, \eqref{converging-sequence-l2-loss} holds, as desired. This complete the proof of part (ii) of Proposition \ref{prop-strongly-convex-loss}.

Now we complete the proof of part (iii) of Proposition \ref{prop-strongly-convex-loss}. 
Take $\varphi_p$ as given in part (iii).
By the DSN assumption, $\widehat \pi_{\bbeta_0} \stackrel{\mathrm{W}}\rightarrow \pi \in \cP_2(\reals)$, we have $\|\bbeta_0\|$ is bounded (over $p$).
Thus, by Lemmas \ref{lem-pseudo-lipschitz-closed-under-fixing-arguments} and \ref{lem-pseudo-lipschitz-closed-under-composition}, we have
$
 \psi_p\left(\bx\right) = \varphi_p(\bbeta_0,\bbeta_0  - \bx),
$
is uniformly pseudo-Lipschitz of order $k$.
Then by Claim \ref{claim-pseudo-lipschitz-change-of-normalization}
\begin{align}
\tilde \varphi_p(\bu) = \psi_p(\bu/\sqrt p)
\end{align}
is \cite{Berthier2017StateFunctions}-uniformly pseudo-Lipschitz of order $k$.
Corollary SE and \eqref{COV} then gives 
\begin{align}\label{arbitrary-loss-applied-to-iterates}
\varphi_p\left(\bbeta_0,\widehat \bbeta^t + \bX^{\mathsf{T}}\br^t \right) &= \varphi_p\left(\bbeta_0,\bbeta_0 - \bu^{t+1}/\sqrt p\right)= \psi_p\left(\bu^{t+1}/\sqrt p\right) = \tilde \varphi_p\left(\bu^{t+1}\right)\nonumber\\
& \stackrel{\mathrm{p}}\simeq \E_{\bz}\left[\tilde \varphi_p(\tau \sqrt p \bz)\right] = \E_{\bz}\left[\varphi_p(\bbeta_0,\bbeta_0 - \tau \bz)\right].
\end{align}
By \eqref{M-AMP-1}, we have
\begin{equation}
\frac{\by - \bX \widehat \bbeta_{\mathsf{cvx}}}{(1 - \mathsf{b})n} = \frac{\by - \bX \widehat \bbeta^t}{(1-\mathsf{b})n} + \frac{\bX(\widehat \bbeta^t - \widehat \bbeta_{\mathsf{cvx}})}{(1-\mathsf{b})n} = \br^t + \frac{\mathsf{b}}{1 - \mathsf{b}}(\br^t - \br^{t-1}) +   \frac{\bX(\widehat \bbeta^t - \widehat \bbeta_{\mathsf{cvx}})}{(1-\mathsf{b})n}.
\end{equation}
Some algebra and the triangle inequality gives
\begin{align}
\Bigg\|\underbrace{\left(\widehat \bbeta^t + \bX^{\mathsf{T}} \br^t\right)}_{:= \ba^t} &- \underbrace{\left(\widehat \bbeta_{\mathsf{cvx}} + \frac{\bX^{\mathsf{T}}(\by - \bX\widehat \bbeta_{\mathsf{cvx}})}{(1 - \mathsf{b})n}\right)}_{:\bb^t}\Bigg\| = \left\|\widehat \bbeta^t - \widehat \bbeta_{\mathsf{cvx}} - \bX^\mathsf{T}\left(\frac{\mathsf{b}}{1 - \mathsf{b}}(\br^t - \br^{t-1}) + \frac{\bX(\widehat \bbeta^t - \widehat \bbeta_{\mathsf{cvx}})}{(1 - \mathsf{b})n}\right) \right|\nonumber\\
&\leq \left\|\widehat \bbeta^t - \widehat \bbeta_{\mathsf{cvx}}\right\| + \|\bX\|_{\mathsf{op}} \left\|\frac{\mathsf{b}}{1 - \mathsf{b}}(\br^t - \br^{t-1}) + \frac{\bX(\widehat \bbeta^t - \widehat \bbeta_{\mathsf{cvx}})}{(1 - \mathsf{b})n}\right\|\nonumber\\
& \leq \left(1 + \frac{\|\bX\|_{\mathsf{op}}^2}{(1 - \mathsf{b})n}\right)\left\|\widehat \bbeta^t - \widehat \bbeta_{\mathsf{cvx}} \right\| +\frac{ \|\bX\|_{\mathsf{op}} \mathsf{b}}{1 - \mathsf{b}} \|\br^t - \br^{t-1}\|,
\end{align}
where we have defined $\ba^t,\bb^t$ for future reference.
Now combining \eqref{p-then-t-residuals-to-zero}, \eqref{p-then-t-iterates-to-estimates}, and $\limsup\limits _{p \rightarrow \infty}^{\mathrm{p}}\|\bX\|_{\mathsf{op}}/\sqrt n < \infty$ (see \cite[Theorem 5.31]{Vershynin2012IntroductionMatrices}) using \eqref{limsup-p-sums} and \eqref{limsup-p-products}, we get 
\begin{gather}\label{a-close-to-b}
\lim_{t \rightarrow \infty} \limsup_{p \rightarrow \infty}^\mathrm{p} \|\ba^t - \bb^t\| = 0.
\end{gather}
In the remainder of the argument, we let $C$ be a constant which does not depend on $p$ or $t$ but which may change at each appearance.
By \eqref{COV}, for each $t$
\begin{align}\label{a-not-too-big}
\|\ba^t\| = \|\bbeta_0 - \bu^{t+1}/\sqrt{p}\| \leq \|\bbeta_0\| + \|\bu^{t+1}\|/\sqrt{p} \stackrel{\mathrm{p}}\simeq \|\bbeta_0\| +  \tau\E_{\bz}[\|\bz\|] \stackrel{\mathrm{p}}\simeq s_2^{1/2}(\pi) + \tau,
\end{align}
where for each $p$ we let $\bz \sim \mathsf{N}(\bzero,\bI_p/p)$, and
in the first probabilistic equality we have used Corollary SE and that $\bu \mapsto \|\bu\|/\sqrt{p}$ is \cite{Berthier2017StateFunctions}-uniformly pseudo-Lipschitz of order 1, and in the second probabilistic equality we have used the DSN assumption \eqref{DSN-assumption}.
Because $\|\bb^t\| \leq \|\ba^t\| + \|\ba^t - \bb^t\|$, by \eqref{a-close-to-b}, \eqref{a-not-too-big}, and \eqref{limsup-p-sums},  we have for every $t$ that
\begin{align}\label{b-not-too-big}
\limsup_{p \rightarrow \infty}^{\mathrm{p}} \|\bb^t\| \leq s_2^{1/2}(\pi) + \tau.
\end{align}
Combining \eqref{a-not-too-big} and \eqref{b-not-too-big} using \eqref{limsup-p-sums}, we have for each $t$
\begin{equation}\label{a-b-pseudo-lipschitz-prefactor-finite}
    \limsup_{p \rightarrow \infty}^{\mathrm{p}} \left(1 + \|\bbeta_0\|^{k-1} + \|\ba^t \|^{k-1} + \|\bb^t\|^{k-1} \right) \leq  C.
\end{equation}
Because the $\varphi_p$ are uniformly pseudo-Lipschitz of order $k$,
\begin{align}
    \lim_{t \rightarrow \infty} \limsup_{p \rightarrow \infty}^{\mathrm{p}} &\left|\varphi_p\left(\bbeta_0,\ba^t \right) - \varphi_p\left(\bbeta_0,\bb^t \right)\right|
    \leq C\lim_{t \rightarrow \infty} \limsup_{p \rightarrow \infty}^{\mathrm{p}} \left(1 + \|\bbeta_0\|^{k-1} + \|\ba^t\|^{k-1} + \|\bb^t\|^{k-1}\right)\|\ba^t - \bb^t\|\nonumber\\
    & \leq C \lim_{t \rightarrow \infty} \limsup_{p \rightarrow \infty}^{\mathrm{p}} \|\ba^t - \bb^t\| = 0,\label{phi-a-close-to-phi-b}
\end{align}
where in the first inequality we have used Definition \ref{def-uniformly-pseudo-lipschitz}, in the second inequality we have used \eqref{limsup-p-products} and \eqref{a-b-pseudo-lipschitz-prefactor-finite}, and in the equality we have used \eqref{a-close-to-b}.
By \eqref{arbitrary-loss-applied-to-iterates}, \eqref{phi-a-close-to-phi-b}, and the triangle inequality, 
\begin{align}
\left|\varphi_p(\bbeta_0,\bb^t) - \E_{\bz}[\varphi(\bbeta_0,\bbeta_0 - \tau \bz)]\right| &\leq \left|\varphi_p(\bbeta_0,\bb^t) - \varphi_p(\bbeta_0,\ba^t)\right| + \left|\varphi_p(\bbeta_0,\ba^t) - \E_{\bz}[\varphi(\bbeta_0,\bbeta_0 - \tau \bz)]\right| \nonumber\\
&\stackrel{\mathrm{p}}\rightarrow 0.
\end{align}
Plugging in for $\bb^t$ yields \eqref{converging-sequence-pl-loss}, as desired.

Thus, we have shown part (iii) and completed the proof of Proposition \ref{prop-strongly-convex-loss} \hfill $\square$

\section{Proof of Theorem \ref{thm-cvx-lower-bound}}\label{sec-proof-of-thm-cvx-lower-bound}

The main technical challenge is that exact asymptotics for the estimation error of penalized least squares estimators rely on several  technical assumptions we would like to avoid.
We summarize the main technical hurdles below.
\begin{enumerate}

    \item \textbf{$\delta > 1$ or strong convexity.} 
    One set of technical assumptions under which exact asymptotics can be established in full generality is that either $\delta > 1$ or the penalties are strong-convexity. 
    Proposition \ref{prop-strongly-convex-loss} leverages this fact in establishing exact asymptotics for oracle estimators when either $\delta > 1$, $\gamma > 0$, or $\rho_p$ is uniformly strongly convex.
    To establish our lower bound when $\delta \leq 1$ and $\rho_p$ need not be uniformly strongly convex, we construct an oracle estimator with oracle parameter $\gamma > 0$ which improves the estimation error of the original estimator.
    Its exact characterization then provides a lower bound on the estimation error of the original estimator.

    \item \textbf{Strong stationarity.} 
    Proposition \ref{prop-strongly-convex-loss} requires the limits \eqref{summary-functions-fixed-asymptotic} exist and satisfy \eqref{fixed-pt-prior-asymptotic} (i.e., strong stationarity), but the $\delta$-bounded width assumption $\{\rho_p\} \in \cC_{\delta,\pi}$ does not require that \eqref{fixed-pt-prior-asymptotic} be satisfied or even that the limits \eqref{summary-functions-fixed-asymptotic} exist.
    To address this, we establish the existence of limits satisfying \eqref{fixed-pt-prior-asymptotic} along certain subsequences of penalties using a compactness argument. 
    This is done in the proof of Lemma \ref{lem-oracle-risk-to-cvx-bound}.

    \item \textbf{Solutions to fixed point equations are appropriately bounded.} 
    We must show the oracle estimator does not have risk which is too small.
    For this we will use the $\delta$-bounded width assumption.
    In fact, this is the only place the $\delta$-bounded width assumption is used in our argument.
    This is done in the proof of Lemma \ref{lem-oracle-risk-to-cvx-bound}.
    For further discussion of the role of the $\delta$-bounded width assumption, see Appendix \ref{sec-role-of-delta-bounded-width-restriction} (though our proof does not use that Appendix).

\end{enumerate}

The proof is organized as follows.
In Section \ref{subsec-bounded-shrinkage-towards-infinity}, we argue that without loss of generality it is enough to consider penalty sequences $\{\rho_p\}$ which satisfy an additional technical assumption.
This technical assumption will be important for the compactness argument mentioned in item 2 above.
In Section \ref{subsec-oracle-bounded-improvement}, we carry out the main technical steps of our proof: defining the oracle estimator and showing that its risk is smaller than the risk of the original estimator but is not much smaller than the convex lower bound \eqref{cvx-lower-bound}. 
The proof of the main technical lemma in that section, Lemma \ref{subsec-oracle-bounded-improvement}, is deferred to Appendix \ref{app-proof-of-lem-oracle-risk-to-cvx-bound}. It is here that the $\delta$-bounded width assumption plays a role.
In Section \ref{subsection-finish-cvx-lower-bound-proof}, we combine the first two parts to finish the proof of the lower bound.
In Section \ref{sec:tightness-cvx-lb}, we show the lower bound is tight when $\delta > 1$.

\subsection{Penalty sequences which do not shrink towards infinity}\label{subsec-bounded-shrinkage-towards-infinity}

The following claim shows that it is enough to prove Theorem \ref{thm-cvx-lower-bound} under the additional assumption that the penalty sequence does not shrink towards infinity (see \eqref{eqdef-bounded-shrinkage-towards-infinity}).
\begin{claim}\label{bounded-minimization-is-enough}
    To show \eqref{cvx-lower-bound} under the conditions of Theorem \ref{thm-cvx-lower-bound}, it is enough to show 
    \begin{equation}\label{cvx-with-bounded-minimization-lower-bound}
        \inf_{\{\rho_p\} \in \cC_{\delta,\pi} \cap \cB}\liminf^{\mathrm{p}}_{p\rightarrow \infty}\|\widehat \bbeta_{\mathsf{cvx}} - \bbeta_0\|^2 \geq \delta\tau^{2}_{\mathsf{reg,cvx}} - \sigma^2,
    \end{equation}
    under the conditions of Theorem \ref{thm-cvx-lower-bound}. (If $\cC_{\delta,\pi} \cap \cB$ is empty, we take the infimum to be infinite).
\end{claim}

\noindent The proof of Claim \ref{bounded-minimization-is-enough} is based on the following lemma. 

\begin{lemma}\label{lem-beta-hat-large-from-prox-zero}
        Fix $\pi \in \cP_2(\reals)$, $\delta \in (0,\infty)$, and $\sigma \geq 0$.
        For a sequence of convex functions $\{\rho_p\}$,
        we have 
        under there exist constants $c_1 >  0$ and $c_2 \geq 0$ depending only on $\pi,\delta,\sigma$ such that
    \begin{equation}\label{beta-hat-large-from-prox-zero}
        \liminf^{\mathrm{p}}_{p \rightarrow \infty} \|\widehat \bbeta_{\mathsf{cvx}} - \bbeta_0\|^2 \geq  \liminf_{p \rightarrow \infty} (c_1\|\mathsf{prox}[\rho_p](\bzero)\| - c_2)^2.
    \end{equation}
    (We use the same convention as in Theorem \ref{thm-cvx-lower-bound} when the minimizing set in \eqref{linear-cvx-estimator} is empty).
\end{lemma}

\begin{proof}[Proof of Lemma \ref{lem-beta-hat-large-from-prox-zero}]
    For each $p$, observe that whenever the minimizing set in \eqref{linear-cvx-estimator} is non-empty, $\frac2n\bX^\mathsf{T}(\by - \bX \widehat \bbeta_{\mathsf{cvx}}) \in \partial \rho_p (\widehat \bbeta_{\mathsf{cvx}})$, whence 
\begin{align*}
&\frac12 \|\bbeta\|^2 + \rho_p(\bbeta) \geq \frac12 \|\bbeta\|^2 + \rho_p(\widehat \bbeta_{\mathsf{cvx}}) + \frac2n (\bbeta - \widehat \bbeta_{\mathsf{cvx}})^\mathsf{T}\bX^\mathsf{T} (\by - \bX \widehat \bbeta_{\mathsf{cvx}}) \\
&\qquad= \frac12 \|\widehat \bbeta_{\mathsf{cvx}}\|^2 +\langle \widehat \bbeta_{\mathsf{cvx}}, \bbeta - \widehat \bbeta_{\mathsf{cvx}} \rangle  + \frac12 \|\bbeta - \widehat \bbeta_{\mathsf{cvx}}\|^2 + \rho_p(\widehat \bbeta_{\mathsf{cvx}})\\
&\qquad\qquad\qquad\qquad\qquad\qquad\qquad\qquad\qquad +\frac2n (\bbeta - \widehat \bbeta_{\mathsf{cvx}})^\mathsf{T}\bX^\mathsf{T} (\by - \bX \widehat \bbeta_{\mathsf{cvx}}) \\
&\qquad\geq \frac12 \|\widehat \bbeta_{\mathsf{cvx}}\|^2 + \rho_p(\widehat \bbeta_{\mathsf{cvx}}) + \left(\frac12 \|\bbeta - \widehat \bbeta_{\mathsf{cvx}}\| - \|\widehat \bbeta_{\mathsf{cvx}}\| - 2\frac{\|\bX\|_{\mathsf{op}}}{\sqrt n} \frac{\|\by-\bX \widehat \bbeta_{\mathsf{cvx}}\|}{\sqrt n}\right)\|\bbeta - \widehat \bbeta_{\mathsf{cvx}}\|\\
&\qquad\geq  \frac12 \|\widehat \bbeta_{\mathsf{cvx}}\|^2 + \rho_p(\widehat \bbeta_{\mathsf{cvx}}) + \left(\frac12 \|\bbeta\|  - \frac32 \|\widehat \bbeta_{\mathsf{cvx}}\| - 2\frac{\|\bX\|_{\mathsf{op}}}{\sqrt n}\frac{\|\by\|}{\sqrt n} - \frac{\|\bX\|_{\mathsf{op}}^2}{n} \|\widehat \bbeta_{\mathsf{cvx}}\|\right)\|\bbeta - \widehat \bbeta_{\mathsf{cvx}}\|.
\end{align*}
By \eqref{sequence-cvx-estimator}, $\frac12 \|\mathsf{prox}[\rho_p](\bzero)\|^2 + \rho(\mathsf{prox}[\rho_p](\bzero)) \leq  \frac12 \|\widehat \bbeta_{\mathsf{cvx}}\|^2 + \rho_p(\widehat \bbeta_{\mathsf{cvx}}) $. Thus, when evaluating the previous display at $\bbeta = \mathsf{prox}[\rho_p](\bzero)$, the expression in parentheses on the right-hand side is non-positive. 
That is,
\begin{align}\label{beta-cvx-shrunk-towards-infty}
    \|\widehat \bbeta_{\mathsf{cvx}}\| &\geq \frac{\frac12 \|\mathsf{prox}[\rho_p](\bzero)\| -2\frac{\|\bX\|_{\mathsf{op}}}{\sqrt n}\frac{\|\by\|}{\sqrt n}}{\frac32 + \frac{\|\bX\|_{\mathsf{op}}^2}{n}} \geq \frac{\frac12 \|\mathsf{prox}[\rho_p](\bzero)\| -2\frac{\|\bX\|_{\mathsf{op}}}{\sqrt n} \frac{\|\bw\| + \|\bX\|_{\mathsf{op}}\|\bbeta_0\|}{\sqrt n}}{\frac32 + \frac{\|\bX\|_{\mathsf{op}}^2}n}.
\end{align}
By \cite[Theorem 5.31]{Vershynin2012IntroductionMatrices} and the HDA assumption, 
\begin{equation}\label{largest-singular-value-convergence}
\|\bX\|_{\mathsf{op}}/\sqrt n  \stackrel{\mathrm{p}}\rightarrow 1 + \sqrt{1/\delta}=: c.
\end{equation}
Then, by the DSN assumption and the Continuous Mapping Theorem, $4\frac{\|\bX\|_{\mathsf{op}}}{\sqrt n} \frac{\|\bw\| + \|\bX\|_{\mathsf{op}}\|\bbeta_0\|}{\sqrt n} \stackrel{\mathrm{p}} \rightarrow 4c(\sigma + cs_2(\pi))$. 
Let $c_1 = \frac{1}{3 + 2 c^2}$ and $c_2 = 4cc_1(\sigma + cs_2(\pi)) + s_2(\pi)$. 
Then by \eqref{beta-cvx-shrunk-towards-infty} and Lemma \ref{lem-limsup-p-identities} from Appendix \ref{app-useful-tools}  whenever the minimizing set in \eqref{linear-cvx-estimator} is non-empty, and the convention $\|\binfty - \bbeta_0\|^2 = \infty$ otherwise, we have \eqref{beta-hat-large-from-prox-zero}.
\end{proof}

\noindent We now establish Claim \ref{bounded-minimization-is-enough}.

\begin{proof}[Proof of Claim \ref{bounded-minimization-is-enough}]
Assume we have shown \eqref{cvx-with-bounded-minimization-lower-bound} under the conditions of Theorem \ref{thm-cvx-lower-bound}.
Now, we assume the conditions of Theorem \ref{thm-cvx-lower-bound} and show the stronger \eqref{cvx-lower-bound}. 

Let $\{p(\ell)\}$ be the subsequence of $\{p\}$ containing exactly those $p$ for which $(c_1 \|\mathsf{prox}[\rho_p](\bzero)\| - c_2)^2 \geq \delta \tau_{\mathsf{reg,cvx}}^2 - \sigma^2 + 1$, and let $\{p'(\ell)\}$ its complement, that is, the subsequence of $\{p\}$ containing exactly those $p$ for which $(c_1 \|\mathsf{prox}[\rho_p](\bzero)\| - c_2)^2 < \delta \tau_{\mathsf{reg,cvx}}^2 - \sigma^2 + 1$.
We permit that one of these subsequences be finite.
It is straightforward to check that
\begin{equation}\label{liminf-split}
\liminf^{\mathrm{p}}_{p\rightarrow \infty}\|\widehat \bbeta_{\mathsf{cvx}}(p) - \bbeta_0\|^2 \geq \min\left\{\liminf^{\mathrm{p}}_{\ell\rightarrow \infty} \|\widehat \bbeta_{\mathsf{cvx}}(p(\ell)) - \bbeta_0\|^2 , \liminf^{\mathrm{p}}_{\ell\rightarrow \infty} \|\widehat \bbeta_{\mathsf{cvx}}(p'(\ell)) - \bbeta_0\|^2 \right\},
\end{equation}
if we adopt the convention that when either of these sequences is finite, the corresponding $\liminf\limits^{\mathrm{p}}$ is $\infty$.
We now check that each expression in the minimum on the right-hand side of \eqref{liminf-split} is bounded below by $\delta \tau_{\mathsf{reg,cvx}}^2 - \sigma^2$.
First, we show that $\liminf\limits^{\mathrm{p}}_{\ell\rightarrow \infty}\|\widehat \bbeta_{\mathsf{cvx}}(p(\ell)) - \bbeta_0\|^2 \geq \delta \tau_{\mathsf{reg,cvx}}^2 - \sigma^2$. If $\{p(\ell)\}$ is finite, there is nothing to check. If $\{p(\ell)\}$ is infinite, then we apply Lemma \ref{lem-beta-hat-large-from-prox-zero}.
Second, we show that $\liminf\limits^{\mathrm{p}}_{\ell\rightarrow \infty} \|\widehat \bbeta_{\mathsf{cvx}}(p'(\ell)) - \bbeta_0\|^2 \geq \delta \tau_{\mathsf{reg,cvx}}^2 - \sigma^2$. If $\{p'(\ell)\}$ is finite, there is nothing to check. If $\{p'(\ell)\}$ is infinite, then $\{\rho_{p'(\ell)}\} \in \cB$ by construction, and we apply the assumption of the claim.
Thus, for all $\{\rho_p\} \in \cC_{\delta,\pi}$, the left-hand side of \eqref{liminf-split} is bounded below by \eqref{cvx-lower-bound}, as desired.
\end{proof}

\subsection{Constructing oracles with not-too-small effective noise}
\label{subsec-oracle-bounded-improvement}

The bulk of the proof of Theorem \ref{thm-cvx-lower-bound} involves constructing a sequence of estimators to which we can apply the exact asymptotics of Proposition \ref{prop-strongly-convex-loss} and whose asymptotic estimation error is not too much smaller than that of the original sequence of estimators.
To do so, we take a subsequence of the $\{\rho_p\}$ and add a small but non-zero oracle term as in \eqref{oracle-penalty}.
The only place in our proof where the $\delta$-bounded width assumption plays a role is in showing that 
small oracle penalties cannot improve the estimation error too much.

\begin{lemma}\label{lem-oracle-risk-to-cvx-bound}
    Consider $\pi \in \cP_2(\reals)$, $\delta \in (0,\infty)$, and $\sigma \geq 0$. 
    Consider an increasing sequence of integers $\{p\}$ and a sequence $\{\rho_p\} \in \cC_{\delta,\pi} \cap \cB$. 
    \begin{enumerate}[(i)]

        \item 
        Consider arbitrary $\delta$ and sequence $\{\rho_p\} \in \cC_{\delta,\pi} \cap \cB$. 
        If $\tau_{\mathsf{lb}} > 0$ is such that $\delta \tau_{\mathsf{lb}}^2 - \sigma^2 < \mathsf{R_{seq,cvx}^{opt}}(\tau_{\mathsf{lb}};\pi)$, 
        then there exists sub-sequence $\{p(\ell)\}$, $\gamma > 0$, $\tau > \tau_{\mathsf{lb}}$, and $\lambda>0$ such that the following is true: with $\cT = (\pi, \{\rho_{p(\ell)}\})$, the quintuplet $\tau,\lambda,\delta,\gamma,\cT$ is strongly stationary.

        \item 
        Consider sequence $\{\rho_p\}$ and $\cT' = (\pi,\{\rho_p\})$.
        If $\delta > 1$,
        and $\tau\geq0$,$\lambda' > 0$ are such that
        $$
            \delta \tau^2 - \sigma^2 = \mathsf{R_{reg,cvx}^\infty}(\tau,\lambda',\cT'),
        $$
        then there exists sub-sequence $\{p(\ell)\}$ and $\lambda>0$ such that the following holds: with $\cT = (\pi,\{\lambda'\rho_p/\lambda\})$, the quintuplet $\tau,\lambda,\delta,\gamma=0,\cT$ is strongly stationary.

    \end{enumerate}

\end{lemma}
\noindent Most of the technical machinery of the proof of Theorem \ref{thm-cvx-lower-bound} is contained in the proof of Lemma \ref{lem-oracle-risk-to-cvx-bound}. The proof of Lemma \ref{lem-oracle-risk-to-cvx-bound} is provided in Appendix \ref{app-proof-of-lem-oracle-risk-to-cvx-bound}. 
In part (i), the reader should have in mind taking $\tau_{\mathsf{lb}}\uparrow \tau_{\mathsf{reg,cvx}}$ and $\varepsilon$ small, so that we may produce strongly stationary quintuplets with $\tau$ not too much smaller than $\tau_{\mathsf{reg,cvx}}$ in the case that $\tau_{\mathsf{reg,cvx}}$ is finite, or diverging in the case that $\tau_{\mathsf{reg,cvx}}$ is infinite.
Part (ii) is only used in establishing tightness of the convex lower bound when $\delta > 1$.

\subsection{Lower bounding the asymptotic loss}
\label{subsection-finish-cvx-lower-bound-proof}

Assume the conditions of Theorem \ref{thm-cvx-lower-bound}.
If $\tau_{\mathsf{reg,cvx}} = 0$, then \eqref{cvx-lower-bound} is trivial. Thus, assume $\tau_{\mathsf{reg,cvx}} > 0$.
We will show that for any $\{\rho_p\} \in \cC_{\delta,\pi} \cap \cB$ and any $\tau_{\mathsf{lb}} > 0$ such that $\delta \tau_{\mathsf{lb}}^2 - \sigma^2 > \mathsf{R_{seq,cvx}^{opt}}(\tau_{\mathsf{lb}};\pi)$,
\begin{equation}\label{cvx-lower-explicit-p-liminf}
    \lim_{p \rightarrow \infty} \P \left(\|\widehat \bbeta_{\mathsf{cvx}} - \bbeta_0 \|^2 < \delta \tau_{\mathsf{lb}}^2 - \sigma^2 \right) = 0.
\end{equation}
We then take $\tau_{\mathsf{lb}}\uparrow \tau_{\mathsf{reg,cvx}}$ such that $\delta \tau_{\mathsf{lb}}^2 - \sigma^2 > \mathsf{R_{seq,cvx}^{opt}}(\tau_{\mathsf{lb}};\pi)$ is satisfied along this sequence, which is permitted by the definition of $\tau_{\mathsf{reg,cvx}}$.
By Claim \ref{bounded-minimization-is-enough}, this is enough.

Assume otherwise. Then for some $\xi > 0$, we may pick a subsequence $\{p(\ell)\}$ such that 
\begin{equation}\label{small-risk-subsequence}
 \P \left( \|\widehat \bbeta_{\mathsf{cvx}}(p(\ell)) - \bbeta_0(p(\ell)) \|^2 < \delta\tau_{\mathsf{lb}}^2 - \sigma^2\right) > \xi
\end{equation}
for all $\ell$.
Observe $\{\rho_{p(\ell)}\} \in \cC_{\delta,\pi} \cap \cB$ because conditions \eqref{delta-bounded-width} and \eqref{eqdef-bounded-shrinkage-towards-infinity} are closed under taking subsequences. 
By Lemma \ref{lem-oracle-risk-to-cvx-bound}(i), we may choose $\gamma > 0$, a further subsequence $\{p'(\ell)\}$, $\tau > \tau_{\mathsf{lb}}$, and $\lambda > 0$ such that, with $\cT = (\pi, \{\rho_{p'(\ell)}^{(\gamma)}\})$, we have that $\tau,\lambda,\gamma,\delta,\cT$ is strongly stationary (here, we have used $\tau_{\mathsf{lb}} > 0$). 
By Proposition \ref{prop-strongly-convex-loss}, we have $\|\widehat \bbeta_{\mathsf{orc}}^{(\gamma)}(p'(\ell)) - \bbeta_0(p'(\ell))\|^2 \xrightarrow[\ell \rightarrow \infty]{\mathrm p} \delta\tau^2 - \sigma^2$, whence
\begin{equation}
    \lim_{\ell \rightarrow \infty} \P\left(\|\widehat \bbeta_{\mathsf{orc}}^{(\gamma)}(p'(\ell)) - \bbeta_0(p'(\ell))\|^2 < \delta\tau_{\mathsf{lb}}^2 - \sigma^2\right) = 0.
\end{equation}
By Lemma \ref{lem-oracle-is-better}, $\|\widehat \bbeta_{\mathsf{orc}}^{(\gamma)}(p'(\ell)) - \bbeta_0(p'(\ell))\|^2 \leq \|\widehat \bbeta_{\mathsf{cvx}}(p'(\ell)) - \bbeta_0(p'(\ell))\|^2$ for all $\ell$ and all realizations of $\bX$, whence 
\begin{equation}
    \lim_{\ell \rightarrow \infty} \P\left(\|\widehat \bbeta_{\mathsf{cvx}}(p'(\ell)) - \bbeta_0(p'(\ell))\|^2 < \delta\tau_{\mathsf{lb}}^2 - \sigma^2\right) = 0,
\end{equation}
contradicting  \eqref{small-risk-subsequence}. We conclude \eqref{cvx-lower-explicit-p-liminf}. 

\subsection{Tightness for $\delta > 1$}\label{sec:tightness-cvx-lb}
Tightness is trivial when both the left and right-hand side of \eqref{cvx-lower-bound} is infinite, so we assume $\tau_{\mathsf{reg,cvx}}^2 < \infty$.
In particular, $\mathsf{R^{opt}_{seq,cvx}}(\tau;\pi)$ is finite for some $\tau$.

Then, by Lemma \ref{lem-r-continuity}, $\mathsf{R^{opt}_{seq,cvx}}(\tau;\pi)$ is finite for all $\tau$ and continuous. 
Thus, by the definition of $\tau_{\mathsf{reg,cvx}}^2$,
we have $\delta \tau_{\mathsf{reg,cvx}}^2 - \sigma^2 = \mathsf{R^{opt}_{seq,cvx}}(\tau_{\mathsf{reg,cvx}};\pi)$.
The infimum in \eqref{R-seq-cvx-opt-asymptotic} can always be achieved by taking a sequence of $\{\{\rho_p^{(k)}\}_p\}_k$ approaching the infimum, and then taking a sequence $\{\rho_p^{(k(p))}\}_p$ where $k(p)$ goes to infinity appropriately as a function of $p$.
By passing to a subsequence, we can assume that the limit infimum is a limit.
Thus, we may assume we have a sequence $\{\rho_p\}_p$ such that
\begin{equation*}
    \delta \tau_{\mathsf{reg,cvx}}^2 - \sigma^2 
        = 
        \lim_{p \rightarrow \infty} \E_{\bbeta_0,\bz}\left[\left\|\mathsf{prox}[\rho_p](\bbeta_0 + \tau \bz) - \bbeta_0\right\|^2\right] = \mathsf{R_{reg,cvx}}(\tau_{\mathsf{reg,cvx}},1,\cT'),
\end{equation*}
where $\cT' = (\pi,\{\rho_p\})$.
Because $\left\|\mathsf{prox}[\rho_p](\bbeta_0 + \tau \bz) - \bbeta_0\right\| \geq \|\mathsf{prox}[\rho_p](\bzero)\| - \tau \|\bz\| - 2 \|\bbeta_0\|$, we may conclude that $\{\rho_p\} \in \cB$.
By Lemma \ref{lem-oracle-risk-to-cvx-bound}(ii),
we may find $\lambda > 0$ and a subsequence $\{p(\ell)\}$ such that $\tau_{\mathsf{reg,cvx}},\lambda,\delta,\gamma=0,\cT=(\pi,\{\rho_{p(\ell)}/\lambda\})$ is strongly stationary. 
By Proposition \ref{prop-strongly-convex-loss},
under the penalty sequence $\{\rho_{p(\ell)}/\lambda\}$ we have
$\|\widehat \bbeta_{\mathsf{cvx}} - \bbeta_0\|^2 \stackrel{\mathrm{p}}\rightarrow \delta \tau_{\mathsf{reg,cvx}}^2 - \sigma^2$.

The proof of Theorem \ref{thm-cvx-lower-bound} is complete. \hfill $\square$

% ------------------------------------------------------------------------------------------------ %

\section{Proof of Lemma \ref{lem-oracle-risk-to-cvx-bound}}\label{app-proof-of-lem-oracle-risk-to-cvx-bound}

Lemma \ref{lem-oracle-risk-to-cvx-bound} contains most of the technical machinery of our proof. 
The argument relies on several lemmas, some of whose proofs are deferred to Appendix \ref{app-proofs-of-app-proof-of-lem-oracle-risk-to-cvx-bound-lemmas}. 

To guide the reader, we first provide a high-level overview of the argument. 

\begin{enumerate}

    \item % 1
    \emph{Finite sample version of fixed-point equations.}
    We begin by defining finite-sample versions of the fixed point equations \eqref{fixed-pt-prior-asymptotic} (see \eqref{fixed-pt-prior-finite} below).
    The solutions to \eqref{fixed-pt-prior-asymptotic} which we construct will be limits to solutions of \eqref{fixed-pt-prior-finite}.

    \item % 2
    \emph{Bounds on possible solutions to finite sample fixed-point equations.}
    By comparing the right and left-hand sides of \eqref{se-var-fixed-pt-finite},
    we place an upper (Lemma \ref{lem-oracle-not-too-bad}) and lower (Lemma \ref{lem-oracle-not-too-good}) bound on the noise variance $\tau^2$ at a solution to the finite-sample fixed point equations.
    The lower bound is the only location in our argument where the $\delta$-bounded width assumption is used, corresponding to the statement that the estimation error of the oracle estimator is not too much smaller than the convex lower bound.

    \item % 3 
    \emph{Existence of solutions to finite sample fixed-point equations.}
    Using a topological argument, we show that solutions to the finite sample fixed-point equations must exist (Lemma \ref{lem-se-fixed-pt-existence-non-asymptotic}). 
    Although the Lemma is quite intuitive (having the flavor of a two-dimensional intermediate value theorem), we could not find a statement of the required result in the literature, and the proof is a bit involved.

    \item % 2
    \emph{From finite-sample fixed points to strongly stationary quintuplets.} 
    Using the existence of and bounds on the solutions to \eqref{fixed-pt-prior-finite}, we apply a compactness arguments to find a subsequence and construct the required strongly stationary quintuplet (see Section \ref{sec:finite-fix-pt-to-strong-stat}).

\end{enumerate}

\subsection{Solutions to finite-sample version of fixed point equations}\label{app-finite-sample-fixed-point-solutions}

We first consider the following finite-sample versions of the fixed point equations \eqref{fixed-pt-prior-asymptotic}:
\begin{subequations}\label{fixed-pt-prior-finite}
\begin{gather}
    \delta \tau^2 - \sigma^2 =  \mathsf{R_{reg,cvx}}(\tau_{\mathsf{orc}},\lambda_{\mathsf{orc}},\cT_p),\label{se-var-fixed-pt-finite}\\
    2\lambda\left(1 - \frac1{\delta(\lambda\gamma+1)}\mathsf{W_{reg,cvx}}(\tau_{\mathsf{orc}},\lambda_{\mathsf{orc}},\cT_p) \right) = 1.\label{se-lam-fixed-pt-finite}
\end{gather}
\end{subequations}
The first step in proving Lemma \ref{lem-oracle-risk-to-cvx-bound} is to establish the existence of solutions to these finite-sample equations and to control their size. This is achieved by the following series of lemmas, whose proofs are provided in Appendix \ref{app-proofs-of-app-proof-of-lem-oracle-risk-to-cvx-bound-lemmas}. 

\begin{lemma}\label{lem-oracle-not-too-bad}
Consider an lsc, proper, convex, function $\rho:\reals^p \rightarrow \reals \cup \{\infty\}$. 
Let $M \geq \|\mathsf{prox}[\rho](\bzero)\| $. 
Let $\cT_p = (\pi,\rho)$.
If $\delta > 1$ or $\gamma > 0$ or $\rho$ is $\kappa$-strongly convex with $\kappa > 0$,
there exists some $\tau_{\mathsf{max}}$ depending only on $\pi,M,\delta,\gamma,\kappa$ (and not on $p$) such that if $\tau,\lambda$ is a solution of \eqref{se-lam-fixed-pt-finite} at $\gamma$ with $\tau \geq \tau_{\mathsf{max}}$, then
\begin{equation}\label{se-large-tau-inequality}
    \delta\tau^{2} - \sigma^2 > \mathsf{R_{reg,cvx}}(\tau_{\mathsf{orc}},\lambda_{\mathsf{orc}},\cT_p).
\end{equation}
\end{lemma}

\noindent The inequality of Lemma \ref{lem-oracle-not-too-bad} says solutions to \eqref{fixed-pt-prior-finite} cannot be too big when either $\delta > 1$, $\gamma > 0$, or $\rho$ is strongly convex. The next lemma establishes --under certain additional restrictions-- the reverse inequality at a value of $\tau$ which is not too small. 

\begin{lemma}\label{lem-oracle-not-too-good}
Consider  $\{\rho_p\} \in \cC_{\delta,\pi}$ and $\tau_{\mathsf{lb}} > 0$ such that $\delta \tau_{\mathsf{lb}}^2 - \sigma^2 < \mathsf{R_{seq,cvx}^{opt}}(\tau_{\mathsf{lb}};\pi)$. 
For each $p$, let $\cT_p = (\pi,\rho_p)$. 
Then we can find $\gamma > 0$, $\tau_{\mathsf{min}} \geq \tau_{\mathsf{lb}}$, and a subsequence $\{p(\ell)\}$ such that for all $p$ in the subsequence we have the following: for all $\lambda$ which solves \eqref{se-lam-fixed-pt-finite} at $\tau_{\mathsf{min}},\gamma$,
\begin{equation}\label{se-small-tau-inequality}
\delta \tau_{\mathsf{min}}^{2} - \sigma^2 <  \mathsf{R_{reg,cvx}}(\tau_{\mathsf{min,orc}},\lambda_{\mathsf{orc}},\cT_p),
\end{equation}
where $\tau_{\mathsf{min,orc}} = \frac{\tau_{\mathsf{min}}}{\lambda \gamma + 1}$ and $\lambda_{\mathsf{orc}} = \frac{\lambda}{\lambda \gamma + 1}$.

\end{lemma}

\noindent Combining Lemmas \ref{lem-oracle-not-too-bad} and \ref{lem-oracle-not-too-good}, the next lemma allows us to choose an oracle parameter such that, along a subsequence of $\{p\}$, there exist solutions to \eqref{fixed-pt-prior-finite} with effective noise parameters $\tau$ which are neither too large or too small. 

\begin{lemma}\label{lem-se-fixed-pt-existence-non-asymptotic}
We have the following.
\begin{enumerate}[(i)]
    
    \item % i
    Assume conditions of Lemma \ref{lem-oracle-not-too-good}. Assume additionally that $\{\rho_p\} \in \cB$. Then we can find $\gamma > 0$, $\tau_{\mathsf{max}} < \infty$, $\lambda_{\mathsf{max}} < \infty$, and a subsequence $\{p(\ell)\}$ such that for all $p$ in the subsequence, there exists solution $\tau,\lambda$ to \eqref{fixed-pt-prior-finite} with 
    $
    (\tau,\lambda) \in [\tau_{\mathsf{lb}},\tau_{\mathsf{max}}] \times [1/2,\lambda_{\mathsf{max}}].
    $

    \item % ii
    If $\delta > 1$ or $\{\rho_p\}$ are $\kappa$-strongly convex with $\kappa > 0$, part (i) holds except we may also take $\gamma = 0$.

\end{enumerate}
\end{lemma}

\noindent The needed characterization of solutions to \eqref{fixed-pt-prior-finite} is complete.

\subsection{From finite-sample fixed points to strongly stationary quintuplets}\label{sec:finite-fix-pt-to-strong-stat}

We now apply the lemmas above to prove Lemma \ref{lem-oracle-risk-to-cvx-bound}.

\begin{proof}[Proof of Lemma \ref{lem-oracle-risk-to-cvx-bound}(i)]

By Lemma \ref{lem-se-fixed-pt-existence-non-asymptotic}(i), we can (and do) choose $\gamma > 0$, $\tau_{\mathsf{max}} < \infty$, $\lambda_{\mathsf{max}} < \infty$, and a subsequence $\{p(\ell)\}$ such that for all $p$ in the subsequence there exists a solution $\tau_p,\lambda_p$ to \eqref{fixed-pt-prior-finite} with $(\tau_p,\lambda_p) \in [\tau_{\mathsf{lb}},\tau_{\mathsf{max}}] \times [1,\lambda_{\mathsf{max}}]$. 
By Bolzano-Weierstrass, we can find a further subsequence $\{p'(\ell)\}$ and $(\tau,\lambda) \in [\tau_{\mathsf{lb}} - \eps,\tau_{\mathsf{max}}] \times [1,\lambda_{\mathsf{max}}] $ such that
\begin{equation}\label{finite-sample-parameters-converge}
(\tau_{p'(\ell)},\lambda_{p'(\ell)})  \rightarrow (\tau,\lambda) \in [\tau_{\mathsf{lb}} ,\tau_{\mathsf{max}}] \times [1,\lambda_{\mathsf{max}}].
\end{equation}
To simplify notation, we write the subsequence as $\{p\}$. 
By \eqref{se-var-fixed-pt-finite} and \eqref{finite-sample-parameters-converge}, we get
\begin{equation}\label{finite-sample-r-converge}
\mathsf{R_{reg,cvx}}(\tau_{p,\mathsf{orc}},\lambda_{p,\mathsf{orc}},\cT_p) \underset{p\rightarrow\infty}{\longrightarrow} \delta\tau^2 - \sigma^2.
\end{equation}
We also have by the definition of $\tau_{\mathsf{orc}}$ and $\lambda_{\mathsf{orc}}$ and \eqref{finite-sample-parameters-converge} that
\begin{equation}\label{finite-sample-oracle-parameters-converge}
    (\tau_{p,\mathsf{orc}},\lambda_{p,\mathsf{orc}}) \rightarrow (\tau_{\mathsf{orc}},\lambda_{\mathsf{orc}}).
\end{equation}
Because $\{\rho_p\} \in \cB$, by Lemma \ref{lem-uniform-modulus-of-continuity}, ${\mathsf{R_{reg,cvx}}}(\tau',\lambda',\cT_p)$ are uniformly Lipschitz continuous on compact sets.
Thus, by \eqref{finite-sample-r-converge} and \eqref{finite-sample-oracle-parameters-converge},
\begin{equation}\label{r-solves-fixed-pt}
\mathsf{R_{reg,cvx}}(\tau_{\mathsf{orc}},\lambda_{\mathsf{orc}},\cT_p)  \underset{p \rightarrow \infty}{\longrightarrow} \delta\tau^2 - \sigma^2. 
\end{equation}
That is, the limit \eqref{summary-functions-fixed-asymptotic} exists at $\tau_{\mathsf{orc}},\lambda_{\mathsf{orc}}$, and the limiting value solves \eqref{se-var-fixed-pt-prior-asymptotic}. 

Similarly, by \eqref{se-lam-fixed-pt-finite}, for each $p$ we have $\mathsf{W_{reg,cvx}}(\tau_{p,\mathsf{orc}},\lambda_{p,\mathsf{orc}},\cT_p) = \delta(\lambda_p\gamma + 1)\left( 1 - \frac{1}{2\lambda_p}\right)$, so that
\begin{equation}\label{finite-sample-w-converge}
\mathsf{W_{reg,cvx}}(\tau_{p,\mathsf{orc}},\lambda_{p,\mathsf{orc}},\cT_p) \underset{p\rightarrow\infty}{\longrightarrow} \delta(\lambda \gamma + 1)\left( 1 - \frac 1{2\lambda}\right).
\end{equation}
By Lemma \ref{lem-uniform-modulus-of-continuity}, $\mathsf{W_{reg,cvx}}(\tau',\lambda',\cT_p)$ is uniformly Lipschitz continuous in $(\tau',\lambda')$ on compact sets.
Thus, by \eqref{finite-sample-w-converge} and \eqref{finite-sample-oracle-parameters-converge},
\begin{equation}\label{w-solves-fixed-pt}
\mathsf{W_{reg,cvx}}(\tau_{\mathsf{orc}},\lambda_{\mathsf{orc}},\cT_p)  \underset{p \rightarrow \infty}{\longrightarrow} \delta(\lambda \gamma + 1)\left( 1 - \frac 1{2\lambda}\right). 
\end{equation}
That is, the limit \eqref{summary-functions-fixed-asymptotic} exists at $\tau_{\mathsf{orc}},\lambda_{\mathsf{orc}}$, and the limiting value solves \eqref{se-lam-fixed-pt-prior-asymptotic}. 

Finally, by Lemma \ref{lem-uniform-modulus-of-continuity}, the functions $\mathsf{R_{reg,cvx}}(\tau',\lambda',\cT_p)$, $\mathsf{W_{reg,cvx}}(\tau',\lambda',\cT_p)$, and $\mathsf{K_{reg,cvx}}(\bT',\lambda',\cT_p)$ are uniformly equicontinuous in $\tau'$, $\lambda'$, and $\bT'$ on bounded sets. 
Further, the convergence \eqref{r-solves-fixed-pt} and \eqref{w-solves-fixed-pt} gives us that $\mathsf{R_{reg,cvx}}(\tau,\lambda,\cT_p)$ and $\mathsf{R_{reg,cvx}}(\tau,\lambda,\cT_p)$ are uniformly bounded over $p$. 
Further, for $\bT = \tau^2 \bI_2$, by \eqref{r-finite-p-def} and \eqref{k-finite-p-def}, we have $\mathsf{K_{reg,cvx}}(\bT,\lambda,\cT_p) = \mathsf{R_{reg,cvx}}(\tau,\lambda,\cT_p)$, so that $\mathsf{K_{reg,cvx}}(\bT,\lambda,\cT_p)$ are uniformly bounded over $p$.
Thus, by the Arzel\'a-Ascoli theorem, we may take a further subsequence $\{p(\ell)\}$ along which the limits \eqref{summary-functions-fixed-asymptotic} exist for all $\tau',\lambda',\bT'$.
We have now established that with $\cT = (\pi,\{\rho_{p(\ell)}\})$, the quintuplet $\tau,\lambda,\gamma,\delta,\cT = (\pi,\{\rho_{p(\ell)}\})$ is strongly stationary. 
\end{proof}

\begin{proof}[Proof of Lemma \ref{lem-oracle-risk-to-cvx-bound}(ii)]
    Because $\mathsf{W_{reg,cvx}}(\tau,\lambda',\cT_p') \in [0,1]$ for all $p$ by \eqref{finite-p-width-bounded-by-1} and \eqref{width-bounds},
    there is a subsequence $\{p(\ell)\}$ such that $\mathsf{W_{reg,cvx}}(\tau,\lambda',\cT_{p(\ell)}')$ converges to a limit $w$. Let $\lambda = \frac1{2(1-w/\delta)}$, so that $1 = 2\lambda(1-w/\delta)$.
    Note $\mathsf{R_{reg,cvx}^\infty}(\tau,\lambda,(\pi,\{\lambda'\rho_{p(\ell)}/\lambda\})) = \mathsf{R_{reg,cvx}^\infty}(\tau,\lambda',(\pi,\{\rho_{p(\ell)}\}))$ and $\mathsf{W_{reg,cvx}^\infty}(\tau,\lambda,(\pi,\{\lambda'\rho_{p(\ell)}/\lambda\})) = \mathsf{W_{reg,cvx}^\infty}(\tau,\lambda',(\pi,\{\rho_{p(\ell)}\}))$.
    Thus, \eqref{fixed-pt-prior-asymptotic} are satisfied for $\cT = (\pi,\{\lambda'\rho_p/\lambda\})$ at $\tau,\lambda,\gamma=0,\delta$.
    Now we may take a further subsequence such that the limits \eqref{summary-functions-fixed-asymptotic} exist for all $\tau',\bT'$ by the same argument used in the proof of Lemma \ref{lem-oracle-risk-to-cvx-bound}(i).
\end{proof}

\section{Proofs of Appendix \ref{app-proof-of-lem-oracle-risk-to-cvx-bound} Lemmas}\label{app-proofs-of-app-proof-of-lem-oracle-risk-to-cvx-bound-lemmas}

\subsection{Proof of Lemma \ref{lem-oracle-not-too-bad}}\label{proof-of-lem-oracle-not-too-bad}

We prove Lemma \ref{lem-oracle-not-too-bad} by controlling the size of $\mathsf{R_{reg,cvx}}(\tau,\lambda,\cT_p)$ for large $\tau$. The following claim is what we need.

\begin{claim}\label{claim-uncentered-prox-l2-bound}
For any lsc, proper, convex $\rho:\reals^p \rightarrow \reals\cup \{\infty\}$ and $\bz \sim \mathsf{N}(0,\bI_p/p)$ independent of $\bbeta_0$, 
\begin{equation}\label{uncentered-l2-bound}
\E_{\bbeta_0,\bz}[\|\mathsf{prox}[\rho](\bbeta_0 + \tau \bz) - \bbeta_0\|^2] \leq \left(\sqrt{\E_{\bbeta_0,\bz}[\langle \tau \bz,\mathsf{prox}[\rho](\bbeta_0 + \tau \bz)\rangle]} + \sqrt{\E_{\bbeta_0,\bz}[\|\mathsf{prox}[\rho](\bbeta_0) - \bbeta_0\|^2]}\right)^2.
\end{equation}
\end{claim}
\begin{proof}[Proof of Claim \ref{claim-uncentered-prox-l2-bound}]
Write
\begin{subequations}\label{l2-split}
\begin{align}
\left\|\mathsf{prox}\left[\rho\right]\left(\bbeta_0 + \tau\bz\right)-\bbeta_0\right\|^2 &\leq \left\|\mathsf{prox}\left[\rho\right]\left(\bbeta_0 + \tau\bz\right)-\mathsf{prox}\left[\rho\right]\left(\bbeta_0\right)\right\|^2\label{l2-split-1}\\
&\quad+ 2\left\|\mathsf{prox}\left[\rho\right]\left(\bbeta_0 + \tau\bz\right)-\mathsf{prox}\left[\rho\right]\left(\bbeta_0\right)\right\| \left\|\mathsf{prox}\left[\rho\right]\left(\bbeta_0\right)-\bbeta_0\right\|\label{l2-split-2}\\
&\quad + \left\|\mathsf{prox}\left[\rho\right]\left(\bbeta_0\right)-\bbeta_0\right\|^2.\label{l2-split-3}
\end{align}
\end{subequations}
First, we bound the first term on the right-hand side. 

By \eqref{prox-firm-non-expansive}, we have 
\begin{equation}
    \|\mathsf{prox}[\rho](\bbeta_0 + \tau \bz) - \mathsf{prox}[\rho](\bbeta_0)\|^2 \leq \langle \tau \bz , \mathsf{prox}[\rho](\bbeta_0 + \tau \bz) - \mathsf{prox}[\rho](\bbeta_0)\rangle.
\end{equation}
Taking expectations of both sides and using that $\E_{\bbeta_0,\bz}[\langle \tau \bz , \mathsf{prox}[\rho](\bbeta_0)\rangle] = 0$ by the independence of $\bbeta_0$ and $\bz$ and the fact that $\E_{\bz}[\bz] = \bzero$, we get
\begin{equation}\label{bound-l2-first-term}
    \E_{\bbeta_0,\bz}[\|\mathsf{prox}[\rho](\bbeta_0 + \tau \bz) - \mathsf{prox}[\rho](\bbeta)\|^2] \leq \E_{\bbeta_0,\bz}[\langle \tau \bz , \mathsf{prox}[\rho](\bbeta_0 + \tau \bz)\rangle ].
\end{equation}
We bound the expectation of \eqref{l2-split-2} by Cauchy-Schwartz.
\begin{align}
\E_{\bbeta_0,\bz}[\|\mathsf{prox}\left[\rho\right]\left(\bbeta_0 + \tau\bz\right)&-\mathsf{prox}\left[\rho\right]\left(\bbeta_0\right)\| \left\|\mathsf{prox}\left[\rho\right]\left(\bbeta_0\right)-\bbeta_0\right\|] \nonumber\\ 
&\leq \sqrt{\E_{\bbeta_0,\bz}[\left\|\mathsf{prox}\left[\rho\right]\left(\bbeta_0 + \tau\bz\right)-\mathsf{prox}\left[\rho\right]\left(\bbeta_0\right)\right\|^2]}\sqrt{\E_{\bbeta_0,\bz}[\|\mathsf{prox}[\rho](\bbeta_0) - \bbeta_0\|^2]}\nonumber\\
&\leq \sqrt{\E_{\bbeta_0,\bz}[\langle \tau \bz,\mathsf{prox}[\rho](\bbeta_0 + \tau \bz)\rangle]}\sqrt{\E_{\bbeta_0,\bz}[\|\mathsf{prox}[\rho](\bbeta_0) - \bbeta_0\|^2]},\label{bound-l2-cross-term}
\end{align}
where in the third line we have used \eqref{bound-l2-first-term}.
Taking the expectation of \eqref{l2-split} and applying bounds  \eqref{bound-l2-first-term}, \eqref{bound-l2-cross-term} gives \eqref{uncentered-l2-bound}.
\end{proof}

We are ready to prove Lemma \ref{lem-oracle-not-too-bad}.
Fix $\gamma \geq 0$ and $\kappa \geq 0$, so that $\rho_p$ is $\kappa$-strongly convex (note, when $\kappa = 0$ we make no strong convexity assumption).
 Consider solutions $\tau,\lambda$ to \eqref{se-lam-fixed-pt-finite} at $\gamma$.
To simplify notation, we denote $\rho_{\mathsf{orc}} = \lambda_{\mathsf{orc}}\rho$.
 By \eqref{r-finite-p-def} and Claim \ref{claim-uncentered-prox-l2-bound}, 
\begin{align}
\frac1{\delta\tau^2} &\mathsf{R_{reg,cvx}}(\tau_{\mathsf{orc}},\lambda_{\mathsf{orc}},\cT_p) =  \frac1{\delta\tau^2}\E_{\bbeta_0,\bz}[\|\mathsf{prox}[\rho_{\mathsf{orc}}](\bbeta_0 + \tau_{\mathsf{orc}} \bz) - \bbeta_0\|^2] \nonumber\\
&\leq \left(\sqrt{\frac{1}{\delta\tau^2} \E_{\bbeta_0,\bz}[\langle \tau_{\mathsf{orc}}\bz,\mathsf{prox}[\rho_{\mathsf{orc}}](\bbeta_0 + \tau_{\mathsf{orc}}\bz)\rangle]} + \sqrt{ \frac1{\delta\tau^2}\E_{\bbeta_0,\bz}[\|\mathsf{prox}[\rho_{\mathsf{orc}}](\bbeta_0) - \bbeta_0\|^2]}\right)^2.\label{risk-upper-bound}
\end{align}

First we bound the second term on the right-hand side of \eqref{risk-upper-bound}.
By \eqref{finite-p-width-bounded-by-1} and \eqref{se-lam-fixed-pt-finite},
$
1 \geq 2\lambda \left(1 - \frac1{\delta(\lambda \gamma + 1)} \frac1{\lambdaorc\kappa+1}\right)
=2\lambda \left(1 - \frac1{\delta} \frac1{\lambda\kappa + \lambda \gamma +1}\right). 
$
If either $\delta > 1$, $\gamma > 0$, or $\kappa > 0$,
the right-hand side diverges to $\infty$ for $\lambda \rightarrow \infty$. Thus, there exists $\lambda_{\mathsf{max}}$ depending only on $\delta,\gamma,\kappa$ such that all solutions $\tau,\lambda$ to \eqref{se-lam-fixed-pt-finite} at $\gamma$ satisfy $\lambda \leq \lambda_{\mathsf{max}}$.
Then we have 
\begin{align*}
\|\mathsf{prox}[\rho_{\mathsf{orc}}](\bbeta_0) - \bbeta_0\| &\leq \|\mathsf{prox}[\rho](\bzero)\| + \|\mathsf{prox}[\lambda_{\mathsf{orc}}\rho](\bzero) - \mathsf{prox}[\rho](\bzero)\| \\
&\qquad\qquad+ \|\mathsf{prox}[\lambda_{\mathsf{orc}}\rho](\bbeta_0) - \mathsf{prox}[\lambda_{\mathsf{orc}}\rho](\bzero)\| + \|\bbeta_0\| \\
&\leq M + M|\lambda_{\mathsf{orc}} - 1| + 2\|\bbeta_0\| 
\leq M(\lambda_{\mathsf{max}} + 2) + 2\|\bbeta_0\|,
\end{align*}
where in the second inequality, we have used \eqref{prox-continuous-in-lambda} and \eqref{prox-is-lipschitz}, and in the third inequality, we have used $\lambda_{\mathsf{orc}} \leq \lambda \leq \lambda_{\mathsf{max}}$.
Thus, 
\begin{equation}\label{risk-upper-bound-term-2}
\E_{\bbeta_0,\bz}[\|\mathsf{prox}[\rho_{\mathsf{orc}}](\bbeta_0) - \bbeta_0\|^2]  \leq 2M^2(\lambda_{\mathsf{max}} + 2)^2 + 8s_2(\pi),
\end{equation}
where $s_2(\pi)$ is the second moment of $\pi$.

Second we bound the first term on the right-hand side of \eqref{risk-upper-bound}.
We bound the first term by using the fact that $\tau,\lambda,\gamma$ solve \eqref{se-lam-fixed-pt-finite}. 
In particular, by \eqref{se-lam-fixed-pt-finite} we have that 
\begin{equation}\label{width-bound-at-solution}
\frac1{\delta(\lambda\gamma+1)} \mathsf{W_{reg,cvx}}(\tau_{\mathsf{orc}},\lambda_{\mathsf{orc}},\cT_p)
= 1 - \frac1{2\lambda} \leq 1 - \frac1{2\lambda_{\mathsf{max}}}
\end{equation}
Then, applying \eqref{w-finite-p-def}, we get
\begin{align*}
     \frac{1}{\delta\tau^2} \E_{\bbeta_0,\bz}[\langle \tau_{\mathsf{orc}}\bz,\mathsf{prox}[\rho_{\mathsf{orc}}](\bbeta_0 + \tau_{\mathsf{orc}}\bz)\rangle] &= \frac1{\delta(\lambda \gamma + 1)^2}\mathsf{W_{reg,cvx}}(\tau_{\mathsf{orc}},\lambda_{\mathsf{orc}},\cT_p) \nonumber\\
    & \leq 1 - \frac{1}{2\lambda_{\mathsf{max}}}.
\end{align*}
Plugging this and \eqref{risk-upper-bound-term-2} into \eqref{risk-upper-bound}, we get
\begin{align}\label{risk-upper-bound-through-M-gamma-pi}
    \frac1{\delta\tau^2} &\mathsf{R_{reg,cvx}}(\tau_{\mathsf{orc}},\lambda_{\mathsf{orc}},\cT_p) \leq \left(\sqrt{1 - \frac{1}{2\lambda_{\mathsf{max}}}} + \sqrt{\frac{2M^2(\lambda_{\mathsf{max}} + 2)^2 + 8s_2(\pi)}{\delta\tau^2}  }\right)^2.
\end{align}
Choose $\tau_{\mathsf{max}}$ such that
\begin{equation}\label{tau-ub-large-enough}
1 > \frac{\sigma^2}{\delta\tau_{\mathsf{max}}^2} + \left(\sqrt{1 - \frac{1}{2\lambda_{\mathsf{max}}}} + \sqrt{\frac{2M^2(\lambda_{\mathsf{max}} + 2)^2 + 8s_2(\pi)}{\tau^2_{\mathsf{max}}\delta}}\right)^2,
\end{equation}
which is possibly because $1-\frac1{2\lambda_{\mathsf{max}}} < 1$. 
This choice depends only on $\pi,M,\gamma,\kappa,\delta$.
This inequality also holds for any $\tau \geq \tau_{\mathsf{max}}$. Chaining \eqref{risk-upper-bound-through-M-gamma-pi} and \eqref{tau-ub-large-enough} and performing some rearrangement, we get that \ref{se-large-tau-inequality} holds, as desired. \hfill $\square$

\subsection{Proof of Lemma \ref{lem-oracle-not-too-good}}\label{proof-of-lem-oracle-not-too-good}

The proof proceeds in three steps. The only place where the $\delta$-bounded width assumption is used is in Case 1 of Step 2.

\noindent {\bf Step 1: Construct interval on which $\delta \tau^2 - \sigma^2 < \mathsf{R^{opt}_{seq,cvx}}(\tau;\pi)$.}

\noindent By the definition of $\mathsf{R^{opt}_{seq,cvx}}(\tau;\pi)$, there exists $\zeta > 0$ such that 
\begin{equation*}
    \delta\tau_{\mathsf{lb}}^2 - \sigma^2 <  \mathsf{R^{opt}_{seq,cvx}}(\tau_{\mathsf{lb}};\pi,p) - \zeta 
\end{equation*}
 eventually.
By the regularity property established in Lemma \ref{lem-r-continuity},
we may pick $\Delta > 0$ such that $\mathsf{R^{opt}_{seq,cvx}}(\tau;\pi,p) > \mathsf{R^{opt}_{seq,cvx}}(\tau_{\mathsf{lb}};\pi,p) - \zeta/3$ and  $\delta\tau^2 - \sigma^2 < \delta \tau_{\mathsf{lb}}^2 - \sigma^2 + \zeta/3$ for all $\tau \in [\tau_{\mathsf{lb}},\tau_{\mathsf{lb}} + \Delta]$.
In particular, for all such $\tau$
\begin{equation}\label{finite-range-se-var-inequality}
    \delta \tau^2 - \sigma^2 \leq \delta \tau_{\mathsf{lb}}^2 - \sigma^2 + \zeta/3 <  \mathsf{R^{opt}_{seq,cvx}}(\tau_{\mathsf{lb}};\pi,p) - 2\zeta/3 < \mathsf{R^{opt}_{seq,cvx}}(\tau;\pi,p) - \zeta/3.
\end{equation}

\noindent {\bf Step 2: Choose oracle parameter with not-too-small oracle effective noise.}

\noindent The meaning of the preceding statement will become clear shortly. 
Let
\begin{equation}\label{eqdef-tau-lb}
    \tau_{\mathsf{min}} = \tau_{\mathsf{lb}} + \Delta.
\end{equation}
Denote 
\begin{equation}\label{max-noise-reduction-allowed}
    \frac{\tau_{\mathsf{lb}}}{\tau_{\mathsf{lb}} + \Delta} = 1 - \theta.
\end{equation}
For simplicity, for the remainder of the proof, we denote the subsequence $\{p(\ell)\}$ as $\{p\}$.
We will show how to choose $\gamma > 0$ such that, for each $p$, any solution $\lambda$ to \eqref{se-lam-fixed-pt-finite} at $\tau_{\mathsf{min}}$, $\gamma$ satisfies 
\begin{equation}\label{noise-reduction-lower-bound}
    \tau_{\mathsf{min,orc}}\geq \tau_{\mathsf{lb}},
\end{equation}
where we have denoted $\tau_{\mathsf{min,orc}} = \frac{\tau_{\mathsf{min}}}{\lambda \gamma + 1}$. 
This is what we mean by ``choose oracle parameter with not-too-small oracle effective noise.'' 
There are two cases.

\begin{itemize}
\item {\bf Case 1: $\delta \leq 1$.}

Because $\{\rho_p\} \in \cC_{\delta,\pi}$, by \eqref{delta-bounded-width}, we can (and do) choose $\bar \lambda > 0$  and $\xi > 0$ such that 
\begin{equation}
    \limsup_{p \rightarrow \infty} \sup_{\lambda > \bar \lambda, \tau' \in [\delta\tau_{\mathsf{min}}/2,\tau_{\mathsf{min}}]} \frac1{\tau'} \E_{\bbeta_0,\bz}[\langle \bz, \mathsf{prox}[\lambda \rho_p](\bbeta_0 + \tau' \bz)\rangle] < \delta(1 - \xi),
\end{equation}
(note that by assumption, $\delta / 2 < 1$, so the interval is non-empty).
Let $\{p(\ell)\}$ be a subsequence of $\{p\}$ such that 
\begin{equation}\label{uniform-delta-bounded-width-along-subsequence}
\sup_{\lambda > \bar \lambda, \tau' \in [\delta\tau_{\mathsf{min}}/2,\tau_{\mathsf{min}}]} \frac1{\tau'} \E_{\bbeta_0,\bz}[\langle \bz, \mathsf{prox}[\lambda \rho_{p(\ell)}](\bbeta_0 + \tau' \bz)\rangle] < \delta(1 - \xi)
\end{equation}
for all $\ell$.
Now choose
\begin{equation}\label{choose-gamma-delta-leq-1}
0 < \gamma < \min\left\{\frac2\delta - 1,\,\frac{\theta}{\bar \lambda},\,\frac{\theta}{1-\theta}2\xi\right\}.
\end{equation}
It is straightforward to check that the right-hand side is positive, so that such $\gamma$ exist.
Now consider any solution $\lambda$ to \eqref{se-lam-fixed-pt-finite} at $\tau_{\mathsf{min}},\gamma$.
Thus, 
\begin{align*}
\frac2\delta - 1 &>\gamma = 2\lambda\gamma\left(1 - \frac1{\delta(\lambda \gamma + 1)}\mathsf{W_{reg,cvx}}(\tau_{\mathsf{min,orc}},\lambda_{\mathsf{orc}},\cT_p)\right) \\ 
&\geq 2\lambda\gamma\left(1 - \frac1{\delta(\lambda \gamma + 1)}\right) = 2\left( \left(\frac1{\lambda \gamma + 1}\right)^{-1}-1\right)\left(1 - \frac1{\delta(\lambda \gamma + 1)}\right),
\end{align*}
where in the first inequality we have used \eqref{choose-gamma-delta-leq-1}, in the first equality we have used \eqref{se-lam-fixed-pt-finite}, and in the second inequality we have used \eqref{finite-p-width-bounded-by-1}.
The right-hand side is strictly decreasing in $\frac1{\lambda \gamma + 1}$. Moreover, the right-hand side equals $\frac2\delta - 1$ when $\frac1{\lambda \gamma + 1} = \frac\delta2$. We conclude that $\frac1{\lambda \gamma + 1} \geq \frac\delta2$, whence
\begin{equation}\label{crude-shrinkage-lower-bound}
\tau_{\mathsf{min}} \geq \tau_{\mathsf{min,orc}} \geq \frac{\delta\tau_{\mathsf{min}}}2,
\end{equation}
where the first inequality holds because trivially $1 \geq \frac1{\lambda \gamma + 1} $.
We now use the crude lower bound of \eqref{crude-shrinkage-lower-bound} to generate the lower bound \eqref{noise-reduction-lower-bound}.
Either $\lambda > \frac1{2\xi}$ or $\lambda \leq \frac1{2\xi}$. 
If $\lambda > \frac1{2\xi}$, then 
\begin{align}
\frac1{\tau_{\mathsf{min,orc}}}\E_{\bbeta_0,\bz}\left[\left\langle  \bz, \mathsf{prox}\left[\lambda_{\mathsf{orc}} \rho\right]\left(\bbeta_0 + \tau_{\mathsf{min,orc}} \bz\right)\right\rangle\right] &= \mathsf{W_{reg,cvx}}(\tau_{\mathsf{min,orc}},\lambda_{\mathsf{orc}},\cT_p) \nonumber\\
&= \delta(\lambda \gamma + 1)\left(1 - \frac1{2\lambda}\right) > \delta\left(1 - \xi \right),
\end{align}
where in the first line, we have used \eqref{w-finite-p-def}, and in the second line, we have used \eqref{se-lam-fixed-pt-finite}.
Combining this with \eqref{uniform-delta-bounded-width-along-subsequence} and \eqref{crude-shrinkage-lower-bound}, we conclude $\bar \lambda \geq \lambda_{\mathsf{orc}}$. 
Thus, $\frac1{\lambda \gamma + 1} = 1 - \frac{\lambda \gamma}{\lambda \gamma + 1} = 1 -  \lambda_{\mathsf{orc}} \gamma  \geq 1 - \bar \lambda \gamma$. 
By \eqref{eqdef-tau-lb}, \eqref{choose-gamma-delta-leq-1} and \eqref{max-noise-reduction-allowed},
$$
\tau_{\mathsf{min,orc}} = \frac{\tau_{\mathsf{lb}} + \Delta}{\lambda \gamma + 1} \geq (\tau_{\mathsf{lb}} + \Delta) \left(1 - \bar \lambda \gamma\right) \geq (\tau_{\mathsf{lb}} + \Delta)(1-\theta) = \tau_{\mathsf{lb}},
$$
so we have \eqref{noise-reduction-lower-bound}.
On the other hand, if $\lambda \leq \frac1{2\xi}$, then by \eqref{choose-gamma-delta-leq-1}
$$
\tau_{\mathsf{min,orc}} = \frac{\tau_{\mathsf{lb}} + \Delta}{\lambda \gamma + 1} \geq \frac{\tau_{\mathsf{lb}} + \Delta}{\gamma/(2\xi) + 1} \geq \frac{\tau_{\mathsf{lb}} + \Delta}{\frac{\theta}{1-\theta} + 1} = (\tau_{\mathsf{lb}} + \Delta)(1-\theta) = \tau_{\mathsf{lb}},
$$
so we also have \eqref{noise-reduction-lower-bound}.
Thus, if we choose $\gamma$ to satisfy \eqref{choose-gamma-delta-leq-1}, then \eqref{noise-reduction-lower-bound} holds at any solution $\lambda$ to \eqref{se-lam-fixed-pt-finite} at $\tau_{\mathsf{min}},\gamma$.

\item {\bf Case 2: $\delta > 1$.}

Choose
\begin{equation}\label{choose-gamma-delta-g-1}
0 \leq \gamma < \frac{2\theta (\delta - 1)}{(1-\theta)\delta}.
\end{equation}
Now consider any solution $\lambda$ to \eqref{se-lam-fixed-pt-finite} at $\tau_{\mathsf{min}},\gamma$.
By \eqref{finite-p-width-bounded-by-1},
\begin{align*}
    1   
        &= 
        2\lambda\left(1 - \frac1{\delta(\lambda \gamma + 1)} \mathsf{W_{reg,cvx}}(\tau_{\mathsf{min,orc}},\lambda_{\mathsf{orc}},\cT_p)\right) \geq 2\lambda \left(1 - \frac1{\delta}\right).
\end{align*}
We conclude that $\lambda \leq \frac{\delta}{2(\delta-1)}$.
Thus, by \eqref{choose-gamma-delta-g-1} and \eqref{max-noise-reduction-allowed},
$$
\tau_{\mathsf{min,orc}} = \frac{\tau_{\mathsf{lb}} + \Delta}{\lambda \gamma + 1} > \frac{\tau_{\mathsf{lb}} + \Delta}{\frac{\delta}{2(\delta-1)}\frac{2\theta(\delta-1)}{(1-\theta)\delta} + 1} = (\tau_{\mathsf{lb}} + \Delta)(1 - \theta) = \tau_{\mathsf{lb}},
$$
so we have \eqref{noise-reduction-lower-bound}.
Thus, if we choose $\gamma$ to satisfy \eqref{choose-gamma-delta-g-1}, then \eqref{noise-reduction-lower-bound} holds at any solution $\lambda$ to \eqref{se-lam-fixed-pt-finite} at $\tau_{\mathsf{min}},\gamma$.
\end{itemize}

\noindent {\bf Step 3: Combine steps 1 and 2.}

\noindent We now provide the construction required by the lemma.
We choose $\gamma, \tau_{\mathsf{min}},$ and subsequence $\{p(\ell)\}$ as in Step 2. 
We showed that, along this sequence, for any $\lambda$ which solves \eqref{se-lam-fixed-pt-finite}, we have \eqref{noise-reduction-lower-bound}. 
Because $\frac1{\lambda\gamma+ 1} \leq 1$, we also have $\tau_{\mathsf{min,orc}} \leq \tau_{\mathsf{lb}} + \Delta$. 
Thus, $\tau_{\mathsf{min,orc}} \in [\tau_{\mathsf{lb}},\tau_{\mathsf{lb}}+\Delta]$, and by \eqref{finite-range-se-var-inequality}, we have
$
\delta \tau_{\mathsf{min}}^2 - \sigma^2 < \delta\tau_{\mathsf{lb}}^2 - \sigma^2 + \zeta/3 < \mathsf{R^{opt}_{seq,cvx}}(\tau_{\mathsf{min,orc}};\pi,p).
$
We conclude \eqref{se-small-tau-inequality}. \hfill $\square$

\subsection{Proof of Lemma \ref{lem-se-fixed-pt-existence-non-asymptotic}}

\begin{proof}[Proof of Lemma \ref{lem-se-fixed-pt-existence-non-asymptotic}]

We prove parts (i) and (ii) in parallel.

Under the conditions of part (i), by Lemma \ref{lem-oracle-not-too-good}, we can (and do) choose $\gamma > 0, \tau_{\mathsf{min}} \geq \tau_{\mathsf{lb}}$, and a subsequence $\{p(\ell)\}$ of $\{p\}$ such that for all $p$ in the subsequence and all $\lambda$ which solves \eqref{se-lam-fixed-pt-finite} at $\tau_{\mathsf{min}},\gamma$, \eqref{se-small-tau-inequality} holds. 
Under the conditions of part (ii), we take $\tau_{\mathsf{min}} = \tau_{\mathsf{lb}}$ and $\gamma = 0$. Now there exists a subsequence such that \eqref{se-small-tau-inequality} holds for any $\lambda$ by the definition of $\mathsf{R_{reg,cvx}^{opt}}$ and $\tau_{\mathsf{lb}}$ (and in particular, it holds for those $\lambda$ solving \eqref{se-lam-fixed-pt-finite}).

Because $\{\rho_{p(\ell)}\} \in \cB$ (indeed, property \eqref{eqdef-bounded-shrinkage-towards-infinity} is closed under taking subsequences), we may choose $M$ such that $M \geq \|\mathsf{prox}[\rho_{p(\ell)}](\bzero)\|$ for all $\ell$. 
By Lemma \ref{lem-oracle-not-too-bad}, under the conditions of parts (i) and (ii) and the respective choices of $\gamma$, we can (and do) choose $\tau_{\mathsf{max}}$ such that if $\tau,\lambda$ is a solution of \eqref{se-lam-fixed-pt-finite} at $\gamma$ with $\tau \geq \tau_{\mathsf{max}}$, then \ref{se-large-tau-inequality} holds.

Choose $\lambda_{\mathsf{max}}> 0$ such that 
\begin{equation}\label{lambda-ub-property}
     2\lambda_{\mathsf{max}} \left(1 - \frac1{\delta(\lambda_{\mathsf{max}} \gamma + \lambda_{\mathsf{max}}\kappa + 1)}\right) > 1,
\end{equation}
where $\kappa = 0$ when $\{\rho_p\}$ is not uniformly strongly convex.
Note that this is possible in part (i) because $\gamma > 0$, and in part (ii) because either $\delta > 1$ or $\kappa > 0$.
Finally, choose $\lambda_{\mathsf{min}} > 0$ such that
\begin{equation}\label{lambda-lb-property}
     2\lambda_{\mathsf{min}} < 1.
\end{equation}
For simplicity, we denote the subsequence $\{p(\ell)\}$ by $\{p\}$ for the remainder of the proof.

For each $p$, denote
\begin{align}
    r_p(\tau,\lambda) &= \delta\tau^2 - \sigma^2 - \mathsf{R_{cvx,cvx}}(\tau_{\mathsf{orc}},\lambda_{\mathsf{orc}},\cT_p),\label{eqdef-little-rp}\\
    w_p(\tau,\lambda) &= 2\lambda \left(1 - \frac1{\delta(\lambda \gamma + 1)}\mathsf{W_{reg,cvx}}(\tau_{\mathsf{orc}},\lambda_{\mathsf{orc}},\cT_p) \right).\label{eqdef-little-wp}
\end{align}
By Lemma \ref{lem-uniform-modulus-of-continuity} and the continuity of the map $(\tau,\lambda) \mapsto \left(\frac{\tau}{\lambda \gamma + 1}, \frac{\lambda}{\lambda\gamma + 1}\right)$ on $(\tau,\lambda) \in [\tau_{\mathsf{min}},\tau_{\mathsf{max}}] \times [\lambda_{\mathsf{min}},\lambda_{\mathsf{max}}]$, we have that $w_p$ and $r_p$ are continuous on $[\tau_{\mathsf{min}},\tau_{\mathsf{max}}] \times [\lambda_{\mathsf{min}},\lambda_{\mathsf{max}}]$. 
To simplify notation in the argument that follows, we will work under the change of variables implemented by the linear bijection
\begin{align}
\iota: [0,1]\times[0,2] &\rightarrow [\tau_{\mathsf{min}}, \tau_{\mathsf{max}}] \times [\lambda_{\mathsf{min}}, \lambda_{\mathsf{max}}],\\
 (a,b) &\mapsto \left( (1-a)\tau_{\mathsf{min}} + a\tau_{\mathsf{max}},  \left(1 - \frac b2\right) \lambda_{\mathsf{min}} + \frac b2 \lambda_{\mathsf{max}}  \right).
\end{align}
The functions $r_p \circ \iota$ and $w_p \circ \iota$ are continuous on $[0,1] \times [0,2]$.
By \eqref{finite-p-width-bounded-by-1}, \eqref{width-bounds}, and \eqref{prox-width-bound}, we have for all $\tau,\lambda$ that $2\lambda \geq w_p(\tau,\lambda) \geq 2\lambda\left(1 - \frac1{\delta(\lambda\gamma+1)(1 + \lambda_{\mathsf{orc}}\kappa)}\right) = 2\lambda\left(1 - \frac1{\delta(1 + \lambda\gamma + \lambda\kappa)}\right)$. 
Thus, by \eqref{lambda-ub-property}, \eqref{lambda-lb-property}, and \eqref{eqdef-little-wp},
\begin{align}\label{w-on-lower-and-upper-edge}
    w_p\circ \iota(a,0) &< 1,\quad w_p\circ \iota (a,2) > 1 \quad \text{for all $a \in [0,1].$}
\end{align}
We seek $(a,b) \in [0,1]\times[0,2]$ such that 
\begin{equation}\label{solution-in-a-and-b}
r_p\circ \iota(a,b) = 0 \quad\text{and}\quad w_p \circ \iota(a,b) = 1.
\end{equation}
The next several paragraphs provide the construction, which essentially amounts to a type of two-dimensional intermediate value theorem.

Let $D_0 = [0,1]\times \{0\}$ and $D_2 = [0,1]\times \{2\}$. 
Let $S = \{(a,b) \in [0,1]\times[0,2]\mid w_p \circ \iota (a,b) \leq 1\}$.
Note that $D_0$ is a connected subset of $S$ by \eqref{w-on-lower-and-upper-edge}.
Let $C_0 = \bigcup C$, where the union is taken over connected sets $C \subset S$ which contain $D_0$. 
The set $C_0$ is connected \cite[Theorem 1.14]{Moise1977Geometric3}, so we are justified in calling $C_0$ ``the connected component of $S$ which contains $D_0$.'' The set $C_0$ is also closed because $S$ is closed and the closure of any connected set is still connected. 
Thus, it is compact.
By \eqref{w-on-lower-and-upper-edge}, $D_2 \cap S = \emptyset$, so that $C_0$ and $D_2$ are disjoint.
Because $C_0$ and $D_2$ are disjoint and compact, they are separated by some Euclidean distance $\xi > 0$. 

For any $\theta > 0$, define
\begin{equation}
    C_{0,\theta}= \{(a,b) \in [0,1]\times [0,2] \mid d((a,b),C_0) \leq \theta\},
\end{equation}
where $d$ denotes Euclidean distance.
Clearly, $C_{0,\theta}$ is closed. For $\theta < \xi/3$, $C_{0,\theta}$ is distance at least $2\xi/3$ from $D_2$.
We consider the lattice on $[0,1] \times [0,2]$ consisting of points $\left(\frac iN, \frac jN\right)$ for $i \in \{0,1,\ldots,N\}$ and $j \in \{0,1,\ldots,2N\}$, where $N$ is chosen to be large enough so that 
\begin{equation}
\theta_N := \frac{\sqrt 5}{N} = \mathsf{diam}\left(\left[\frac iN, \frac{i+2}N\right] \times \left[\frac jN, \frac{j+1}N\right]\right) < \frac{\xi}3.
\end{equation} 
Here, $\mathsf{diam}$ denotes the supremal distance between two points contained in a set. 
We define a set of points $\cV$ and line segments $\cE$ as follows. 
The vertex set $\cV$ is
\begin{equation}
    \cV = \left\{ \bv_{ij} := \left(\frac iN, \frac jN\right) \Bigm\vert i \in \{0,1,\ldots,N\},\, j \in \{0,1,\ldots,2N\} \right\}.
\end{equation}
The edge set $\cE$ contains ``horizontal'' edges $E^{H}_{ij} := \left\{\left(\frac iN, \frac jN \right),\left( \frac{i+1}N, \frac jN \right)\right\}$ and ``vertical'' edges $E^{V}_{ij} := \left\{\left(\frac iN, \frac jN\right),\left(\frac iN, \frac{j+1}N\right)\right\}$ for certain values of $i,j$, as we now specify.
\begin{description}
    \item[Horizontal edges.] The edge $E^H_{ij} \in \cE$ if and only if the following are all true.
    \begin{enumerate}[(i)]
        \item $i \in \{0,\ldots,N-1\}$ and $j \in \{1,\ldots,2N-1\}$ (ie. we exclude edges along the bottom or top edge of $[0,1] \times [0,2]$).
        \item Either (i) $j-i$ is even and exactly one of the open rectangles $\Big(\frac iN, \frac {i+2}N\Big) \times \Big(\frac jN, \frac{j+1}N\Big)$ and $\Big(\frac {i-1}N, \frac {i+1}N\Big) \times \Big( \frac {j-1}N, \frac jN\Big)$ has non-empty intersection with  $C_{1,\xi/3}$, or (ii) $j-i$ is odd and exactly one of the open rectangles $\Big(\frac {i-1}N, \frac {i+1}N\Big) \times \Big( \frac jN, \frac{j+1}N\Big)$ and $\Big(\frac {i}N, \frac {i+2}N\Big) \times \Big( \frac {j-1}N, \frac jN\Big)$ has non-empty intersection with $C_{1,\xi/3}$.
    \end{enumerate}
    \item[Vertical edges.] The edge $E^V_{ij} \in \cE$ if and only if the following are all true.
    \begin{enumerate}[(i)]
        \item $i \in \{0,\ldots,2N-1\}$ and $j \in \{1,\ldots,N-1\}$ (ie. we exclude edges along the left or right edge of $[0,1] \times [0,2]$).
        \item $j - i$ is even and exactly one of the open rectangles $\Big(\frac{i - 2}N, \frac iN\Big)\times \Big(\frac jN, \frac{j+1}N\Big)$ and $\Big(\frac iN, \frac {i+2}N\Big)\times \Big(\frac jN, \frac{j+1}N\Big)$ has non-empty intersection with $C_{1,\xi/3}$.
    \end{enumerate}
    \end{description}
    
\begin{remark}
    To interpret the preceding definitions, the reader should have in mind the following picture. We tile the rectangle $[0,1] \times [0,2]$ with ``bricks'' of width 2 and height 1 whose alignment is offset by 1 in neighboring rows (as is done in \cite[Theorem 4.4]{Moise1977Geometric3}). The collection of edges we have specified delineates the outer-boundary of the union of bricks in the tiling which intersect $C_{0,\xi/3}$ (excluding the shared boundary with $[0,1] \times [0,2]$ itself). We should think of think of this as a more topologically well-behaved approximation to the boundary of $C_{0}$ itself.
\end{remark}

\noindent We establish the following series of claims. 

\begin{claim}\label{claim-edges-far-from-sets}
    For all edges $E \in \cE$, all points $\bp \in E$ are distance at least $\theta_N$ and at most $2\theta_N$ from $C_0$ and at least $\xi/3$ from $D_2$.
\end{claim}

\begin{proof}[Proof of Claim \ref{claim-edges-far-from-sets}] 
Note each edge is contained in the boundary of each of the rectangles invoked in its definition. 
That is, for horizontal edges $E^H_{ij}$ with $j- i$ even, we have 
$
E^H_{ij} \in \Big[\frac iN, \frac {i+2}N\Big] \times \Big[ \frac jN, \frac{j+1}N\Big]$ and $\Big[\frac {i-1}N, \frac {i+1}N\Big] \times \Big[ \frac {j-1}N, \frac jN\Big]$, and for $j - i$ odd we have $E_{ij}^H \in  \Big[\frac {i-1}N, \frac {i+1}N\Big] \times \Big[ \frac jN, \frac{j+1}N\Big]$ and $\Big[\frac {i}N, \frac {i+2}N\Big] \times \Big[ \frac {j-1}N, \frac jN\Big]$. 
For vertical edges, we have $E^V_{ij} \in \Big[\frac{i - 2}N, \frac iN\Big]\times \Big[\frac jN, \frac{j+1}N\Big]$ and $\Big[\frac iN, \frac {i+2}N\Big]\times \Big[\frac jN, \frac{j+1}N\Big]$.
Thus, all edges $E \in \cE$ are contained in the boundary of a rectangle which does not intersect $C_{0,\xi/3}$, so that all $\bp \in E$ are distance at least $\xi/3 > \theta_N$ from $C_0$.
Also, all edges $E \in \cE$ are contained in the boundary of a rectangle of diameter $< \theta_N$ with non-empty intersection with $C_{0,\xi/3}$. 
Because every point of $C_{0,\xi/3}$ is distance at most $\xi/3$ from $C_0$, we see that all $\bp \in E$ are distance at most $\xi/3 + \theta_N < 2\xi/3$ from $C_0$. 
Because $C_0$ and $D_2$ are separated by distance $\xi$, all $\bp \in E$ are distance at least $\xi/3$ from $D_2$.
We have established Claim \ref{claim-edges-far-from-sets}.
\end{proof}

\begin{claim}\label{claim-degree-is-0-or-2}
    For $i \neq 0$ or $N$ and $j \neq 0$ or $2N$, the vertex $\bv_{ij}$ is the endpoint of either 0 or 2 edges in $\cE$. (That is, this applies to vertices not on the boundary of $[0,1] \times [0,2])$.
\end{claim}

\begin{proof}[Proof of \ref{claim-degree-is-0-or-2}] The only edges which possibly have endpoint $\bv_{ij}$ are vertical edges $E^V_{ij},E^V_{i(j-1)}$ and horizontal edges $E^H_{ij},E^H_{(i-1)j}$.
First, consider that $j - i$ is even. Then $E^V_{i(j-1)} \not \in \cE$ because $j - 1 - i$ is not even.
There are three rectangles whose intersection with $C_{0,\theta_N}$ determine the membership of the remaining three edges, $E^V_{ij}$, $E^H_{ij}$, and $E^H_{(i-1)j}$, in $\cE$. They are $\Big(\frac{i - 2}N, \frac iN\Big)\times \Big(\frac jN, \frac{j+1}N\Big)$, $ \Big(\frac iN, \frac {i+2}N\Big)\times \Big(\frac jN, \frac{j+1}N\Big)$, and $ \Big(\frac {i-1}N, \frac {i+1}N\Big)\times \Big(\frac {j-1}N, \frac jN\Big)$. 
The edge $E^V_{ij}$ is in $\cE$ if exactly one of the first two rectangles has non-empty intersection with $C_{0,\theta_N}$; the edge $E^H_{ij}$ is in $\cE$ if exactly one of the last two has non-empty intersection with $C_{0,\theta_N}$; and the edge $E^H_{(i-1)j}$ is in $\cE$ if exactly one of the first and last rectangle has has non-empty intersection with $C_{0,\theta_N}$. 
Thus, if exactly one or two of the three rectangles has non-empty intersection with $C_{0,\theta_N}$, then two of the edges $E^V_{ij},E^H_{ij},E^H_{(i-1)j}$ is in $\cE$; otherwise, none of these edges are in $\cE$. 
The case $j-i$ odd is similar.
This establishes Claim \ref{claim-degree-is-0-or-2}.
\end{proof}

\begin{claim}\label{claim-degree-is-0-or-1}
    If $i = 0$ or $N$ or $j = 0$ or $2N$, then the vertex $\bv_{ij}$ is the endpoint of either 0 or 1 edges in $\cE$.
\end{claim}

\begin{proof}[Proof of Claim \ref{claim-degree-is-0-or-1}] For $i = 0$, it is easy to check that the only edge which could be in $\cE$ without violating conditions (i) is $E^H_{0j}$. The other cases are similar, establishing Claim \ref{claim-degree-is-0-or-1}.
\end{proof}

Though we have defined $\cV$ and $\cE$ as sets of points and line segments in the plane, we may think of them as vertices and edges in a graph $\cG = (\cV,\cE)$.
Claims \ref{claim-degree-is-0-or-2} and \ref{claim-degree-is-0-or-1} establish by elementary graph theory that the graph is partitioned into connected components, each of which is a path whose endpoints are on the boundary of $[0,1]\times[0,2]$ and whose other vertices are in the interior of $[0,1]\times[0,2]$.
These paths contain each of the vertices in the path exactly once.

\begin{claim}\label{claim-path-connects-left-and-right}
    There is a path $\bp_0,\ldots,\bp_K$ in the graph $\cG$ such that $ \bp_0 \in \{0\} \times [0,2]$ and $\bp_K \in \{1\} \times [0,2]$, the left and right boundary of $[0,1]\times [0,2]$.
\end{claim}

\begin{proof}[Proof of Claim \ref{claim-path-connects-left-and-right}] Observe that $\Big(0,\frac2N\Big) \times \Big(0,\frac1N\Big)$ intersects $C_{0,\theta_N}$ because it is distance 0 from $[0,1]\times \{0\} = D_0\subset C_0$. 
Also, $\Big(-\frac1N,\frac1N\Big) \times \Big( \frac{2N-1}N, 2 \Big)$ does not intersect $C_{0,\theta_N}$ because it has diameter $\theta_N < \xi/3$ and intersects $D_2$, which has distance at least $\xi$ from $C_0$.
Thus, there is a $j_{\mathsf{max}}$ the maximal value of $j$ such that $\Big(\frac{i(j)}{N}, \frac{i(j) + 2}N\Big) \times \Big(\frac{j - 1}N, \frac jN\Big)$ has non-empty intersection with $C_{0,\theta_N}$, where we have denoted $i(j) = -1$ if $j$ is even and $i(j) = 0$ if $j$ is odd.
By the definition of $\cE$, we see that $E^H_{0j_{\mathsf{max}}} \in \cE$ and $j_{\mathsf{max}}$ is the maximal $j$ for which this is true.
Let $\bp_0 = \bv_{0j_{\mathsf{max}}}$ and $\bp_0,\bp_1,\ldots,\bp_K$ be the connected path in $\cG$ to which $\bp_0$ belongs.
We claim $\bp_K \in \{1\} \times [0,2]$. 
We have already established that $\bp_K$ is on the boundary of $[0,1] \times [0,2]$, so we only need to eliminate the possibility that it belongs to the top, bottom, or left boundaries.
Because $\bp_K$ is contained in an edge $E \in \cE$, we have $\bp_K \not \in D_2$, the top boundary, by Claim \ref{claim-edges-far-from-sets}.
Similarly, $\bp_K \not \in D_0$, the bottom boundary, because $D_0 \subset C_0$ and, by Claim \ref{claim-edges-far-from-sets}, $\bp_K$ is distance at least $\theta_N$ from $C_0$.
Finally, consider that $\bp_K$ were in $\{0\} \times [0,2]$, the left boundary. Then the final edge in the path is $E^V_{0j}$ for some $j \neq j_{\mathsf{max}}$. By the definition of $j_{\mathsf{max}}$, we in fact have $j < j_{\mathsf{max}}$. Also, $j > 0$ because otherwise $\bp_{K-1}$ is also on the boundary of $[0,1]\times[0,2]$.
If we connect $\bp_K = \left(0,\frac jN\right)$ and $\bp_0 = \left(0,\frac{j_{\mathsf{max}}}N\right)$ by a line-segment, then $\bp_0,\bp_1,\ldots,\bp_K$ are the vertices of a polygon $P$ (formally, the union of line segments connecting the adjacent vertices and $\bp_0,\bp_K$). 
By \cite[Theorem 2.1]{Moise1977Geometric3}, $\reals^2 \setminus P$ has two connected components which are disconnected from each other, one of which is bounded and one of which is unbounded.
It is straightforward to check that the open rectangle $\Big(0,\frac{1}N\Big)\times\Big(\frac{j_{\mathsf{max}}-1}{N},\frac{j_{\mathsf{max}}}N\Big)$ is in the bounded component,\footnote{This can be established rigorously by computing the ``index'' in the sense of \cite[Lemma 2.2]{Moise1977Geometric3} of a point $\bp$ in its interior. Compute the index via a horizontal ray which starts at $\bp$ and points left. This ray intersect the polygon in 1 point, so has index 1. See \cite{Moise1977Geometric3} for details.} and $D_0$ is in the unbounded component (because $j > 0$).
But $\Big(\frac{i(j_{\mathsf{max}})}{N},\frac{i(j_{\mathsf{max}}) + 2}N\Big)\times\Big(\frac{j_{\mathsf{max}}-1}{N},\frac{j_{\mathsf{max}}}N\Big)$ intersects $C_{0,\xi/3}$ but not the polygon, and $D_0 \subset C_{0,\xi/3}$ and $C_{0,\xi/3}$ is connected, which contradicts that $\Big(0,\frac{1}N\Big)\times\Big(\frac{j_{\mathsf{max}}-1}{N},\frac{j_{\mathsf{max}}}N\Big)$ and $D_0$ are contained in disconnected components of $\reals^2 \setminus P$. Thus, we conclude $\bp_K \not \in \{0\} \times[0,2]$, the left boundary. We have established Claim \ref{claim-path-connects-left-and-right}.
\end{proof}

Now we construct such a path for a sequence $N \rightarrow \infty$. That is, for each $N$ we have a path $\bp_0^{(N)},\ldots,\bp_{K_N}^{(N)}$ such that $\bp_0^{(N)} \in \{0\}\times[0,2]$, $\bp_{K_N}^{(N)} \in \{1\} \times [0,2]$, and whose edges satisfy Claim \ref{claim-edges-far-from-sets}.
By compactness, we may take a subsequence $\{N(\ell)\}$ of $\{N\}$ such that $\bp_{N(\ell)0} \rightarrow \bp_{\mathsf{left}}$ and $\bp_{N(\ell)K_{N(\ell)}} \rightarrow \bp_{\mathsf{right}}$ for some $\bp_{\mathsf{left}},\bp_{\mathsf{right}}$. Because by Claim \ref{claim-edges-far-from-sets} the points $\bp_{N(\ell)0}$ and $\bp_{N(\ell)K_{N(\ell)}}$ are between distance $\theta_N$ and $2\theta_N$ from $C_0$, we have that $\bp_{\mathsf{left}},\bp_{\mathsf{right}} \in \partial C_0$. Thus, $w_p \circ \iota(\bp_{\mathsf{left}}) = w_p \circ \iota(\bp_{\mathsf{right}}) = 1$. 
Thus, by  Lemmas \ref{lem-oracle-not-too-bad} and \ref{lem-oracle-not-too-good}, we have $r_p \circ \iota (\bp_{\mathsf{left}}) < 0$ and $r_p \circ \iota(\bp_{\mathsf{right}}) > 0$.
By the continuity of $r_p \circ \iota$, we have for sufficiently large $\ell$ that $r_p \circ \iota (\bp_{N(\ell)0}) > 0$ and $r_p \circ \iota(\bp_{N(\ell)K_{N(\ell)}}) < 0$.
Then, by the Intermediate Value Theorem along the path $\bp_{N(\ell)0},\ldots,\bp_{N(\ell)K_{N(\ell)}}$, we have for each $\ell$ sufficiently large a point $\bp_{N(\ell)}$ on the path such that $r_p \circ\iota (\bp_{N(\ell)}) = 0$.
By compactness, there exists a further subsequence $\{N'(\ell)\}$ of $\{N(\ell)\}$ such that $\bp_{N'(\ell)} \rightarrow \bp^*$. 
By continuity, we have $r_p \circ \iota(\bp^*) = 0$. 
Further, because $\bp_{N'(\ell)}$ is between distance $\theta_{N'(\ell)}$ and $2\theta_{N'(\ell)}$ from $C_0$, we have in fact that $\bp^* \in \partial C_0$, whence $w_p \circ \iota(\bp^*) = 1$.
With $(\tau,\lambda) = \iota(\bp^*)$, we have that $(\tau,\lambda) \in [\tau_{\mathsf{reg,cvx}} - \eps,\tau_{\mathsf{max}}] \times [1/2,\lambda_{\mathsf{max}}]$ and $\tau,\lambda$ solves \eqref{se-var-fixed-pt-finite}, \eqref{se-lam-fixed-pt-finite}, as desired.
\end{proof}

\section{Proofs for Section \ref{sec-beyond-square-error}: beyond mean square error}
\label{app:app-proofs-for-beyond-square-error}

\begin{proof}[Proof of Proposition \ref{prop-beyond-squared-error}]
    By the strong stationarity of $\tau,\lambda,\delta,\cT$, we have by \eqref{r-finite-p-def}, \eqref{summary-functions-fixed-asymptotic}, \eqref{fixed-pt-prior-asymptotic}, that $\E_{\tilde \bbeta_0,\bz}\left[\|\mathsf{prox}[\lambda \rho_p](\tilde\bbeta_0 + \tau \bz) - \tilde\bbeta_0\|^2\right]$ is bounded, where $\tilde \beta_{0j} \stackrel{\mathrm{iid}}\sim \pi/\sqrt{p}$.
    By Jensen's, also $\E_{\tilde \bbeta_0,\bz}\left[\|\mathsf{prox}[\lambda \rho_p](\tilde\bbeta_0 + \tau \bz) - \tilde\bbeta_0\|\right]$ is bounded.
    By the triangle inequality that 
    \begin{align*}
    \|\mathsf{prox}[\lambda \rho_p](\bzero)\| &\leq \|\mathsf{prox}[\lambda \rho_p](\tilde\bbeta_0 + \tau\bz) - \tilde\bbeta_0\| + \|\tilde\bbeta_0\|\\
    &\qquad+ \|\mathsf{prox}[\lambda \rho_p](\bzero) - \mathsf{prox}[\lambda \rho_p](\tilde\bbeta_0 + \tau \bz)\| \\
    &\leq  \|\mathsf{prox}[\lambda \rho_p](\tilde\bbeta_0 + \tau\bz) - \tilde\bbeta_0\| + \|\tilde\bbeta_0\|+ \|\tilde\bbeta_0 + \tau \bz\|,
    \end{align*}
    where in the second inequality we have applied \eqref{prox-is-lipschitz} from Appendix \ref{app-proximal-operator-identities}. 
    Taking expectations on both sides, we get that $\mathsf{prox}[\lambda \rho_p](\bzero)$ is bounded.
    Further, again by \eqref{prox-is-lipschitz} from Appendix \ref{app-proximal-operator-identities}, we have that the sequence (in $p$) of functions $(\bx_1,\bx_2) \mapsto (\bx_1, \mathsf{prox}[\lambda \rho_p](\bx_2))$ is uniformly pseudo-Lipschitz of order 1 and bounded at $(\bzero,\bzero)$. 
    Then, by Lemma \ref{lem-pseudo-lipschitz-closed-under-composition} from Appendix \ref{app-useful-tools}, we have the sequence of functions $(\bx_1,\bx_2) \mapsto \ell_p(\bx_1,\mathsf{prox}[\lambda \rho_p](\bx_2))$ is uniformly pseudo-Lipschitz of order $k$. 
    By Proposition  \ref{prop-strongly-convex-loss}(iii) applied to $\tilde \tau, \tilde \lambda, \delta,\tilde \cT$, we then have
    \begin{equation}\label{big-noise-arbitrary-loss}
    \ell_p\left(\bbeta_0,\mathsf{prox}[\lambda \rho_p] \left(\widehat \bbeta_{\mathsf{cvx}} +  2\lambda\frac{\bX^\mathsf{T}(\by - \bX \widehat \bbeta_{\mathsf{cvx}})}{n}\right)\right) \stackrel{\mathrm{p}}\simeq \E_{\bz}\left[\ell_p\left(\bbeta_0,\bbeta_0 + \tau \bz\right)\right],
    \end{equation}
    where we have used that either $\delta > 1$ or $\{\rho_p\} \in \cC_*$.
    Further, by Lemma \ref{lem-pseudo-lipschitz-closed-under-random-element}, the sequence of functions $(\bx_1,\bx_2) \mapsto \E_{\bz} \left[\ell_p(\bx_1,\mathsf{prox}[\lambda \rho_p](\bx_2 + \sqrt{\tau^2 - \tilde \tau^2} \bz)\right]$ is uniformly pseudo-Lipschitz of order $k$. 
    By Proposition \ref{prop-strongly-convex-loss}(iii) applied to $\tilde \tau, \tilde \lambda, \delta,\tilde \cT$, we then have under either (i) the HDA and RSN assumption, or (ii) the HDA and DSN assumptions if $\tilde \rho_p$ are symmetric, that
    \begin{align}
    &\E_{\bz}\left[\ell_p\left(\bbeta_0 , \mathsf{prox}\left[\lambda  \rho_p\right]\left(\widehat {\tilde \bbeta}_{\mathsf{cvx}} +  2\lambda\frac{\bX^\mathsf{T}(\by - \bX \widehat{\tilde \bbeta}_{\mathsf{cvx}})}{n} + \sqrt{\tau^2 - \tilde \tau^2}\bz\right) \right)\right]\nonumber \\
    &\qquad\qquad\qquad\qquad\qquad\qquad\qquad\qquad\qquad\stackrel{\mathrm{p}}\simeq \E_{\bz_1,\bz_2}\left[\ell_p\left(\bbeta_0,\bbeta_0 + \tilde \tau \bz_1 + \sqrt{\tau^2 - \tilde \tau^2} \bz_2\right)\right]\nonumber\\
    &\qquad\qquad\qquad\qquad\qquad\qquad\qquad\qquad\qquad= \E_{\bz}\left[\ell_p\left(\bbeta_0,\bbeta_0 + \tau \bz \right)\right],\label{expected-noise-expansion-arbitrary-loss}
    \end{align}
    where $\bz_1,\bz_2 \sim \mathsf{N}(\bzero,\bI_p/p)$ are independent and we have used that either $\delta > 1$ or $\{\rho_p\} \in \cC_*$.
    By Lemma \ref{lem-emp-to-exp-non-symm}, we have that 
    \begin{align}
        &\E_{\bz}\left[\ell_p\left(\bbeta_0 , \mathsf{prox}\left[\lambda  \rho_p\right]\left(\widehat {\tilde \bbeta}_{\mathsf{cvx}} +  2\lambda\frac{\bX^\mathsf{T}(\by - \bX \widehat{\tilde \bbeta}_{\mathsf{cvx}})}{n} + \sqrt{\tau^2 - \tilde \tau^2}\bz\right) \right)\right] \nonumber \\
        &\qquad\qquad\qquad\qquad\stackrel{\mathrm{p}}\simeq \ell_p\left(\bbeta_0 , \mathsf{prox}\left[\lambda  \rho_p\right]\left(\widehat {\tilde \bbeta}_{\mathsf{cvx}} +  2\lambda\frac{\bX^\mathsf{T}(\by - \bX \widehat{\tilde \bbeta}_{\mathsf{cvx}})}{n} + \sqrt{\tau^2 - \tilde \tau^2}\bz\right) \right).\label{noise-expansion-arbitrary-loss}
    \end{align}
    Combining \eqref{big-noise-arbitrary-loss}, \eqref{expected-noise-expansion-arbitrary-loss}, and \eqref{noise-expansion-arbitrary-loss}, and using the definition \eqref{post-processing}, we get \eqref{post-processing-loss-equivalent}, as desired.
\end{proof}

\begin{proof}[Proof of Theorem \ref{conj-beyond-squared-error}]
    By the same argument as in Claim \ref{bounded-minimization-is-enough}, 
    it is enough to show \eqref{beyond-squared-error-lower-bound} for $\{\rho_p\} \in \cC_* \cap \cB$.
    Assume for the sake of contradiction that the left-hand side of \eqref{beyond-squared-error-lower-bound} is less than the right-hand side.
    By passing to a subsequence, we 
    may assume we have $\{\rho_p\} \in \cC_*$ such that 
    \begin{equation*}
        \lim^{\mathrm{p}}_{p \rightarrow \infty} \frac1p \sum_{j=1}^p \ell\left(\sqrt{p} \beta_{0j},\sqrt{p}\widehat \beta_{\mathsf{cvx},j}\right) < \E_{\beta_0,z}[\ell(\beta_0,\eta(\beta_0 + \tau_{\mathsf{reg,cvx}}z)].
    \end{equation*}
    By Lemma \ref{lem-oracle-risk-to-cvx-bound},
    we may find a further subsequence $\{p(\ell)\}$, $\tau \geq \tau_{\mathsf{reg,cvx}}$, and $\lambda > 0$ such that with $\cT = (\pi,\{\rho_{p(\ell)}\})$,
    the quintuplet $\tau,\lambda,\delta,\gamma = 0,\cT$ is strongly stationary.
    By the KKT conditions for \eqref{linear-cvx-estimator},
    \begin{equation*}
        \widehat \bbeta_{\mathsf{cvx}} = \mathsf{prox}[\lambda \rho]\left(\widehat \bbeta_{\mathsf{cvx}} + 2\lambda \frac{\bX^\top(\by - \bX \widehat \bbeta_{\mathsf{cvx}})}{n}\right),
    \end{equation*}
    whence Proposition \ref{prop-strongly-convex-loss}(iii) implies (either under the HDA and RSN assumptions, or, if the penalties are symmetric, with the RSN assumption replaced by the DSN assumption)
    \begin{align*}
        \lim^{\mathrm{p}}_{p\rightarrow \infty} \frac1p \sum_{j=1}^p \ell\left(\sqrt{p}\beta_{0j},\sqrt{p}\widehat \beta_{\mathsf{cvx},j}\right)
            &\stackrel{\mathrm{p}}\simeq
            \E_{\bz}\left[\frac1p \sum_{j=1}^p \ell\left(\sqrt{p}\beta_{0j},\sqrt{p}\mathsf{prox}[\lambda \rho_p](\bbeta_0 + \tau_{\mathsf{reg,cvx}}\bz)_j\right)\right ]
            \\
            &\stackrel{\mathrm{p}}\simeq
            \E_{\tilde\bbeta_0,\bz}\left[\frac1p \sum_{j=1}^p \ell\left(\sqrt{p}\tilde\beta_{0j},\sqrt{p}\mathsf{prox}[\lambda \rho_p](\tilde \bbeta_0 + \tau_{\mathsf{reg,cvx}}\bz)_j\right)\right ] 
            \\
            &\geq
            \E_{\beta_0,z}[\ell(\beta_0,\eta(\beta_0 + \tau_{\mathsf{reg,cvx}}))],
    \end{align*}
    where the second inequality holds by Lemma \ref{lem-emp-to-exp-non-symm} under the DSN and RSAN assumption or by Lemma \ref{lem-pseudo-lipschitz-empirical-to-expectation} when $\rho_p$ are symmetric and the DSN assumption holds (here $\tilde \beta_{0j} \stackrel{\mathrm{iid}}\sim \pi/\sqrt{p}$;
    and the final inequality holds by the optimality of $\eta$ with respect to the loss $\ell$.
    Moreover, if $\eta \neq \mathsf{prox}[\lambda\rho_p]$, which occurs when $\eta$ is not a proximal operator, this inequality is strict.

    By Theorem \ref{thm-cvx-lower-bound}, when $\delta > 1$ the convex lower bound is strict. 
    As we saw in its proof in Section \ref{sec:tightness-cvx-lb},
    tightness holds because there exists $\{\rho_p\} \in \cC_*$ and $\lambda \geq 0$ such that with $\cT = (\pi,\{\rho_p\})$ we have that $\tau_{\mathsf{reg,cvx}},\lambda,\gamma=0,\delta,\cT$ is strongly stationary.
    Thus, for any Lipschitz $\eta'$, by Proposition \ref{prop-strongly-convex-loss} we have
    \begin{equation*}
        \lim^{\mathrm{p}}_{p \rightarrow \infty} \frac1p \sum_{j=1}^p \ell\left(\sqrt{p}\beta_{0j},\eta'\left(\sqrt{p} \widehat \bbeta_{\mathsf{cvx},j} + 2\lambda\frac{[\bX^\mathsf{T}(\by - \bX \widehat \bbeta_{\mathsf{cvx}})]_j}{n}\right)\right) = \E_{\beta_0,z}[\ell(\beta_0,\eta'(\beta_0 + \tau_{\mathsf{reg,cvx}}z))].
    \end{equation*}
    Because the set of Lipschitz functions is dense in $L_2(\pi * \normal(0,\tau_{\mathsf{reg,cvx}}^2))$,
    taking the infimum over $\eta'$ gives $\E_{\beta_0,z}[\ell(\beta_0,\eta(\beta_0,\tau_{\mathsf{reg,cvx}}z))]$ on the right-hand side.
    This completes the proof.
\end{proof}

\section{The role of the $\delta$-bounded width assumption}\label{sec-role-of-delta-bounded-width-restriction}

The primary weakness of Theorem \ref{thm-cvx-lower-bound} is its restriction to sequences of convex functions in $\cC_{\delta,\pi}$. For $\delta > 1$, this is no restriction at all.
In this section, we provide some reflection on the nature of the restriction for $\delta < 1$ and the role it plays in Theorem \ref{thm-cvx-lower-bound}.
No other sections or appendices depend upon the results in this appendix.

First, we observe that for $\delta < 1$, Theorem \ref{thm-cvx-lower-bound} does not hold  if we instead take the infimum in \eqref{cvx-lower-bound} over $\{\rho_p\} \in \cC$, the collection of all sequences of convex penalties.

\begin{claim}\label{claim-beat-statistical-risk}
    Take $\rho_p = 0$ (so $\{\rho_p\} \not \in \cC_{\delta,\pi}$).
    Under the RSN assumption, if $\delta < 1$,
    there exists a random sequence $\widehat \bbeta_{\mathsf{cvx}}$ such that for each $p$ we have $\widehat \bbeta_{\mathsf{cvx}} \in \arg\min_{\bbeta} \frac1n\|\by - \bX \bbeta\|^2 + \rho_p(\bbeta)$ with probability 1 but
    $$
    \lim^{\mathrm{p}}_{p \rightarrow \infty} \|\widehat \bbeta_{\mathsf{cvx}} - \bbeta_0\|^2 = \frac{\delta \sigma^2}{1-\delta}.
    $$
    For some such values of $\pi,\delta,\sigma$, we have $\frac{\delta \sigma^2}{1 - \delta} < \delta \tau_{\mathsf{reg,stat}}^2 - \sigma^2 \leq \delta \tau_{\mathsf{reg,cvx}}^2 - \sigma^2$.
\end{claim}

\begin{proof}
    For sufficiently large $p$, we have $p > n$ because $n/p \rightarrow \delta < 1$.
    Take such sufficiently large $p$.
    Define 
    \begin{equation}\label{cheating-cvx-estimator}
    \widehat \bbeta_{\mathsf{cvx}} = \arg\min_{\bbeta} \left\{  \|\bbeta - \bbeta_0\|^2 \Bigm\vert \bbeta \in \arg\min_{\bbeta'} \left\{\frac1n\|\by - \bX\bbeta\|^2\right\}\right\}.
    \end{equation}
    Clearly $\widehat \bbeta_{\mathsf{cvx}} \in \arg\min_{\bbeta} \frac1n\|\by - \bX \bbeta\|^2 + \rho_p(\bbeta)$. 
    Let the singular value decomposition of $\bX$ be $\bU\bS\bV^\mathsf{T}$ where $\bU \in \reals^{n \times n}$ is orthonormal, $\bS \in \reals^{n\times n}$ is diagonal, and $\bV \in \reals^{p \times n}$ has orthonormal columns. Let $\bV_\perp \in \reals^{p\times(p-n)}$ have orthonormal columns orthogonal to those of $\bV$. Because $p> n$, this makes sense, and moreover, $\bX$ is full-rank with probability 1, whence $\bS$ is non-singular.
    We parameterize $\bbeta$ as $\bV \bb + \bV_\perp \bb_\perp$ for $\bb \in \reals^n$ and $\bb_\perp \in \reals^{p-n}$. 
    Then
    $$
    \frac1n \|\by - \bX \bbeta\|^2 = \frac1n\left\|\by - \bU \bS \bV^\mathsf{T}(\bV \bb + \bV_\perp \bb_\perp)\right\|^2 = \frac1n \|\by - \bU \bS \bb\|^2 = \frac1n \left\|\bU^\mathsf{T}\by - \bS \bb\right\|.
    $$
     Because $\bS$ is non-singular,
    $$
    \arg\min_{\bbeta}\left\{\frac1n \|\by - \bX \bbeta\|^2\right\} = \left\{ \bV \bS^{-1} \bU^\mathsf{T} \by + \bV_\perp \bb_\perp \mid \bb_\perp \in \reals^{p-n}\right\}.
    $$
    Observe
    \begin{align}
    \left\|\bV \bS^{-1} \bU^\mathsf{T} \by + \bV_\perp \bb_\perp - \bbeta_0\right\|^2 &= \left\|\bV \bS^{-1} \bU^\mathsf{T} \by + \bV_\perp \bb_\perp - \bV\bV^\mathsf{T}\bbeta_0 - \bV_\perp \bV_\perp^\mathsf{T}\bbeta_0\right\|^2 \nonumber \\ 
    &=\left \|\bS^{-1} \bU^\mathsf{T} \by - \bV^\mathsf{T}\bbeta_0\right\|^2 + \left\|\bb_\perp - \bV_\perp^\mathsf{T}\bbeta_0\right\|^2.
    \end{align} 
    This is minimized at $\bb_\perp = \bV_\perp^{\mathsf{T}}\bbeta_0$, whence
    \begin{equation}\label{cheating-beta-cvx-matrix-form}
        \widehat \bbeta_{\mathsf{cvx}} = \bV \bS^{-1} \bU^\mathsf{T} \by + \bV_\perp \bV_\perp^\mathsf{T} \bbeta_0.
    \end{equation}
    Now consider the oracle estimator with parameter $\gamma$.
    The objective we must minimize is
    \begin{align*}
        \frac1n \|\by - \bX \bbeta\|^2 + \frac\gamma2 \|\bbeta - \bbeta_0\|^2 &= \frac1n \left\|\bU^\mathsf{T} \by - \bS \bb\right\|^2 + \frac\gamma2 \left\|\bb - \bV^\mathsf{T} \bbeta_0\right\|^2 + \frac\gamma2\left\|\bb_\perp - \bV_\perp^\mathsf{T} \bbeta_0\right\|^2\\
    &= \left(\bb - \ba\right)^\mathsf{T}\left(\bS^2/n + \gamma \bI_n/2\right)\left(\bb - \ba\right)+ \frac\gamma2\left\|\bb_\perp - \bV_\perp^\mathsf{T} \bbeta_0\right\|^2,
    \end{align*}
    where $\ba = \left(\bS^2/n + \gamma \bI_n/2\right)^{-1}\left(\bS \bU^\mathsf{T} \by/n + \gamma \bV^\mathsf{T}\bbeta_0/2\right)$. 
    Thus,
    \begin{align}
        \widehat \bbeta_{\mathsf{cvx}}^{(\gamma)}  = \bV \left(\bS^2/n + \gamma \bI_n/2 \right)^{-1} \left(\bS \bU^\mathsf{T} \by/n + \gamma \bV^\mathsf{T}\bbeta_0/2\right) + \bV_\perp \bV_\perp^\mathsf{T}\bbeta_0.
    \end{align}
    We get
    \begin{align}
        \|\widehat \bbeta_{\mathsf{cvx}} - \widehat \bbeta_{\mathsf{cvx}}^{(\gamma)}\| &= \left\|\bV (\bS^2/n)^{-1} \bS \bU^\mathsf{T} \by/n - \bV \left(\bS^2/n + \gamma \bI_n/2 \right)^{-1} \left(\bS \bU^\mathsf{T} \by/n + \gamma \bV^\mathsf{T}\bbeta_0/2\right)\right\|\nonumber\\
        &=  \left\|(\bS^2/n)^{-1} \bS \bU^\mathsf{T} \by/n - \left(\bS^2/n + \gamma \bI_n/2 \right)^{-1} \left(\bS \bU^\mathsf{T} \by/n + \gamma \bV^\mathsf{T}\bbeta_0/2\right)\right\|\nonumber\\
        &\leq \left\| \left((\bS^2/n)^{-1} - \left(\bS^2/n + \gamma \bI_n/2 \right)^{-1}  \right) \bS\bU^\mathsf{T}\by/n \right\| \nonumber\\
        &\qquad\qquad\qquad\qquad\qquad\qquad + \left\| \left(\bS^2/n + \gamma \bI_n/2 \right)^{-1} \gamma \bV^\mathsf{T} \bbeta_0/2 \right\|\nonumber\\
        &\leq  \left\| (\bS^2/n)^{-1} - \left(\bS^2/n + \gamma \bI_n/2 \right)^{-1} \right\|_{\mathsf{op}} \frac{\left\|\bS\right\|_{\mathsf{op}}}{\sqrt n} \frac{\|\by\|}{\sqrt n}\nonumber\\
        &\qquad\qquad \qquad\qquad\qquad\qquad+ \frac\gamma2 \left\|\left(\bS^2/n + \gamma \bI_n/2 \right)^{-1}\right\|_{\mathsf{op}} \left\|\bV^\mathsf{T} \bbeta_0\right\|\nonumber\\
        &= \left|\frac1{\sigma_{\mathsf{min}}(\bX)^2/n} - \frac1{\sigma_{\mathsf{min}}(\bX)^2/n + \gamma/2}\right| \frac{\|\bX\|_{\mathsf{op}}}{\sqrt n} \frac{\|\by\|}{\sqrt n} + \frac\gamma2 \frac1{\sigma_{\mathsf{min}}(\bX)^2/n + \gamma/2} \left\|\bV^\mathsf{T}\bbeta_0\right\|\nonumber\\
        &= \varepsilon(\gamma) O_p(1), \label{orc-converges-to-super-orc}
    \end{align}
    for some deterministic function $\varepsilon(\gamma) \downarrow 0$ as $\gamma \rightarrow 0$ and $O_p(1)$ tight over both $p$ and $\gamma$, where $\sigma_{\mathsf{min}}(\bX)$ is the minimal non-zero singular value of $\bX$ and we have used that $\|\bX\|_{\mathsf{op}}/\sqrt{n}$ and $\sigma_{\mathsf{min}}(\bX)/\sqrt{n}$ both converge in probability to constants by \cite[Theorem 5.31]{Vershynin2012IntroductionMatrices}. 
    
    Let $\cT = (\pi,\{\rho_p = 0\})$.
    Because $\rho_p = 0$, for all $\tau,\lambda$, $\mathsf{prox}[\lambda \rho_p](\bbeta_0 + \tau \bz) - \bbeta_0 = \bbeta_0 + \tau \bz - \bbeta_0 = \tau \bz$. 
    Then by \eqref{r-finite-p-def} and \eqref{w-finite-p-def}, we have that $\mathsf{R_{reg,cvx}^\infty}(\tau,\lambda,\cT) = \tau^2$ and $\mathsf{W_{reg,cvx}^\infty}(\tau,\lambda,\cT) = 1$.
    Thus, at oracle parameter $\gamma$, the fixed point equations \eqref{fixed-pt-prior-asymptotic} are equivalent to
    \begin{gather}
        \delta \tau^2 - \sigma^2 = \frac{\tau^2}{(\lambda \gamma + 1)^2}
        \;\;\;
        \text{and}
        \;\;\;
        2\lambda\left(1 - \frac1\delta \frac{1}{\lambda \gamma + 1} \right) = 1.\label{trivial-fixed-pt}
    \end{gather}
    It is straightforward to see that such a solution exists: we may choose non-negative $\lambda$ to solve the second equation in \eqref{trivial-fixed-pt} by the intermediate value theorem; at this value of $\lambda$ we have $\frac1{\lambda \gamma + 1} < \delta$, whence setting $\tau^2 = \frac{\sigma^2}{\delta - (\lambda\gamma+1)^{-2}}$ solves the first equation in \eqref{trivial-fixed-pt}.
    Then $\tau,\lambda,\gamma,\delta,\cT$ is strongly stationary. 
    
    Equation \eqref{trivial-fixed-pt} implies that $\frac{1}{\lambda \gamma + 1} < \delta$, which implies $\lambda > \frac{\delta^{-1} - 1}\gamma \rightarrow \infty$ as $\gamma \rightarrow 0$ because $\delta < 1$.
    We conclude that $\frac{1}{\lambda \gamma + 1}= \delta\left(1-\frac1{2\lambda}\right) \rightarrow\delta$ as $\gamma \rightarrow 0$. 
    Then, writing equation \eqref{trivial-fixed-pt} as $\delta (\lambda \gamma + 1)^2 \frac{\tau^2}{(\lambda \gamma + 1)^2} - \sigma^2 = \frac{\tau^2}{(\lambda \gamma + 1)^2}$, we get $\frac{\tau^2}{(\lambda \gamma + 1)^2} = \frac{\sigma^2}{\delta(\lambda \gamma + 1)^2 - 1} \rightarrow \frac{\sigma^2}{\delta^{-1} - 1} = \frac{\delta \sigma^2}{1-\delta}$.
    In particular, by Proposition \ref{prop-strongly-convex-loss}, we have
    \begin{equation}\label{trivial-oracle-convergence}
    \lim_{\gamma \rightarrow 0} \lim^{\mathrm{p}}_{p \rightarrow \infty}  \|\widehat \bbeta_{\mathsf{cvx}}^{(\gamma)} - \bbeta_0\|^2 = \frac{\delta \sigma^2}{1-\delta}.
    \end{equation}
    By \eqref{orc-converges-to-super-orc}, we have $ \|\widehat \bbeta_{\mathsf{cvx}} - \bbeta_0\|^2 = \|\widehat \bbeta_{\mathsf{cvx}}^{(\gamma)} - \bbeta_0\|^2 + \varepsilon(\gamma)O_p(1)$. Taking $\gamma \rightarrow 0$ and applying \eqref{trivial-oracle-convergence}, we get
    $$
    \lim_{p \rightarrow \infty}^{\mathrm{p}}\|\widehat \bbeta_{\mathsf{cvx}} - \bbeta_0\|^2 =  \frac{\delta\sigma^2}{1-\delta},
    $$
    as desired.
    
    It is easy to construct examples in which this is smaller than $\delta\tau_{\mathsf{reg,stat}}^2 - \sigma^2$ and $\tau_{\mathsf{reg,stat}}\leq \tau_{\mathsf{reg,cvx}}$.
    Here is one construction. 
    Observe that all solutions $\tau_{\mathsf{reg,stat}}^2$ to \eqref{tau-stat-is-stationary} must satisfy $\tau_{\mathsf{reg,stat}}^2 \geq \sigma^2/\delta$.
    Thus, for fixed $\sigma$, we have $\lim_{\delta \rightarrow 0} (\delta \tau_{\mathsf{reg,stat}}^2 - \sigma^2) = \lim_{\delta \rightarrow 0} \mathsf{mmse}_{\pi}(\tau_{\mathsf{reg,stat}}^2) = \lim_{\tau \rightarrow \infty} \mathsf{mmse}_{\pi}(\tau^2) = s_2(\pi)$ \cite[Eq.~(61)]{DongningGuo2011EstimationError}. 
    Moreover, $\lim_{\delta \rightarrow 0} \frac{\delta \sigma^2}{1-\delta} = 0$. 
    Thus, if $s_2(\pi) > 0$ (which is true unless $\pi$ is a point mass at 0), then for sufficiently small $\delta$ we have $\frac{\delta \sigma^2}{1-\delta} <\delta \tau_{\mathsf{reg,stat}}^2 - \sigma^2$. When the minimizer of \eqref{tau-stat-def} is unique, we have $\delta \tau_{\mathsf{reg,stat}}^2 - \sigma^2 \leq \delta \tau_{\mathsf{reg,cvx}}^2 - \sigma^2$ by Theorem \ref{thm-cvx-cannot-beat-alg}. Because the minimizer of \eqref{tau-stat-def} is unique for almost every $(\delta,\sigma)$ (w.r.t. Lebegesgue measure), for some $\sigma$ there are arbitrarily large $\delta$ at which the minimizer of \eqref{tau-stat-def} is unique.
    This completes the construction. 
    \end{proof}
Of course, Claim \ref{claim-beat-statistical-risk} does not --indeed, could not-- imply that we can achieve smaller than Bayes risk using convex M-estimation.
The construction of $\widehat \bbeta_{\mathsf{cvx}}$ given in \eqref{cheating-cvx-estimator} uses knowledge of $\bbeta_0$, so is information theoretically inaccessible to the statistician.
Indeed, even though our measurements and our penalty are completely uninformative along directions parallel to the null space of $\bX$, the estimator $\widehat \bbeta_{\mathsf{cvx}}$ in \eqref{cheating-cvx-estimator} achieves perfect estimation along these directions, as captured by the term $\bV_\perp \bV_\perp^\mathsf{T} \bbeta_0$ in \eqref{cheating-beta-cvx-matrix-form}.

The counterexample of Claim \ref{claim-beat-statistical-risk} demonstrates that the conclusion of Theorem \ref{thm-cvx-lower-bound} is too strong to remove all restrictions on the penalty sequence in \eqref{cvx-lower-bound}.
This is because Theorem \ref{thm-cvx-lower-bound} applies to all mechanisms for breaking ties between members of the minimizing set, even those which rely on knowledge of $\bbeta_0$. 
The counterexample of Claim \ref{claim-beat-statistical-risk} uses an uninformative penalty. When $\rho_p = 0$, the set of minimizers is large, and we have much to gain from breaking ties by looking at $\bbeta_0$, something which the conditions of Theorem \ref{thm-cvx-lower-bound} do not prohibit.

This discussion is perhaps unsurprising given the way in which the $\delta$-bounded width assumption is used in the proof of Theorem \ref{thm-cvx-lower-bound}.
The $\delta$-bounded width assumption is used only in the proof of Lemma \ref{lem-oracle-risk-to-cvx-bound}.
This Lemma shows that the oracle estimator which exploit knowledge of $\bbeta_0$ does so weakly enough that it achieves loss at best negligibly smaller than the convex lower bound (see Lemma \ref{lem-oracle-risk-to-cvx-bound}).
It is not hard to show that when $\rho_p = 0$, even arbitrarily weak oracles can dramatically improve the performance of the M-estimator by allowing us to estimating $\bbeta_0$ exactly correctly along the null space of $\bX$.
Said more generally (but more heuristically), when the minimizing set of the original M-estimator is large --as it is when $\delta < 1$ and $\rho_p = 0$-- arbitrarily weak oracles break ties in the way that best exploits knowledge of $\bbeta_0$. 
Thus, arbitrarily weak oracles achieve a non-negligible improvement over the convex lower bound.
Indeed, this is essentially what we have used in the proof of Claim \ref{claim-beat-statistical-risk}. 

This is not to say that the statistician can do better by choosing penalty sequence from $\cC$ rather than $\cC_{\delta,\pi}$.
Without making statements which are fully precise, we conjecture that (i) no convex M-estimator which with high-probability returns a singleton minimizing set (or perhaps even a minimizing set which is ``small'' in an appropriate sense) can achieve asymptotic loss smaller than $\delta\tau_{\mathsf{reg,cvx}}^2 - \sigma^2$, and  (ii) no convex procedure which breaks ties among members of the minimizing set with a polynomial-time algorithm can achieve asymptotic loss smaller than $\delta\tau_{\mathsf{reg,amp}*}^2 - \sigma^2$.\footnote{Perhaps the lower bound is even larger than this, because we are requiring that we use convex M-estimation for at least a part of the procedure.}
If (i) is true, then it is possible to expand, at least slightly, the set over which we take the infimum in Theorem \ref{thm-cvx-lower-bound}.
We suspect that the restriction $\{\rho_p\} \in \cC_{\delta,\pi}$ corresponds closely, though not exactly, to the restriction that the minimizing set be ``small'' in the appropriate sense.
Resolving (i) would require identifying the appropriate weaker condition.
Successfully resolving statement (ii) would require addressing some of the deepest and most insurmountable problems in the theory of computational complexity.
Exploring whether and in what sense any of these speculations is true is beyond the scope of the current work.

\section{Proof of Proposition \ref{prop-achieving-the-bound}}\label{app-proof-of-prop-achieving-the-bound}

\begin{proof}[Proof of Proposition \ref{prop-achieving-the-bound}(i)]
    In fact,
    we prove \eqref{minimal-loss-upper-bound}
    when the infimum is taken over $\{\rho_p\} \in \cC_*$, the sequences of uniformly strongly convex penalties:
    \begin{equation}\label{minimal-loss-upper-bound-str-cvx}
        \inf_{\{\rho_p\} \in \cC_{*}} \lim^{\mathrm{p}}_{p \rightarrow \infty} \|\widehat \bbeta_{\mathsf{cvx}} - \bbeta_0\|^2 \leq \delta\tau^2 - \sigma^2.
    \end{equation}
    This is stronger than \eqref{minimal-loss-upper-bound}.
    If we show \eqref{minimal-loss-upper-bound-str-cvx},
    we can conclude that under RSN assumption \eqref{minimal-loss-upper-bound} holds also when the limit in probability is replaced by $\lim_{p \rightarrow \infty}\E_{\bbeta_0,\bw,\bX}\big[\| \hat \bbeta_{\mathsf{cvx}} - \bbeta_0\|_2^2\big]$ by applying Lemma \ref{lem-strongly-convex-DSN-to-RSN-risk} and using that $\cC_* \subset \cC_{\delta,\pi}$ (see Proposition \ref{claim-a2-for-strong-convex-penalty}).

    The proof of \eqref{minimal-loss-upper-bound-str-cvx} proceeds in three steps.\\
    \noindent {\bf Step 1: Construct lsc, proper, convex $\rho: \reals \rightarrow \reals$ such that $\mathsf{prox}[\rho]$ is the Bayes estimator.}
    This construction is provided in \cite[pg.~14567]{Bean2013}. We provide most of the details for completeness.
    Let $p_Y(x)$ be the density of $\pi * \mathsf{N}(0,\tau^2)$ (recall, $\tau > 0$, so that this exists).
    Let $m(y) = - \tau^2 \log p_Y(y)$ and $p_2(x) = \frac12 x^2$. 
    By assumption, $m$ is convex.
    Observe that up to the additive constant $\tau^2 \log \left( \sqrt{2\pi} \tau\right)$
    \begin{equation}\label{negative-log-density}
        m(y) = - \tau^2 \log \int e^{-\frac1{2\tau^2}(y - x)^2}\pi(\mathrm{d}x) = \frac12 y^2 - \tau^2 \log \int e^{ \frac1{\tau^2} y x - \frac1{2\tau^2}x^2 }\pi(\mathrm{d}x).
    \end{equation}
    We identify the second term on the right-hand side --up to a multiplicative and additive constant-- as the cumulant generating function of the probability distribution with density proportional to $e^{-\frac1{2\tau^2}x^2}$ with respect to $\pi$. 
    This term can be written as $p_2(y) - m(y)$.
    Because, for all $y$, $e^{ \frac1{\tau^2} y x - \frac1{2\tau^2}x^2 }$ is bounded over $x$, this term is finite for all $y$, and by \cite[Theorem 1.13]{Brown1986FundamentalsTheory}, it is infinitely differentiable and lsc, proper, and convex in $y$.  
    
    For an lsc, proper, convex $f: \reals \rightarrow \reals \cup\{\infty\}$, the \emph{Fenchel-Legendre conjugate} $f^*$ is defined by $f^*(g) = \sup_{x \in \reals}\{gx - f(x)\}$.
    Define
    \begin{equation}
        \rho = (p_2 - m)^* - p_2.
    \end{equation}
    This makes sense because we have argued that $p_2 - m$ is lsc, proper, and convex. Moreover, as argued in \cite[pg. 14567]{Bean2013} by appeal to \cite[Proposition 9.b]{Moreau1965ProximiteHilbertien}, $\rho$ so defined is convex.\footnote{Roughly, this is because $p_2 - m$ is ``less convex'' than $p_2$, so $(p_2 - m)^*$ is ``more convex'' than $p_2$.}
    Define the \emph{Moreau envelope} of $\rho$ by $\mathsf{M}[\rho](y) = \inf_x \left\{\frac12(y-x)^2 + \rho(x)\right\}$. 
    Repeating the argument of \cite[pg.~14567]{Bean2013}, we have
    \begin{align}
        \mathsf{M}[\rho](y) &= \inf_x \left\{\frac12(y-x)^2 + \rho(x)\right\}= p_2(y) - \sup_x \left\{yx - (\rho(x) + p_2(x))\right\}\nonumber\\
    & = \left(p_2 - (\rho + p_2)^*\right)(y) = m(y),
    \end{align}
    where we use that for any lsc, proper, convex $f$, we have $f^{**} = f$ \cite[Theorem 12.2]{Rockafellar1997ConvexAnalysis}.
    By a fundamental identity for Moreau envelopes (see \cite[pg.~14567]{Bean2013} and references therein), we get 
    \begin{equation}\label{moreau-derivative-to-prox}
    \frac{\mathrm{d}}{\mathrm{d}y}\mathsf{M}[\rho](y) = \mathsf{prox}[\rho^*](y) = y - \mathsf{prox}[\rho](y).
    \end{equation}
    Let $\eta:\reals \rightarrow \reals$ be the Bayes estimator with respect to $\ell_2$-loss in the scalar model $y = \beta_0 + \tau z$ where $\beta_0 \sim \pi$ and $z \sim \mathsf{N}(0,1)$ independent of $\beta_0$. 
    That is $\eta(y) = \E_{\beta_0,z}[\beta_0|y]$.
   Recall $\tau > 0$.
    Thus, by Tweedie's formula (Lemma \ref{lem-tweedies-formula}), $\eta(y) = y - m'(y)$. By comparison with \eqref{moreau-derivative-to-prox}, we conclude
    \begin{equation}\label{bayes-as-prox}
        \eta(y) = \mathsf{prox}[\rho](y).
    \end{equation} 
    
    \noindent {\bf Step 2: Strongly stationary $\tau,\lambda,\gamma=0,\delta,\cT$ with uniformly strongly convex penalty.}
    
    \noindent 
    We have $\mathsf{mmse}_\pi(\tau^2) = \E_{\beta_0,z}[(\eta(y) - \beta_0)^2]$, whence by \eqref{bayes-as-prox} and the assumption of the proposition
    \begin{equation}\label{proximal-bayes-risk}
    \delta \tau^2 - \sigma^2 > \mathsf{mmse}_\pi(\tau^2) = \E_{\beta_0,z}\left[(\mathsf{prox}[\rho](\beta_0 + \tau z) - \beta_0)^2\right].
    \end{equation}
    Observe also that for any $f:\reals\rightarrow \reals$ measurable for which the following expectations exist and are finite,
    \begin{align*}
        \E_{\beta_0,z}\left[f(y)(\E_{\beta_0,z}[\beta_0|y] - \beta_0)\right] &= \E_{\beta_0,z}\left[\E_{\beta_0,z}[f(y)(\E_{\beta_0,z}[\beta_0|y] - \beta_0)\mid y]\right]\\
    &=  \E_{\beta_0,z}\left[f(y)\E_{\beta_0,z}[\E_{\beta_0,z}[\beta_0|y] - \beta_0 \mid y]\right] = 0.
    \end{align*}
    Let $f(y) = y - \mathsf{prox}[\rho](y)$ in the previous display and recall $y = \beta_0 + \tau z$. After rearrangement and using the $\E_{\beta_0,z}[z\beta_0] = 0$,
    \begin{equation}\label{proximal-bayes-width}
        \frac1\tau\E_{\beta_0,z}[z\,\mathsf{prox}[\rho](y)] = \frac{1}{\tau^2}\E_{\beta_0,z}\left[(\mathsf{prox}(y)-\beta_0)^2\right] < \delta - \frac{\sigma^2}{\tau^2} \leq \delta.
    \end{equation}
     
    Now consider $\kappa > 0$ and define
    \begin{equation}\label{bayes-to-rho-perturbation}
    \rho^{(\kappa)}(x) = \rho(x) + \frac\kappa2 x^2.
    \end{equation}
    Then by \eqref{sequence-cvx-estimator}
    \begin{align}
    \mathsf{prox}[\rho^{(\kappa)}](y) &= \arg\min_x \left\{\frac12(y-x)^2 + \rho(x) + \frac\kappa2 x^2\right\} = \arg\min_x \left\{\frac12\left(\frac1{1+\kappa}y-x\right)^2 + \frac1{1+\kappa}\rho(x) \right\}\nonumber \\
    &= \mathsf{prox}\left[\frac1{1+\kappa}\rho\right]\left(\frac1{1+\kappa}y\right).\label{prox-with-ridge-identity}
    \end{align}
    First, we will choose $\kappa > 0$ sufficiently small such that \eqref{proximal-bayes-risk} and \eqref{proximal-bayes-width} still hold with $\rho$ replaced by $\rho^{(\kappa)}$.
    To make our notation more compact, we let $c_\kappa = \frac1{1+\kappa}$.
    Let $a = \mathsf{prox}[\rho](y) - \beta_0$ and $b = \mathsf{prox}\left[\rho^{(\kappa)}\right](y) - \beta_0$.
    We have
    \begin{align}
        |a| &\leq |\mathsf{prox}[\rho](0)| + |\mathsf{prox}[\rho](y) - \mathsf{prox}[\rho](0)| + |\beta_0|\leq |\mathsf{prox}[\rho](0)| + |y| + |\beta_0|,\label{a-perturbed-bayes-bound}\\
    |b| &\leq  |\mathsf{prox}[\rho](0)| + \left|\mathsf{prox}\left[c_\kappa \rho \right](0) - \mathsf{prox}[\rho](0)\right| + \left|\mathsf{prox}\left[c_\kappa \rho\right]\left(c_\kappa y\right) - \mathsf{prox}\left[c_\kappa \rho\right](0)\right| + |\beta_0|\nonumber\\
    &\leq |\mathsf{prox}[\rho](0)| + |\mathsf{prox}[\rho](0)| \left|c_\kappa - 1\right| + \left|c_\kappa y\right| + \beta_0\nonumber\\
    &\leq (c_\kappa + 2)|\mathsf{prox}[\rho](0)| + |c_\kappa y| + |\beta_0|,\label{b-perturbed-bayes-bound}\\
    |a - b| &\leq \left|y - \mathsf{prox}[\rho](y)\right| \left|c_\kappa - 1\right|\leq \left(|y| + |\mathsf{prox}[\rho](y) - \mathsf{prox}[\rho](0)| + |\mathsf{prox}[\rho](0)|\right)|c_\kappa - 1|\nonumber\\
    &\leq \left(|\mathsf{prox}[\rho](0)| + 2|y|\right)|c_\kappa - 1|,\label{a-b-perturbed-bayes-bound}
    \end{align}
    where in \eqref{a-perturbed-bayes-bound}, we have used \eqref{prox-is-lipschitz}, and in both \eqref{b-perturbed-bayes-bound} and  \eqref{a-b-perturbed-bayes-bound}, we have used \eqref{prox-is-lipschitz} and \eqref{prox-continuous-in-lambda}.
    Then by \eqref{a-perturbed-bayes-bound} and \eqref{b-perturbed-bayes-bound}, we have $|a| \vee |b| \leq (c_\kappa + 2) |\mathsf{prox}[\rho](0)| + |y| + |\beta_0|$. Applying this bound, Jensen's inequality, \eqref{norm-squared-difference-inequality}, and \eqref{a-b-perturbed-bayes-bound}, we conclude
    \begin{align}
        &\left|\E_{\beta_0,z}\left[(\mathsf{prox}[\rho](y) - \beta_0)^2\right] - \E_{\beta_0,z}\left[(\mathsf{prox}[\rho^{(\kappa)}](y) - \beta_0)^2\right]\right|  \leq \E_{\beta_0,z}\left[|a^2 - b^2|\right]\nonumber\\
        &\qquad \leq 2\E_{\beta_0,z}\Big[ \Big((c_\kappa + 2) |\mathsf{prox}[\rho](0)| + |y| + |\beta_0|\Big)\Big(|\mathsf{prox}[\rho](0)| + 2|y|\Big)\Big]|c_\kappa - 1| \xrightarrow[\kappa \rightarrow 0]{} 0,\label{bayes-to-perturbed-risk-bound}
    \end{align}
    because $c_\kappa - 1 \rightarrow 0$ as $\kappa \rightarrow 0$, and the expectation is bounded.
    Also, by Jensen's inequality, Cauchy-Schwartz, and \eqref{a-b-perturbed-bayes-bound},
    \begin{align}
        &\left|\frac1\tau\E_{\beta_0,z}\left[ z \mathsf{prox}[\rho](y)\right] - \frac1\tau \E_{\beta_0,z}[z \mathsf{prox}[\rho^{(\kappa)}](y)]\right| \leq \frac1\tau \E_{\beta_0,z}\left[|z( a-b)| \right] \leq \frac1\tau \E_{\beta_0,z}[(a-b)^2]^{1/2}\nonumber\\
    &\qquad\qquad \leq \frac1\tau \E_{\beta_0,z}\big[(|\mathsf{prox}[\rho](0)| + 2|y|)^2\big](c_\kappa - 1)^2 \xrightarrow[\kappa \rightarrow 0]{} 0.\label{bayes-to-perturbed-width-bound}
    \end{align}
    By \eqref{proximal-bayes-risk}, \eqref{proximal-bayes-width}, \eqref{bayes-to-perturbed-risk-bound}, and \eqref{bayes-to-perturbed-width-bound}, we can (and do) choose $\kappa$ sufficiently small that
    \begin{gather}
    \E_{\beta_0,z}\left[(\mathsf{prox}[\rho^{(\kappa)}](y) - \beta_0)^2\right] < \delta \tau^2 - \sigma^2,\label{prox-perturbed-risk}\\
    \frac1\tau\E_{\beta_0,z}\left[z\mathsf{prox}[\rho^{(\kappa)}](y)\right] < \delta.\label{prox-perturbed-width}
    \end{gather}
    
    We now will define an lsc, proper, convex function $\tilde \rho:\reals \rightarrow \reals$ such that \eqref{prox-perturbed-risk} holds with equality and \eqref{prox-perturbed-width} holds with the same strict inequality when $\rho^{(\kappa)}$ is replaced by $\tilde \rho$.
    By \eqref{prox-perturbed-risk}, we may choose $c \in \reals$ such that 
    \begin{equation}\label{shift-prox-to-achieve-risk}
        \E_{\beta_0,z}\left[(\mathsf{prox}[\rho^{(\kappa)}](y) + c - \beta_0)^2\right] = \delta \tau^2 - \sigma^2.
    \end{equation}
    Define the lsc, proper, convex function $\tilde \rho:\reals \rightarrow \reals$ by $\tilde \rho(x) = -cx + \rho^{(\kappa)}(x - c)$. 
    Then by \eqref{prox-shift}, we have $\mathsf{prox}[\tilde \rho](y) = \mathsf{prox}[\rho^{(\kappa)}](y) + c$. Then  \eqref{shift-prox-to-achieve-risk} can be written 
    \begin{gather}\label{rho-stationary-risk}
        \E_{\beta_0,z}\left[\left(\mathsf{prox}\left[\tilde\rho\right](y) - \beta_0\right)^2\right] = \delta \tau^2 - \sigma^2
    \end{gather}
    Further, by \eqref{prox-perturbed-width} and because $\E_{z}[zc] = 0$, we have
    \begin{gather}\label{rho-stationary-width}
    \frac1\tau\E_{\beta_0,z}\left[z\mathsf{prox}[\tilde \rho](y)\right] = \frac1\tau \E_{\beta_0,z}\left[z \mathsf{prox}[\rho^{(\kappa)}](y)\right] < \delta.
    \end{gather}

    Let $\lambda = \frac12\left(1 - \frac1{\delta\tau}\E_{\beta_0,z}\left[z\mathsf{prox}[\tilde \rho](y)\right]\right)^{-1}$, where $\lambda > 0$ by \eqref{rho-stationary-width}. 
    For each $p$, define the lsc, proper, symmetric, convex function $\rho_p: \reals^p \rightarrow \reals \cup \{\infty\}$ by
    \begin{equation}\label{rho-to-rho-p}
    \rho_p(\bx) = \frac1{\lambda p}\sum_{j=1}^p \tilde \rho (\sqrt p x_j).
    \end{equation}
    By \eqref{prox-change-of-scaling} and \eqref{prox-of-separable}, we have for each $1 \leq j \leq p$ and $\by \in \reals^p$ that $\mathsf{prox}[\lambda \rho_p](\by)_j = \frac1{\sqrt p} \mathsf{prox}[\tilde \rho](\sqrt p y_j)$.
    For each $p$, let $\tilde \bbeta_0 \in \reals^p$ be random with coordinates distributed iid from $\pi/\sqrt p$ and $\bz \sim \mathsf{N}(\bzero,\bI_p/p)$.
    These coordinates are denoted $\tilde \beta_{0j}$ and $z_j$.
    As above, $\beta_0,z$ denote independent random variables distributed from $\pi$ and $\mathsf{N}(0,1)$, respectively.
    For any $\tau' \geq 0$ and $\bT' \in S_+^2$, we have 
    \begin{subequations}\label{p-independent-rwk}
    \begin{align}
         \E_{\tilde \bbeta_0,\bz}\left[\left\|\mathsf{prox}\left[\lambda \rho_p\right](\tilde \bbeta_0 + \tau' \bz) - \tilde \bbeta_0\right\|^2\right] &= \frac1p\sum_{j=1}^p \E_{\beta_{0j},z_j} \left[(\mathsf{prox}[ \tilde \rho](\sqrt p (\tilde \beta_{0j} + \tau' z_j)) - \sqrt p\tilde \beta_{0j})^2\right]\nonumber \\
     &= \E_{\beta_0,z}\left[(\mathsf{prox}[ \tilde \rho](\beta_0 + \tau' z) - \beta_0)^2\right],\label{multivariate-prox-perturbed-r}\\
    \frac1\tau\E_{\tilde \bbeta_0,\bz}\left[\langle \bz , \mathsf{prox}[\lambda  \rho_p ](\tilde\bbeta_0 + \tau' \bz)\rangle\right] &= \frac1{\tau p}\sum_{j=1}^p\E_{\tilde \beta_{0j},z_j}\left[\sqrt{p}z_j\mathsf{prox}[  \tilde \rho ](\sqrt p(\tilde\beta_{0j} + \tau' z_j)\right]  \nonumber\\
    &= \frac1\tau \E_{\beta_0,z}\left[z\mathsf{prox}[ \tilde \rho ](\beta_0 + \tau' z)\right], \label{multivariate-prox-perturbed-w}
    \end{align}
    \begin{align}
     &\E_{\tilde\bbeta_0\bz_1,\bz_2}\left[\left\langle\mathsf{prox}\left[\lambda \rho_p\right]\left(\tilde \bbeta_0 + \bz_1\right)-\tilde \bbeta_0,\mathsf{prox}\left[\lambda\rho_p\right]\left(\tilde \bbeta_0 + \bz_2\right)-\tilde \bbeta_0\right\rangle\right]\nonumber\\
     &\qquad\qquad = \frac1p \sum_{j=1}^p \E_{\tilde \beta_{0j},z_{1j},z_{2j}}\left[\left(\mathsf{prox}\left[ \tilde \rho \right]\left(\sqrt{p}(\tilde \beta_{0j} + z_{1j})\right)-\sqrt{p}\tilde \beta_{0j}\right)\left(\mathsf{prox}\left[ \tilde \rho \right]\left(\sqrt{p}(\tilde \beta_{0j} + z_{2j})\right)-\sqrt{p}\tilde \beta_{0j}\right)\right]\nonumber\\
     &\qquad\qquad =  \E_{ \beta_0,z_1,z_2}\left[\left(\mathsf{prox}\left[ \tilde \rho \right]\left( \beta_0 + z_1\right)- \beta_0\right)\left(\mathsf{prox}\left[ \tilde \rho \right]\left( \beta_0 + z_2\right)-\beta_0\right)\right].\label{multivariate-gaussian-perturbed-k}
    \end{align}
    \end{subequations}
    Let $\cT = (\pi,\{\rho_p\})$. 
    We see the limits \eqref{summary-functions-fixed-asymptotic} exist for all $\tau' \geq0$, $\bT' \succeq \bzero$ at the $\lambda$ we have defined.
    By \eqref{rho-stationary-risk}, \eqref{multivariate-prox-perturbed-r}, \eqref{multivariate-prox-perturbed-w}, and the definition of $\lambda$, we see that equations \eqref{fixed-pt-prior-asymptotic} are satisfied at $\tau,\lambda,\gamma=0,\delta,\cT$.
    Thus, $\tau,\lambda,\gamma=0,\delta,\cT$ is strongly stationary.
    
    \noindent {\bf Step 3: Exactly characterize the asymptotic risk.}
    
    \noindent By \eqref{bayes-to-rho-perturbation} and \eqref{rho-to-rho-p}, observe that $\rho_p$ has uniform strong convexity parameter $\kappa > 0$.
    Because $\tau,\lambda,\gamma=0,\delta,\cT$ is strongly stationary, by Proposition \ref{prop-strongly-convex-loss} we have $\|\widehat \bbeta_{\mathsf{cvx}} - \bbeta_0\|^2 \stackrel{\mathrm{p}}\rightarrow \delta\tau^2 - \sigma^2$ where $\widehat \bbeta_{\mathsf{cvx}}$ is defined as with respect to the penalties  \eqref{rho-to-rho-p}. 
    This holds under the HDA and either the RSN or DSN assumptions because the penalties are symmetric.
    
    Thus, by construction, we see that the risk $\delta \tau^2 - \sigma^2$ is achieved on the class $\cC_*$ of uniformly strongly convex sequences of estimators, whence \eqref{minimal-loss-upper-bound-str-cvx} follows.
\end{proof}

\noindent To prove Proposition \ref{prop-achieving-the-bound}.(ii), we will need the following lemma.

\begin{lemma}\label{lem-non-log-concave-gap}
    Consider $\pi \in \cP_2(\reals)$ and $\tau > 0$ such that $\pi * \mathsf{N}(0,\tau^2)$ does not have log-concave density with respect to Lebesgue measure on $\reals$.
    Then 
    \begin{equation}\label{optimal-prox-strict-inequality}
        \mathsf{R^{opt}_{seq,cvx}} (\tau;\pi) > \mathsf{mmse}_\pi(\tau^2).
    \end{equation}
\end{lemma}  

\begin{proof}[Proof of Lemma \ref{lem-non-log-concave-gap}]
    Throughout this proof, we will let $Y \sim \pi * \mathsf{N}(0,\tau^2)$ be a random variable.
    Because $\tau > 0$, $\pi * \mathsf{N}(0,\tau^2)$ has density with respect to Lebesgue measure which is infinitely continuously differentiable.
    Call this density $p_Y$.
    Let $\eta:\reals \rightarrow \reals$ be the Bayes estimator of $\beta_0$ given observation $\beta_0 + \tau z$ where $\beta_0 \sim \pi$ and $z \sim \mathsf{N}(0,1)$ independent of $\beta_0$.
    By Tweedie's formula (Lemma \ref{lem-tweedies-formula}),
    \begin{equation}
    \eta(y) = y + \tau^2 \frac{\mathrm{d}}{\mathrm{d}y}\log p_Y(y).
    \end{equation}
    Because $\pi *\mathsf{N}(0,\tau^2)$ is not log-concave and has infinitely continuously differentiable density, there exists $v \in \reals$ and $\xi,\varepsilon > 0$ such that $\frac{\mathrm{d}^2}{\mathrm{d}y^2}\log p_Y(y) > \xi/\tau^2$ on $[v - 2\eps,v+2\eps]$. Thus,
    $$
    \eta'(y) > 1 + \xi \quad \text{ on } \quad [v - 2\eps,v+2\eps].
    $$
    Then, for any 1-Lipschitz function $\eta_{\mathrm{Lip}}:\reals \rightarrow \reals$,
    either $|\eta(y) - \eta_{\mathrm{Lip}}(y)| \geq \xi \varepsilon$ for $y \in [v+\varepsilon,v+2\varepsilon]$ (if $\eta_{\mathrm{Lip}}(v) \leq \eta(v)$),
    or $|\eta(y) - \eta_{\mathrm{Lip}}(y)| \geq \xi \varepsilon$ for $y \in [v-2\varepsilon,v-\varepsilon]$ (if $\eta_{\mathrm{Lip}}(v) \geq \eta(v)$).
    Thus, for any 1-Lipshictz function,
    \begin{equation}\label{eq:lipsch-gap}
        \E_{Y\sim p_Y}[(\eta_{\mathrm{Lip}}(Y) - \eta(Y))^2] \geq \xi \varepsilon \min\{p_Y([v+\varepsilon,v+2\varepsilon]),p_Y([v-2\varepsilon,v-\varepsilon])\} =: \Delta > 0.
    \end{equation}

    Consider $\bbeta_0\in \reals^p$ with coordinates distributed iid from $\pi/\sqrt p$ and $\bz \sim \mathsf{N}(0,\bI_p/p)$ independent of $\bbeta_0$. 
    Let $\by = \bbeta_0 + \tau \bz$. 
    Clearly, $\sqrt p \by$ has coordinates distributed iid from $\pi * \mathsf{N}(0,\tau^2)$.
    We define the application of $\eta$ to a vector by
    $
    \eta(\by)_j = \frac1{\sqrt p}\eta(\sqrt p y_j).
    $
    This agrees with  \eqref{tweedies-formula} when $p = 1$, so no confusion should result.
    Observe that 
    $
    \eta(\by) = \E_{\bbeta_0,\bz}[\bbeta_0 | \by].
    $
    Because $\mathsf{prox}[\rho](\by)_j - \eta(y_j)$ is uncorrelated with $\eta(y_j) - \beta_{0j}$ conditional on $\by_{-j},\bbeta_{0,-j}$ (where these denote the coordinates of $\by,\bbeta_0$ excluding coordinate $j$), we have
    \begin{align}\label{risk-to-bayes-comparison}
    \E_{\bbeta_0,\bz}\left[(\mathsf{prox}[\rho](\by)_j - \beta_{0j})^2|\by_{-j},\bbeta_{0,-j}\right] &= \E_{\bbeta_0,\bz}\left[(\mathsf{prox}[\rho](\by)_j - \eta(y_j))^2|\by_{-j},\bbeta_{0,-j}\right]\nonumber \\
    &\qquad\qquad\qquad\qquad + \E_{\bbeta_0,\bz}\left[(\eta(y_j) - \beta_{0j})^2|\by_{-j},\bbeta_{0,-j}\right]\nonumber \\
    &=\E_{\bbeta_0,\bz}\left[(\mathsf{prox}[\rho](\by)_j - \eta(y_j))^2|\by_{-j},\bbeta_{0,-j}\right] + \mathsf{mmse}_\pi(\tau^2)/p.
    \end{align}
    
    For any lsc, proper, convex $\rho:\reals^p \rightarrow \reals \cup \{\infty\}$,
    fixing $\by_{-j}$ the function $y_j \mapsto \mathsf{prox}[\rho](\by)$ is 1-Lipschitz, 
    whence in fact $y_j \mapsto \mathsf{prox}[\rho](\by)_j$ is 1-Lipschitz.
    Then by \eqref{eq:lipsch-gap}, 
    $$
        \E_{\bbeta_0,\bz}\left[(\mathsf{prox}[\rho](\by)_j - \eta(y_j))^2|\by_{-j},\bbeta_{0,-j}\right] \geq \Delta/p, \;\; \text{almost surely.}
    $$ 
    We conclude
    \begin{equation*}
        \E_{\bbeta_0,\bz}[\|\mathsf{prox}[\rho](\by) - \bbeta_0\|^2]
        =
        \sum_{j=1}^p \E_{\bbeta_0,\bz}[\E_{\bbeta_0,\bz}[(\mathsf{prox}[\rho](\by)_j - \beta_{0j})^2|\by_{-j},\bbeta_{0,-j}]]
        \geq 
        \mathsf{mmse}_\pi(\tau^2) + \Delta.
    \end{equation*}
    The proof is complete.
\end{proof}

\noindent We are ready to prove the second part of Proposition \ref{prop-achieving-the-bound}.

\begin{proof}[Proof of Proposition \ref{prop-achieving-the-bound}.(ii)]
    By Lemma \ref{lem-non-log-concave-gap}, we have  $\mathsf{R^{opt}_{seq,cvx}} (\tau;\pi) > \mathsf{mmse}_\pi(\tau^2)$. 
    By assumption, $\mathsf{mmse}_\pi(\tau^2) \geq \delta \tau^2 - \sigma^2$. Thus, $\mathsf{R^{opt}_{seq,cvx}} (\tau;\pi) > \delta \tau^2 - \sigma^2$.
    By Lemma \ref{lem-r-continuity}, the left and right-hand sides are continuous in $\tau$, so that there exists $\tau' > \tau$ with $\mathsf{R^{opt}_{seq,cvx}} (\tau';\pi) > \delta {\tau'}^2 - \sigma^2$. Then by \eqref{eqdef-tau-reg-cvx}, $\tau_{\mathsf{reg,cvx}} \geq \tau' > \tau$.
    Proposition \ref{prop-achieving-the-bound}(ii) then follows from Theorem \ref{thm-cvx-lower-bound}.
\end{proof}

\noindent Finally, we prove the third part of Proposition \ref{prop-achieving-the-bound}.

\begin{proof}[Proof of Proposition \ref{prop-achieving-the-bound}.(iii)]
    If $\pi * \normal(0,\tau^2_{\mathsf{reg,amp}*})$ has log concave density,
    so too does $\pi * \normal(0,\tau^2)$ for all $\tau > \tau_{\mathsf{reg,amp}*}$.
    By the definition of $\tau_{\mathsf{reg,amp}*}$ (Eq.\ \eqref{alg-bound-sup-to-inf-form-def}), we have 
    $\delta \tau^2 - \sigma^2 > \mathsf{mmse}_\pi(\tau^2)$ for all such $\tau$. 
    Then by Proposition \ref{prop-achieving-the-bound}(i), Theorem \ref{thm-cvx-lower-bound}, and the fact that $\cC_* \subset \cC_{\delta,\pi}$,
    we have
    \begin{equation*}
        \delta\tau^2 - \sigma^2 
            \geq 
            \inf_{\{\rho_p\} \in \cC_*}  \lim_{p \rightarrow \infty}^{\mathrm p} \| \widehat \bbeta_{\mathsf{cvx}} - \bbeta_0\|^2 
            \geq
            \inf_{\{\rho_p\} \in \cC_{\delta,\pi}}  \lim_{p \rightarrow \infty}^{\mathrm p} \| \widehat \bbeta_{\mathsf{cvx}} - \bbeta_0\|^2 
            \geq 
            \delta \tau_{\mathsf{cvx,reg}}^2 - \sigma^2.
    \end{equation*}
    Taking $\tau \downarrow \tau_{\mathsf{reg,amp}*}$ gives $\tau^2_{\mathsf{reg,amp}*} = \tau_{\mathsf{reg,cvx}}^2$.

    If $\pi * \normal(0,\tau^2_{\mathsf{reg,amp}*})$ does not have log concave density,
    then because $\delta \tau^2_{\mathsf{reg,amp}*} - \sigma^2 = \mathsf{mmse}_\pi(\tau^2_{\mathsf{reg,amp}*})$, we have $\tau_{\mathsf{reg,cvx}}^2 > \tau^2_{\mathsf{reg,amp}*}$ by Proposition \ref{prop-achieving-the-bound}.

    The argument for $\tau_{\mathsf{reg,stat}}^2$ is completely analogous.
\end{proof}

\section{Connection with the random signal and noise model}\label{app-proofs-of-lemmas-DSN-to-RSN}

In this appendix we state and prove two lemmas that provide explicit connection between the deterministic and random signal and noise models.
The first lemma will allow us to extend lower bounds on the $\liminf$ of sequences of estimation errors.
\begin{lemma}\label{lem-DSN-to-RSN-lower-bound}
    Fix $\pi \in \cP_2(\reals)$, $\delta \in (0,\infty)$, and $\sigma \geq 0$. 
    Consider any sequence of estimators $\{\widehat \bbeta\}$ (ie.\ measurable functions of $\by,\bX$ and potentially some auxiliary noise). 
    Assume that the HDA and DSN assumptions imply that for some constant $c$ we have 
    \begin{equation}\label{generic-probabilistic-bound}
    \liminf_{p \rightarrow \infty}^{\mathrm p} \|\widehat \bbeta - \bbeta_0\|^2 \geq c.
    \end{equation}
    Then under the HDA and RSN assumption (where the randomness in $\bbeta_0$ and $\bw$ is independent of the auxiliary noise used to construction $\widehat \bbeta$), we have 
    \begin{equation}\label{generic-expectation-bound}
    \liminf_{p \rightarrow \infty} \E_{\bbeta_0,\bw,\bX}\left[\|\widehat \bbeta - \bbeta\|^2\right] \geq c.
    \end{equation}
\end{lemma}
\begin{proof}[Proof of Lemma \ref{lem-DSN-to-RSN-lower-bound}]
    By \cite[Lemma 8.4]{Bickel1981SomeBootstrap}, if $\beta_{0j} \stackrel{\mathrm{iid}}\sim \pi/\sqrt p$ for $\pi \in \cP_2(\reals)$, then
    \begin{equation}
        d_{\mathrm{W}}(\widehat \pi_{\bbeta_0},\pi) \stackrel{\mathrm{as}}\rightarrow 0,
    \end{equation}
    where $\widehat \pi_{\bbeta_0}$ is as in \eqref{DSN-assumption}.
    Further, under assumption RSN, by the strong law of large numbers, $\frac1n\|\bw\|^2 \stackrel{\mathsf{as}}\rightarrow \sigma^2$. 
    Thus, under the RSN assumption the sequences $\{\bbeta_0\}$, $\{\bw\}$ satisfy the DSN assumption with probability 1.
    Thus, if \eqref{generic-probabilistic-bound} holds under the DSN assumption, we have under the RSN assumption that for all $\eps > 0$
    \begin{equation}
         \P_{\bbeta_0,\bw,\bX}\left( \|\widehat \bbeta - \bbeta_0\|^2 > c - \eps\Bigm\vert\bbeta_0,\bw\right) \xrightarrow[p\rightarrow\infty]{\mathrm{as}} 1.
    \end{equation}
    Observe by bounded convergence
    \begin{align*}
        \E_{\bbeta_0,\bw}[\|\widehat \bbeta - \bbeta\|^2] &= \E_{\bbeta_0,\bw,\bX}\left[\E_{\bbeta_0,\bw,\bX}\left[\|\widehat \bbeta - \bbeta\|^2 \Bigm\vert \bbeta_0,\bw\right]\right]\\
        &\geq (c-\eps)\E_{\bbeta_0,\bw,\bX}\left[\P_{\bbeta_0,\bw,\bX}\left(\|\widehat \bbeta - \bbeta\|^2 > c-\eps \Bigm\vert \bbeta_0,\bw\right)\right] \rightarrow c-\eps.
    \end{align*}
    Taking $\eps \downarrow 0$ gives \eqref{generic-expectation-bound}.
\end{proof}

\noindent Observe that Lemma \ref{lem-DSN-to-RSN-lower-bound} applies to any sequence of estimators $\{\widehat \bbeta\}$ defined in any way. 
In particular, the estimators need not be defined via convex M-estimation. 
The second lemma will allow us to extend the exact loss characterization of Proposition \ref{prop-strongly-convex-loss}.
\begin{lemma}\label{lem-strongly-convex-DSN-to-RSN-risk}
    Fix $\pi \in \cP_2(\reals)$, $\delta \in (0,\infty)$, and $\sigma \geq 0$. 
    Consider a sequence $\{\rho_p\} \in \cC_*$ and the corresponding M-estimators \eqref{linear-cvx-estimator} (which always exist and are unique by strong convexity).
    Assume that the HDA and DSN assumptions imply that for some constant $c$ we have
    \begin{equation}\label{convex-probabilistic-convergence}
    \|\widehat \bbeta_{\mathsf{cvx}} - \bbeta_0\|^2 \stackrel{\mathrm{p}}\rightarrow c.
    \end{equation}
    Then under the HDA and RSN assumption
    \begin{equation}\label{convex-average-convergence}
    \lim_{p \rightarrow \infty} \E_{\bbeta_0,\bw,\bX}\left[\|\widehat \bbeta_{\mathsf{cvx}} - \bbeta_0\|^2\right] = c.
    \end{equation}
\end{lemma}
\begin{proof}[Proof of Lemma \ref{lem-strongly-convex-DSN-to-RSN-risk}]
    Let $\gamma > 0$ be such that $\rho_p$ is strongly convex with parameter $\gamma$ for all $p$.
    Because $\rho_p$ is strongly convex, it has a unique minimizer, which we will denote by $\bm_p$. 
    First we show that $\|\bm_p\|$ is bounded in $p$.
    Without loss of generality, we may assume $\rho_p(\bm_p) = 0$ for all $p$. Thus, $\rho_p(\bbeta) \geq \frac{\gamma}2 \|\bbeta - \bm_p\|^2$ for all $p$ and all $\bbeta \in \reals^p$.
    By \eqref{linear-cvx-estimator},
    \begin{align*}
        \frac1n\|\by - \bX \widehat \bbeta_{\mathsf{cvx}}\|^2 + \frac\gamma2 \|\widehat \bbeta_{\mathsf{cvx}} - \bm_p\|^2 \leq \frac1n \|\by - \bX \widehat \bbeta_{\mathsf{cvx}}\|^2 + \rho_p(\widehat \bbeta_{\mathsf{cvx}}) \leq \frac1n \|\by - \bX \bm_p\|^2.
    \end{align*}
    By optimality, we have that $\frac2n\bX^\mathsf{T}(\by - \bX \widehat \bbeta_{\mathsf{cvx}}) \in \partial \rho_p(\widehat \bbeta_{\mathsf{cvx}})$.
    Thus,
    \begin{align*}
        \rho_p(\widehat \bbeta_{\mathsf{cvx}}) \geq \rho_p(\bm_p) &\geq \rho_p(\widehat \bbeta_{\mathsf{cvx}}) + \frac2n(\by - \bX \widehat \bbeta_{\mathsf{cvx}})^\mathsf{T}\bX(\bm_p - \widehat \bbeta_{\mathsf{cvx}}) + \frac\gamma2 \|\bm_p - \widehat \bbeta_{\mathsf{cvx}}\|^2\\
        &\geq \rho_p(\widehat \bbeta_{\mathsf{cvx}}) - \frac2n \|\by - \bX \widehat \bbeta_{\mathsf{cvx}}\|\|\bX\|_{\mathsf{op}}\|\bm_p - \widehat \bbeta_{\mathsf{cvx}}\| +  \frac\gamma2 \|\bm_p - \widehat \bbeta_{\mathsf{cvx}}\|^2.
    \end{align*}
    Also,
    $$
    \|\by - \bX \widehat \bbeta_{\mathsf{cvx}}\| \leq \|\by - \bX \bbeta_0\| + \|\bX(\widehat \bbeta_{\mathsf{cvx}} - \bbeta_0)\| \leq \|\bw\| +\|\bX\|_{\mathsf{op}}\|\widehat \bbeta_{\mathsf{cvx}} - \bbeta_0\|.
    $$
    Combining the previous two displays,
    $$
    \|\bm_p - \widehat \bbeta_{\mathsf{cvx}}\| \leq \frac4\gamma \frac{\|\bX\|_{\mathsf{op}}}{\sqrt n} \left(\frac{\|\bw\|}{\sqrt n} + \frac{\|\bX\|_{\mathsf{op}}}{\sqrt n} \|\widehat \bbeta_{\mathsf{cvx}} - \bbeta_0\|\right).
    $$
    In particular,
    \begin{align*}
        \|\bm_p\| &\leq \|\bbeta_0\| + \|\widehat \bbeta_{\mathsf{cvx}} - \bbeta_0\| + \|\bm_p - \widehat \bbeta_{\mathsf{cvx}}\|\\
        &\leq  \|\bbeta_0\| + \|\widehat \bbeta_{\mathsf{cvx}} - \bbeta_0\| + \frac4\gamma \frac{\|\bX\|_{\mathsf{op}}}{\sqrt n} \left(\frac{\|\bw\|}{\sqrt n} + \frac{\|\bX\|_{\mathsf{op}}}{\sqrt n} \|\widehat \bbeta_{\mathsf{cvx}} - \bbeta_0\|\right).
    \end{align*}
    The random variable $\|\bX\|_{\mathsf{op}} /\sqrt n$ is tight by \cite{Anderson2010}, the random variable $\|\bw \|/\sqrt n$ is tight by the law of large numbers, and $\|\widehat \bbeta_{\mathsf{cvx}} - \bbeta_0\|$ is tight under the DSN assumption by assumption.
    Because $\|\bm_p\|$ is deterministic, it must be bounded in $p$. Let $M$ be such that $\|\bm_p\| \leq M$ for all $p$.

    Now we turn to proving \eqref{convex-average-convergence} under the RSN assumption.
    As in the proof of Lemma \ref{lem-DSN-to-RSN-lower-bound}, we have that the sequences $\{\bbeta_0\}, \{\bw\}$ satisfy the DSN assumption with probability 1.
    Thus, we have \eqref{convex-probabilistic-convergence}. By Vitali's convergence theorem, we only need to verify that $\|\widehat \bbeta_{\mathsf{cvx}} - \bbeta_0\|^2$ is uniformly integrable over $p$ \cite[Theorem 16.14]{Billingsley2012ProbabilityMeasure}.
    Observe that for any $\bbeta \in \reals^p$,
    \begin{align*}
        \frac1n \|\by - \bX \widehat \bbeta_{\mathsf{cvx}}\|^2 + \rho_p(\widehat \bbeta_{\mathsf{cvx}}) &\geq \frac1n \|\by\|^2 - \frac2n \by^\mathsf{T}\bX \widehat \bbeta_{\mathsf{cvx}} + \frac\gamma2 \|\widehat \bbeta_{\mathsf{cvx}} - \bm_p\|^2\\
        &\geq \frac1n \|\by\|^2 - 2\frac{\|\bX^\mathsf{T}\by\|}{n} \|\widehat \bbeta_{\mathsf{cvx}}\| + \frac\gamma2 \|\widehat \bbeta_{\mathsf{cvx}}\|^2  - \gamma M \|\widehat \bbeta_{\mathsf{cvx}}\|\\
        &\geq \frac1n\|\by\|^2 + \frac\gamma4\|\widehat \bbeta_{\mathsf{cvx}}\|^2 - \frac1\gamma \left(2\frac{\|\bX^\mathsf{T}\by\|}n + \gamma M\right)^2\\
        &\geq \frac1n\|\by\|^2 + \frac\gamma4\|\widehat \bbeta_{\mathsf{cvx}}\|^2 - \frac{8\|\bX^\mathsf{T}\by\|^2}{\gamma n^2} - 2\gamma M^2.
    \end{align*}
    Further, recalling $\rho_p(\bm_p) = 0$, by \eqref{linear-cvx-estimator} and the triangle inequality
    \begin{align*}
        \frac1n\|\by - \bX \bm_p\|^2 + \rho_p(\bm_p) &=  \frac1n\|\by - \bX \bm_p\|^2 \leq \frac2n \|\by - \bX \bbeta_0\|^2 + \frac2n\|\bX(\bm_p - \bbeta_0)\|^2\\
        &= \frac2n \|\bw\|^2 + \frac2n \|\bX(\bm_p - \bbeta_0)\|^2.
    \end{align*}
    Combining the previous two displays, we get
    \begin{align}
        \|\widehat \bbeta_{\mathsf{cvx}}\|^2 &\leq \frac4\gamma \left(\frac1n \|\by - \bX \widehat \bbeta_{\mathsf{cvx}}\|^2 + \rho(\widehat \bbeta_{\mathsf{cvx}})  - \frac1n\|\by\|^2  +  \frac{8\|\bX^\mathsf{T}\by\|^2}{\gamma n^2} + 2 \gamma M^2\right)\nonumber\\
        &\leq  \frac4\gamma \left(\frac1n \|\by - \bX \bm_p\|^2 + \rho(\bm_p)  +  \frac{8\|\bX^\mathsf{T}\by\|^2}{\gamma n^2} + 2 \gamma M^2\right)\nonumber\\
        &\leq \frac4\gamma\left(\frac2n\|\bw\|^2+ \frac2n\|\bX (\bm_p - \bbeta_0)\|^2  + \frac8\gamma \left(\frac{\|\bX^\mathsf{T}\bw\|}n + \frac{\|\bX^\mathsf{T}\bX\bbeta_0\|}n\right)^2 + 2\gamma M^2\right)\nonumber\\
        &\leq   \frac4\gamma \left(\frac2n \|\bw\|^2 +  \frac4n \|\bX\bm_p\|^2+ \frac4n\|\bX\bbeta_0\|^2  +  \frac{16}{\gamma n^2}\|\bX^\mathsf{T}\bw\|^2 + \frac{16}{\gamma n^2}\|\bX^\mathsf{T}\bX\bbeta_0\|^2 + 2 \gamma M^2\right).\label{beta-hat-bound}
    \end{align}
    We show the right-hand side is uniformly integrable one term at a time.
    First, we recall two well-known facts about uniform integrability, which we state without proof.
    \begin{claim}\label{ui-of-empirical-mean}
        If the collection (over $j$) $\{A_j\}$ is uniformly integrable, then the collection (over $p$) $\left\{\frac1p\sum_{i=1}^p A_i\right\}_p$ is uniformly integrable.
    \end{claim}
    \begin{claim}\label{ui-of-independent-products}
        If $\{A_p\}$ and $\{B_p\}$ are uniformly integrable and for each $p$ the random variables $A_p$ and $B_p$ are defined on the same probability space and are independent, then $\{A_pB_p\}$ are uniformly integrable.
    \end{claim}
    \noindent First, the $\frac2n \|\bw\|^2$ are uniformly integrable by Claim \ref{ui-of-empirical-mean} because the $w_j^2$ are integrable from the same distribution. 
    Second, $\frac4n \|\bX\bm_p\|^2 \sim 4\|\bm_p\|^2\chi^2_n/n \stackrel{\mathrm{d}}= \frac{4\|\bm_p\|^2}n \sum_{i=1}^n Z_i^2$, where $Z_i \stackrel{\mathrm{iid}}\sim \mathsf{N}(0,1)$. 
    By Claim \ref{ui-of-empirical-mean}, the $\frac1n \sum_{i=1}^n Z_i^2 $ are uniformly integrable, and because the $\|\bm_p\|^2$ are bounded, the $\frac{4\|\bm_p\|^2}n \sum_{i=1}^n Z_i^2$, and hence the $\frac4n \|\bX\bm_p\|^2$, are the uniformly integrable by Claim \ref{ui-of-independent-products}.
    Third, $\frac4n \|\bX\bbeta_0\|^2 = \frac4n \sum_{i=1}^n \left(\sum_{j=1}^p X_{ij}\beta_{0j}\right)^2$. Observe that conditional on $\bbeta_0$, the random variable $\sum_{j=1}^p X_{ij}\beta_{0j}$ has distribution $\mathsf{N}(0,\|\bbeta_0\|^2)$, so that $\left(\sum_{j=1}^p X_{ij}\beta_{0j}\right)^2 \stackrel{\mathrm{d}}= Z^2\|\bbeta_0\|^2$ for $Z \sim \mathsf{N}(0,1)$ independent of $\bbeta_0$.
    Observe that the $\|\bbeta_0\|^2 = \frac1p \sum_{j=1}^p (\sqrt p \beta_{0j})^2$ are uniformly integrable (over $p$) by Claim \ref{ui-of-empirical-mean} because $\sqrt p \beta_{0j} \sim \pi \in \cP_2(\reals)$ for all $p$.
    Then, by Claim \ref{ui-of-independent-products}, the $Z^2\|\bbeta_0\|^2$, and hence the $\left(\sum_{j=1}^p X_{ij}\beta_{0j}\right)^2 $, are uniformly integrable. 
    Then, by Claim \ref{ui-of-empirical-mean},  the $\frac4n\|\bX\bbeta_0\|^2$ are uniformly integrable.
    Fourth, the $\frac{16}{\gamma n^2}\|\bX^\mathsf{T}\bw\|^2$ are uniformly integrable by the same argument (just replace $\bbeta_0$ with $\bw/\sqrt{n}$ and switch $i,j$ and $n,p$).
    Fifth, and lastly, we show the $\frac{16}{\gamma n^2}\|\bX^\mathsf{T}\bX\bbeta_0\|^2 $ are uniformly integrable.
    We have
    \begin{align*}
        \frac{\|\bX^\mathsf{T} \bX \bbeta_0\|^2}{n^2} & = \frac1n\sum_{j=1}^p \left(\frac1{\sqrt n}\sum_{i=1}^n X_{ij} [\bX \bbeta_0]_i\right)^2  = \frac1n\sum_{j=1}^p \left(\frac1{\sqrt n} \sum_{i=1}^nX_{ij}^2 \beta_{0j} + \frac1{\sqrt n} \sum_{i=1}^n \sum_{l \neq j}^pX_{ij} X_{il}\beta_{0l}\right)^2\\
    &\leq 2   \underbrace{\sum_{j=1}^p \left(\frac1n \sum_{i=1}^n X_{ij}^2 \beta_{0j}\right)^2}_{:= a} + \underbrace{\frac2n \sum_{j=1}^p \left( \frac1{\sqrt n} \sum_{i=1}^n \sum_{l \neq j}^pX_{ij} X_{il}\beta_{0l}\right)^2}_{:= b}.
    \end{align*}
    We write $a$ as $\frac2{n^2p} \sum_{j=1}^p\sum_{i_1,i_2=1}^n X_{i_1j}^2X_{i_2j}^2(\sqrt p \beta_{0j})^2$, which are uniformly integrable by Claim \ref{ui-of-empirical-mean} because the $X_{i_1j}^2X_{i_2j}^2(\sqrt p \beta_{0j})^2$ are integrable and have one of only two possible distributions (depending on whether $i_1 = i_2$ or $i_1 \neq i_2$) which do not depend on $n,p$.
    Now we consider $b$.
    We denote the columns of $\bX$ by $\bX_j$. 
    Observe that conditional on $\bX_j,\bbeta_0$, we have $\frac1{\sqrt n} \sum_{i=1}^n \sum_{l \neq j}^pX_{ij} X_{il}\beta_{0l} \sim \mathsf{N}\left(0,\left(\sum_{l \neq j}^p \beta_{0l}^2\right) \frac{\|\bX_j\|^2}n\right)$, so that in fact, $\left( \frac1{\sqrt n} \sum_{i=1}^n \sum_{l \neq j}^pX_{ij} X_{il}\beta_{0l}\right)^2 \stackrel{\mathrm{d}}= Z^2\left(\sum_{l \neq j}^p \beta_{0l}^2\right) \frac{\|\bX_j\|^2}n$ for $Z \sim \mathsf{N}(0,1)$ independent of $\bbeta_0,\bX$.
    Observe that the $\left(\sum_{l \neq j}^p \beta_{0l}^2\right)$ are uniformly integrable because they are dominated by $\|\bbeta_0\|^2$, whose uniform integrability we already established. Further, the $\frac{\|\bX_j\|^2}n = \frac1n \sum_{i=1}^n X_{ij}^2$ are uniformly integrable by Claim \ref{ui-of-empirical-mean}.
    Then the $Z^2\left(\sum_{l \neq j}^p \beta_{0l}^2\right) \frac{\|\bX_j\|^2}n$ are uniformly integrable by two applications of Claim \ref{ui-of-independent-products}, because they are the product of three independent and uniformly integrable terms.
    Thus, the $b$'s are uniformly integrable by Claim \ref{ui-of-empirical-mean}, and the $\frac{\|\bX^\mathsf{T} \bX \bbeta_0\|^2}{n^2} $ are uniformly integrable by the uniform integrability of the $a$'s and $b$'s.
    We conclude that the right-hand side of \eqref{beta-hat-bound} is uniformly integrable, whence the $\|\widehat \bbeta_{\mathsf{cvx}}\|^2$ are uniformly integrable.
    
    Because
    \begin{align*}
        \|\widehat \bbeta_{\mathsf{cvx}} - \bbeta_0\|^2 \leq 2\left(\|\widehat \bbeta_{\mathsf{cvx}}\|^2 + \|\bbeta_0\|^2\right),
    \end{align*}
    and $\|\bbeta_0\|^2$ are uniformly integrable,
    we also have the $\|\widehat \bbeta_{\mathsf{cvx}} - \bbeta_0\|^2$ are uniformly integrable, completing the proof.
\end{proof}

\section{Proof of Proposition \ref{prop-bayes-risk-asymptotic}, Proposition \ref{prop-bAMP-achieves-alg-bound}, and equivalence of $\tau_{\mathsf{reg,amp}*}$ and $\tau_{\mathsf{reg,amp}}$}\label{app-proof-of-prop-bayes-risk-asymptotic-and-bAMP-achives-alg-bound}

\begin{proof}[Proof of Proposition \ref{prop-bayes-risk-asymptotic}]\label{proof-of-prop-bayes-risk-asymptotic}
Model \cite[eq.~(1)]{barbier2018optimal} is equivalent to our model \eqref{linear-model} under the following change of variable (with the notation of \cite{barbier2018optimal} on the left).
    \begin{align*}
    \bPhi &\leftarrow \bX, & \bX^*  &\leftarrow \sqrt p \bbeta_0, & Y_\mu &\leftarrow y_i, & A_\mu  &\leftarrow w_i,
    &(m,n)  &\leftarrow (n,p),\\
    \varphi(x,a) &= x + a, & P_0 &\leftarrow \pi, & \alpha &\leftarrow \delta, & r &\leftarrow 1/\tau^2, & \rho &\leftarrow s_2(\pi),
    \end{align*}
    where we have used an equal sign for any quantity which we do not have our own notation for.
    The authors of \cite{barbier2018optimal} denote by $X_0,Z_0$ independent random scalars distributed from $P_0$ and $\mathsf{N}(0,1)$ respectively. In our notation, we denote by $\beta_0,z$ independent random scalars distributed from $\pi$ and $\mathsf{N}(0,1)$ respectively. 
    We will also denote the random scalar $y = \beta_0/\tau + z$.
    To avoid clutter, we will write $s_2$ in place of $s_2(\pi)$ for the remainder of the proof.
    The authors of \cite{barbier2018optimal} define in Eq.~(5) (where we have already converted to our notation)
    \begin{align}
    \psi_{\pi}(1/\tau^2) &= \E_{\beta_0,z} \log \int  e^{\beta_0\beta/\tau^2 + \beta z/\tau - \beta^2/2\tau^2} \pi(\mathrm{d}\beta) = \E_{\beta_0,z} \log \left( e^{\frac12(\beta_0/\tau + z)^2} \int e^{-\frac 12(\beta_0/\tau + z -\beta/\tau )^2}\right) \pi(\mathrm{d}\beta)\nonumber\\
    &=  \frac{s_2}{2\tau^2} + \frac12  + \E_{\beta_0,z}  \int e^{-\frac 12(y -\beta/\tau )^2}\pi(\mathrm{d} \beta) 
    = \frac{s_2}{2\tau^2} - i(\tau^2),\label{psi-pi}
    \end{align}
    where the last line follows by \eqref{eqdef-i}.
    Their $P_{\mathsf{out}}\left(Y_\mu\big|\frac1{\sqrt n}\left[\bPhi \bX^*\right]_{\mu}\right)$ is the conditional density (w.r.t. Lebesgue measure) of $Y_\mu \bigm\vert \frac1{\sqrt n}\left[\bPhi \bX^*\right]_{\mu}$ (in their notation), which in our notation is $P_{\mathsf{out}}\left(y|x\right) = \frac1{\sqrt{2\pi}\sigma}\exp\left(-\frac1{2\sigma^2}(y-x)^2\right)$. 
    The authors of \cite{barbier2018optimal} denote by $V,W$ independent random scalars distributed from $\mathsf{N}(0,1)$ and by $\tilde Y_0$ a random scalar distributed from $P_{\mathsf{out}}(\cdot|\sqrt{q} V + \sqrt{s_2 - q}W)$. 
    In our notation and with our choice of $P_{\mathsf{out}}$, we denote by $z_1,z_2$ independent random scalars distributed from $\mathsf{N}(0,1)$ (corresponding to $V,W$ respectively) and observe that $\sqrt q z_1 + \sqrt{s_2 - q}z_2 + \sigma z_3 \sim P_{\mathsf{out}}(\cdot|\sqrt{q} z_1 + \sqrt{s_2 - q}z_2)$ where $z_3 \sim \mathsf{N}(0,1)$ independent of $z_1,z_2$.
    The authors of \cite{barbier2018optimal}  define in Eq.\ (6) (where we have already converted to our notation)
    \begin{align}
    \Psi_{P_{\mathsf{out}}}(q;s_2) &= \E_{z_1,z_2,z_3}\log \int \frac1{\sqrt{2\pi}}e^{-\frac12 w^2}P_{\mathsf{out}}\left(\sqrt{q}z_1 + \sqrt{s_2 - q}z_2 + \sigma z_3| \sqrt{q}z_1 + \sqrt{s_2 - q}w\right)\mathrm{d}w\nonumber\\
    &= \E_{z_1,z_2,z_3}\log \int \frac1{\sqrt{2\pi}}e^{-\frac12 w^2}\frac1{\sqrt{2\pi}\sigma}\exp\left(-\frac1{2\sigma^2}\left(\sqrt{s_2 - q}z_2 + \sigma z_3 - \sqrt{s_2-q}w\right)^2\right) \mathrm{d}w\nonumber\\ 
    &= \E_{z_2,z_3}\log\left(\frac1{\sqrt{2\pi(\sigma^2 + s_2 - q)}} \exp\left(-\frac1{2(\sigma^2 + s_2 - q)} \left(\sqrt{s_2 - q}z_2 + \sigma z_3 \right)^2\right)\right)\nonumber\\
    &= - \frac12\log 2\pi - \frac12 \log(\sigma^2 + s_2 - q)  - \frac12 .\label{psi-p-out}
    \end{align}
    The authors of \cite{barbier2018optimal} define in Eq.\ (4)
    \begin{equation}\label{f-rs}
    f_{\mathsf{RS}}(q,1/\tau^2;s_2) = \psi_{\pi}(1/\tau^2) + \delta\Psi_{P_{\mathsf{out}}}(q;s_2) - \frac{q}{2\tau^2}.
    \end{equation}
    In Eq.\ (3) they define a parameter $q^*$ via a variational formula which in Theorem 1 they show to be equivalent to defining $q^*$ as the first coordinate of
    \begin{equation}\label{free-energy-maximization}
    \arg\max_{(q,\tau)\in\Gamma} f_{\mathsf{RS}}(q,1/\tau^2;s_2),
    \end{equation}
    whenever maximizers exist and the first coordinate of maximizing pairs is unique,
    where
    $$
    \Gamma = \left\{(q,\tau) \in [0,s_2] \times [0,\infty] \Bigm\vert \frac{\mathrm{d}}{\mathrm{d} q} f_{\mathsf{RS}}(q,1/\tau^2;s_2) = \frac{\mathrm{d}}{\mathrm{d} \tau^{-2}} f_{\mathsf{RS}}(q,1/\tau^2;s_2)  = 0\right\}.
    $$
    Some calculus applied to \eqref{psi-pi}, \eqref{psi-p-out}, and \eqref{f-rs} shows that $ \frac{\mathrm{d}}{\mathrm{d} q} f_{\mathsf{RS}}(q,1/\tau^2;s_2) = 0$ if and only if $s_2 - q = \delta\tau^2 - \sigma^2$. 
    That is,
    \begin{equation}\label{q-in-terms-of-tau}
        (q,\tau) \in \Gamma \Rightarrow q = s_2 - \delta \tau^2 + \sigma^2.
    \end{equation}
    We claim that maximizing $f_{\mathsf{RS}}$ over $\Gamma$ is equivalent to maximizing $f_{\mathsf{RS}}$ over the larger set $q = s_2 - \delta \tau^2 + \sigma^2$, as we now show.
    By  \eqref{psi-pi}, \eqref{psi-p-out}, and \eqref{f-rs}, we have
    \begin{align}
    f_{\mathsf{RS}}(s_2 - \delta \tau^2 + \sigma^2,1/\tau^2;s_2) &= \frac{s_2}{2\tau^2} - i(\tau^2) - \frac\delta2 \log2\pi - \frac\delta2 \log(\delta\tau^2) - \frac\delta2 - \frac{s_2 - \delta \tau^2 + \sigma^2}{2\tau^2} \nonumber \\ 
    &= -\left(\frac{\sigma^2}{2\tau^2} - \frac{\delta}2 \log \left(\frac{\sigma^2}{\tau^2}\right) + i(\tau^2)\right) + C 
    = - \phi(\tau^2) + C,\label{f-rs-after-q-elimination}
    \end{align}
    where $C$ is a constant which depends only on $\delta,\sigma^2$ and numerical constants.
    For $\tau \rightarrow 0$ and $\tau \rightarrow \infty$, we see from \eqref{f-rs-after-q-elimination} that $f_{\mathsf{RS}}(s_2 - \delta \tau^2 + \sigma^2,1/\tau^2;s_2)$ goes to $- \infty$, so that it is maximized at a point for which $\frac{\mathrm{d}}{\mathrm{d}\tau^{-2}} f_{\mathsf{RS}}(\delta \tau^2 - \sigma^2,1/\tau^2;s_2) = 0$. Because $\frac{\mathrm{d}}{\mathrm{d} q} f_{\mathsf{RS}}(q,1/\tau^2;s_2)\Big|_{q =s_2 - \delta \tau^2 + \sigma^2} = 0$, we have that $f_{\mathsf{RS}}(s_2 - \delta \tau^2 + \sigma^2,1/\tau^2;s_2) $ is maximized at a $\tau$ for which $\frac{\mathrm{d}}{\mathrm{d} \tau^{-2}} f_{\mathsf{RS}}(q,1/\tau^2;s_2)\Big|_{q =s_2 - \delta \tau^2 + \sigma^2}  = 0$. 
    That is, any maximizer $(q,\tau)$ of $f_{\mathsf{RS}}$ which satisfies $q = s_2 - \delta \tau^2 + \sigma^2$ must lie in $\Gamma$, as claimed.
    In particular, maximizing $f_{\mathsf{RS}}$ over the weaker constraint $q = s_2 - \delta \tau^2 + \sigma^2$ yields the same maximizers as maximizing $f_{\mathsf{RS}}$ over the stronger constraint $(q,\tau) \in \Gamma$.
    
    To summarize, all solutions to \eqref{free-energy-maximization} are constructed in the following way: 
    let $\tau^* \in \arg\max_\tau\{-\phi(\tau^2)\} = \arg\min_\tau \phi(\tau^2)$ and let $q^* = s_2 - \delta \tau^2 + \sigma^2$.
    Further, by \eqref{q-in-terms-of-tau}, when $\tau^*$  is the unique minimizer of $\phi$, we have $q$* is the unique first coordinate of maximizers of \eqref{free-energy-maximization}.
    We see that $\tau^* = \tau_{\mathsf{reg,stat}}$.
     After converting into our notation, Theorem 2 of \cite{barbier2018optimal} and their Eq.\ (8) state that under certain assumptions which we will list, $\lim_{p \rightarrow \infty}\E_{\bbeta_0,\bw,\bX}\left[\|\E_{\bbeta_0,\bw,\bX}[\bbeta_0|\by,\bX] - \bbeta_0 \|^2\right] = s_2 - q^*$.  
    The assumptions they require are that $\pi \in \cP_\infty(\reals)$, $\E_{\bbeta_0,\bw,\bX}\left[\left|\sum_{j=1}^p X_{1j} \beta_{0j} + w_1\right|^{2+\gamma}\right]$ is bounded for some $\gamma > 0$ (for us, it is bounded for all $\gamma > 0)$, the function $\varphi$ is continuous, $\sigma > 0$, and the minimizer $q^*$ is unique.
    These are all satisfied in our setting when the minimizer of $\phi$ is unique. 
    By equation \eqref{tau-stat-is-stationary} and because $s_2 - q^* = \delta \tau_{\mathsf{reg,stat}}^2 - \sigma^2 $, equation \eqref{bayes-risk-asymptotic} follows.
    
    Finally, by \cite[Theorm 2]{barbier2018optimal}, we have for fixed $\sigma^2$ that the maximizer $q^*$ is unique for almost every $\delta$ (w.r.t. Lebesuge measure). By Fubini's theorem, this holds for almost every $(\delta,\sigma)$ (w.r.t. Lebesgue measure).
    \end{proof}
    
In the proof of Corollary \ref{cor-alg-lower-bound} in Section \ref{sec:lower-bounds-and-benchmarks}, we use the following claim.
\begin{claim}\label{claim-alg-boudn-sup-to-inf-form-is-generic}
    For any $\pi \in \cP_2(\reals)$, the equality $\tau_{\mathsf{reg,amp}}^2 = \tau_{\mathsf{reg,amp}*}^2$ holds for almost every value of $\delta,\sigma$ (w.r.t.\ Lebesgue measure).
\end{claim}

\begin{proof}[Proof of Claim \ref{claim-alg-boudn-sup-to-inf-form-is-generic}]
    Comparing \eqref{tau-alg-def} and \eqref{alg-bound-sup-to-inf-form-def}, we see $\tau_{\mathsf{reg,amp}*}^2 \geq \tau_{\mathsf{reg,amp}}^2$ always. 
    Consider the case that $\tau_{\mathsf{reg,amp}} < \tau_{\mathsf{reg,amp}*}$.
    By \eqref{tau-alg-def}, for all $\tau \in (\tau_{\mathsf{reg,amp}}, \tau_{\mathsf{reg,amp}*}]$ we have $\delta \tau^2 - \sigma^2 \geq \mathsf{mmse}_\pi(\tau^2)$. 
    By \eqref{alg-bound-sup-to-inf-form-def}, for all $\tau > \tau_{\mathsf{reg,amp}*}$, we have $\delta \tau^2 - \sigma^2 > \mathsf{mmse}_\pi(\tau^2)$. By the continuity of $\mathsf{mmse}_\pi(\tau^2)$ \cite{DongningGuo2011EstimationError}, we have $\delta \tau_{\mathsf{reg,amp}*}^2 - \sigma^2 = \mathsf{mmse}_\pi(\tau_{\mathsf{reg,amp}*}^2)$. 
    Combining these three observations, we conclude by the differentiability of $\mathsf{mmse}_\pi(\tau^2)$ at $\tau_{\mathsf{reg,amp}*} > 0$ \cite{DongningGuo2011EstimationError} that $\delta = \frac{\mathrm{d}}{\mathrm{d} \tau^2} \mathsf{mmse}_\pi(\tau^2)\Big|_{\tau = \tau_{\mathsf{reg,amp}*}^2}$ and $\sigma^2 = \delta \tau_{\mathsf{reg,amp}*}^2 - \mathsf{mmse}_\pi(\tau_{\mathsf{reg,amp}*}^2)$. 
    Thus, the set of $\delta,\sigma^2$ for which $\tau_{\mathsf{reg,amp}}^2 = \tau_{\mathsf{reg,amp}*}^2$ holds is contained within the set 
    $$
    \left\{\left(\frac{\mathrm{d}\phantom{\tau^2}}{\mathrm{d}\tau^2}\mathsf{mmse}_\pi(\tau^2), \tau^2 \frac{\mathrm{d}\phantom{\tau^2}}{\mathrm{d}\tau^2}\mathsf{mmse}_\pi(\tau^2) - \mathsf{mmse}_\pi(\tau^2)\right)\Bigm\vert \tau^2 >0\right\},
    $$ 
    which has Lebesgue measure 0 because $\mathsf{mmse}_\pi$ is infinitely differentiable \cite{DongningGuo2011EstimationError}.
\end{proof}

\begin{proof}[Proof of Proposition \ref{prop-bAMP-achieves-alg-bound}]\label{proof-of-prop-bAMP-achieves-alg-bound}
    We will prove the proposition under the DSN assumption. Because the DSN assumption holds almost surely under the RSN assumption, the proposition also holds under the RSN assumption.

    The proposition is nearly an instance of Theorem 14 of \cite{Berthier2017StateFunctions},
    except that $\eta_t$ as we have defined it need not be Lipschitz continuous, which is required by \cite{Berthier2017StateFunctions}.
    Versions of Proposition \ref{prop-bAMP-achieves-alg-bound} appear elsewhere in the literature (e.g., \cite{Berthier2017StateFunctions,barbier2017}), 
    though often without proof, citing works in which state evolution for AMP is established for Lipschitz denoisers \cite{BM-MPCS-2011,javanmard2013state,Berthier2017StateFunctions}. 
    For the sake of completeness,
    we address here the minor technical difficulty that arises when $\eta_t$ is not Lipschitz using a truncation technique which is standard in the AMP literature.

    The truncation argument requires the following lemma.
    \begin{lemma}\label{lem:eta-linear-growth}
        There exist constants $C_t > 0$ such that for each $t$, $|\eta_t(y)| \leq C_t(1 + |y|)$.
    \end{lemma}

    \begin{proof}
        Assume $K$ is such that $\pi([-K,K]) \geq 1/2$.
        For $y > K$,
        we have 
        \begin{align*}
            \E_{\beta_0,z}[\beta_0|\beta_0+\tau_tz= y] &\leq y + \int_0^\infty \P_{\beta_0,z}(\beta_0 \geq y + t | \beta_0 + \tau_t z = y) \de t\\
            &\leq 2y+K + \int_0^\infty \frac{\int_{2y+K + t}^\infty \exp(-(y-s)^2/(2\tau_t^2)) \pi(\de s)}{\int_{-\infty}^\infty \exp(-(y-s)^2/(2\tau_t^2))\pi(\de s)} \de t\\
            & \leq 2y+K + \frac{\exp(-(y+K)^2/(2\tau_t^2))}{\exp(-(y+K)^2/(2\tau_t^2))/2} = 2y+K+2.
        \end{align*}
        A similar argument shows that for $y < - K$, $\E_{\beta_0,z}[\beta_0|\beta_0+\tau_tz= y] \geq 2y - K - 2$.
        This establishes the lemma.
    \end{proof}

    Define $\eta_{M,t}(y) = \eta_t(y) \ones_{ |y| \leq M} + \eta_t(M)\ones_{ y > M } + \eta_t(-M) \ones_{y < M}$.
    The reason for defining this truncation is that, because $\eta_t$ has continuous first derivative, $\eta_{M,t}$ is Lipschitz continuous.

    Define $\tau_{M,0}^2 = \tau_0^2$ and
    \begin{gather}
        \tau_{M,t+1}^2 = \frac1\delta(\sigma^2 + \E_{\beta_0,z}[(\eta_{M,t}(\beta_0 + \tau_{M,t}z) - \beta_0)^2]), \;\; t \geq 0,\label{B-AMP-SE-M}\\
        \mathsf{b}_{M,t} = \frac1\delta \E_{\beta_0,z}[\eta_{M,t-1}'(\beta_0 + \tau_{M,t-1}z)]. 
    \end{gather}
    The truncated Bayes AMP iteration is
    \begin{equation}\label{eq:truncated-AMP}
        \begin{gathered}
            \br_M^t = \frac{\by - \bX \widehat \bbeta^t}{n} + \mathsf{b}_{M,t} \br_M^{t-1},\\
            \widehat \bbeta_M^{t+1} = \eta_{M,t}(\widehat \beta_{M,t}^t + \bX^\top \br_M^t).
        \end{gathered}
    \end{equation}
    To prove Proposition \ref{prop-bAMP-achieves-alg-bound},
    we will use Theorem 14 of \cite{Berthier2017StateFunctions} to establish state evolution for the iteration \eqref{eq:truncated-AMP} and then show that for large $M$, iteration \eqref{eq:truncated-AMP} approximates iteration \eqref{B-AMP}.

    \noindent \textbf{State evolution for truncated Bayes AMP}.

    We claim for any fixed $t$,
    \begin{equation}\label{eq:truncated-est-err}
        \lim_{p \rightarrow \infty}^{\mathrm p} \| \widehat \bbeta_M^t - \bbeta_0\|^2 = \E_{\beta_0,z}[(\eta_{M,t}(\beta_0 + \tau_{M,t} z)^2 - \beta_0^2)].
    \end{equation}
    This statement follow directly from Theorem 14 of \cite{Berthier2017StateFunctions} because, due to the truncation, all the technical conditions of that theorem are satisfied, as we now show.

    First, Theorem 14 of \cite{Berthier2017StateFunctions} is related to our setting by the following change of variables (with the notation of \cite{Berthier2017StateFunctions} on the left).
    \begin{equation}\label{B-AMP-COV}
    \begin{aligned}
        \bA &\leftarrow \frac1{\sqrt n}\bX, & \btheta_0 &\leftarrow\sqrt p \bbeta_0, & \by &\leftarrow \sqrt{\frac pn} \by, & \bw &\leftarrow \sqrt{ \frac pn } \bw, & \widehat \btheta^t &\leftarrow \sqrt p \widehat \bbeta^t_M,\\
        \br^t &\leftarrow \sqrt{np}\br^t_M, & (m,n) &\leftarrow (n,p),             & \eta_t(\bx)_j &\leftarrow \eta_{M,t}(\sqrt{p}x_j), & \sigma_w^2 &\leftarrow \sigma^2/\delta, & \mathsf{b}_t &\leftarrow \mathsf{b}_{M,t}.
    \end{aligned}
    \end{equation}

    We must check conditions (C1) - (C6) of \cite{Berthier2017StateFunctions} and one more condition which we list as equation \eqref{bMt-condition} below.
    (C1) holds by assumption; (C2) holds because the posterior mean is continuously differentiable to all orders,\footnote{See \cite[Theorem 2.7.1]{Lehmann2005TestingHypotheses}. Because the posterior mean as a function of $y$ under Gaussian noise is the mean of an exponential family with natural parameter $y/\tau^2$, this theorem applies.}\label{footnote-smooth-bayes-estimator}, so it is Lipschitz on compact intervals, and $\eta_{M,t}$ defined by $\eta_{M,t}(\bx)_j = \eta_{M,t}(\sqrt{p}x_j)/\sqrt{p}$ is uniformly Lipschitz;
    and (C3) and (C4) hold by the DSN assumption.

    For (C5), we must check that $\lim_{p \rightarrow \infty}\E_{\bz}\left[\left\langle \bbeta_0, \eta_{M,t}(\bbeta_0 + \tau \bz)\right\rangle\right]$ exists and is finite.
    Note the functions $(\bx_1,\bx_2) \mapsto \bx_1$ and $(\bx_1, \bx_2) \mapsto \eta_{M,t}(\bx_1 + \tau \bx_2)$ are uniformly pseudo-Lipschitz of order 1 (the first trivially, the second by (C2)), and their norm evaluated at $\bzero$ is bounded over $p$.
    Because these functions are symmetric, Lemma \ref{lem-pseudo-lipschitz-empirical-to-expectation}, which gives
    \begin{align}
    \lim_{p \rightarrow \infty}\E_{\bz}[\langle \bbeta_0, &\eta_{M,t}(\bbeta_0 + \tau \bz)\rangle] = \lim_{p \rightarrow \infty} \E_{\tilde \bbeta_0,\bz} \left[\left\langle \tilde \bbeta_0, \eta_{M,t}( \tilde\bbeta_0 + \tau \bz)\right\rangle\right]\nonumber \\ 
    &= \lim_{p \rightarrow \infty} \frac1p \sum_{j=1}^p \E_{\tilde \bbeta_0,\bz}\left[\sqrt p \tilde \beta_{0j} \eta_{M,t} (\sqrt p \tilde \beta_{0j} + \tau \sqrt p z_j)\right]= \E_{\beta_0,z}[ \beta_0\eta_{M,t}(\beta_0 + \tau z)],\label{C5-limit}
    \end{align}
    where $\tilde \bbeta_0$ has coordinates distributed iid from $\pi/\sqrt p$, and $\beta_0\sim \pi,\,z\sim\mathsf{N}(0,1)$ independent.
    Because $\eta_{M,t}$ is bounded, the expectation on the right-hand side is finite, and (C5) is established.

    For (C6), we must show that or any $s,t$ and any $2\times 2$ covariance matrix $\bT \in S_+^2$, the limit  $\lim_{p \rightarrow \infty}\E_{\bz_1,\bz_2}\left[\left\langle \eta_{M,s}(\bbeta_0 + \bz_1), \eta_{M,t}(\bbeta_0 + \bz_2)\right\rangle\right]$ exists and is finite, where $(\bz_1,\bz_2) \sim \mathsf{N}(\bzero, \bT \otimes \bI_p / p)$.
    This is shown in the same way we establisehd (C5).

    Under the change of variables (with the notation of \cite{Berthier2017StateFunctions} on the left) $\tau_t^2 \leftarrow \tau_{M,t}^2$, 
    iteration (206), (207) of \cite{Berthier2017StateFunctions} becomes the scalar iteration \eqref{B-AMP-SE-M}. 
    Under this change of variables, 
    the condition given by equation (208) of \cite{Berthier2017StateFunctions}
    becomes
    \begin{equation}\label{bMt-condition}
    \mathsf{b}_{M,t} \stackrel{\mathrm p}\simeq \frac1n \E_{\bz}\left[\mathrm{div}\, \eta_{M,t-1}\left(\bbeta_0 + \tau_{M,t-1}\bz\right)\right],
    \end{equation}
    where $\bz \sim \mathsf{N}(\bzero,\bI_p/p)$.
    Note $\tau_{M,t-1} > 0$ by induction: for the base case, use $s_2(\pi) > 0$, 
    and then use that the right-hand side of \eqref{B-AMP-SE-M} must be positive whenever $\tau_{M,t} > 0$ because perfect recovery with a non-trivial prior is impossible under Guassian corruption.
    Then by Gaussian integration by parts (Lemma \ref{lem-steins-lemma}), we have 
    \begin{equation}\label{steins-for-BAMP}
     \frac1n \E_{\bz}\left[\mathrm{div}\, \eta_{M,t-1}\left(\bbeta_0 + \tau_{M,t-1}\bz\right)\right] = \frac p{\tau_{M,t-1}n} \E_{\bz}\left[\langle \bz,\eta_{M,t-1}\left(\bbeta_0 + \tau_{M,t-1}\bz\right)\rangle\right].
    \end{equation}

    As the dimension $p$ varies, the functions $(\bx_0,\bx_1) \mapsto \bx_1$ and $(\bx_0,\bx_1) \mapsto \eta_{M,t-1}(\bx_0 + \tau_{M,t-1}\bx_1)$ are uniformly pseudo-Lipschitz of order 1 (the first trivially, the second by (C2)).
    By the same argument as in (C5), their norm when evaluated at $\bzero$ is bounded (over $p$).
    Thus, by Lemma \ref{lem-pseudo-lipschitz-closed-under-inner-product}, the functions $(\bx_0,\bx_1) \mapsto  \langle \bx_1, \eta_{M,t-1}(\bx_0 + \tau_{M,t-1} \bx_1)\rangle$ are uniformly pseudo-Lipschitz of order 2.
     Because these functions are also symmetric and $\{\bbeta_0\}$ satisfies the DSN assumption \eqref{DSN-assumption},
    we may apply Lemma \ref{lem-pseudo-lipschitz-empirical-to-expectation}, which gives
    \begin{align}
        \E_{\bz}&\left[\langle \bz,\eta_{M,t-1}\left(\bbeta_0 + \tau_{M,t-1}\bz\right)\rangle\right] \stackrel{\mathrm{p}}\simeq \E_{\tilde \bbeta_0,\bz}\left[\left\langle \bz,\eta_{M,t-1}(\tilde \bbeta_0 + \tau_{M,t-1}\bz)\right\rangle\right] \nonumber \\
        &= \frac1p \sum_{j=1}^p \E_{\tilde \bbeta_0,\bz} \left[\left\langle \sqrt p z_{0j} \eta_{M,t-1}(\sqrt p \tilde \beta_{0j} + \sqrt p z_{0j})\right\rangle \right]
        = \frac1p \sum_{j=1}^p \E_{\beta_0,z} \left[\left\langle z \eta_{M,t-1}\left( \beta_0 + z_0\right)\right\rangle \right]\nonumber \\
        &= \tau_{M,t-1} \E_{\beta_0,z}\left[\eta_{M,t-1}'(\beta_0 + \tau_{M,t-1}z)\right],\label{width-limit-BAMP}
    \end{align}
    where we have taken $\tilde \bbeta_0$ with coordinates distributed iid from $\pi/\sqrt p$, in the second equality we have used \eqref{B-AMP-full-p-denoiser}, in the third line we have taken $\beta_0 \sim \pi,\, z\sim \mathsf{N}(0,1)$ independent, and in the fourth equality we have used Lemma \ref{lem-steins-lemma} and the fact that $\eta_{M,t-1}:\reals \mapsto \reals$ is Lipschitz (see (C2)).
    By the HDA assumption, $n/p \rightarrow \delta$, whence \eqref{steins-for-BAMP} and \eqref{width-limit-BAMP} yield \eqref{bMt-condition}.

    \noindent {\bf Estimation error of truncated Bayes AMP}

    \noindent Having checked the above conditions, we may apply Theorem 14 of \cite{Berthier2017StateFunctions}.
    Because $\eta_{M,t}:\reals^p\rightarrow \reals^p$ are uniformly pseudo-Lipschitz of order 1 by (C2) and $\|\eta_{M,t}(\bzero) - \bzero\| = \|\eta_{M,t}(\bzero)\|$ is uniformly (over $p$) bounded by the argument in (C5), we have by Lemma \ref{lem-pseudo-lipschitz-closed-under-inner-product} that the functions $(\bx_0,\bx_1) \mapsto \left\|\eta_{M,t}\left(\bx_1\right) - \bx_0\right\|^2$ are uniformly pseudo-Lipschitz of order 2. 
    Thus, by Claim \ref{claim-pseudo-lipschitz-change-of-normalization}, the functions $\Psi_p(\bx_0,\bx_1) := \left\|\eta_{M,t}\left(\bx_1/\sqrt p\right) - \bx_0/\sqrt p\right\|^2$ are \cite{Berthier2017StateFunctions}-uniformly pseudo-Lipschitz of order 2.\footnote{See \eqref{berthier-pseudo-lipschitz}. This terminology just refers to the use of the notion of being uniformly pseudo-Lipschitz under the different choice of normalization used by \cite{Berthier2017StateFunctions}. Thus, it tells us the functions to which we are able to apply their theorem.}
    Under the change of variables \eqref{B-AMP-COV}, we have $\widehat \btheta^t + \bA^\mathsf{T} \br^t \leftarrow \sqrt p(\widehat \bbeta^t_M + \bX \br^t_M)$ and $\btheta_0 \leftarrow \sqrt p \bbeta_0$. Then (justification follows equations)
    \begin{align}
    &\left\|\widehat \bbeta^{t+1}_M - \bbeta_0\right\|^2 =  \left\|\eta_{M,t}\left(\widehat \bbeta^{t}_M + \bX^\mathsf{T}\br^{t}_M\right) - \bbeta_0\right\|^2 \stackrel{\mathrm{p}}\simeq \E_{\bz} \left[\left\|\eta_{M,t}\left(\bbeta_0 + \tau_{M,t} \bz\right) - \bbeta_0\right\|^2\right]\nonumber\\
    &\qquad\qquad\stackrel{\mathrm{p}}\simeq \E_{\tilde \bbeta_0,\bz} \left[\|\eta_{M,t}(\tilde \bbeta_0 + \tau_{M,t} \bz) - \bbeta_0\|^2\right]
    = \frac1p \sum_{j=1}^p \E_{\tilde \bbeta_0,\bz}\left[(\eta_{M,t}(\sqrt p \tilde \beta_{0j} + \tau_{M,t} \sqrt p z_j ) - \sqrt p \beta_{0j})^2\right]\nonumber\\
    &\qquad\qquad = \E_{\beta_0,z}\left[\left(\eta_{M,t}\left(\beta_0 + \tau_{M,t} z\right) - \beta_0\right)^2\right],\label{truncated-iterates-loss}
    \end{align}
    where in the first equality we have used \eqref{B-AMP}; in the first line we have taken $\bz \sim \mathsf{N}(\bzero,\bI_p/p)$; in the first probabilistic equality we have used Theorem 14 of \cite{Berthier2017StateFunctions} (in particular, Eq.~(210) applied to $\psi_p$); in the second probabilistic equality we have used Lemma \ref{lem-pseudo-lipschitz-empirical-to-expectation}; in the second line we have taken $\tilde \bbeta_0$ with coordinates distributed iid from $\pi/\sqrt p$; in the third line we have used \eqref{B-AMP-full-p-denoiser}; and in the fourth line we have taken $\beta_0 \sim \pi,\,z\sim \mathsf{N}(0,1)$ independent.

    \noindent {\bf The truncated and untruncated state evolutions are close}

    \noindent By induction, for each $t \geq 0$, we have
    \begin{equation}\label{tau-M-to-tau}
    \lim_{M \rightarrow \infty} \tau_{M,t} = \tau_t
    \;\;\;
    \text{and}
    \;\;\;
    \lim_{M \rightarrow \infty} \mathsf{b}_{M,t} = \mathsf{b}_t.
    \end{equation} 
    Indeed, the base case $(t=0)$ holds by definition. 
    For the induction step, assume $\tau_{M,t} \rightarrow \tau_{t}$.
    Denote $\eta(\cdot;\tau)$ the Bayes estimator at noise level $\tau$. That is, $\eta(y;\tau) = \E_{\beta_0,z}[\beta_0|\beta_0 + \tau z = y]$.
    By the same argument we used in (C2), the Bayes estimator $\eta$ is continuous in $\tau$ for $\tau > 0$.
    Then, by the inductive hypothesis, $\eta(y;\tau_{M,t}) \xrightarrow[M \rightarrow \infty]{} \eta(y;\tau_{t})$ for all $y \in \reals$. 
    Further, because for $\beta_0,z$ fixed we have $M > \beta_0 + \tau_{t} z$ for sufficiently large $M$, we have 
    \begin{equation}\label{pointwise-eta-Mt-convergence}
    \eta_{M,t}(\beta_0 + \tau_{M,t} z) \xrightarrow[M \rightarrow \infty]{} \eta(\beta_0 + \tau_{t} z;\tau_{t}) \text{ pointwise.} 
    \end{equation}
    Also, $|\eta_{M,t}(\beta_0 + \tau_{M,t} z)|<|\eta(\beta_0 + \tau_{M,t}z;\tau_{M,t})|$ and
    the collection of random variables $\{\eta(\beta_0 +\tau z;\tau)^2 \mid \tau \geq 0\}$ is uniformly integrable because $\eta(\beta_0 +\tau z;\tau)^2 = \E_{\beta_0,z} [\beta_0|\beta_0 + \tau z]^2 \leq \E_{\beta_0,z}[\beta_0^2|\beta_0 + \tau z]$, and $\E_{\beta_0,z}[ \E_{\beta_0,z}[\beta_0^2|\beta_0 + \tau z]\mathbf{1}_{ \E_{\beta_0,z}[\beta_0^2|\beta_0 + \tau z] > C}] = \E_{\beta_0,z}[\beta_0^2\mathbf{1}_{ \E_{\beta_0,z}[\beta_0^2|\beta_0 + \tau z] > C}]$ becomes uniformly small for sufficiently large $C$ because $\P_{\beta_0,z}( \E_{\beta_0,z}[\beta_0^2|\beta_0 + \tau z] > C ) \leq \frac{s_2(\pi)}{C}$ by Markov's inequality.
    Thus, in fact, the collection $\{(\eta_{M,t}(\beta_0 + \tau_{M,t} z) - \beta_0)^2\}$ over all values of $M$ and $t$ is uniformly integrable.
    By Vitali's Convergence Theorem (see e.g. \cite[Theorem 5.5.2]{Durrett2010Probability:Examples}) and \eqref{pointwise-eta-Mt-convergence}, we have
    \begin{align}
    \E_{\beta_0,z}\left[\left(\eta_{M,t}\left(\beta_0 + \tau_{M,t} z\right) - \beta_0\right)^2\right] &\xrightarrow[M \rightarrow \infty]{} \E_{\beta_0,z}\left[\left(\eta\left(\beta_0 + \tau_{t} z,\tau_{t}\right) - \beta_0\right)^2\right] = \mathsf{mmse}_\pi\left(\tau_{t}^2\right).\label{fixed-t-se-convergence}
    \end{align} 
    By \eqref{B-AMP-SE-2}, \eqref{B-AMP-SE-M} and \eqref{fixed-t-se-convergence}, we have \eqref{tau-M-to-tau}.
    The induction is complete, so in fact $\tau_{M,t} \rightarrow \tau_t$ as $M \rightarrow \infty$ holds for all $t$.
    A similar argument shows that convergence of $\mathsf{b}_{M,t}$.

    \noindent {\bf The truncated and untruncated state evolutions are close}

    \noindent We claim
    \begin{equation*}
        \lim_{M \rightarrow \infty} \limsup^{\mathrm{p}}_{p\rightarrow \infty}\|\widehat \bbeta_{M}^t - \widehat \bbeta^t \|^2 = 0.
    \end{equation*}
    This follows inductively by combining $|\eta_t(y) - \eta_{M,t}(y)| \leq C_t(1+|y|)\mathbf{1}_{|y|\geq M}$ (Lemma \ref{lem:eta-linear-growth}), $\mathsf{b}_{M,t} \rightarrow \mathsf{b}_t$, and the boundedness (in probability) of $\|\widehat \bbeta_M^t\|^2$. 
    Thus, by \eqref{eq:truncated-est-err} and \eqref{tau-M-to-tau},
    we conclude \eqref{M-to-infty-iterates-loss}.

    \noindent {\bf Bayes AMP achieves noise variance $\tau_{\mathsf{reg,amp}*}^2$}

    \noindent Now we prove \eqref{large-t-bAMP-SE}.
    Because $s_2(\pi) > 0$, for all $\tau > 0$, we have $\mathsf{mmse}_\pi(\tau^2) < s_2(\pi)$. Thus, for $\tau^2 \geq \frac1\delta(\sigma^2 + s_2(\pi))$, we have $\delta \tau^2 - \sigma^2 \geq s_2(\pi) > \mathsf{mmse}_\pi(\tau^2)$. Thus, by \eqref{alg-bound-sup-to-inf-form-def} and \eqref{B-AMP-SE-1} and the continuity of $\mathsf{mmse}_\pi(\tau^2)$ in $\tau^2$, we have $\tau_0 > \tau_{\mathsf{reg,amp}*}$.
    Further, if $\tau > \tau_{\mathsf{reg,amp}*}$, because $\mathsf{mmse}_\pi(\tau^2)$ is strictly increasing in $\tau$ (see \cite[Eq.~(65)]{DongningGuo2011EstimationError}), we have $\frac1\delta\left(\sigma^2 + \mathsf{mmse}_\pi(\tau^2)\right) > \frac1\delta \left(\sigma^2 + \mathsf{mmse}_\pi(\tau_{\mathsf{reg,amp}*}^2)\right) = \tau_{\mathsf{reg,amp}*}^2$.
    Thus, by \eqref{B-AMP-SE-2} and induction, we have $\tau_{t-1} > \tau_t > \tau_{\mathsf{reg,amp}*}$ for all $t$.
    Further, because $\mathsf{mmse}_\pi(\tau^2)$ is continuous in $\tau^2$ \cite{DongningGuo2011EstimationError}, for all $\eps > 0$ such that $\tau_{\mathsf{reg,amp}*} + \eps < \tau_0$, we have $\inf_{\tau \in [\tau_{\mathsf{reg,amp}*} + \eps,\tau_0]} \left\{\tau^2 - \frac1\delta \left(\sigma^2 + \mathsf{mmse}_\pi(\tau^2)\right) \right\} > 0$.
    Thus, for all $t$ such that $\tau_t > \tau_{\mathsf{reg,amp}*} + \eps$, we have $\tau_t - \tau_{t+1}$ is bounded below by a positive constant.
    Thus, for $t$ sufficiently large we must have $\tau_t \leq \tau_{\mathsf{reg,amp}*} + \eps$. Because this is true for all sufficiently small $\eps > 0$, we have $\limsup_{t\rightarrow \infty} \tau_t \leq \tau_{\mathsf{reg,amp}*}$.
    Because we also have $\tau_t > \tau_{\mathsf{reg,amp}*}$ for all $t$, we in fact have \eqref{large-t-bAMP-SE}.

    Eq.\ \eqref{bAMP-achieves-bound} now follows by \eqref{M-to-infty-iterates-loss} and taking $t$ sufficiently large.
\end{proof}

\section{Proof of Theorem \ref{thm-high-low-snr-gaps}}\label{sec:ThmGaps}

\begin{proof}[Proof of Theorem \ref{thm-high-low-snr-gaps}.(i)]
    Throughout the proof, we will drop $\pi$ from our notation for the moments of $\pi$. That is, we write $s_k$ in place of $s_k(\pi)$.
    Observe that $\mathsf{mmse}_\pi(\tau^2) \leq s_2$. Also, $\mathsf{R^{opt}_{seq,cvx}}(\tau;\pi) \leq s_2$ because at each $p$ we may take in \eqref{R-seq-cvx-opt-finite} the function $\rho_p(\bx) = \mathbb{I}_{\bx = \bzero}$ which is $0$ when $\bx = \bzero$ and $\infty$ otherwise.
Thus,
    \begin{gather}
    \frac{\sigma^2}{\delta} \leq \frac{\sigma^2 + \mathsf{mmse}_\pi(\tau^2)}{\delta}  \leq \frac{\sigma^2 + s_2}{\delta},\label{tau-stat-bound}\\
     \frac{\sigma^2}{\delta}\leq \frac{\sigma^2 +\mathsf{R^{opt}_{seq,cvx}}(\tau;\pi)}{\delta} \leq  \frac{\sigma^2 + s_2}{\delta}.\label{tau-cvx-bound}
    \end{gather}
    By \eqref{tau-stat-bound} and \eqref{tau-stat-is-stationary}, $\frac{\sigma^2}{\delta} \leq \tau_{\mathsf{reg,stat}}^2 \leq\frac{\sigma^2 + s_2}{\delta}$, whence $\tau_{\mathsf{reg,stat}} = \frac{\sigma}{\sqrt \delta} + O\left(\frac1{\sqrt{\delta}}\right)$, where we have used the inequality that for $a,b\geq 0$ we have $\sqrt{a + b} \leq \sqrt{a} + \sqrt{b}$.
    By Lemma \ref{lem-r-continuity}, $\mathsf{R^{opt}_{seq,cvx}}(\tau;\pi)$ is continuous in $\tau$, whence by \eqref{eqdef-tau-reg-cvx}, we have $\delta \tau_{\mathsf{reg,cvx}}^2 - \sigma^2 = \mathsf{R^{opt}_{seq,cvx}}(\tau_{\mathsf{reg,cvx}};\pi)$. Combined with \eqref{tau-cvx-bound}, we have $\frac{\sigma^2}{\delta} \leq \tau_{\mathsf{reg,cvx}}^2 \leq\frac{\sigma^2 + s_2}{\delta}$, whence $\tau_{\mathsf{reg,cvx}} =  \frac{\sigma}{\sqrt \delta} + O\left(\frac1{\sqrt{\delta}}\right)$ as well.
    
    With $\beta_0\sim \pi,\,z\sim \mathsf{N}(0,1)$ independent and $y = \beta_0 + \tau z$, we have (justification to follow) 
    \begin{align}
        \frac{\mathrm{d}}{\mathrm{d}\tau}\mathsf{mmse}_\pi(\tau^2) &= -\frac2{\tau^3}\frac{\mathrm{d}}{\mathrm{d}\tau^{-2}} \mathsf{mmse}_\pi(\tau^2) = \frac2{\tau^3} \E_{\beta_0,z}\left[\E_{\beta_0,z}\left[(\beta_0 - \E_{\beta_0,z}[\beta_0|y])^2\big|y\right]^2\right]\nonumber\\
    &\leq  \frac2{\tau^3} \E_{\beta_0,z}\left[(\beta_0 - \E_{\beta_0,z}[\beta_0|y])^4\right] 
    \leq 32\sqrt{24} \tau,\label{mmse-derivative-bound}
    \end{align}
    where in the second equality we have used \cite[Proposition 9]{DongningGuo2011EstimationError}, in the first inequality we have used Jensen's inequality, and in the second inequality we have used \cite[Proposition 5]{DongningGuo2011EstimationError}.
    Further, because $\pi \in \cP_6(\reals)$, $\tau^{-2} \mapsto \mathsf{mmse}_\pi(\tau^2)$ is continuously differentiable to second order on $[0,\infty)$ \cite[Proposition 7]{DongningGuo2011EstimationError}, whence $\frac{\mathrm{d}}{\mathrm{d}\tau^{-2}}\mathsf{mmse}_{\pi}(\tau^2)$ is bounded for $\tau \geq C$ for any $C > 0$. Combined with \eqref{mmse-derivative-bound} (which bounds the derivative for small $\tau$), we get that $\mathsf{mmse}_\pi(\tau^2)$ is Lipschitz in $\tau$ on the entirety of its domain $[0,\infty)$.
     Because $\mathsf{R^{opt}_{seq,cvx}}(\tau;\pi)$ is also Lipschitz in $\tau$, we have by Theorem \ref{thm-cvx-lower-bound}
    $$
    \Delta(\pi,\delta,\sigma) \geq \mathsf{R^{opt}_{seq,cvx}}(\tau_{\mathsf{reg,cvx}};\pi) - \mathsf{mmse}_\pi(\tau_{\mathsf{reg,stat}}^2) = \mathsf{R^{opt}_{seq,cvx}} (\sigma/\sqrt{\delta};\pi) - \mathsf{mmse}_\pi(\sigma^2/\delta) + O(1/\sqrt \delta).
    $$
   Thus, we have \eqref{gap-high-snr}.
\end{proof}

\begin{proof}[Proof of Theorem \ref{thm-high-low-snr-gaps}.(ii)]
    As in the proof of part (i), throughout the proof, we will drop $\pi$ from our notation for the moments of $\pi$. That is, we write $s_k$ in place of $s_k(\pi)$.
    By \cite[Eq. (61)]{DongningGuo2011EstimationError}, we have
    \begin{align}
        \mathsf{mmse}_\pi(\tau^2) &= s_2 - s_2^2 \frac1{\tau^2} + \frac12 \left(2s_2^3 - s_3^2\right) \frac1{\tau^4}
        - \frac16 \left(15 s_2^4 - 12 s_2 s_3^2 - 6 s_2^2s_4 + s_4^2\right) \frac1{\tau^6}+ O\left(\frac1{\tau^8}\right),\label{mmse-low-snr-expansion}
    \end{align}
    where $O\left(\frac1{\tau^8}\right)$ hides constants depending only on the moments of $\pi$. 
    Define $\kappa^2 = \frac{\sigma^2}{\delta}\left(1 + \frac{s_2}{\sigma^2}\right)$ and $\Delta = \kappa^2 - \tau_{\mathsf{reg,stat}}^2$.
    By \eqref{tau-stat-is-stationary} and some rearrangement, we have 
    \begin{equation}\label{fixed-pt-in-Delta}
    s_2 - \delta \Delta = \mathsf{mmse}_\pi\left(\kappa^2  - \Delta\right).
    \end{equation}
    For the remainder of the proof, $O$ will also hide constants depending $\delta$ (in addition to the moments of $\pi$), but will not depend on $\sigma^2$ and likewise on $\kappa^2$ or $\Delta$.
    We see that $\Delta \leq \frac{s_2}{\delta} = O(1)$, so that by \eqref{mmse-low-snr-expansion} we have
    \begin{align}
        \mathsf{mmse}_\pi\left(\kappa^2 - \Delta \right) &= s_2 - \frac{s_2^2}{\kappa^2} \left(1 + \frac{\Delta}{\kappa^2} + \frac{\Delta^2}{\kappa^4} + O \left(\frac{1}{\kappa^6}\right) \right)+ \frac12 \frac{2s_2^3 - s_3^2}{\kappa^4} \left(1 + O\left(\frac{\Delta}{\kappa^2}\right)\right)\nonumber\\
        &\quad- \frac16 \frac{15 s_2^4 - 12 s_2 s_3^2 - 6 s_2^2s_4 + s_4^2} {\kappa^6} \left(1 + O\left(\frac{\Delta}{\kappa^2}\right)\right)
        + O\left(\frac{1}{\kappa^8}\right).\label{mmse-Delta-expansion}
    \end{align}
    Comparing with \eqref{fixed-pt-in-Delta} and using $\Delta = O(1)$, we see that 
    $
    \Delta = \frac{s_2^2}{\delta\kappa^2} + O\left(\frac1{\kappa^4}\right).
    $
    Thus, we have
    \begin{align*}
    \frac{s_2^2}{\kappa^2} \left(1 + \frac{\Delta}{\kappa^2} + \frac{\Delta^2}{\kappa^4} + O \left(\frac{1}{\kappa^6}\right) \right) &= \frac{s_2^2}{\kappa^2} + \frac{s_2^4}{\delta \kappa^6} + O\left(\frac1{\kappa^8}\right),\\
     \frac12 \frac{2s_2^3 - s_3^2}{\kappa^4} \left(1 + O\left(\frac{\Delta}{\kappa^2}\right)\right) &=  \frac12 \frac{2s_2^3 - s_3^2}{\kappa^4} + O\left(\frac1{\kappa^8}\right).
    \end{align*}
    Plugging into \eqref{mmse-Delta-expansion}, we have
    \begin{align}
        \mathsf{mmse}_\pi(\tau_{\mathsf{reg,stat}}^2) & = \mathsf{mmse}_\pi(\kappa^2 - \Delta) = s_2 - \frac{s_2^2}{\kappa^2}  + \frac12 \frac{2s_2^3 - s_3^2}{\kappa^4}\nonumber\\
         &\qquad - \frac{s_2^4}{\delta \kappa^6} - \frac16 \frac{15 s_2^4 - 12 s_2 s_3^2 - 6 s_2^2s_4 + s_4^2} {\kappa^6} + O\left(\frac1{\kappa^8}\right).\label{mmse-tau-stat-expansion}
    \end{align}
    We now write this expansion in terms of the signal-to-noise parameter $\mathsf{snr}$. 
    Applying the definition of $\kappa^2$, we have $\frac{s_2}{\kappa^2} = \delta\, \mathsf{snr}(1 - \mathsf{snr} + \mathsf{snr}^2) + O(\mathsf{snr}^4)$, $\frac{s_2^2}{\kappa^4} = \delta^2 \mathsf{snr}^2(1 - 2\, \mathsf{snr}) + O(\mathsf{snr}^4)$, and $\frac{s_2^3}{\kappa^6} = \delta^3\,\mathsf{snr}^3 + O(\mathsf{snr}^4)$.
    Plugging into \eqref{mmse-tau-stat-expansion} and rearranging, we get
    \begin{align}
    \mathsf{mmse}_\pi(\tau_{\mathsf{reg,stat}}^2) &= s_2 - s_2 \delta\, \mathsf{snr} 
    + s_2 \left(\delta + \delta^2 \left(1 - \frac{s_3^2}{2s_2^3}\right)\right)\mathsf{snr}^2\nonumber \\
    &\quad+ s_2 \left(- \delta - \delta^2\left(3 - \frac{s_3^2}{s_2^3}\right) - \delta^3 \left(\frac52 - 2 \frac{s_3^2}{s_2^3} - \frac{s_4}{s_2^2} + \frac{s_4^2}{6s_2^4}\right)\right)\mathsf{snr}^3 + O(\mathsf{snr}^4).\label{mmse-tau-stat-in-zeta-expansion}
    \end{align}
   
    Now let $\tau_{\mathsf{ridge}}$ solve
    \begin{align}\label{optimal-ridge-effective-noise}
        \delta \tau^2 -\sigma^2 = \mathsf{mmse}_{\mathsf{N}(0,s_2)}(\tau^2).
    \end{align}
   We will show that ridge regression with appropriately chosen regularization achieves risk $\mathsf{mmse}_{\mathsf{N}(0,s_2)}(\tau_{\mathsf{ridge}}^2)$. 
    Let
    $
    \rho_p(\bx) =  \frac{\sigma^2}{\delta s_2} \|\bx\|^2.
    $
    The risk of this estimator has been studied previously by \cite{karoui2013asymptotic}. 
    We repeat the analysis here for completeness.
        Let 
    \begin{align}\label{optimal-ridge-effective-lam}
        \lambda_{\mathsf{ridge}} = \frac{\delta\tau_{\mathsf{ridge}}^2}{2\sigma^2}.
    \end{align}
    Observe then that $\mathsf{prox}[\lambda_{\mathsf{ridge}} \rho_p](\by) = \frac{s_2}{s_2 + \tau_{\mathsf{ridge}}^2} \by$.
    Let $\cT = (\pi,\{\rho_p\})$.
    Observe that $\mathsf{R_{reg,cvx}^\infty}(\tau_{\mathsf{ridge}},\lambda_{\mathsf{ridge}},\cT) = \frac{s_2\tau^2}{s_2 + \tau^2} = \mathsf{mmse}_{\mathsf{N}(0,s_2)}(\tau_{\mathsf{ridge}}^2)$ and $\mathsf{W_{reg,cvx}^\infty}(\tau_{\mathsf{ridge}},\lambda_{\mathsf{ridge}},\cT) = \frac{s_2}{s_2 + \tau_{\mathsf{ridge}}^2}$. 
    One can then check using \eqref{optimal-ridge-effective-noise} and \eqref{optimal-ridge-effective-lam} that $\tau_{\mathsf{ridge}},\lambda_{\mathsf{ridge}}$ solve \eqref{fixed-pt-prior-asymptotic} at $\gamma=0$. 
    Because $\rho_p$ are uniformly strongly convex, by Proposition \ref{prop-strongly-convex-loss}.(ii) and Lemma \ref{lem-strongly-convex-DSN-to-RSN-risk}, we have
    \begin{equation}\label{ridge-asymptotic-risk}
    \lim_{p \rightarrow \infty}  \E_{\bbeta_0,\bw,\bX} \left[\|\widehat \bbeta_{\mathsf{cvx}} - \bbeta_0\|^2\right] = \mathsf{R_{orc,cvx}^\infty}(\tau_{\mathsf{ridge}},\lambda_{\mathsf{ridge}},\cT) = \mathsf{mmse}_{\mathsf{N}(0,s_2)}(\tau_{\mathsf{ridge}}^2).
    \end{equation}
 Because $\tau_{\mathsf{ridge}}$ solves \eqref{optimal-ridge-effective-noise}, we in fact have that formula \eqref{mmse-tau-stat-in-zeta-expansion} holds for $\mathsf{mmse}_{\mathsf{N}(0,s_2)}(\tau_{\mathsf{ridge}}^2)$ after replacing the moments with those of $\mathsf{N}(0,s_2)$. 
That is,
    \begin{align}
    \mathsf{mmse}_{\mathsf{N}(0,s_2)}(\tau_{\mathsf{ridge}}^2) &= s_2 - s_2 \delta \,\mathsf{snr} + s_2 \left(\delta + \delta^2\right)\mathsf{snr}^2 + s_2 \left(- \delta - 3 \delta^2 - \delta^3 \right)\mathsf{snr}^3 + O(\mathsf{snr}^4),\label{ridge-in-zeta-expansion}
    \end{align}
    where we have used that the third moment of $\mathsf{N}(0,s_2)$ is 0 and fourth moment of $\mathsf{N}(0,s_2)$ is $3s_2^2$.

By \cite[Eq.~(65)]{DongningGuo2011EstimationError}, we have
    \begin{align*}
    -\frac{\mathrm{d}}{\mathrm{d} \tau^{-2}} \mathsf{mmse}_\pi(\tau^2) &= \E_{\beta_0,z}\big[(\E_{\beta_0,z} \left[(\E_{\beta_0,z}[\beta_0|y] - \beta_0)^2|y\right])^2\big]
     \leq \E_{\beta_0,z}\big[\E_{\beta_0,z} \left[(\E_{\beta_0,z}[\beta_0|y] - \beta_0)^4|y\right]\big]\\
    &\leq 8 \E_{\beta_0,z}\left[\E_{\beta_0,z}[\beta_0|y]^4 + \beta_0^4\right] \leq 16 \E_{\beta_0,\bz}[\beta_0^4] = 16 s_4,
\end{align*}
where $y$ denotes $\beta_0 + \tau z$.
Thus, $\frac{\mathrm{d}}{\mathrm{d} \tau^2} \mathsf{mmse}_\pi(\tau^2) = -  \frac1{\tau^4} \frac{\mathrm{d}}{\mathrm{d} \tau^{-2}} \mathsf{mmse}_\pi(\tau^2) \leq \frac{16 s_4}{\tau^4}$.
Thus, for sufficiently large $\tau$, the derivative of the right-hand side of \eqref{potential-derivative} is strictly negative, so that for sufficiently large $\sigma$ there can be at most one solution to \eqref{tau-stat-is-stationary} in the region $[\sigma^2/\delta,\infty)$. 
But all solutions $\tau_{\mathsf{reg,stat}}^2$ must satisfy $\tau_{\mathsf{reg,stat}}^2 \geq \sigma^2/\delta$, whence the minimizer of \eqref{tau-stat-def} is unique for sufficiently large $\sigma$.
Then, by Proposition \ref{prop-bayes-risk-asymptotic} and Eq. \eqref{ridge-asymptotic-risk}, for sufficiently large $\sigma$ we have $\Delta(\pi,\delta,\sigma) \leq \mathsf{mmse}_{\mathsf{N}(0,s_2)}(\tau_{\mathsf{ridge}}^2) - \mathsf{mmse}_\pi(\tau_{\mathsf{reg,stat}}^2)$. 
Combining \eqref{mmse-tau-stat-in-zeta-expansion} and \eqref{ridge-in-zeta-expansion}, we get \eqref{gap-low-snr-third-moment}, as desired.    
 \end{proof}

\section{Proofs for Section \ref{sec-examples}: examples}
\label{app:ProofExamples}

\subsection{Proof of Proposition \ref{claim-a2-for-strong-convex-penalty}}

This follows from inequality \eqref{prox-width-bound} proved in Appendix \ref{app-proximal-operator-identities}, which gives
\begin{align}
     \frac1{\tau}\E_{\bbeta_0,\bz}\left[\left\langle \bz, \mathsf{prox}\left[\lambda \rho_p\right]\left(\bbeta_0 + \tau\bz\right)\right\rangle\right] &\leq \frac1{\lambda \gamma + 1},\label{strong-convexity-bound}
\end{align}
because $\lambda \rho_p$ has strong convexity parameter $\lambda \gamma$.
The right-hand side of \eqref{strong-convexity-bound} does not depend upon $\tau$ or $p$.
Thus, choosing $1/(\bar \lambda \gamma + 1) < \delta$ yields the proposition.

\subsection{Proof of Proposition \ref{claim-a2-for-convex-constraints}}
claim
  Observe $\mathsf{prox}[\lambda \rho_p](\bx) = \Pi_{C_p}(\bx)$ for all $\lambda$, where $\Pi_{C_p}$ denotes projection onto the set $C_p$.
    Further, observe that $\E_{\bbeta_0,\bz}[\langle \bz, \bbeta_0\rangle] = 0$.
    Thus, we must show 
    \begin{equation}\label{delta-bounded-width-projection-form}
    \limsup_{p\rightarrow \infty}\sup_{\tau \in T} \frac1{\tau}\E_{\bbeta_0,\bz}\left[\left\langle \bz, \Pi_{C_p}\left(\bbeta_0 + \tau\bz\right) -\bbeta_0\right\rangle\right] < \delta.
    \end{equation}
    First, we argue conditionally on $\bbeta_0$, which for now we treat as fixed.
    To simplify notation, we translate our problem---both $\bbeta_0$ and $C_p$---by $-\bbeta_0$, so that we may without loss of generality consider $\bbeta_0 = \bzero$. 
    In the translated problem, denote  $\bb = \Pi_{C_p}(\bzero)$.
    Then
    \begin{align}\label{inner-product-z-with-proj}
        \left\langle \tau \bz , \Pi_{C_p}\left(\tau \bz \right) \right\rangle 
    &= \underbrace{\left\langle \tau \bz - \bb, \Pi_{C_p}\left(\tau \bz \right) - \bb\right\rangle}_{(*)} +  \underbrace{\left\langle \bb  , \Pi_{C_p}\left(\tau \bz\right) - \bb \right\rangle + \langle  \tau\bz , \bb \rangle}_{(**)}.
    \end{align}
    For $t \in [0,1]$, we have $t \bb + (1-t) \Pi_{C_p}\left(\tau \bz\right) \in C_p$. 
    Thus,
    $$
    \frac{\mathrm{d}}{\mathrm{d} t}\left\|\tau \bz - \left(t \bb + (1-t) \Pi_{C_p}\left(\tau \bz\right)\right)\right\|^2\Bigg|_{t = 0} \geq 0.
    $$
    Some rearrangement gives 
    \begin{align}\label{inner-product-z-with-proj-1}
        (*) &\geq \|\Pi_{C_p}(\tau \bz) - \bb\|^2 \geq \|\Pi_{C_p}(\tau \bz)\|^2 - 2 \|\bb\| \|\tau \bz\|.
    \end{align}
    Cauchy-Schwartz gives 
    \begin{align}\label{inner-product-z-with-proj-2}
    (**) &\geq  -\|\bb\|\Big(\|\Pi_{C_p}(\tau \bz) - \bb\| + \|\tau\bz\|\Big) \geq -2\|\bb\|\|\tau \bz\|,
    \end{align}
    where in the second inequality we have used that projections onto convex sets are 1-Lipschitz.
    Also, if $\|\Pi_{C_p}( \tau \bz)\| > \eps$, then $\Pi_{C_p}( \tau \bz) \in T_{C_p \cap B^c(\bzero,\eps)}(\bzero)$. 
    Thus,
    \begin{align}\label{T-closer-than-C}
    \|\tau \bz -\Pi_{C_p}(\tau \bz) \| \geq \|\tau \bz - \Pi_{T_{C_p \cap B^c(\bzero,\eps)}}(\tau \bz)\| \quad \text{if}\quad \|\Pi_{C_p}(\tau \bz)\| > \eps.
    \end{align}
    Thus, if $\|\Pi_{C_p}(\tau \bz)\| > \eps$,
    \begin{align*}
    \|\Pi_{T_{C_p \cap B^c(\bzero,\eps)}}(\tau \bz)\|^2 &= \|\tau \bz\|^2 - \|\tau \bz - \Pi_{T_{C_p \cap B^c(\bzero,\eps)}}(\tilde\bz)\|^2\\
    &\geq \|\tau \bz\|^2 - \|\tau \bz - \Pi_{C_p}(\tau \bz)\|^2 \\
    &= 2 \langle \tau \bz, \Pi_{C_p}(\tau \bz) \rangle - \|\Pi_{C_p}(\tau \bz)\|^2\\
    &\geq \langle \tau \bz, \Pi_{C_p}(\tau \bz)\rangle - 4\|\bb\| \|\tau \bz\|,
    \end{align*}
    where in the second line, we have used \eqref{T-closer-than-C}, and in the last line, we have used \eqref{inner-product-z-with-proj}, \eqref{inner-product-z-with-proj-1}, and \eqref{inner-product-z-with-proj-2}.
    We conclude that 
    \begin{align*}
    \langle \tau \bz, \Pi_{C_p}(\tau \bz) \rangle &\leq \left(\|\Pi_{T_{C_p \cap B^c(\bzero,\eps)}}(\tau \bz)\|^2 + 4 \|\bb\| \|\tau \bz\|\right)\ones_{\|\Pi_{C_p}(\tau \bz)\| > \eps} + \|\tau \bz\| \eps \ones_{\|\Pi_{C_p}(\tau \bz)\| \leq \eps}\\
    &\leq \|\Pi_{T_{C_p \cap B^c(\bzero,\eps)}}(\tau \bz)\|^2 + (4 \|\bb\|+ \eps) \|\tau \bz\|.
    \end{align*}
    Substituting the value of $\bb$, undoing the translation, and averaging over $\bbeta_0$ and $\bz$, we get
    \begin{align}
        \E_{\bbeta_0,\bz} \left[\left\langle \tau \bz, \Pi_{C_p}\left(\bbeta_0+\tau \bz\right) - \bbeta_0 \right \rangle \right] &\leq \E_{\bbeta_0,\bz}\left[\|\Pi_{T_{C_p \cap B^c(\bbeta_0,\eps)}}\left(\tau \bz\right)\|^2\right] + \E_{\bbeta_0,\bz}\left[\left(4 d\left(\bbeta_0,C_p\right)+ \eps\right)\left\|\tau \bz\right\|\right]\\
    &= \tau^2 \E_{\bbeta_0}\left[w\left(T_{C_p \cap B^c(\bbeta_0,\eps)}\right)\right] + \tau\E_{\bbeta_0,\bz}\left[\left(4 d\left(\bbeta_0,C_p\right)+ \eps\right)\left\| \bz\right\|\right].
    \end{align}
    By independence of $\bbeta_0$ and $\bz$, and \eqref{beta0-eventually-in-Cp}, we get 
    \begin{equation}
    \limsup_{p \rightarrow \infty} \E_{\bbeta_0,\bz}\left[\left(4 d\left(\bbeta_0,C_p\right)+ \eps\right)\left\|\tau \bz\right\|\right] \leq  \eps \tau \E_{\bbeta_0,\bz}\left[\|\bz\|\right] \leq \eps\tau.
    \end{equation}
    Fix compact $[\tau_{\mathsf{min}},\tau_{\mathsf{max}}] \subset
    (0,\infty)$.
\begin{equation}
    \limsup_{p\rightarrow \infty}\sup_{\tau \in T} \frac1{\tau}\E_{\bbeta_0,\bz}\left[\left\langle \bz, \Pi_{C_p}\left(\bbeta_0 + \tau\bz\right) -\bbeta_0\right\rangle\right] 
\le    \limsup_{p\rightarrow \infty}\E_{\bbeta_0}\left[w\left(T_{C_p \cap B^c(\bbeta_0,\eps)}\right)\right]+\frac{\eps}{\tau_{\mathsf{min}}} =\overline{\delta}(\eps)+\frac{\eps}{\tau_{\mathsf{min}}} ,
    \end{equation}
where we defined $\overline{\delta}(\eps) :=   \limsup_{p\rightarrow \infty}\E_{\bbeta_0}\left[w\left(T_{C_p \cap B^c(\bbeta_0,\eps)}\right)\right]$. 
The claim  \eqref{delta-bounded-width-projection-form} follows from taking the limit $\eps\to 0$ and using 
Eq.~\eqref{asymptotic-width-bounded-by-delta}.

\subsection{Proof of Proposition \ref{claim-a2-for-separable-penalties}}

    Applying the change of scaling identity for proximal operators (see Appendix \ref{app-proximal-operator-identities}, Eq.~\eqref{prox-change-of-scaling}), we get
\begin{equation}
    \mathsf{prox}[\lambda \rho_p](\bbeta_0 + \tau \bz)_j = \mathsf{prox}[\lambda \rho]\left(\sqrt{p}(\beta_{0j} + \tau z_j)\right)/\sqrt{p},
\end{equation}
so that 
\begin{align}
    \frac1{\tau}\E_{\bbeta_0,\bz}[\langle \bz, \mathsf{prox}[\lambda \rho_p](\bbeta_0 + \tau \bz)\rangle] &= \frac1{ \tau}  \E_{\beta_{0},z}  [z \mathsf{prox}[\lambda \rho](\beta_0 + \tau z) ]\nonumber.
\end{align}
Having removed the dependence on $p$, the left-hand side of \eqref{delta-bounded-width} becomes 
\begin{align*}
    \sup_{\lambda > \bar \lambda,\tau \in T} \frac1{\tau} \E_{\beta_0,z}[ z \mathsf{prox}[\lambda \rho](\beta_0 + \tau z)].
\end{align*}
It is easy to check using \eqref{sequence-cvx-estimator} that for any fixed $\beta_0,z,\tau$, we have $\lim_{\lambda \rightarrow \infty} \mathsf{prox}[\lambda \rho](\beta_0 + \tau z) = \Pi_{C}(\beta_0 + \tau z)$. 
Further, if $m \in C$, we have by the 1-Lipschitz property of the proximal operator \eqref{prox-is-lipschitz} that $|\mathsf{prox}[\lambda \rho](\beta_0 + \tau z)| < |m| + |\beta_0| + \tau |z|$.
By dominated convergence, we have
\begin{align}
\lim_{\lambda \rightarrow \infty}   \frac1{ \tau} \E_{\beta_0,z}[ z \mathsf{prox}[\lambda \rho](\beta_0 + \tau z)] &=   \frac1{ \tau} \E_{\beta_0,z}[ z \Pi_C(\beta_0 + \tau z)]. \end{align}
By Gaussian integration by parts (Appendix \ref{app-proximal-operator-identities}, Eq.~\eqref{prox-gaussian-IBP}),
\begin{align}
  \frac1{ \tau} \E_{\beta_0,z}[ z \Pi_C(\beta_0 + \tau z)] = \E_{\beta_0,z}[\ones_{\beta_0 + \tau z \in C}]= \P_{\beta_0,z}(\beta_0 + \tau z \in C).
\end{align}

First assume the $\delta$-bounded width assumption is satisfied.
Observe $\sup_{\lambda \geq \bar \lambda} \frac1{ \tau} \E_{\beta_0,z}[ z \mathsf{prox}[\lambda \rho](\beta_0 + \tau z)] \geq \P_{\beta_0,z}(\beta_0 + \tau z \in C)$,
whence for any compact $T \subset (0,\infty)$, we have $\delta > \sup_{\tau \in T} \P_{\beta_0,z}(\beta_0 + \tau z \in C)$.
Because $\tau \mapsto \P_{\beta_0,z}(\beta_0 + \tau z \in C)$ is continuous and converges to 0 as $\tau \rightarrow \infty$, we have $\sup_{\tau > \eps} \P_{\beta_0,z}(\beta_0 + \tau z \in C) = \sup_{\tau \in [\eps,M]} \P_{\beta_0,z}(\beta_0 + \tau z \in C)$ for some finite $M$.
We conclude $\sup_{\tau > \eps}\P_{\beta_0,z}(\beta_0 + \tau z \in C) < \delta$.

Conversely, assume $\sup_{\tau > \eps}\P_{\beta_0,z}(\beta_0 + \tau z \in C) < \delta$.
Because $\mathsf{prox}[\lambda \rho]$ is 1-Lipschitz, we have that $\tau \rightarrow \frac1{ \tau} \E_{\beta_0,z}[ z \mathsf{prox}[\lambda \rho](\beta_0 + \tau z)]$ is $L$-Lipschitz on compact $T = [\tau_{\mathsf{min}},\tau_{\mathsf{max}}]\subset(0,\infty)$ for some sufficiently large $L$ depending on $T$ (a complete argument of this fact occurs in a more general setting in the proof of Lemma \ref{lem-uniform-modulus-of-continuity} in Appendix \ref{app-regularity-lemmas}).
Thus, in order to pick $\bar \lambda$ such that $\sup_{\lambda \geq \bar \lambda, \tau \in T}\frac1{ \tau} \E_{\beta_0,z}[ z \mathsf{prox}[\lambda \rho](\beta_0 + \tau z)] < \delta$, we can choose a $\frac{\delta - \sup_{\tau} \P(\beta_0 + \tau z \in C)}{4L}$-cover of $T$ and a $\bar \lambda$ sufficiently large such that $ \frac1{ \tau} \E_{\beta_0,z}[ z \mathsf{prox}[\lambda \rho](\beta_0 + \tau z)] < \frac{\delta + \sup_{\tau} \P(\beta_0 + \tau z \in C)}{2}$ for all $\lambda > \bar \lambda$ and all $\tau$ in the cover.

\subsection{Proof of Proposition \ref{claim-a2-for-owl-penalty}}

    Fix any compact interval $T = [\tau_{\mathsf{min}},\tau_{\mathsf{max}}] \subset (0,\infty)$. 
    First, we describe how to choose a $\bar \lambda$ for which \eqref{delta-bounded-width} holds, and then we will prove that our choice works.
    Pick $\eps$ with
    \begin{equation}\label{choice-of-epsilon-a2-for-owl-1}
        \frac{\delta^2\tau_{\mathsf{min}}^2}{4(s_2(\pi)+ \tau_{\mathsf{max}}^2)} > \eps > 0,
    \end{equation}
    such that, for any $A\subseteq \reals$,
    \begin{equation}\label{choice-of-epsilon-a2-for-owl-2}
        \P(A) \leq \eps \quad\Longrightarrow\quad \E_{\beta_0}[\beta_0^2\ones_{A}] < \delta^2\tau_{\mathsf{min}}^2/32 \quad\text{and}\quad
 \E_z[ \tau_{\mathsf{max}}^2z^2\ones_{A}] < \delta^2\tau_{\mathsf{min}}^2/32\, .
    \end{equation} 
    Pick $t$ such that 
    \begin{equation}\label{choice-of-t-a2-for-owl}
    \P_{\beta_0}(|\beta_0| > t) < \eps \quad \text{and} \quad \P_{z}(|\tau z| > t) < \eps.
    \end{equation}
    Finally pick $\xi > 0$ such that $j \leq (1-\eps)p$ implies $\kappa_j^{(p)} > \xi$. We claim that \eqref{delta-bounded-width} is satisfied for $\bar \lambda = 2t/\xi$. 
    
    First we recall that the proximal operator for the OWL penalty satisfies (e.g.\ see Lemma 3.1 of \cite{Su2016SLOPEMinimax})
    \begin{equation}\label{prox-of-owl-l2-bound}
    \|\mathsf{prox}[\lambda \rho_p](\bx)\| \leq \|(|\bx| - \lambda \bkappa^{(p)}/\sqrt{p})_+\|,
    \end{equation}
    where $|\bx|$ denotes the coordinate-wise absolute value and $(\cdot)_+$ denotes the coordinate-wise positive part.
    By Cauchy-Schwartz, for any $\tau,\lambda$,
    \begin{align}
        \frac1{\tau}\E_{\bbeta_0,\bz}[\langle \bz, \mathsf{prox}[\lambda \rho_p](\bbeta_0 + \tau \bz ) \rangle] & \leq \frac1{\tau} \sqrt{\E_{\bz}[\| \bz \|^2]\E_{\bbeta_0,\bz}[\|\mathsf{prox}[ \rho_p ](\bbeta_0 + \tau \bz)\|^2]}\nonumber\\
        &\leq \frac1{\tau} \sqrt{\E_{\bbeta_0,\bz}[\|(|\bbeta_0 + \tau \bz| - \lambda \bkappa^{(p)}/\sqrt{p})_+\|^2]},\label{owl-guass-width-upper-bound-1}
    \end{align}
    where the last inequality holds by \eqref{prox-of-owl-l2-bound}.
    For $\lambda > \bar \lambda$,
    \begin{align*}
        \left\|\left(|\bbeta_0 + \tau \bz| - \frac{\lambda \bkappa^{(p)}}{\sqrt{p}}\right)_+\right\|^2 &= \sum_{j=1}^{\lfloor (1-\eps)p\rfloor} \left(|\beta_{0j} + \tau z_j| -\frac{\lambda \kappa_j^{(p)}}{\sqrt{p}}\right)_+^2 + \sum_{j=\lfloor (1-\eps)p\rfloor + 1}^p \left(|\beta_{0j} + \tau z_j| - \frac{\lambda \kappa_j^{(p)}}{\sqrt{p}}\right)_+^2\\
        &\leq \sum_{j=1}^{\lfloor (1-\eps)p\rfloor} \left(\beta_{0j} + \tau z_j\right)^2\ones_{|\beta_{0j} + \tau z_j| > \frac{2t}{\sqrt{p}}} + \sum_{j=\lfloor (1-\eps)p\rfloor + 1}^p \left(\beta_{0j} + \tau z_j\right)^2\\
        &\leq 2\sum_{j=1}^{\lfloor (1-\eps)p\rfloor}\left(\beta_{0j}^2 + \tau^2 z_j^2 \right)(\ones_{|\beta_{0j}| > \frac{t}{\sqrt{p}}} + \ones_{|\tau z_j| > \frac{t}{\sqrt{p}}}) + \sum_{j=\lfloor (1-\eps)p\rfloor + 1}^p \left(\beta_{0j} + \tau z_j\right)^2,
    \end{align*}
    where in the first inequality, we have used that $\lambda \kappa_j^{(p)}/\sqrt{p} \geq 2t/\sqrt{p}$ because $\lambda > 2t/\xi$ and $\kappa_j^{(p)} > \xi$ for $j \leq (1-\eps)p$ and that for any $x,y \in \reals$ we have $(|x| - y)_+^2 \leq x^2\ones_{|x|>y}$; and in the second inequality, we have used that $\left(\beta_{0j} + \tau z_j\right)^2 \leq 2\beta_{0j}^2 + 2\tau^2 z_j^2$ and $\ones_{|\beta_{0j} + \tau z_j| > \frac{2t}{\sqrt p}} \leq \ones_{|\beta_{0j}| > \frac{t}{\sqrt p}} + \ones_{|\tau z_j| > \frac{t}{\sqrt p}}$.
    Taking expectations, inequality \eqref{owl-guass-width-upper-bound-1} becomes
    \begin{align*}
        \frac1{\tau}\E_{\bbeta_0,\bz}[\langle  \bz, \mathsf{prox}[\lambda \rho_p](\bbeta_0 + \tau \bz ) \rangle]  &\leq \frac1{\tau} \sqrt{ 2\E_{\beta_0,z}\left[\left(\beta_{0}^2 + \tau^2 z^2 \right)(\ones_{|\beta_{0}| > t} + \ones_{|\tau z| > t})
        \right] + \eps (\E_{\beta_0}[\beta_0^2] + \tau^2)  }\\
        &\leq \frac1\tau \sqrt{\delta^2 \tau_{\mathsf{min}}^2/4 + \delta^2\tau_{\mathsf{min}}^2/4}\\
        &\leq \delta/\sqrt{2} < \delta,
    \end{align*}
    where in the second inequality we have bounded the first term under the square-root by \eqref{choice-of-epsilon-a2-for-owl-2} and the second term under the square-root by \eqref{choice-of-epsilon-a2-for-owl-1}, and in the third inequality, we have used that $\tau_{\mathsf{min}} \leq \tau$.
    This completes the proof.

\section{Proximal operator identities}\label{app-proximal-operator-identities}

We collect here various identities and properties of proximal operators, defined in \eqref{sequence-cvx-estimator}. Many arguments are included because they are not well-known, others for the reader's convenience. Throughout this section, $\rho:\reals^p \rightarrow \reals \cup\{\infty\}$ is an lsc, proper, convex function which is $\gamma$-strongly convex for $\gamma \geq 0$ (if $\rho$ is not strongly convex, we take $\gamma = 0$).

\begin{itemize}
    \item We have the following sub-gradient identity, which follows by the KKT conditions applied to \eqref{sequence-cvx-estimator}.
    \begin{equation}\label{prox-to-rho-subgradient}
        \by - \mathsf{prox}[\rho](\by) \in \partial \rho(\mathsf{prox}[\rho](\by)).
    \end{equation}
    
    \item We have the following fixed point identity, which follows from \eqref{prox-to-rho-subgradient}.
    \begin{equation}\label{prox-fixed-points-are-minimizers}
        \by = \mathsf{prox}[\rho](\by) \Longleftrightarrow \text{$\by$ minimizes $\rho$}.
    \end{equation}
    
    \item $\mathsf{prox}[\rho]$ is firmly non-expansive \cite[p.~131]{Parikh2013ProximalAlgorithms}. That is,
    \begin{equation}\label{prox-firm-non-expansive}
        \langle \by - \by' , \mathsf{prox}[\rho](\by) - \mathsf{prox}[\rho](\by') \rangle \geq (1+\gamma)\|\mathsf{prox}[\rho](\by) - \mathsf{prox}[\rho](\by')\|^2.
    \end{equation}
    
    \item  $\mathsf{prox}[\rho]$ is $(1+\gamma)^{-1}$-Lipschitz \cite{Parikh2013ProximalAlgorithms}. That is, for $\by,\by' \in \reals^p$,
        \begin{equation}\label{prox-is-lipschitz}
            \|\mathsf{prox}[\rho](\by) - \mathsf{prox}[\rho](\by')\| \leq (1+\gamma)^{-1}\|\by - \by'\|.
        \end{equation}
        This follows by applying Cauchy-Schwartz to the left-hand side of \eqref{prox-firm-non-expansive} and rearrangement.
        
        \item $\mathsf{prox}[\lambda\rho]$ satisfies the following continuity property in regularization parameter $\lambda$. For $\lambda > 0$, $\lambda' \geq 0$, we have
    \begin{equation}\label{prox-continuous-in-lambda}
                \|\mathsf{prox}[\lambda \rho](\by) - \mathsf{prox}[\lambda' \rho](\by)\| \leq \|\by - \mathsf{prox}[\lambda \rho](\by)\|\left|\frac{\lambda '}{\lambda} - 1\right|,
    \end{equation}
    as we now argue.
    For simplicity, denote $\ba = \mathsf{prox}[\lambda \rho](\by)$. 
    By \eqref{prox-to-rho-subgradient}, we have $\by - \ba \in \lambda \partial \rho(\ba)$. 
    Scaling by $\frac{\lambda'}{\lambda}$, we have $\frac{\lambda'}{\lambda}(\by - \ba) \in \lambda' \partial \rho(\ba)$.  Thus, 
    \begin{equation}\label{prox-subgradient-change-with-lambda}
        \left(\frac{\lambda '}{\lambda} - 1\right)(\by - \ba) \in \partial \left(\frac12\|\by - \bx\|^2 + \lambda ' \rho(\bx)\right)\Big|_{\bx = \ba}.
    \end{equation}
    Denote $\ba' = \mathsf{prox}[\lambda'\rho](\by)$.
    We have
    \begin{align*}
    \frac12\|\by - \ba\|^2 + \lambda'\rho(\ba) &\geq  \frac12\|\by - \ba'\|^2 + \lambda'\rho(\ba') + \frac12 \|\ba - \ba'\|^2\\
    &\geq \frac12\|\by - \ba\|^2 + \lambda'\rho(\ba) + \left\langle \left(\frac{\lambda'}{\lambda} - 1\right)(\by - \ba) , \ba' - \ba \right\rangle +\|\ba' - \ba\|^2,
    \end{align*}
    where in both inequalities we have used the strong convexity of $\bx \mapsto \frac12 \|\by - \bx\|^2 + \lambda' \rho(\bx)$, in the first inequality we have used that this function has sub-gradient $\bzero$ at $\ba'$ by optimality, and in the second inequality we have used \eqref{prox-subgradient-change-with-lambda}.
    By Cauchy-Schwartz, rearrangement, and substitution for the values of $\ba$ and $\ba'$, we get \eqref{prox-continuous-in-lambda}.

        \item $\mathsf{prox}[\rho]$ is almost everywhere differentiable. This follows because $\mathsf{prox}[\rho]$ is Lipschitz \cite{Evans2015MeasureFunctions}. 
        Whenever we write the divergence $\mathrm{div}\,\mathsf{prox}[\rho]$ and the Jacobian $\mathrm{D}\,\mathsf{prox}[\rho]$, they are understood to be defined almost everywhere.
        For all $\by$ for which the left-hand sides are defined, we have by \eqref{prox-is-lipschitz},
        \begin{gather}
                \mathrm{div}\,\mathsf{prox}[\rho](\by) \leq \frac{p}{1+\gamma},\label{prox-divergence-bound}\\
        \left\|\mathrm{D}\,\mathsf{prox}[\rho](\by)\right\|_{\mathsf{op}} \leq \frac{1}{1+\gamma},\label{prox-jacobian-bound}
        \end{gather}
        and by \eqref{prox-firm-non-expansive}, 
        \begin{equation}\label{prox-jacobian-non-negative-definite}
                \mathrm{D}\,\mathsf{prox}[\rho](\by) \succeq \bzero.
        \end{equation}
        By \eqref{prox-jacobian-non-negative-definite}, we have
        \begin{equation}
                \mathrm{div}\,\mathsf{prox}[\rho](\by) \geq 0.\label{prox-divergence-non-negative}
        \end{equation}
        
        \item We may apply Stein's Lemma (i.e.\ Gaussian integration by parts) to proximal operators. That is, for any $\tau \geq 0$,
        \begin{equation}\label{prox-gaussian-IBP}
             \E_{\bz}\left[\langle \bz, \mathsf{prox}[\rho](\bbeta_0 + \tau \bz)\rangle\right] = \frac{\tau}{p}\E_{\bz}[\mathrm{div}\,\mathsf{prox}[\rho](\bbeta_0 + \tau \bz)],
        \end{equation}
        where $\bz \sim \mathsf{N}(\bzero,\bI_p/p)$.
    To see this, observe that real-valued function $z_i \mapsto \mathsf{prox}[\rho](\bbeta_0 + \tau\bz)_i$, holding the other $z_j$'s fixed, is Lipschitz continuous. Thus, it is the indefinite integral of its almost-everywhere derivative. Applying Stein's lemma \cite{Stein1981EstimationDistribution}, averaging over the other $z_j$'s, and summing over $i$ yields \eqref{prox-gaussian-IBP}.
    
    \item By \eqref{prox-jacobian-bound} and \eqref{prox-gaussian-IBP}, we have for any $\tau \geq 0$ and $\bb \in \reals^p$,
        \begin{equation}\label{prox-width-bound}
                \E_{\bz}\left[\langle \bz , \mathsf{prox}[\rho](\bb + \tau\bz)\rangle\right] \leq \frac{\tau}{1+\gamma},
        \end{equation}
        where $\bz \sim \mathsf{N}(\bzero,\bI_p/p)$.

        \item Proximal operators obey the following identity under the change of scaling of $\rho$ \cite[p. 130]{Parikh2013ProximalAlgorithms}. Let $\tilde \rho(\bx) = a \rho(b\bx)$. Then
        \begin{equation}\label{prox-change-of-scaling}
            \mathsf{prox}[\tilde \rho](\by) = \mathsf{prox}\left[ab^2\rho\right](b\by)/b.
    \end{equation}
    
    \item Proximal operators shift by a constant under the following perturbation. For $\bc \in \reals^p$ fixed, let $\tilde \rho(\bx) = -  \langle \bc, \bx \rangle + \rho(\bx - \bc)  $.  We have,
    \begin{align}
        \mathsf{prox}[\tilde \rho](\by) &= \arg\min_{\bx} \left\{\frac12\|\by - \bx\|^2 - \langle \bc, \bx \rangle + \rho(\bx - \bc)\right\}\nonumber\\
        &= \arg\min_{\bx}\left\{\frac12 \|\by - (\bx - \bc)\|^2 + \rho(\bx - \bc) \right\}\nonumber\\
        &= \mathsf{prox}[\rho](\by) + \bc.\label{prox-shift}
    \end{align}
    
    \item Proximal operators of separable functions are separable. In particular, if $\rho(\bx) = \sum_{j=1}^p f(x_j)$ for some lsc, proper, convex $f:\reals \rightarrow \reals \cup \{\infty\}$, then for all $j$,
    \begin{equation}\label{prox-of-separable}
        \mathsf{prox}[\rho](\by)_j = \mathsf{prox}[f](y_j).
    \end{equation}

    \item The oracle penalty correspons to a decrease in both noise level and regularization.
    Precisely,
    \begin{gather}\label{oracle-prox-noise-reduction-form}
        \mathsf{prox}[\lambda\rho^{(\gamma)}]\left(\bbeta_0 + \tau \bz\right) = \mathsf{prox}\left[\frac{\lambda}{\lambda\gamma + 1}\rho\right]\left(\bbeta_0 + \frac{\tau}{\lambda \gamma + 1}\bz\right).
    \end{gather}
    Indeed, for any $\by$
    \begin{align*}
        \mathsf{prox}[\lambda \rho^{(\gamma)}](\by) &= \arg\min_{\bx} \left\{\frac12 \|\by - \bx\|^2 + \lambda\rho(\bx) + \frac{\lambda\gamma}{2}\|\bx - \bbeta_0\|^2\right\}\nonumber\\
        &= \arg\min_{\bx} \left\{\frac12 \left\|\frac{\lambda \gamma \bbeta_0 + \by}{\lambda \gamma + 1} - \bx\right\|^2 + \frac{\lambda}{\lambda \gamma +1}\rho(\bx) \right\}
        = \mathsf{prox}\left[\frac{\lambda}{\lambda \gamma + 1} \rho\right]\left(\frac{\lambda \gamma \bbeta_0 + \by}{\lambda \gamma  +1}\right).
    \end{align*}
    Eq.~\eqref{oracle-prox-noise-reduction-form} corresponds to $\by = \bbeta_0 + \tau \bz$.

\end{itemize}

\section{Useful tools}\label{app-useful-tools}

We omit proof for the first five lemmas, which are easy to verify. Lemmas \ref{lem-pseudo-lipschitz-closed-under-inner-product}, \ref{lem-pseudo-lipschitz-closed-under-fixing-arguments}, and \ref{lem-pseudo-lipschitz-closed-under-random-element} appear as Lemmas 20, 21, and 22 of \cite{Berthier2017StateFunctions}.
The remaining lemmas in this section are well-known, and we provide citations for each.

\begin{lemma}\label{lem-limsup-p-identities}
    The probabilistic limit supremum satisfies the following. 
    For any real valued random variables $X_p,Y_p$ such that, for each $p$, $X_p$ and $Y_p$ are defined on the same probability space,
    \begin{gather}
    \limsup_{p\rightarrow \infty}^{\mathrm{p}} X_p+Y_p \leq \left(\limsup_{p \rightarrow \infty}^{\mathrm{p}} X_p\right)+\left(\limsup_{p \rightarrow \infty}^{\mathrm{p}} Y_p\right).\label{limsup-p-sums}
    \end{gather}
    If $X_p \geq 0$ and $\mathrm{p}-\limsup_{p \rightarrow \infty} X_p < \infty$, then $X_p = O_{\mathrm{p}}(1)$.
    If $X_p, Y_p \geq 0$, then
    \begin{gather}
        \limsup_{p\rightarrow \infty}^{\mathrm{p}} X_pY_p \leq \left(\limsup_{p \rightarrow \infty}^{\mathrm{p}} X_p\right)\left(\limsup_{p \rightarrow \infty}^{\mathrm{p}} Y_p\right).\label{limsup-p-products}
    \end{gather}   
\end{lemma}

\begin{lemma}[Lemma 20 in \cite{Berthier2017StateFunctions}]\label{lem-pseudo-lipschitz-closed-under-inner-product}
    Consider two sequences $f: (\reals^p)^{\ell_1} \rightarrow \reals^p$ and $g: (\reals^p)^{\ell_2} \rightarrow \reals^p$, $p \geq 1$, of uniformly pseudo-Lipschitz functions of order $k$ such that $\|f(\bzero)\|$, $\|g(\bzero)\|$ are bounded over $p$. The sequence of functions $\varphi:(\reals^p)^{\ell_1} \times (\reals^p)^{\ell_2} 
\rightarrow \reals$, $p \geq 1$ defined by $\varphi(\bx,\by) = \langle f(\bx),g(\by) \rangle$ is uniformly pseudo-Lipschitz of order $2k$.
\end{lemma}

\begin{lemma}[Lemma 21 in \cite{Berthier2017StateFunctions}]\label{lem-pseudo-lipschitz-closed-under-fixing-arguments}
    Let $t$, $s$, and $k$ be any three positive integers. Consider a sequence (in $p$) of $\bx_1,\ldots,\bx_s \in \reals^p$ such that $\|\bx_j\| \leq c_j$ for some constants $c_j$ independent of $p$, for $j = 1,\ldots,s$, and a sequence (in $p$) of uniformly pseudo-Lipschitz functions $\varphi_p: (\reals^p)^{t+s} \rightarrow \reals$. Then the sequence of functions $\phi_p(\cdot) := \varphi_p(\cdot,\bx_1,\ldots,\bx_s)$ is also uniformly pseudo-Lipschitz of order $k$.
\end{lemma}

\begin{lemma}[Lemma 22 in \cite{Berthier2017StateFunctions}]\label{lem-pseudo-lipschitz-closed-under-random-element}
    Let $t$ be any positive integer. Consider a sequence (in $p$) of uniformly pseudo-Lipschitz functions $\varphi_p : (\reals^p)^t \rightarrow \reals$ of order $k$.
    The sequence of functions $\phi_p: (\reals^p)^t \rightarrow \reals$ such that $\phi_p(\bx_1,\ldots,\bx_t) = \E_{\bz}\left[\varphi_p(\bx_1,\ldots,\bx_{t-1},\bx_t + \tau \bz)\right]$ 
    where $\bz \sim \mathsf{N}(\bzero,\bI_p/p)$ and $\tau \geq 0$ does not depend on $p$, is also uniformly pseudo-Lipschitz of order $k$.
\end{lemma}

\begin{lemma}\label{lem-pseudo-lipschitz-closed-under-composition}
    If $f:\reals^p \rightarrow \reals^n$ is a sequence (in $p$) of uniformly pseudo-Lipschitz functions of order $k$ and $g: \reals^n \rightarrow \reals^m$ is a sequence (in $p$) of uniformly pseudo-Lipschitz functions of order $l$ such that $\|g(\bzero)\|$ is bounded, then $g \circ f: \reals^p \rightarrow \reals^m$ is a sequence of uniformly pseudo-Lipschitz functions of order $kl$.
\end{lemma}

\begin{lemma}[Stein's lemma \cite{Stein1981EstimationDistribution}]\label{lem-steins-lemma} 
    Let $f:\reals^p \rightarrow \reals$ be ``almost differentiable'' in the sense that there exists measurable $\nabla f: \reals^p \rightarrow \reals^p$ such that, for all $\bdelta \in \reals^p$,
    $$
    f(\bx + \bdelta) - f(\bx) = \int_0^1 \langle \bdelta ,\nabla f(\bx + t \bdelta)\rangle \mathrm{d}t,
    $$
    for almost every $\bx \in \reals^p$. In particular, this is satisfied if $f$ is pseudo-Lipschitz.
    If $(\bz_1,\bz_2) \sim \mathsf{N}(\bzero,\bT \otimes \bI_p/p)$ for some $\bT \in S_+^2$,  then
    \begin{equation}
        \E_{\bz_1,\bz_2}\left[\langle \bz_1, f(\bz_2)\rangle\right] = \frac{T_{12}}p \E[ \mathrm{div}\, f(\bz_2)].
    \end{equation}
\end{lemma}

\begin{lemma}[Tweedie's Formula. Eq.~(2.8) in \cite{Efron2011TweediesBias}]\label{lem-tweedies-formula}  
    Fix $\tau > 0$. Fix $\pi $ any probability measure on $\reals$. Let $y = \beta_0 + \tau z$ where $\beta_0 \sim \pi$, $z \sim \mathsf{N}(0,1)$ independent. Let $p_Y$ denote the density of $y$ with respect to Lebesgue measure. Then
    \begin{equation}\label{tweedies-formula}
        \E_{\beta_0,z}\left[\beta_0 | y\right] = y + \tau^2 \frac{\mathrm{d}}{\mathrm{d} y} \log p_Y(y).
    \end{equation}
    (See \cite{Efron2011TweediesBias} for a discussion of earlier references for this remarkable formula.)
\end{lemma}

\begin{lemma}[Gaussian concentration of Lipschitz functions. Theorem 5.6 in \cite{Boucheron2016ConcentrationIndependence}]\label{lem-gaussian-lipschitz-concentration}
    If $\bz \sim \mathsf{N}(\bzero,\bI_p/p)$ and $f:\reals^p \rightarrow \reals$ is an $L$-Lipschitz function, then for all $t > 0$,
    \begin{equation}\label{gaussian-lipschitz-concentration}
        \P_{\bz}\left(\left|f(\bz) - \E_{\bz}[f(\bz)]\right| \geq t \right) \leq e^{-\frac{p}{2L^2} t^2}.
    \end{equation}
\end{lemma}

\begin{lemma}[Gaussian Poincar\'e inequality. Theorem 3.20 in \cite{Boucheron2016ConcentrationIndependence}] \label{gaussian-poincare-inequality} 
Let $\bz \sim \mathsf{N}(\bzero,\bI_p/p)$ and $\varphi:\reals^p \rightarrow \reals$ be continuous and weakly differentiable. Then, for some universal constant $c$,
\begin{equation}
    \Var[\varphi(\bz)] \leq \frac{c}{p}\E_{\bz}\left[\|\nabla \varphi (\bz) \|^2\right].
\end{equation}
\end{lemma}

\end{appendix}

\end{document}